\documentclass[11pt]{amsbook}
\usepackage{amsfonts, amsthm, amssymb, amsmath, stmaryrd, color, enumerate, relsize}
\usepackage[pdftex]{graphicx}
\usepackage{mathrsfs,array}
\usepackage{xy}
\usepackage{hyperref}
\usepackage{tikz}
\input xy
\xyoption{all}

\DeclareMathOperator{\Spa}{Spa}

\DeclareMathOperator{\Hom}{Hom}

\DeclareMathOperator{\Spec}{Spec}

\renewcommand*{\tilde}{\widetilde}

\newcommand{\dotimes}{{\otimes^{\mathbb L}}}

\numberwithin{equation}{section}
\newtheorem{theorem}{Theorem}
\numberwithin{theorem}{section}
\newtheorem{lemma}[theorem]{Lemma}
\newtheorem{corollary}[theorem]{Corollary}
\newtheorem{proposition}[theorem]{Proposition}

\newtheorem{problem}[theorem]{Problem}
\newtheorem{assumption}[theorem]{Assumption}

\theoremstyle{definition}
\newtheorem{remark}[theorem]{Remark}
\newtheorem{exercise}[theorem]{Exercise}
\newtheorem{warning}[theorem]{Warning}
\newtheorem{definition}[theorem]{Definition}
\newtheorem{question}[theorem]{Question}
\newtheorem{example}[theorem]{Example}

\newtheorem{construction}[theorem]{Construction}

\date{\today}

\title{Condensed Mathematics and Complex Geometry}

\author{Dustin Clausen, Peter Scholze}

\begin{document}

\maketitle

\tableofcontents

\chapter*{Condensed Mathematics and Complex Geometry}

\section*{Preface}

This is a slightly revised version of lectures notes for a course in Summer 2022 joint between Bonn and Copenhagen, intended as a stable citable version.\\

Many thanks go to Ko Aoki and Mohan Ramachandran for many comments, and in particular for many pointers to the literature.\\

Dustin Clausen and Peter Scholze\hfill{May 2026}

\newpage

\section{Lecture I: Introduction}

Over the last few years, we have been working on an alternative foundation for the development of a (very general) ``analytic geometry'', based on a new foundation for combining algebra and topology, in the framework of ``condensed mathematics''. These ideas have been laid out in the lectures \cite{Condensed} and \cite{Analytic}, given in 2019--2020.

The goal of this course will be to make these developments more concrete by concentrating on the case of complex-analytic geometry, and instead of trying to develop new kinds of geometry, we will here merely try to redevelop the classical theory, but from a different point of view. More precisely, we aim to reprove some important theorems for compact complex manifolds, including:

\begin{enumerate}
\item Finiteness of coherent cohomology;
\item Serre Duality;
\item In the algebraic case, GAGA;
\item (Grothendieck--)Hirzebruch--Riemann--Roch.
\end{enumerate}

The proofs will be very different from previous proofs. Notably, at least for the first three results, our proofs will be of a local nature; even better, we will formulate versions of these results that are true even in the non-compact (sometimes also called non-proper) case. Moreover, we would like to say that our proofs are proofs by ``formal nonsense'' and in particular analysis-free. Let us try to be somewhat more precise by what we have in mind here.

The most basic notion in complex analysis is that of a holomorphic function, in one variable. The claim that this is a well-behaved notion is contained in the following theorem.

\begin{theorem}\label{thm:holomorphicdefinition} There is a (necessarily unique) sheaf $\mathcal O$ on the topological space $\mathbb C$, such that for any open disc
\[
D=D(x,r)=\{z\in \mathbb C\mid |z-x|<r\}\subset \mathbb C,
\]
one has
\[
\mathcal O(D)=\{\sum_{n=0}^\infty a_n (T-x)^n\mid a_n\in \mathbb C, \forall r'<r, a_n r'^n\to 0\},
\]
with obvious transition maps $\mathcal O(D)\to \mathcal O(D')$ for $D'\subset D$.

Moreover, for any such disc $D$, the sheaf cohomology groups $H^i(D,\mathcal O)=0$ for $i>0$.
\end{theorem}

Here, the ``obvious'' transition maps are those that are compatible with evaluation at points. More precisely, if $z\in D(x,r)$ and $f\in \mathcal O(D)$, i.e. $f=\sum_{n=0}^\infty a_n (T-x)^n$, then $f(z)=\sum_{n=0}^\infty a_n (z-x)^n$ is absolutely convergent in $\mathbb C$, so $\mathcal O(D)$ maps (injectively) to the functions from $D$ to $\mathbb C$. Then $\mathcal O(D)\to \mathcal O(D')$ is the unique map making
\[\xymatrix{
\mathcal O(D)\ar[r]\ar[d] & \mathrm{Map}(D,\mathbb C)\ar[d]\\
\mathcal O(D')\ar[r] & \mathrm{Map}(D',\mathbb C)
}\]
commute. (One could also write down this map very explicitly in terms of power series.)

In other words, there should be a notion of ``holomorphic function'' on $D$, a certain kind of function $D\to \mathbb C$, and checking holomorphicity can be done locally. On the other hand, any such, a priori locally defined, function, should have on $D$ a global convergent Taylor series expansion.

The usual proof of this theorem is based on treating holomorphic functions as special smooth functions, and most critically on Cauchy's integral formula, based on integrating along paths. This makes it possible to compute the Taylor coefficients $a_n$ in terms of integrals around the circle $|z-x|=r'$ for any $r'<r$. For the final cohomology statement, one uses the exact sequence
\[
0\to \mathcal O\to C^\infty\xrightarrow{\overline{\partial}} C^\infty\to 0,
\]
and uses the good properties of the $\overline{\partial}$-differential operator, as well as bump functions to see that the higher cohomology of the sheaf of $C^\infty$-functions vanishes. This is analysis!

Already Weierstra\ss\ was trying to avoid the use of integrals in the proof of this theorem; he gave an argument replacing the integrals by approximating finite sums, but in essence it is still the same argument. In a letter to Schwarz dated October 3, 1875, Weierstra\ss\ writes (cf.~\cite[Chapter 8, Section 2.2.3]{Remmert}):

``Je mehr ich \"uber die Principien der Functionentheorie nachdenke -- und ich thue dies unabl\"assig --, um so fester wird meine \"Uberzeugung, dass diese auf dem Fundamente algebraischer Wahrheiten aufgebaut werden muss, und dass es deshalb nicht der richtige Weg ist, wenn umgekehrt zur Begr\"undung einfacher und fundamentaler algebraischer S\"atze das `Transzendente', um mich kurz auszudr\"ucken, in Anspruch genommen wird -- so bestechend auch auf den ersten Ablick z.B.~die Betrachtungen sein m\"ogen, durch welche Riemann so viele der wichtigsten Eigenschaften algebraischer Functionen entdeckt hat. (Dass dem Forscher, so lange er sucht, jeder Weg gestattet sein muss, versteht sich von selbst; es handelt sich nur um die systematische Begr\"undung.)''

In this course, we will try to develop complex geometry in this spirit.

Let us also remark that the above strategy of proof is extremely different from the strategy of proof used to prove the corresponding result in $p$-adic geometry (known as ``Tate acyclicity''). Part of our goal is to develop foundations for analytic geometry that treat archimedean and non-archimedean geometry on equal grounds; and we will proceed by making archimedean geometry more similar to non-archimedean geometry. In very rough outline, we will make $\mathbb C[T]$ into an ``analytic ring'', which will have an associated ``analytic spectrum'' $\mathrm{AnSpec} \mathbb C[T]$. Moreover, any $D$ as above will correspond to an open subset of $\mathrm{AnSpec} \mathbb C[T]$, and the theorem basically amounts to a computation of the structure sheaf on $\mathrm{AnSpec} \mathbb C[T]$ (that will exist, and have acyclicity properties, by general theory). This computation will in fact be rather simple. So in this approach, one discovers holomorphic functions as functions on certain open subsets of the analytic spectrum $\mathrm{AnSpec} \mathbb C[T]$.

Let us step back for a moment and take a big picture look at the development of different kinds of geometry. Basically, things started with the analysis of functions of several real variables,
\[
f: \mathbb R^n\to \mathbb R^m,
\]
restricting to continuous, differentiable, or smooth functions. The focus here was always on the infinitesimal variation of such functions; and functions are notably something that is determined by its values (and vice-versa, any collection of values determines a function). Assuming $n$ and $m$ even and identifying $\mathbb R^2=\mathbb C$, one can also restrict to complex-differentiable functions. The first big surprise in the theory is that this already implies smoothness, and in fact that the function is given by a convergent power series expansion. Here, a small shift happens: Any function is uniquely given by a formula (its power series expansion). Generally, complex analysis becomes much more rigid, and in fact is often extremely close to algebraic geometry, which deals just with polynomial functions (with complex coefficients). Pictorially:

\[
\mathrm{Map}(\mathbb C,\mathbb C)\supset C^0(\mathbb C,\mathbb C)\supset C^1(\mathbb C,\mathbb C)\supset \ldots\supset C^\infty(\mathbb C,\mathbb C)\supset \mathcal O(\mathbb C)\supset \mathbb C[T]\supset \mathbb Z[T]
\]

This goes from arbitrary maps (``set theory'') to continuous maps (``topology'') to differentiable functions (``analysis'') to smooth functions (``differential geometry'') to holomorphic functions (``complex analysis'') to complex polynomials (``algebraic geometry'') to integer polynomials (``arithmetic geometry''). Historically (at least on a very superficial level), the development started with some class of differentiable functions, implicit in the first works on analysis, and then proceeded rightwards by putting more and more stringent conditions. Proceeding along those lines, however, made functions much more into formulas -- any element of $\mathbb Z[T]$ is just something like $3T^4 + 7T - 4$, and in fact can be evaluated not just in $\mathbb C$, but in any field, for example $\mathbb F_p$.

Traditionally, objects on the right are considered more complicated than objects towards the left. We will take the opposite point of view -- what could be simpler than polynomials with integer coefficients? (Indeed, these are the simplest kinds of commutative rings.) Thus, we will start from $\mathbb Z[T]$ and try to get to $\mathbb C[T]$ and $\mathcal O(D)$ in a ``formal nonsense'' manner -- i.e., we want to find ways to freely adjoint the ``correct'' convergent power series.

In other words, we need a way to ``formally'' adjoin convergent sums. This is formalized in the notion of an ``analytic ring'' developed in \cite{Condensed} and \cite{Analytic}. The idea is the following. An analytic ring should be a topological ring $A$ together with a certain (topological) $A$-module
\[
\mathrm{Summable}(\mathbb N,A)\subset A^{\mathbb N}
\]
of ``summable'' sequences, such that, for any null-sequence $a_0,a_1,\ldots,a_n,\ldots\in A$, encoded in a continuous map
\[
f: \mathbb N\cup \{\infty\}\to A
\]
(with $f(\infty)=0$), and any summable sequence $(x_0,x_1,\ldots,x_n,\ldots)\in \mathrm{Summable}(\mathbb N,A)$, one can form a new element
\[
\sum_{n=0}^\infty x_n a_n\in A.
\]

In other words, any continuous map $f: \mathbb N\cup\{\infty\}\to A$ with $f(\infty)=0$ should extend uniquely to a continuous map
\[
\tilde{f}: \mathrm{Summable}(\mathbb N,A)\to A.
\]
It turns out that it's more convenient to generalize convergent sequences (for example, in general continuity cannot be checked in terms of convergent sequences). Namely, $\mathbb N\cup\infty$ is a special kind of profinite set $S$ (recall that profinite sets are also known as totally disconnected compact Hausdorff spaces). One should define for any profinite set $S$ a topological $A$-module $\mathcal M_A(S)$ of ``$A$-valued measures on $S$'', with a map $S\to \mathcal M_A(S)$ sending any $s\in S$ to the ``Dirac measure'' $\delta_s\in \mathcal M_A(S)$, such that for any continuous map $f: S\to A$, there is a unique extension
\[
\tilde{f}: \mathcal M_A(S)\to A: \mu\mapsto \int f \mu.
\]

In other words, an analytic ring $A$ is roughly the datum of a topological ring $A$, together with a topological $A$-module $\mathcal M_A(S)$ for any profinite set $S$, satisfying some compatibilities. One would like the basic example to be $A=\mathbb R$, with $\mathcal M_{\mathbb R}(S)$ given by the space of Radon measures, i.e.
\[
\mathcal M_{\mathbb R}(S)=\mathrm{Hom}_{\mathbb R}^{\mathrm{cont}}(\mathrm{Cont}(S,\mathbb R),\mathbb R).
\]
Unfortunately, as we will discuss below, this will not define an example of an analytic ring, due to some subtle behaviour of the real numbers. But similar constructions carried over $\mathbb Z$ or $\mathbb Z_p$ (for a prime $p$) do give very interesting examples of analytic rings, discussed in \cite{Condensed}.

We note that any analytic ring $A$ will in fact give rise to a notion of ``complete'' topological $A$-modules $M$. Namely, $M$ is complete if for any $f: S\to M$ as before, there is a unique extension
\[
\tilde{f}: \mathcal M_A(S)\to M: \mu\mapsto \int f\mu.
\]

\begin{example} Let $A=\mathbb R$ with $\mathcal M_{\mathbb R}(S)$ given by the space of Radon measures as above. Then for any complete locally convex topological $\mathbb R$-vector space $V$, any map $f: S\to V$ extends uniquely to a map
\[
\tilde{f}: \mathcal M_{\mathbb R}(S)\to V: \mu\mapsto \int f\mu.
\]
\end{example}

On the other hand, we would like to use the very categorical methods of algebraic geometry in this setting of analytic rings. In particular, we would like to work in the setting of abelian categories.

\begin{problem} Topological abelian groups do not form an abelian category!
\end{problem}

The issue is very basic. Indeed, consider the map
\[
\mathbb R_{\mathrm{disc}}\to \mathbb R_{\mathrm{nat}}
\]
from $\mathbb R$ with the discrete topology to $\mathbb R$ with its natural topology. This has trivial kernel, trivial cokernel, but is not an isomorphism. This happens in fact whenever one equips an abelian group $M$ with two distinct topologies, one finer than the other, so it is an extremely pervasive problem.

On the other hand, in the presentation above, we started to look at topological abelian groups $M$ only in terms of the corresponding functor
\[
\underline{M}: \mathrm{ProFin}^{\mathrm{op}}\to \mathrm{Ab}: S\mapsto \mathrm{Cont}(S,M)
\]
taking any profinite set $S$ to the continuous maps from $S$ to $M$. This will be an example of a condensed abelian group, which is just a functor
\[
\mathrm{ProFin}^{\mathrm{op}}\to \mathrm{Ab}
\]
satisfying some simple axioms.

Condensed abelian groups form an abelian category! For example, there is a short exact sequence
\[
0\to \underline{\mathbb R_{\mathrm{disc}}}\to \underline{\mathbb R_{\mathrm{nat}}}\to Q\to 0
\]
where $Q$ is the condensed abelian group given by the functor
\[
Q(S) = \mathrm{Cont}(S,\mathbb R)/\mathrm{LocConst}(S,\mathbb R).
\]
Note that the ``underlying abelian group'' $Q(\ast)=0$, but still $Q\neq 0$ as $Q(S)\neq 0$ for infinite profinite sets -- for example, taking $S=\mathbb N\cup \{\infty\}$, not every convergent sequence in $\mathbb R$ is eventually constant.

Just like condensed abelian groups, one can define condensed sets, as functors $\mathrm{ProFin}^{\mathrm{op}}\to \mathrm{Set}$, satisfying some simple conditions (recalled in the next lecture). There is a functor
\[
\mathrm{Top}\to \mathrm{Cond}: X\mapsto (\underline{X}: S\mapsto \mathrm{Cont}(S,X)).
\]
Restricted to compactly generated spaces, this functor is fully faithful, so in practice condensed sets are an enlargement of topological spaces. But there are many more objects in $\mathrm{Cond}$; notably, nonseparated quotient objects.

\begin{definition}[slightly preliminary]\label{def:analyticringprelim} An analytic ring is a condensed ring $A$ together with a condensed $A$-module $\mathcal M_A(S)$ for any profinite set $S$, with a map $S\to \mathcal M_A(S)$ of ``Dirac measures'', subject to some compatibility. In particular, the full subcategory
\[
\{\mathcal M_A\text-\mathrm{complete\ condensed}\ A\text-\mathrm{modules}\}\subset \mathrm{Cond}(A),
\]
of those condensed $A$-modules $M$ such that for all $S$ and all $f: S\to M$, there is a unique extension $\tilde{f}: \mathcal M_A(S)\to M$, should be an abelian subcategory of $\mathrm{Cond}(A)$, stable under extensions.
\end{definition}

We can now explain why $\mathbb R$ (or $\underline{\mathbb R}$ -- we will soon become lax with our use of underlines, treating everything as condensed from the start...) equipped with the spaces of signed Radon measures is not an analytic ring.

\begin{example}\label{entropy} For any $p\in [1,\infty]$, let $\ell^p(\mathbb N)$ be the Banach space of $p$-summable sequences $(x_0,x_1,\ldots)$ of real numbers, equipped with the $\ell^p$-norm
\[
||(x_n)_n||_{\ell^p} = (\sum_n |x_n|^p)^{1/p}.
\]
Then $\ell^1(\mathbb N)\subset \ell^2(\mathbb N)$ with dense image; in condensed $\underline{\mathbb R}$-vector spaces, we can pass to the quotient:
\[
0\to \underline{\ell^1(\mathbb N)}\to \underline{\ell^2(\mathbb N)}\to Q\to 0.
\]
Now consider the (non-linear) map
\[
g: \ell^1(\mathbb N)\to \ell^2(\mathbb N): (x_n)_n\mapsto (x_n\log|x_n|)_n.
\]
It turns out that the composite
\[
\overline{g}: \underline{\ell^1(\mathbb N)}\xrightarrow{g} \underline{\ell^2(\mathbb N)}\to Q
\]
is a (nonzero) map of condensed $\underline{\mathbb R}$-vector spaces. On the other hand, restricted to the basis vectors, the map is identically $0$. In other words, the sequence
\[
\mathbb N\cup\{\infty\}\to Q
\]
that is constant $0$ admits two distinct extensions to the space of summable sequences $\ell^1(\mathbb N)$ (which is basically $\mathcal M_{\mathbb R}(\mathbb N\cup\{\infty\})$): The map that is constant $0$, and $\overline{g}$. That is, $Q$ is not $\mathcal M_{\mathbb R}$-complete, contrary to the requirement put in Definition~\ref{def:analyticringprelim}.
\end{example}

\newpage

\section{Lecture II: Liquid vector spaces}

Recall from the previous lecture that our first goal is to produce a sheaf $\mathcal{O}$ of $\mathbb{R}$-vector spaces on the topological space $\mathbb{C}$, whose value on the disk $D=D(x,r)$ of radius $r>0$ centered at a point $x\in \mathbb{C}$ is given by
\[
\mathcal O(D)=\{\sum_{n=0}^\infty a_n (T-x)^n\mid a_n\in \mathbb C, \forall r'<r, a_n r'^n\to 0\}.
\]
Moreover, we want to see that $\mathcal{O}$ has no higher cohomology on these disks $D$, so that this description of the sections is ``robust''.

As described in the previous lecture, we will approach this old result by new methods, which are much more algebraic in nature as opposed to analytic.  For this, the first step is to recognize more structure on these vector spaces of sections.  In fact, we need to promote $\mathcal O(D)$ to an object of an \emph{abelian tensor category}, such that this promotion in some sense encodes a structure of topological vector space on $\mathcal O(D)$.  The category being abelian is useful in order to control sheaf cohomology; and the tensor product, though its use is not a priori obvious, is actually also crucial for our approach.

The purpose of this lecture is to introduce the required abelian tensor category, whose objects are known as \emph{$p$-liquid vector spaces}.  

As mentioned, the idea is that we want to encode a topological vector space structure on $\mathcal O(D)$.  However, as indicated in the previous lecture, the formalism of topological vector spaces is not appropriate, because topological spaces don't mix well with algebra: in particular, topological abelian groups do not form an abelian category.

The solution is to find a replacement for the very notion of topological space, one for which it is much more convenient to add on algebraic structures such as that of a vector space.  This replacement is known as \emph{condensed set}.

The building blocks for condensed sets are the following very special kinds of topological spaces.

\begin{definition} Let $\operatorname{Prof}$ denote the category of compact Hausdorff totally disconnected topological spaces.\end{definition}

There are several useful characterizations of this class of topological spaces.  We have that $S\in \operatorname{Prof}$ if and only if $S$ is compact Hausdorff and there is a basis for the topology of $S$ consisting of \emph{clopen} subsets.  Another equivalent condition is that $S$ should be homeomorphic to some inverse limit of finite discrete spaces:
$$S = \varprojlim_{i\in I} S_i.$$
In fact, there is a canonical choice for this inverse limit diagram: we can take $I$ to be the set of disjoint union decompositions of $S$ into (finitely many) nonempty clopen subsets, with partial order given by refinement. Equivalently, an element of $I$ is an isomorphism class of continuous surjections $S\twoheadrightarrow S_i$ with $S_i$ a finite discrete space, and this shows how to build the inverse limit diagram.

This canonical index diagram is cofiltered, so every $S\in\operatorname{Prof}$ is a cofiltered inverse limit of finite sets.  In fact, this yields an equivalence of categories
$$\operatorname{Prof}\simeq \operatorname{Pro}(\operatorname{finite Sets})$$
with the Pro-category of finite sets, giving a ``combinatorial'' as opposed to topological description of our basic building blocks.  For this reason, objects of $\operatorname{Prof}$ are called simply \emph{profinite sets} for short.

More general condensed sets will be glued from objects of $\operatorname{Prof}$ via the formalism of sheaves.  To implement this, we consider the following finitary Grothendieck topology on $\operatorname{Prof}$.

\begin{definition}
A finite collection $(f_i:S_i\rightarrow S)_{i\in I}$ of maps in $\operatorname{Prof}$ with common codomain is called a \emph{covering} if the induced map
$$\bigsqcup_{i\in I} S_i \twoheadrightarrow S$$
is surjective.
\end{definition}

The category $\operatorname{Prof}$ has all pullbacks (in fact, all limits), and it is trivial to check the axioms of a Grothendieck topology.  However, before we use this to define a category of sheaves, we need to get around the technical difficulty that $\operatorname{Prof}$ is not essentially small.  For this we introduce a ``cutoff'' cardinal $\kappa$, which we assume to be infinite.

\begin{definition}
A profinite set $S$ is called \emph{$\kappa$-small} if the set of clopen subsets of $S$ has cardinality $<\kappa$.  The full subcategory of $\kappa$-small profinite sets is denoted
$$\operatorname{Prof}_\kappa\subset\operatorname{Prof}.$$
\end{definition}

We can restrict the Grothendieck topology on $\operatorname{Prof}$ to $\operatorname{Prof}_\kappa$, and then make the following definition.

\begin{definition}
A \emph{$\kappa$-condensed set} is a sheaf of sets on $\operatorname{Prof}_\kappa$.  The category of $\kappa$-condensed sets is denoted by $\operatorname{CondSet}_\kappa$.
\end{definition}

Explicitly, an $X\in \operatorname{CondSet}_\kappa$ is a functor
$$X: \operatorname{Prof}_\kappa^{\mathrm{op}}\rightarrow Set$$
such that:

\begin{enumerate}
\item If $(S_i)_{i\in I}$ is a finite collection of objects of $\operatorname{Prof}_\kappa^{\mathrm{op}}$, then
$$X(\bigsqcup_{i\in I} S_i)\overset{\sim}{\rightarrow} \prod_{i\in I} X(S_i);$$
\item If $f:T\twoheadrightarrow S$ is a surjective map in $\operatorname{Prof}_\kappa$, then
$$X(S)\overset{f^*}{\hookrightarrow} X(T)$$
is injective with image equal to the set of those $x\in X(T)$ for which $p_1^\ast x=p_2^\ast x$, where $p_1,p_2:T\times_ST\rightarrow T$ are the two projections.
\end{enumerate}

There are two ways of thinking about this definition.  The first is that instead of encoding a topological space structure on $X$ directly, we encode, for every $S\in \operatorname{Prof}$, the data of an abstract set $X(S)$ which we think of as specifying the set of ``continuous maps from $S$ to $X$''.  From this perspective, the above structure and conditions should make intuitive sense:
\begin{enumerate}
\item the contravariant functoriality $X: \operatorname{Prof}_\kappa^{\mathrm{op}}\rightarrow Set$ comes from the idea that a ``continuous map'' $S\rightarrow X$ and a continuous map $T\rightarrow S$ should compose to give a ``continuous map'' $T\rightarrow X$;
\item the first axiom encodes that giving a ``continuous map'' out of a disjoint union should be the same as separately giving a ``continuous map'' out of each constituent piece;
\item the second axiom encodes that if $f:T\twoheadrightarrow S$, then giving a ``continuous map'' out of $S$ should be the same as giving a ``continuous map'' out of $T$ which is constant on the fibers of $f$.
\end{enumerate}
To make this last point more palatable, note that a continuous surjection between compact Hausdorff spaces is a topological quotient map.

The other way of think about this definition, actually equivalent to the first though perhaps psychologically different, comes back to the idea that profinite sets are the building blocks for arbitrary condensed sets.  More formally, the Yoneda embedding gives a fully faithful functor
$$\operatorname{Prof}_\kappa \hookrightarrow \operatorname{CondSet}_\kappa$$
defined by
$$S\mapsto (\underline{S}: T\mapsto \operatorname{Cont}(T,S)),$$
allowing us to view $\kappa$-small profinite sets as special kinds of $\kappa$-condensed sets; and an arbitrary $\kappa$-condensed set is built from $\kappa$-small profinite sets via colimits in $\operatorname{CondSet}_\kappa$.

From this perspective, the role of the sheaf conditions is perhaps a bit more tangible.  If we had used the presheaf category instead of sheaf category in the definition, then $\operatorname{CondSet}_\kappa$ would be \emph{freely} generated under colimits by $\operatorname{Prof}_\kappa$, so every relation between colimits would be essentially diagrammatic in nature. The role of the sheaf conditions is to enforce that certain colimits we already have in $\operatorname{Prof}_\kappa$ should be preserved by the Yoneda embedding, and thus be visible in $\operatorname{CondSet}_\kappa$.  Indeed:
\begin{enumerate}
\item the first sheaf condition says that given finitely many profinite sets $S_i$, we have
$$\bigsqcup \underline{S_i} \overset{\sim}{\rightarrow} \underline{\bigsqcup S_i};$$
\item the second sheaf condition says that that if $T\twoheadrightarrow S$, thereby presenting $S$ as the quotient of $T$ by the equivalence relation $T\times_S T$, then $\underline{S}$ is the quotient of $\underline{T}$ by the equivalence relation $\underline{T\times_S T}$.
\end{enumerate}

Carrying around the cardinal $\kappa$ is sometimes a bit annoying.  If $\kappa < \kappa'$, then there is a pullback functor
$$\pi^\ast: \operatorname{CondSet}_\kappa\rightarrow \operatorname{CondSet}_{\kappa'},$$
characterized by the fact that it preserves colimits and restricts to the obvious inclusion $\operatorname{Prof}_\kappa\rightarrow \operatorname{Prof}_{\kappa'}$ under the Yoneda embeddings.  Now it turns out that, for $\kappa$ and $\kappa'$ restricted to a cofinal class of cardinals (e.g., strong limit cardinals of sufficiently large cofinality), the functor $\pi^\ast$ is fully faithful and preserves limits and internal hom's of any fixed size.  This means that the colimit category
$$\operatorname{CondSet} := \varinjlim_\kappa \operatorname{CondSet}_\kappa$$
of \emph{condensed sets} is just as well-behaved as each individual $\operatorname{CondSet}_\kappa$ in terms of its topos-theoretic exactness properties, because calculations can always be done at some ``finite'' level.  Moreover, $\operatorname{CondSet}$ is generated under colimits by arbitrary profinite sets in the same way $\operatorname{CondSet}_\kappa$ is generated under colimits by $\kappa$-small profinite sets.

Let's try to understand a bit about what condensed sets can look like by carving out a hierarchy of full subcategories of $\operatorname{CondSet}$:

$$\operatorname{qcProj}\subset\operatorname{Prof}\subset\operatorname{qcqs}\subset\operatorname{qs}\subset\operatorname{CondSet}.$$

We already know about $\operatorname{Prof}$: these were our basic building blocks, the profinite sets.  The other three full subcategories are given by some general properties an object in a topos can have.  We review the relevant definitions in the current context.

\begin{definition}  A condensed set $X$ is called \emph{quasicompact} or \emph{qc} for short if for all collections $(X_i\rightarrow X)_{i\in I}$ of maps of condensed sets with target $X$, indexed by an arbitrary set $I$, it holds that if
$$\bigsqcup_i X_i\twoheadrightarrow X,$$
then there exists a finite subset $I_0\subset I$ such that
$$\bigsqcup_{i\in I_0} X_i\twoheadrightarrow X.$$
\end{definition}

Since $\operatorname{CondSet}$ is generated by $\operatorname{Prof}$, every $X\in\operatorname{CondSet}$ admits a surjection from a coproduct of profinite sets, and it follows that in the above definition we could restrict to the case where all $X_i$ are profinite without any change in the resulting notion.  Since the Grothendieck topology on $\operatorname{Prof}$ is finitary, every profinite set is itself quasicompact, and in fact we find that an $X\in \operatorname{CondSet}$ is quasicompact if and only if there exists an $S\in\operatorname{Prof}$ and a surjection
$$S\twoheadrightarrow X.$$

Thus, a condensed set is qc if and only if it is a quotient of a profinite set.  However, this quotient can be quite arbitrary, and in particular non-Hausdorff: an example is $X=\mathbb{R}/\mathbb{R}^\delta$ from the previous lecture.  This is a quotient of $\mathbb{R}/\mathbb{Z}$, which is in turn a quotient of the profinite set $\prod_{\mathbb{N}} \{0,1\}$ via binary expansions.

The next condition rules out such non-Hausdorff behavior.

\begin{definition}  A condensed set $X$ is called \emph{quasiseparated} or \emph{qs} for short if the diagonal map $X\rightarrow X\times X$ is quasicompact, meaning for all $S\in \operatorname{Prof}$ and all pairs of maps $f,g:S\rightarrow X$, the condensed set $S\times_XS$ is quasicompact.
\end{definition}

This is the analog of the Hausdorff condition for condensed sets.  Indeed, a topological space is Hausdorff if and only if the diagonal is a closed inclusion; but on the other hand, the condition in the above definition that $S\times_XS$ be quasicompact is equivalent to the condition that the inclusion $S\times_XS \subset S\times S$ be represented by a closed subset of the topological space $S\times S$, thanks to the following lemma, whose proof we include to give a feel for this kind of elementary argument.

\begin{lemma}
Let $S\in \operatorname{Prof}$. For an arbitrary condensed subset $T\subset S$, we have that $T$ is qc $\Leftrightarrow$ $T$ is represented by a closed subset of $S$ $\Leftrightarrow$ $T\in  \operatorname{Prof}$.
\end{lemma}
\begin{proof}
The nontrivial implication is that $T$ qc $\Rightarrow$ $T$ is represented by a closed subset of $S$.  If $T$ is qc, we can choose an $S'\in\operatorname{Prof}$ and a surjection $S'\twoheadrightarrow T$.  The composition $$f:S'\twoheadrightarrow T\subset S$$
is a map between objects of $\operatorname{Prof}$.  Let $f(S')$ denote the topological space image of this map of compact Hausdorff spaces.  Then $f(S')$ is closed in $S$, so it suffices to show that $f(S')=T$ as subobjects of $S$.  However, $f(S')$ is the quotient of $S'$ by the equivalence relation $S'\times_{f(S')} S'$ in $\operatorname{CondSet}$ by definition of the Grothendieck topology.  On the other hand, $S'\times_{f(S')} S'=S'\times_TS'$ because $T,f(S')\subset S$, and $T$ is the quotient of $S'$ by $S'\times_T S'$ on completely general topos-theoretic grounds.
\end{proof}

Having the notions of qc and qs, we get the notion of qcqs: a condensed set is qcqs if it is both qc and qs.  This is the topos-theoretic analog of compact Hausdorff; but in fact, in this case it is much more than an analogy.

\begin{proposition}
The full subcategory $\operatorname{qcqs}\subset\operatorname{CondSet}$ is equivalent to the category $\operatorname{CHaus}$ of compact Hausdorff spaces.
\end{proposition}

The idea behind this equivalence is that both qcqs condensed sets and compact Hausdorff spaces can be described in terms of profinite sets in exactly the same way: they are quotients of profinite sets by closed (equivalently, profinite) equivalence relations.  The functor $\operatorname{CHaus}\overset{\sim}{\rightarrow} \operatorname{qcqs}$ implementing this equivalence is also just the usual way
$$X\mapsto (\underline{X}: S\mapsto \operatorname{Cont}(S,X))$$
of assigning a condensed set to a topological space, and the functor backwards can be described as sending a qcqs $X$ to its underlying set $X(\ast)$ equipped with the topology where the closed subsets are those of the form $Y(\ast)$ for some qc condensed subset $Y\subset X$.

We can also describe the more general quasiseparated condensed sets.

\begin{proposition}
A condensed set is $qs$ if and only if it is a filtered union of $qcqs$ subobjects.  The full subcategory $\operatorname{qs}\subset\operatorname{CondSet}$ is equivalent to the full subcategory category of $\operatorname{Ind}(\operatorname{CHaus})$ consisting of those ind-compact Hausdorff spaces which can be represented by an ind-system with injective transition maps.
\end{proposition}

The fully faithful functor $X \mapsto \underline{X}$ from compactly generated weak Hausdorff topological spaces to condensed sets lands inside $\operatorname{qs}$:  it sends $X$ to the union of all its compact Hausdorff subsets, viewing these as qcqs condensed sets.  However, there are many more quasi-separated condensed sets than just these: given a CGWH topological space $X$, we can take any family of compact Hausdorff subsets of $X$ which is closed under finite union, and take the union in condensed sets over just those instead.  If that family of compact subsets still determines the topology of $X$ (for example, if $X$ is metrizable and we take the family of countable compact subsets), this gives a new qs condensed set which is, however, indistinguishable from $X$ from a topological perspective.  In fact, the notion of qs condensed set exactly matches the notion of \emph{compactological space} developed by Waelbrock, \cite{waelbroeck2006topological}.

Via this description in terms of ind-compact Hausdorff spaces, the qs condensed sets are fairly accessible in standard topological terms.  On the other hand, disjoint unions of profinite sets are qs, and a subobject of a qs condensed set is qs.  It follows that any condensed set is a quotient of a qs condensed set by a qs equivalence relation.  This gives some idea of the general nature of condensed sets.

It remains to discuss the most special class $\operatorname{qcProj}$ of condensed sets, the \emph{quasicompact projective} objects.

\begin{definition}
A condensed set $X$ is called \emph{projective} if for all $Y\in\operatorname{CondSet}$, every surjection $Y\twoheadrightarrow X$ admits a section.
\end{definition}

Equivalently, $X$ is projective if for all surjections $A\twoheadrightarrow B$ of condensed sets, the induced map $A(X)\twoheadrightarrow B(X)$ is a surjective map of sets.

A condensed set is quasicompact projective if it is both quasicompact and projective.  Such an $X$, being quasicompact, admits a surjection from a profinite set; being projective, this surjection splits, so $X$ is a retract of a profinite set, and hence itself is a profinite set.  Thus we see that qc projective condensed sets are the same thing as projective profinite sets: profinite sets $S$ such that for every surjection $S'\twoheadrightarrow S$ of profinite sets splits.

These are the same as the \emph{extremally disconnected} profinite sets studied by Gleason in \cite{Gleason}.  The main result on them is that there are ``enough'' extremally disconnected profinite sets:

\begin{lemma}
For every $S\in \operatorname{Prof}$, there is a projective $S'\in \operatorname{Prof}$ with a surjection $S'\twoheadrightarrow S$.
\end{lemma}
\begin{proof}
We take $S'=\beta S^\delta$, the Stone-Cech compactification of the underlying set of $S$.  The universal property of Stone-Cech compactification gives both the required surjection and the proof that $S'$ is projective.
\end{proof}

This means that the projective profinite sets also generate the category of condensed sets.  In some ways they are a more convenient set of generators, because of their projectivity.  For example, a map $X\rightarrow Y$ of condensed sets is surjective (resp.\ injective, resp.\ an isomorphism) if and only if $X(S)\rightarrow Y(S)$ is surjective (resp.\ injective, resp.\ an isomorphism) as a map of sets, for all projective profinite $S$.  For ``surjective'', this fails on more general profinite sets.

Coming back to our cut-off cardinals $\kappa$, it is convenient to require $\kappa$ to be a strong limit cardinal ($S<\kappa \Rightarrow 2^S < \kappa$), because then the above lemma also holds inside $\operatorname{Prof}_\kappa$.  Moreover, then $\kappa$-smallness for a profinite set as defined above is equivalent to underlying set being $\kappa$-small.  In fact, for all practical purposes we know, nothing would be lost by restricting the theory to $\operatorname{CondSet}_\kappa$ for $\kappa$ equal to the first uncountable strong limit cardinal, the supremum of finite iterations of the power set function applied to the natural numbers.

These projective profinite sets, though they never show up in practice as explicit objects of interest, are very convenient abstract entities to have around when discussing the homological algebra of \emph{condensed abelian groups}, which constitute our next topic.

\begin{definition}
A \emph{condensed abelian group} is an abelian group object of $\operatorname{CondSet}$, or equivalently, a sheaf of abelian groups on $\operatorname{Prof}_\kappa$ for some $\kappa$.  The category of condensed abelian groups is denoted $\operatorname{CondAb}$.
\end{definition}

General topos-theoretic considerations show that $\operatorname{CondAb}$ is an abelian category with all limits and colimits, and that filtered colimits in $\operatorname{CondAb}$ are exact.  However, due to the existence of enough qc projective profinite sets, we get even better exactness properties.  If $S\in\operatorname{qcProj}$, then the free condensed abelian group on $S$
$$\mathbb{Z}[S]$$
is a \emph{compact projective} object of the abelian category $\operatorname{CondAb}$.  This means that
$$\operatorname{Hom}_{\operatorname{CondAb}}(\mathbb{Z}[S],-):\operatorname{CondAb}\rightarrow \operatorname{Ab}$$
commutes with both filtered colimits and cokernels.  It follows that it in fact commutes with all colimits.  It also obviously commutes with all limits, so it commutes with both limits and colimits.

Moreover, these $\mathbb{Z}[S]$ for $S\in\operatorname{qcProj}$ provide enough compact projectives to generate $\operatorname{CondAb}$.  This can be expressed in two equivalent ways:
\begin{enumerate}
\item a map $f:A\rightarrow B$ of condensed abelian groups is an isomorphism if and only if $\operatorname{Hom}(\mathbb{Z}[S],A)\rightarrow \operatorname{Hom}(\mathbb{Z}[S],B)$ is an isomorphism for all $S\in\operatorname{qcProj}$;
\item every condensed abelian group is isomorphic to a cokernel of a map between direct sums of condensed abelian groups of the form $\mathbb{Z}[S]$.
\end{enumerate}

From this we deduce that $\operatorname{CondAb}$ inherits all the exactness properties of $\operatorname{Ab}$.  Indeed, any question about interchanging limits and colimits in $\operatorname{CondAb}$ can be tested on Hom's out of the $\mathbb{Z}[S]$, and then it reduces to the same claim in $\operatorname{Ab}$. In particular, infinite products are exact, a feature not shared by abelian group objects in most toposes.

Now we can finally discuss our main objects of study: the \emph{liquid vector spaces}.  As mentioned in the previous lecture, the idea is that we want to restrict to a certain class of condensed abelian groups, which we consider to be ``complete" real vector spaces in some sense.  And this completeness is expressed in terms of the existence and uniqueness of certain infinite sums, or more precisely, integrals of continuous functions along measures.

We start with the set of (signed) Radon measures $\mathcal{M}(S)$ on a profinite set $S$.  In the previous lecture this was defined as the continuous dual of $C(S;\mathbb{R})$, the Banach space of continuous functions on $S$.  But for our purposes it will be convenient to introduce an alternate more direct perspective on $\mathcal{M}(S)$.

Classically, a measure assigns to every Borel subset of $S$ a real number, such that certain axioms are satisfied, in particular countable additivity and bounded variation.  But for profinite $S$, the description simplifies considerably: it suffices to specify the measure $\mu(T)$ only for clopen subsets of $S$, and then only finite additivity and bounded variation need to be imposed.  That is,
$$\mathcal{M}(S) \subset \{\mu:\operatorname{Clopen}(S)\rightarrow\mathbb{R} \mid \mu(\emptyset)=0, \mu(T\sqcup T')=\mu(T)+\mu(T')\}$$
is given by $\mathcal{M}(S) = \bigcup_{C>0} \mathcal{M}(S)_{\leq C}$, where
$$\mathcal{M}(S)_{\leq C} = \{\mu \mid |\mu(T_1)|+ \ldots+ |\mu(T_n)|\leq C\text{ if } S = T_1\sqcup\ldots\sqcup T_n\}.$$

There is another way of organizing this definition, using the presentation $S=\varprojlim_i S_i$ of $S$ as a filtered inverse limit of finite sets, namely
$$\mathcal{M}(S)_{\leq C} = \varprojlim_i \mathcal{M}(S_i)_{\leq C}\subset \varprojlim_i \mathbb{R}^{\oplus S_i},$$
where $\mathcal{M}(S_i)_{\leq C}$ is the closed ball of radius $C$ centered at the origin with respect to the $\ell^1$-norm on $\mathbb{R}^{\oplus S_i}$.  This also describes $\mathcal{M}(S)$ as a qs condensed set: it is the sequential union of the $\mathcal{M}(S)_{\leq C}$, which are compact Hausdorff spaces, being inverse limits of closed balls in finite-dimensional spaces.

The appearance of the $\ell^1$-norm here is necessary to make $\mathcal{M}(S)$ match the dual of continuous functions, but in this last presentation there's nothing to suggest we couldn't use a different norm.  In fact this freedom of changing norms will be crucial for us.

More precisely, for $0<p\leq 1$, we can consider the $\ell^p$-norm on $\mathbb{R}^S$ for a finite set $S$, defined by 
$$\lVert (x_s)_{s\in S}\rVert_p = \sum_{s\in S} \vert x_s\vert^p.$$
This norm still satisfies the triangle inequality (this uses $p\leq 1$).  Note that the closed disk of radius $C$ in the $\ell^p$ norm is smaller than in the $\ell^1$-norm, so we will be shrinking our space of measures.  More drastically, for $p<1$ the closed disks in the $\ell^p$ norm are no longer convex!  This constitutes a more radical departure from standard functional analysis, which makes extensive use of convexity.  However, nonconvex functional analysis has been studied to some extent, with Kalton in particular proving some interesting and nontrivial results which we used as inspiration, see \cite{KaltonConvexityType}.

Anyway, we can consider
$$\mathcal{M}_p(S) = \bigcup_{C>0} \mathcal{M}(S)_{\ell^p\leq C},$$
where
$$\mathcal{M}(S)_{\ell^p\leq C} = \varprojlim_i \mathcal{M}(S_i)_{\ell^p\leq C}\subset \varprojlim_i \mathbb{R}^{\oplus S_i},$$
in perfect analogy to the above.  In other words, instead of measures of bounded variation, we take measures of bounded $p$-variation.  However, we cannot directly use this space of measures for defining $p$-liquid vector spaces either, because of similar phenomena as given in Example \ref{entropy} of the previous lecture.  Instead, we need to fix a $0<p\leq 1$, and then consider
$$\mathcal{M}_{<p}(S) = \bigcup_{q<p} \mathcal{M}_{q}(S).$$
Note that there is a natural inclusion $S\rightarrow \mathcal{M}_{<p}(S)$ of condensed sets, singling out the Dirac measures (unit coordinate vectors in the finite-dimensional approximations in the above definitions).  With this space of measures in place, we make the main definition.

\begin{definition}
Let $V\in\operatorname{CondAb}$ and $0<p\leq 1$.  We say that $V$ is \emph{$p$-liquid} if for every $S\in\operatorname{Prof}$ and every map of condensed sets $f:S\rightarrow V$, there is a unique map of condensed abelian groups $\widetilde{f}:\mathcal{M}_{<p}(S)\rightarrow V$ such that $\widetilde{f}|_S = f$.
\end{definition}

In the next lecture, we will state the main theorem on $p$-liquid condensed abelian groups, which implies that they form a very robust abelian subcategory of condensed abelian groups; and moreover, every $p$-liquid condensed abelian group has a unique and functorial structure of condensed $\mathbb{R}$-module.  But first, we want to explore the meaning of this definition by specializing it to the case of quasi-separated $V$.  The result is the following.

\begin{theorem}\label{thm:qspliquid}
Let $V$ be a qs condensed $\mathbb{R}$-module and $0<p\leq 1$.  Then $V$ is $p$-liquid if and only if for every $q<p$, every quasicompact subobject of $V$  is contained in a quasicompact $q$-convex subobject of $V$.
\end{theorem}

A quasi-compact subobject $K$ of a qs condensed $\mathbb{R}$-module is said to be \emph{$q$-convex} if for all finite sets $S$, the map
$$(\mathbb{R}^{\oplus S})_{\ell^q\leq 1}\times K^S\rightarrow V,$$
$$ ((\lambda_s),(x_s))\mapsto \sum_s \lambda_s x_s,$$
lands inside $K$.  If $q=1$, this is the usual notion of absolute convexity.

Recall that a qc subobject of a qs condensed set is qcqs, hence corresponds to a compact Hausdorff space.  In particular, the above $q$-convexity on $K$ is really a pointwise condition.  In essence, the theorem says that a $p$-liquid qs condensed $\mathbb{R}$-module is one which can be exhausted by $q$-convex compact Hausdorff subobjects, for all $q<p$.

One implication in the above theorem is easy to prove:

\begin{proof}
Suppose $V$ is $p$-liquid, and let $K$ be a quasicompact subobject of $V$.  Choose a surjection $S\twoheadrightarrow K$ with $S\in\operatorname{Prof}$.  By definition of $p$-liquid, the composite $f:S\rightarrow K \subset V$ extends to a map of condensed abelian groups from $\mathcal{M}_{q}(S)$ for all $q<p$.  Note that this extended map is necessarily $\mathbb{R}$-linear, because we can test this on the compact Hausdorff subobjects $\mathcal{M}(S)_{\ell^q\leq C}$ and then argue pointwise using the density of $\mathbb{Q}$ in $\mathbb{R}$.  On the other hand, from the definitions we see that the quasicompact subset
$$\mathcal{M}(S)_{\ell^q\leq 1}\subset \mathcal{M}_{q}(S)$$
is $q$-convex.  It follows that its image under the $\mathbb{R}$-linear map $\mathcal{M}_{q}(S)\rightarrow V$ is also quasicompact and $q$-convex.  Since it clearly contains $K$, this proves one direction of the above theorem.
\end{proof}

The other implication is a bit more tricky, and the argument can be found in the appendix.

Thus we have a convenient criterion for a qs condensed $\mathbb{R}$-module to be $p$-liquid.  We can use this to produce some examples.  First of all, the spaces $\mathcal{M}_{<p}(S)$ themselves clearly satisfy the hypotheses, so they are themselves $p$-liquid.  They are the most basic examples, being (by definition!) the free $p$-liquid vector spaces on the profinite sets $S$.

But there is another class of examples, closely related to classical functional analysis, which also play an important role in our theory.  These are the \emph{Banach spaces}, or more generally $p$-Banach spaces for our fixed $0<p\leq 1$.

\begin{definition}
Let $0<p\leq 1$.  For an abstract $\mathbb{R}$-vector space $V$, a \emph{$p$-norm} on $V$ is a map
$$\lVert \cdot \rVert: V\rightarrow \mathbb{R}_{\geq 0}$$
such that:
\begin{enumerate}
\item $\lVert v\rVert=0 \Leftrightarrow v=0$ for $v\in V$;
\item $\lVert v+w\rVert\leq \lVert v\rVert + \lVert w\rVert$ for $v,w\in V$;
\item $\lVert \lambda v\rVert = \vert \lambda\vert^p \lVert v\rVert$ for $v\in V$, $\lambda\in\mathbb{R}$.
\end{enumerate}
Such a $p$-norm induces a topological vector space structure on $V$, and those topological vector spaces which arise from $p$-norms on $V$ for which $V$ is \emph{complete} are called \emph{$p$-Banach spaces}.
\end{definition}

Note that if $q<p$ and $V$ is $p$-Banach, then it is also $q$-Banach: we can replace the norm by $\lVert\cdot\rVert^{q/p}$ without affecting the topology.  In particular, every Banach space, which is a $1$-Banach space in the above terminology, is also $p$-Banach for all $p$.

By definition the topology of a $p$-Banach space is sequential, hence compactly generated, so we can equivalently encode a $p$-Banach $V$ in terms of its associated condensed $\mathbb{R}$-vector space.  We claim that any $p$-Banach is $p$-liquid.  To prove this, and for other purposes, it's nice to have control over compact subsets of $p$-Banach spaces.  Good information is provided by the following lemma of Grothendieck, \cite{GrothendieckTensor} p.\ 112:

\begin{lemma}\label{compactnullsequence}
Let $0<p\leq 1$, and let $V$ be a $p$-Banach space.  Then:
\begin{enumerate}
\item Every compact subset of $V$ is contained in the closure of the $p$-convex hull of a nullsequence.
\item The closure of the $p$-convex hull of any nullsequence is compact (and $p$-convex).
\end{enumerate}
\end{lemma}

Evidently, this implies that a $p$-Banach space is $p$-liquid, given the criterion provided by the above theorem.  Since arbitrary inverse limits of $p$-liquid spaces are $p$-liquid, we deduce that arbitrary inverse limits of $p$-Banach spaces are $p$-liquid.  This covers the condensed $\mathbb{R}$-modules associated to arbitrary \emph{complete locally $p$-convex} topological vector spaces, showing that the most classical notion of completeness in functional analysis does yield condensed vector spaces which are complete in our sense, i.e.\ $p$-liquid.\\

But besides the desire to connect with classical functional analysis, there is another reason for us to study $p$-Banach spaces, or rather simply ordinary Banach spaces: it turns out they also play a foundational role in our theory.  This is for the following two reasons:

\begin{enumerate}
\item As mentioned, the most fundamental $p$-liquid vector spaces are the measure spaces $\mathcal{M}_{<p}(S)$, and these are far from being Banach.  (They are countable unions of compact subsets, which only happens for finite-dimensional Banach spaces.) However, their \emph{dual}, i.e.\ internal Hom to $\mathbb{R}$, identifies with the Banach space $C(S;\mathbb{R})$ of continuous maps $S\rightarrow \mathbb{R}$, and these duals play an important role.
\item It turns out that Banach spaces have many convenient calculational properties which are not shared by the $\mathcal{M}_{<p}(S)$, so for some arguments it is preferable to build more complicated $p$-liquid spaces from Banach spaces instead of spaces of measures.  We will see some examples in the coming lectures.\\
\end{enumerate}

\textbf{Exercise 1.}  Recall that an object $X$ of a category $\mathcal{C}$ is said to be \emph{compact} if the functor $\operatorname{Hom}(X;-): \mathcal{C}\rightarrow \operatorname{Set}$ commutes with filtered colimits.  For $X\in \operatorname{CondSet}$, show that there are implications
$$X \text{ qcqs }\Rightarrow X\text{ compact }\Rightarrow X\text{ quasicompact}.$$\\

\textbf{Exercise 2.}  Find examples showing that neither implication in the previous exercise is reversible.\\

\textbf{Exercise 3.} Let $\{U_i\}_{i\in I}$ be an arbitrary open cover of a topological space $X$.  Show that $\underline{X}$ is the quotient of $\bigsqcup_i \underline{U_i}$ by the equivalence relation $\bigsqcup_{i,j} \underline{U_i\cap U_j}$ in $\operatorname{CondSet}_\kappa$ for all $\kappa$.\\

\textbf{Exercise 4.} Let $0<p\leq 1$.  Suppose $V$ is a qs condensed $\mathbb{R}$-vector space and $K\subset V$ is a qc $p$-convex subset such that $V=\mathbb{R}\cdot K$.  Define a $p$-norm on $V(\ast)$ such that $K(\ast)$ is the closed unit ball and show that $V(\ast)$ is complete with respect to this $p$-norm, thus giving a new qs condensed $\mathbb{R}$-vector space $V^{Ban_p}$.  Show that there is an injective map of condensed $\mathbb{R}$-vector spaces
$$V^{Ban_p}\rightarrow V$$
which is the identity on $\ast$-valued points, and is natural in $V$ as a condensed $\mathbb{R}$-vector space.\newpage

\section*{Appendix to Lecture II: Quasiseparated liquid vector spaces}

In this appendix, we prove the converse direction in Theorem~\ref{thm:qspliquid}. We will use the following lemma.

\begin{lemma}\label{approximate}
Let $T$ be a compact Hausdorff space, $0<p\leq 1$, and $K\subset V$ a quasicompact $p$-convex subobject of a qs condensed $\mathbb{R}$-module $V$.  Then:
\begin{enumerate}
\item $\bigcap_{\lambda>0}\lambda\cdot K = \{0\}$;
\item for any sequence $f_1,f_2\ldots$ of continuous maps $T\rightarrow K$, if there is a nullsequence $\lambda_1,\lambda_2\ldots$ of positive real numbers such that $f_n(t) - f_m(t) \in \lambda_n\cdot K$ for all $t\in T$ and $n\leq m$, then the $f_n$ converge pointwise to a continuous map
$$f_\infty:T\rightarrow K.$$
\end{enumerate}
\end{lemma}
\begin{proof}
For (1), let $v\in \bigcap_{\lambda>0}\lambda\cdot K$.  Then $\mathbb{R}\cdot v\subset K$.  Hence the closure of $\mathbb{R}\cdot v$ is a compact Hausdorff topological $\mathbb{R}$-vector space.  But the only such is $0$: for example, by Pontryagin duality this is equivalent to the fact that a discrete topological $\mathbb{R}$-vector space must be $0$.

For (2), since $K$ is compact Hausdorff, we can extend the map $f:T\times \mathbb{N}\rightarrow K$ defined by $f(t,n)=f_n(t)$ to the Stone-Cech compactification:
$$\widetilde{f}:\beta(T\times\mathbb{N})\rightarrow K.$$
Recall that $\beta(T\times\mathbb{N})$ is a compact Hausdorff space containing $T\times\mathbb{N}$ as a dense subspace.  We will show that this $\widetilde{f}$ factors through the more refined compactification $T\times (\mathbb{N}\cup\infty)$, giving the claim. 

As surjective maps of compact Hausdorff spaces are quotient maps, it suffices to check this pointwise.  The subsets of the form $T'\times \mathbb{N}_{\geq N}\subset T\times \mathbb{N}$, for $T'\subset T$ a compact neighborhood of $t$ and $N\in\mathbb{N}$, are the intersections with $T\times \mathbb{N}$ of a fundamental system of neighborhoods of $t\times\infty$ in $T\times (\mathbb{N}\cup\infty)$, so it suffices to show that
$$\bigcap_{T',N} \overline{f(T'\times \mathbb{N}_{\geq N})}$$
consists of a single point of $K$ (which will give the value $\widetilde{f}(t,\infty)$).

However, from the hypothesis we have for fixed $N$ that
$$f(T'\times \mathbb{N}_{\geq N})\subset f_N(T') + \lambda_N\cdot K.$$
The latter subset is compact, and taking the intersection over $T'$ we deduce that
$$\bigcap_{T'}\overline{f(T'\times \mathbb{N}_{\geq N})}\subset f_N(t) + \lambda_N\cdot K.$$
If $v$ is any given limit point of the sequence $f_1(t),f_2(t),\ldots$ in $K$, then $v - f_N(t)$ is a limit point of the $f_n(t)- f_N(t)$, hence lies in $\lambda_N\cdot K$.  We deduce that
$$\bigcap_{T'} \overline{f(T'\times \mathbb{N}_{\geq N})}\subset v + \lambda_N\cdot (K + K)$$
for all $N$.  However, $K+K\subset C\cdot K$ for some $C>0$ depending only on $p$ because of the $p$-convexity of $K$.  Thus, intersecting over all $N$ and using part (1), we deduce the claim.  
\end{proof}

Now we turn to the proof of the reverse direction of Theorem~\ref{thm:qspliquid}: if $V$ is a qs condensed $\mathbb{R}$-module which is exhausted by quasicompact $q$-convex subobjects for all $q<p$, then $V$ is $p$-liquid.  The difficulty is that we will need to understand how to build maps out of these spaces of measures $\mathcal{M}_q(S)$.  For this, we will make use of the \emph{discretization} of $\mathcal{M}_q(S)$ described in \cite[Lecture VI]{Analytic}, based on passing from the condensed ring $\mathbb{R}$ to the condensed ring $$\mathbb{Z}((T))_r$$
of Laurent series with integer coefficients which converge absolutely on the closed disk of radius $r$, for some fixed arbitrary $r<1$.  For any $0<\lambda<r$, we get a surjective homomorphism
$$\mathbb{Z}((T))_r\rightarrow \mathbb{R}$$
by plugging in $T=\lambda$.  One very relevant advantage of $\mathbb{Z}((T))_r$ over $\mathbb{R}$ is that each element of $\mathbb{Z}((T))_r$ is \emph{canonically} a limit of elements of the discrete ring $\mathbb{Z}[T,T^{-1}]$, given by truncating the power series.

Quite remarkably, this also passes to spaces of measures.  For a profinite set $S$, we have

$$\mathcal{M}(S;\mathbb{Z}((T))_{r}) = \bigcup_{C>0} \mathcal{M}(S;\mathbb{Z}((T))_{r})_{\leq C},$$
where
$$\mathcal{M}(S;\mathbb{Z}((T))_{r})_{\leq C}\subset \prod_{n\in\mathbb{Z}} \mathbb{Z}[S]\cdot T^n$$
is the condensed --- in fact, profinite --- subset whose $S'$-valued points are specified pointwise, and on a point are given by those
$$\sum_n c_n T^n, c_n\in \mathbb{Z}[S]$$
such that $\sum_n \lVert c_n\rVert_{\ell^1} r^n \leq C$.  Then \cite[Proposition 6.9]{Analytic} says that there is a natural isomorphism of condensed $\mathbb{R}$-modules
$$\mathcal{M}(S;\mathbb{Z}((T))_{r}) \otimes_{\mathbb{Z}((T))_{r}}\mathbb{R} = \mathcal{M}_q(S)$$
again induced by plugging in $\lambda$ for $T$.  Here $0<q<1$ is such that $\lambda^q=r$.

Thus, to map $\mathbb{R}$-linearly out of $\mathcal{M}_q(S)$ is the same as mapping $\mathbb{Z}((T))_{r}$-linearly out of $\mathcal{M}(S;\mathbb{Z}((T))_r)$, and this can be easier because measures in $\mathcal{M}(S;\mathbb{Z}((T))_{r})$ are \emph{canonically} limits of finite sums of Dirac measures with coefficients in $\mathbb{Z}[T,T^{-1}]$!  We will use this in the following proof of the reverse direction.

\begin{proof}
Suppose given a qs condensed $\mathbb{R}$-vector space $V$ such that for all $q<p$, every quasicompact subobject of $V$ is contained in a quasicompact $q$-convex subobject of $V$.  Let $S\in\operatorname{Prof}$ with a map
$$f:S\rightarrow V.$$
We want to extend $f$ uniquely to a map $\widetilde{f}:\mathcal{M}_{<p}(S)\rightarrow V$ of condensed abelian groups.  First, note that uniqueness is automatic, because the sums of Dirac measures with $\mathbb{Q}$-coefficients are dense in each qc subset of $\mathcal{M}_{<p}(S)$.  Also for this reason, it suffices to show that for all $q<p$, the map $f$ extends to an $\mathbb{R}$-linear map
$$\widetilde{f}:\mathcal{M}_q(S)\rightarrow V.$$

Now, the image of $f$ is a qc subset of $V$, hence by assumption on $V$ it lies in a $q$-convex subset.  Actually, it even lands in a $q'$-convex qc subset for some $q'>q$.  Thus it suffices to show that if $f:S\rightarrow V$ lands in a $q'$-convex qc subset $K\subset V$, then it extends to $\mathcal{M}_q(S)$ for all $q<q'$.

Now we use the trick of passing to $\mathcal{M}(S;\mathbb{Z}((T))_{r})$.  By the above discussion, it suffices to produce a $\mathbb{Z}((T))_{r}$-linear map
$$\mathcal{M}(S;\mathbb{Z}((T))_{r})\rightarrow V$$
which restricts to $f$, where the right-hand side is a $\mathbb{Z}((T))_{r}$-module via evaluation at $\lambda=r^{1/q}$.  Now, using density of $\mathbb{Z}[T,T^{-1}]$-linear sums of Dirac measures in every qc subset of $\mathcal{M}(S;\mathbb{Z}((T))_{r})$, the $\mathbb{Z}((T))_{r}$-linearity will be automatic if we just show that there exists a map of condensed \emph{sets}
$$\widetilde{f}:\mathcal{M}(S;\mathbb{Z}((T))_{r})\rightarrow V$$
such that on $(\mathbb{Z}[T,T^{-1}])[S]$ it gives the $\mathbb{Z}[T,T^{-1}]$-linear extension of $f$.

We produce this $\widetilde{f}$ as a limit of approximations.  For $n\geq 0$ we can consider the projection
$$\mathcal{M}(S;\mathbb{Z}((T))_{r})\rightarrow \bigoplus_{i\leq n} \mathbb{Z}[S]\cdot T^i$$
and compose with the $\mathbb{Z}[T,T^{-1}]$-linear map $(\mathbb{Z}[T,T^{-1}])[S]\rightarrow V$ induced by $f:S\rightarrow V$ to deduce maps $f_n:\mathcal{M}(S;\mathbb{Z}((T))_{r})\rightarrow V$ which we will claim converge to our desired $\widetilde{f}$.

Clearly, if we plug a $\mathbb{Z}[T,T^{-1}]$-linear sum of Dirac measures into these $f_n$, we get an eventually constant sequence with the correct desired limit value.  Using Lemma \ref{approximate} applied to the source compact Hausdorff space $\mathcal{M}(S;\mathbb{Z}((T))_{r})_{\leq C}$ for arbitrary $C$, to finish the proof it therefore suffices to show two things: first, $f_n$ restricted to $\mathcal{M}(S;\mathbb{Z}((T))_{r})_{\leq C}$ lands in some constant multiple of $K$, and second, there is a nullsequence $(\lambda_n)$ such that
$$f_n-f_m\in \lambda_n\cdot K$$
if $n\leq m$, again restricting to $\mathcal{M}(S;\mathbb{Z}((T))_{r})_{\leq C}$.

To see that $f_n$ restricted to $\mathcal{M}(S;\mathbb{Z}((T))_{r})_{\leq C}$ lands in some constant multiple of $K$, recall that $\mathcal{M}(S;\mathbb{Z}((T))_{r})$ goes to $\mathcal{M}_q(S)$ under evaluation $\lambda=T$, so the bound on the coefficients of the finite $\mathbb{Z}[T,T^{-1}]$-linear combination of Dirac measures in $\mathcal{M}(S;\mathbb{Z}((T))_{r})$ implies an $\ell^q$-bound on the $\mathbb{R}$-coefficients once we plug in $\lambda$.  As $K$ is $q'$-convex, it is also $q$-convex, so this implies that claim.

To see that there is a nullsequence $(\lambda_n)$ as desired, note that for $a>1$ we have
$$(f_m - f_n)(\sum c_i T^i) = \sum_{n<i\leq m} f(c_i) \lambda^i = a^{-n} \sum_{n<i\leq m} f(c_i) a^{n-i}(a\lambda)^i$$
where $c_i\in\mathbb{Z}[S]$ and we use $f$ also to denote the $\mathbb{Z}$-linear extension of $f$.  But now we can choose $a$ so that evaluation at $T=a\lambda$ produces the space of $q'$-measures from $\mathcal{M}(S;\mathbb{Z}((T))_{r})$, instead of $q$-measures as we got from $T=\lambda$, and we deduce the desired claim as $a^{-n}$ is a nullsequence and the sum on the right hand side lands in a constant multiple of $K$ for the same reason as before, using $q'$-convexity of $K$.
\end{proof}\newpage

\section{Lecture III: Liquid vector spaces, redux}

In today's lecture, we spent some more time discussing condensed sets, condensed abelian groups, and discussed the main theorem on liquid vector spaces.

Recall from the last lecture that condensed sets are sheaves on the category $\operatorname{Prof}$ of profinite sets. (Up to set-theoretic issues, which we will from now on ignore --- they are resolved as in the last lecture.) In other words, those are functors
\[
X: \operatorname{Prof}^{\mathrm{op}}\to \operatorname{Set}: S\mapsto X(S)
\]
satisfying some simple conditions. Actually, these conditions simplify even further if one passes to the generating subcategory $\operatorname{Extr}=\operatorname{qcProj}\subset \operatorname{Prof}$ of extremally disconnected profinite sets. Then a condensed set is equivalently given by a functor
\[
X: \operatorname{Extr}^{\mathrm{op}}\to \operatorname{Set}
\]
that preserves finite products. Equivalently, $X(\emptyset)=\ast$, and $X(S_1\sqcup S_2)\to X(S_1)\times X(S_2)$ is bijective for all $S_1,S_2\in \operatorname{Extr}$. In particular, the category of condensed sets is almost given by a functor category. This is what gives the theory of condensed sets an extremely algebraic nature. We note that for any $S\in \mathrm{Extr}$, the functor
\[
\operatorname{CondSet}\to \operatorname{Set}: X\mapsto X(S)
\]
commutes with all limits, all filtered colimits, and preserves surjections. (In fact, it commutes with all so-called sifted colimits.) It does not, however, in general commute with finite coproducts, except in the trivial case $S=\ast$. The functor
\[
\operatorname{CondSet}\to \operatorname{Set}: X\mapsto X(\ast)
\]
should be thought of as the ``underlying set'' functor, but it is not conservative (as shown by the example of $\mathbb R/\mathbb R^\delta$).

\begin{remark} Both $\mathrm{Prof}$ and $\mathrm{Extr}$ (or rather the opposite categories) admit simple algebraic descriptions. Namely, $S\mapsto C(S,\mathbb F_2)$ induces an equivalence between $\mathrm{Prof}^{\mathrm{op}}$ and the category of Boolean algebras (i.e.~those rings where all elements satisfy $x^2=x$); while $\mathrm{Extr}\subset \mathrm{Prof}$ corresponds to the category of complete Boolean algebras. (Here, a Boolean algebra is complete if any (possibly infinite) set of elements has a supremum, where a Boolean algebra is partially ordered by $x\leq y$ if $xy=x$. Equivalently, a Boolean algebra is complete if and only if it is a retract of a Boolean algebra of the form $\prod_I \mathbb F_2$ for some set $I$.) Thus, condensed sets are the same thing as finite-product-preserving functors from complete Boolean algebras to sets. (It is even enough to consider the subcategory of Boolean algebras of the form $\prod_I \mathbb F_2$.)

However, going down this route one tends to lose track of the topological intuition, and we will not use the perspective of Boolean algebras.
\end{remark}

Next, we discussed condensed abelian groups $\operatorname{CondAb}$. These can be equivalently thought of as abelian group objects in $\operatorname{CondSet}$, or as sheaves of abelian groups on $\operatorname{Prof}$. We note that as finite coproducts agree with finite products, it is now the case that for $S\in \operatorname{Extr}$, the functor
\[
\operatorname{CondAb}\to \operatorname{Ab}: X\mapsto X(S)
\]
commutes with all limits and all colimits. This implies that $\operatorname{CondAb}$ has excellent categorical properties: Like sheaves of abelian groups on any topos, it has all limits and all colimits, and filtered colimits are exact, but in this case it is also true that arbitrary products are exact. In fact, it admits a class of compact projective generators, namely the free condensed abelian groups $\mathbb Z[S]$ for $S\in \mathrm{Extr}$. We recall here that the forgetful functor
\[
\operatorname{CondAb}\to \operatorname{CondSet}
\]
admits a left adjoint, the ``free condensed abelian group'' functor
\[
\operatorname{CondSet}\to \operatorname{CondAb}: X\mapsto \mathbb Z[X].
\]
Concretely, this is given by the sheafification of $S\mapsto \mathbb Z[X(S)]$ (and the sheafification is important here, even for $S\in \operatorname{Extr}$).

\begin{remark} The construction $\mathbb Z[X]$ for $X$ say a CW complex is closely related to classical constructions due to Dold--Thom. The underlying set of $\mathbb Z[X]$ is given by $\mathbb Z[X(\ast)]$, i.e.~finite collections of points of $X$ weighted with integers. The condensed structure is made so that if two points collide, the corresponding integers add up. This can also be understood as the group completion of the free abelian monoid $\mathbb N[X]$, which can be written as
\[
\mathbb N[X] = \bigsqcup_{n\geq 0} \mathrm{Sym}^n(X)
\]
where $\mathrm{Sym}^n(X) = X^n/\Sigma_n$ is the condensed set of unordered $n$-tuples of elements of $X$.
\end{remark}

If $X=S\in \mathrm{Prof}$, one can describe $\mathbb Z[S]$ explicitly.

\begin{proposition}[{\cite[Proposition 2.1]{Analytic}}]\label{prop:freecondab} Let $S=\varprojlim_i S_i$ be a profinite set, written as a cofiltered limit of finite sets $S_i$. The free condensed abelian group $\mathbb Z[S]$ is quasiseparated, and can explicitly be written as a countable union of profinite subsets, as follows:
\[
\mathbb Z[S] \cong \bigcup_{n>0} \varprojlim_i \mathbb Z[S_i]_{\ell^1 \leq n}\subset \varprojlim_i \mathbb Z[S_i].
\]
\end{proposition}

Here $\mathbb Z[S_i]$ is the finite free abelian group on the finite set $S_i$, and the subset $\ell^1\leq n$ is the part of $\ell^1$-norm at most $n$; this produces a finite subset of $\mathbb Z[S_i]$, so $\varprojlim_i \mathbb Z[S_i]_{\ell^1\leq n}$ is a profinite set. We note that by the discreteness of $\mathbb Z$, the precise choice of norm is irrelevant here: any bound on the $\ell^1$-norm also gives a bound on the $\ell^p$-norm for all $0<p<1$.

We note that this description is remarkably close to the description of the space of signed Radon measures $\mathcal M_1(S)$ -- but over $\mathbb Z$, this construction produces the uncompleted free objects.

Another important structure is the tensor product of condensed abelian groups. In fact, like on any topos, condensed abelian groups have a symmetric monoidal tensor product, representing bilinear maps. Concretely, for $M,N\in \operatorname{CondAb}$, the tensor product $M\otimes N$ is the sheafification of the presheaf $S\mapsto M(S)\otimes N(S)$. In this case, one can in fact understand what happens under sheafification. Namely, for any $S\in \operatorname{Extr}$,
\[
(M\otimes N)(S) = M(S)\otimes_{\mathbb Z(S)} N(S)
\]
where $\mathbb Z(S) = \mathrm{Cont}(S,\mathbb Z)$ is the algebra of continuous integer-valued functions on $S$. (Indeed, this functor will already send finite disjoint unions to finite products.)

There is also a partial right adjoint to tensor product, the internal Hom functor $\underline{\operatorname{Hom}}(-,-)$, so that
\[
\operatorname{Hom}(M,\underline{\operatorname{Hom}}(N,P))\cong \operatorname{Hom}(M\otimes N,P).
\]
Moreover, all operations can be derived (using existence of projective resolutions), leading to derived tensor products $-\dotimes-$ and derived internal Hom $\underline{\operatorname{RHom}}$.

At this point, it may be helpful to say a few words about flat objects, and projective objects. Generally, it turns out that flat objects are plentiful and well-behaved, but projective objects are extremely scarce, and fragile.

\begin{proposition}\label{prop:flattorsionfree} A condensed abelian group $M\in \operatorname{CondAb}$ is flat if and only if for all $S\in \operatorname{Extr}$, the value $M(S)$ is torsion-free (as an abelian group). In that case, $M(S)$ is torsion-free for all $S\in \operatorname{Prof}$.
\end{proposition}

\begin{proof} In the forward direction, it is enough to look at tensor products of $M$ with short exact sequences of abstract abelian groups (where everything commutes with evaluation at $S$). In the converse direction, if $M(S)$ is torsion-free for all $S\in \operatorname{Extr}$, then tensoring is exact on the presheaf level. But sheafification is also exact.

For the final sentence, note that if $S\in \operatorname{Prof}$ and $T\in \operatorname{Extr}$ is a cover of $S$, then $M(S)\hookrightarrow M(T)$, showing that torsion-freeness of $M(T)$ implies torsion-freeness of $M(S)$.
\end{proof}

Concerning projective objects, we have the following basic characterization.

\begin{proposition}\label{prop:projectiveretract} A condensed abelian group $M\in \operatorname{CondAb}$ is compact projective if and only if it is a retract of $\mathbb Z[S]$ for some $S\in \operatorname{Extr}$; equivalently, if it is a retract of $\mathbb Z[\beta I]$ for some infinite set $I$. In that case, there is an isomorphism $M\oplus \mathbb Z[\beta I]\cong \mathbb Z[\beta I]$.
\end{proposition}

\begin{proof} We already know that $\mathbb Z[S]$ is compact projective when $S\in \operatorname{Extr}$. In the other direction, any surjection $\mathbb Z[\beta I]\to M$ must split, showing that $M$ is a retract. For the final statement, we only sketch the argument. Choose a bijection $I\cong I\times \mathbb N$ and regard $\mathbb Z[\beta (I\times \mathbb N)]$ as a certain completion of
\[
\bigoplus_{\mathbb N} \mathbb Z[\beta I],
\]
where $\mathbb Z[\beta I]\cong M\oplus M'$, and then reorder
\[
(M\oplus M')\oplus (M\oplus M')\oplus\ldots\cong  M\oplus (M'\oplus M)\oplus (M'\oplus M)\oplus \ldots ,
\]
to find an isomorphism (after recompleting)
\[
\mathbb Z[\beta (I\times \mathbb N)]\cong M\oplus \mathbb Z[\beta (I\times \mathbb N)].\qedhere
\]
\end{proof}

The following question is open, and closely related to an open question on injective Banach spaces \cite[Section 1.6.1]{SepInjBanach}.

\begin{question} Are all compact projective condensed abelian groups isomorphic to $\mathbb Z[S]$ for some $S\in \operatorname{Extr}$?
\end{question}

We note that one can find non-extremally disconnected $S$ such that $\mathbb Z[S]$ is compact projective; for example, if $S$ is obtained by identifying two points in the boundary of $\beta \mathbb N$, then $\mathbb Z[S]\cong \mathbb Z[\beta \mathbb N]$.

In any case, this indicates that there are extremely few compact projectives. They are also destroyed by basic operations:

\begin{proposition}\label{prop:tensorproductcompactprojective} For any infinite sets $I$ and $J$, the tensor product $\mathbb Z[\beta I]\otimes \mathbb Z[\beta J]=\mathbb Z[\beta I\times \beta J]$ is not projective.
\end{proposition}

The proposition is not completely trivial. In the appendix to this lecture, we include a proof of a strengthening, answering a question in \cite[Section 6.4.1]{SepInjBanach} in the negative.

This leads to the following important warning.

\begin{warning} For $S\in \operatorname{Extr}$ and any $M\in \operatorname{CondAb}$, clearly the Ext-groups
\[
\operatorname{Ext}^i(\mathbb Z[S],M) = H^i(S,M)=0
\]
vanish for $i>0$. However, the internal Ext-groups
\[
\underline{\operatorname{Ext}}^i(\mathbb Z[S],M)\in \operatorname{CondAb}
\]
are nonzero in general. Indeed, for $S'\in \operatorname{Extr}$, one has
\[
\underline{\operatorname{Ext}}^i(\mathbb Z[S],M)(S') = \operatorname{Ext}^i(\mathbb Z[S']\otimes \mathbb Z[S],M) = \operatorname{Ext}^i(\mathbb Z[S'\times S],M),
\]
which may be nonzero by the previous proposition.
\end{warning}

While the problem mentioned in the warning has never really led to any problems in situations of practical interest (see especially the appendix for some positive results), it is sometimes annoying in the theory development.

One might hope in particular that the issue discussed in this warning might disappear if $M$ is suitably ``complete''. The answer is yes and no:

\begin{enumerate}
\item In nonarchimedean functional analysis, when working with ``solid'' modules (see \cite{Condensed}), the issue disappers. In that case, $\mathbb Z[S]^\blacksquare$ is compact projective for all $S\in \operatorname{Prof}$, and hence (as profinite sets are stable under products) tensor products of compact projectives are still compact projective.

\item In archimedean functional analysis, we deal with ``liquid'' vector spaces, and in this case it turns out that the warning stays as acute. (In fact, in the appendix, we will prove Proposition~\ref{prop:tensorproductcompactprojective} by proving that the issue persists even in the world of Banach spaces.)
\end{enumerate}

With this, let us now turn back to liquid $\mathbb R$-vector spaces. Recall that for any profinite set $S=\varprojlim_i S_i$ (with finite sets $S_i$), we have the space of signed Radon measures
\[\begin{aligned}
\mathcal M_1(S) &= \underline{\operatorname{Hom}}(C(S,\mathbb R),\mathbb R)\\
&= \left\{\mu: \{\mathrm{clopen\ }U\subset S\}\to \mathbb R\mid \begin{array}{c} \mu(\emptyset)=0, \mu(U_1\sqcup U_2)=\mu(U_1)+\mu(U_2)\\
\exists c: \forall U_1,\ldots,U_n\ \mathrm{disjoint}: \sum_j |\mu(U_j)|\leq c
\end{array}\right\}\\
&=\bigcup_{c>0} \varprojlim_i \mathbb R[S_i]_{\ell^1\leq c}.
\end{aligned}\]
Notice that this has the prototypical form of a quasiseparated condensed set, as a union of compact Hausdorff subsets.

\begin{remark} Note that this definition of Radon measures on profinite sets is straightforward, as one has to define the measure only on open and closed subsets. In contrast, for general compact Hausdorff spaces $X$, the usual definition of $\mathcal M(X)$ is rather complicated, involving Borel sets, countable additivity, and inner/outer regularity. The perspective of condensed sets gives a different approach. Namely, from the condensed perspective, $X$ is a quotient of a profinite set $S$ by a profinite equivalence relation $R\subset S\times_X S$. In that case, there is an exact sequence
\[
\mathcal M(R)\to \mathcal M(S)\to \mathcal M(X)\to 0
\]
which can also be used to define $\mathcal M(X)$ as the quotient of $\mathcal M(S)$ by $\mathcal M(R)$.
\end{remark}

It turned out that this choice of measures does not define an abelian category of ``$\mathcal M_1$-complete'' $\mathbb R$-vector spaces; in fact, it is necessary to include non-locally convex objects. This led us to consider, for $0<p\leq 1$, the variant
\[\begin{aligned}
\mathcal M_p(S) &= \left\{\mu: \{\mathrm{clopen\ }U\subset S\}\to \mathbb R\mid \begin{array}{c} \mu(\emptyset)=0, \mu(U_1\sqcup U_2)=\mu(U_1)+\mu(U_2)\\
\exists c: \forall U_1,\ldots,U_n\ \mathrm{disjoint}: \sum_j |\mu(U_j)|^p\leq c
\end{array}\right\}\\
&=\bigcup_{c>0} \varprojlim_i \mathbb R[S_i]_{\ell^p\leq c}.
\end{aligned}\]

\begin{remark}\label{rem:ell0} These spaces shrink as $p$ gets smaller. One might even define a limiting case for $p=0$, defining
\[
\mathcal M_0(S)=\bigcup_{c>0} \varprojlim_i \mathbb R[S_i]_{\ell^0\leq c}
\]
where the condition $\ell^0\leq c$ means that at most $\lfloor c\rfloor$ coefficients are nonzero, and the $\ell^1$-norm is at most $c$. (Once one bounds the number of nonzero coefficients, the precise choice of norm does not matter.) As an analogue of Proposition~\ref{prop:freecondab}, it turns out that this is just the uncompleted free condensed $\mathbb R$-vector space on $S$
\[
\mathbb R[S] = \mathcal M_0(S).
\]
\end{remark}

For a fixed $p$, the issue persists: There is a nonzero map
\[
\ell^p(\mathbb N)\to \ell^2(\mathbb N)/\ell^p(\mathbb N): (x_n)_n\mapsto (x_n\log|x_n|)_n
\]
of condensed $\mathbb R$-vector spaces, whose restriction to basis vectors is zero. However, its restriction to $\ell^q(\mathbb N)$ with $q<p$ is zero. Thus, we consider the variant
\[
\mathcal M_{<p}(S) = \bigcup_{q<p} \mathcal M_q(S).
\]
All these spaces of measures have the subset $S\subset \mathcal M_q(S)\subset \mathcal M_{<p}(S)$ of Dirac measures.

Now we can state the main theorem on liquid vector spaces. In this course, we will use this as a black box; for the proof we refer to \cite[Lecture VI -- IX]{Analytic}. The theorem is quite a mouthful. The essential content is that $(\mathbb R,\mathcal M_{<p})$ (and $(\mathbb Z,\mathcal M_{<p})$) is an analytic ring in the sense of \cite{Condensed}; most of this theorem is then a formal consequence of general properties of analytic rings.

\begin{theorem}\label{thm:liquidmain} Let $V$ be a condensed abelian group. The following are equivalent.
\begin{enumerate}
\item For all $S\in \operatorname{Prof}$, all $f: S\to V$, and all $q<p$, there is a unique extension of $f$ to a map $\tilde{f}_q: \mathcal M_q(S)\to V$.
\item For all $S\in \operatorname{Prof}$ and all $f: S\to V$, there is a unique extension of $f$ to a map $\tilde{f}: \mathcal M_{<p}(S)\to V$.
\item One can write $V$ as a cokernel of a map
\[
\bigoplus_j \mathcal M_{<p}(T_j)\to \bigoplus_i \mathcal M_{<p}(S_i).
\]
\end{enumerate}
The class of such $V$ is an abelian subcategory $\mathrm{Liquid}_p$ of $\mathrm{CondAb}$ stable under all limits, all colimits, all extensions, and all internal Hom's (and internal Ext's). In fact, on the level of derived categories, the functor
\[
D(\mathrm{Liquid}_p)\to D(\mathrm{CondAb})
\]
is fully faithful, and the essential image consists of all those $C\in D(\mathrm{CondAb})$ for which all cohomologies $H^i(C)$ lie in $\mathrm{Liquid}_p$. For all $C\in D(\mathrm{Liquid}_p)$, all $S\in \operatorname{Prof}$, and all $q<p$, the maps
\[
\underline{\operatorname{RHom}}(\mathcal M_{<p}(S),C)\to \underline{\operatorname{RHom}}(\mathcal M_q(S),C)\to \underline{\operatorname{RHom}}(\mathbb Z[S],C)
\]
are isomorphisms (giving a derived and internal version of conditions (1) and (2)).

The embedding $\mathrm{Liquid}_p\subset \mathrm{CondAb}$ admits a left adjoint $M\mapsto M^{\mathrm{liq}}$, called ($p$-)liquidification, whose left derived functor defines a left adjoint to the inclusion
\[
D(\mathrm{Liquid}_p)\to D(\mathrm{CondAb}).
\]
Liquidification commutes with all colimits and sends $\mathbb Z[S]$ to $\mathcal M_{<p}(S)$. In particular, for $S=\ast$, it sends $\mathbb Z$ to $\mathbb R$.

There is a unique symmetric monoidal tensor product $-\otimes_{\mathbb R_{<p}}-$ on $\mathrm{Liquid}_p$ making $M\mapsto M^{\mathrm{liq}}$ symmetric monoidal, i.e.~such that
\[
M^{\mathrm{liq}}\otimes_{\mathbb R_{<p}} N^{\mathrm{liq}}\cong (M\otimes N)^{\mathrm{liq}}.
\]
Concretely, $-\otimes_{\mathbb R_{<p}}-$ represents bilinear maps in $\mathrm{Liquid}_p$. Its left derived functor defines a symmetric monoidal tensor product on $D(\mathrm{Liquid}_p)$ making the left adjoint $D(\mathrm{CondAb})\to D(\mathrm{Liquid}_p)$ symmetric monoidal.

All functors just mentioned (liquidification, forgetful functors, and tensor products) commute with all colimits (resp.~all direct sums on the level of derived categories).

As the liquidification of $\mathbb Z$ is $\mathbb R$, it follows that the tensor unit of $\mathrm{Liquid}_p$ is $\mathbb R$, and thus all objects of $\mathrm{Liquid}_p$ have unique and functorial $\mathbb R$-module structures. All the above results on the relation between $\mathrm{Liquid}_p$ and $\mathrm{CondAb}$ are also true about the relation between $\mathrm{Liquid}_p$ and $\mathrm{Cond}(\mathbb R)$.
\end{theorem}

\begin{remark} As $\mathbb Z^{\mathrm{liq}}=\mathbb R^{\mathrm{liq}} = \mathbb R$, it follows that $(\mathbb R/\mathbb Z)^{\mathrm{liq}}=0$. In fact, more generally, for any compact abelian group $A$, one has $A^{\mathrm{liq}} = 0$. Thus, for example, $\mathbb Z_2^{\mathrm{liq}}=0$.
\end{remark}

The following remark may be safely ignored (and reading it may actually lead to confusion, as it generalizes the concept of liquidity in a way not captured by the above discussion, but also not relevant for this course).

\begin{remark} When we tried to find an analytic ring structure on $\mathbb R$ (and eventually found the $p$-liquid analytic ring structure), we also wondered whether this analytic ring structure on $\mathbb R$ might admit non-archimedean analogues, say over $\mathbb Q_2$ (or any prime in place $2$, but the letters $p$, $q$ and $\ell$ are already taken...). And indeed, it does! More precisely, for any $0\leq p\leq \infty$, we can define
\[
\mathcal M_p(S,\mathbb Q_2) = \bigcup_{c>0} \varprojlim_i \mathbb Q_2[S_i]_{\ell^p\leq c}
\]
where for $0<p<\infty$, the $\ell^p$-norm is defined as usual by $\sum_{s\in S_i} |x_s|^p$ (which is now subadditive, as $|x+y|^p\leq \max(|x|^p,|y|^p)\leq |x|^p+|y|^p$ in the nonarchimedean case), the $\ell^0$-norm is defined as in Remark~\ref{rem:ell0}, and the $\ell^\infty$-norm is the supremum norm. Then $\mathcal M_0(S,\mathbb Q_2)=\mathbb Q_2[S]$, while $\mathcal M_\infty(S,\mathbb Q_2)$ is the dual of $C(S,\mathbb Q_2)$. It turns out that all pairs
\[
(\mathbb Q_2,\mathcal M_0), (\mathbb Q_2,\mathcal M_{<p})_{0<p\leq \infty}, (\mathbb Q_2,\mathcal M_\infty)
\]
define (pairwise distinct) analytic rings. For the first, all condensed $\mathbb Q_2$-vector spaces are complete. In the middle, we have various classes of $p$-liquid $\mathbb Q_2$-vector spaces for varying $p$. Finally, $\mathcal M_\infty$-complete $\mathbb Q_2$-vector spaces are exactly the solid $\mathbb Q_2$-vector spaces. Thus, over nonarchimedean fields, there is a whole line $0\leq p\leq \infty$ of possible ``completeness'' conditions, interpolating from arbitrary to ``liquid'' to ``solid''.

Over $\mathbb R$, one has to stop at $p=1$, because of the triangle inequality. This picture is extremely reminiscent of the Berkovich space of $\mathbb Z$, which includes a full $(0,\infty)$ at each nonarchimedean place, but only a half-open interval $(0,1]$ at the archimedean place.

All these various notions of liquidity over $\mathbb R$ and the various nonarchimedean local fields like $\mathbb Q_2$ can be understood as specializations of a notion of liquid module over a ring of overconvergent arithmetic Laurent series $\mathbb Z((T))_{>r}$, constructed in \cite[Lecture VI -- IX]{Analytic}. In the current lectures, we will only consider the case over $\mathbb R$, and so as in the theorem, ``liquid'' will imply $\mathbb R$-vector space for us, discarding these analogues of liquid vector spaces over nonarchimedean fields.
\end{remark}

Coming back to the discussion of flat and projective objects, we discuss flat objects in the following proposition (and projective objects in the appendix).

\begin{theorem} For any $S\in \mathrm{Prof}$, the object $\mathcal M_{<p}(S)\in \mathrm{Liquid}_p$ is flat. Moreover, for any quasiseparated $p$-liquid $\mathbb R$-vector space $V$, there is a natural isomorphism
\[
\mathcal M_{<p}(S)\otimes_{\mathbb R_{<p}} V\cong \bigcup_{q<p} \bigcup_{\substack{K\subset V\\q\text-\mathrm{convex}}} \mathcal M_q(S,K).
\]
Here, for finite $S$, we define
\[
\mathcal M_q(S,K)\subset \mathbb R[S]\otimes_{\mathbb R} V
\]
as the $q$-convex hull of the $|S|$ closed subsets $K\cong [s]\otimes K\subset \mathbb R[S]\otimes_{\mathbb R} V$, $s\in S$, and in general
\[
\mathcal M_q(S,K) = \varprojlim_i \mathcal M_q(S_i,K)
\]
for $S=\varprojlim_i S_i$ profinite.
\end{theorem}

\begin{proof} Consider the functor
\[
V\mapsto \mathcal M_{<p}(S;V) = \bigcup_{q<p} \bigcup_{\substack{K\subset V\\q\text-\mathrm{convex}}} \mathcal M_q(S,K),
\]
which is an endofunctor on quasiseparated $p$-liquid $\mathbb R$-vector spaces. (Indeed, each $\mathcal M_q(S,K)$ is a compact $q$-convex subset, and all transition maps are injective.) This is an exact functor: If
\[
0\to V\to V'\to V''\to 0
\]
is an exact sequence of quasiseparated $p$-liquid $\mathbb R$-vector spaces, then
\begin{enumerate}
\item the map $\mathcal M_{<p}(S;V)\to \mathcal M_{<p}(S;V')$ is injective, directly from the definition. In fact, this holds for any injection $V\hookrightarrow V'$, not necessarily with quasiseparated quotient.
\item the image of $\mathcal M_{<p}(S;V)\to \mathcal M_{<p}(S;V')$ is the kernel of $\mathcal M_{<p}(S;V')\to \mathcal M_{<p}(S;V'')$, as for any compact $q$-convex $K'\subset V'$ with image $K''\subset V''$, any element in the kernel of $\mathcal M_q(S;K')\to \mathcal M_q(S;K'')$ actually defines an element of $\mathcal M_q(S;K)$, where $K=V\cap K'$.
\item the map $\mathcal M_{<p}(S;V')\to \mathcal M_{<p}(S;V'')$ is surjective, as any compact $q$-convex $K''\subset V''$ is the image of some compact $q$-convex $K'\subset V'$, in which case $\mathcal M_q(S;K')\to \mathcal M_q(S;K'')$ is surjective (by Tychonoff). Indeed, as $V'\to V''$ is surjective, there is some compact subset $K_0\subset V'$ with image $K''$, and this is contained in a compact $q$-convex subset $K_1\subset V'$. Then the intersection of $K_1$ with the preimage of $K''$ is a compact $q$-convex subset $K'\subset V'$ with image $K''$.
\end{enumerate}

In particular, the functor commutes with finite direct sums. In fact, it commutes with infinite direct sums (as any compact $K\subset \bigoplus_i V_i$ is already contained in some finite direct sum). Moreover, for $V=\mathcal M_{<p}(S')$, there is a natural isomorphism
\[
\mathcal M_{<p}(S;\mathcal M_{<p}(S'))\cong \mathcal M_{<p}(S\times S'),
\]
by writing out the definitions.

Now any $p$-liquid condensed $\mathbb R$-vector space $V$ admits a resolution
\[
\ldots\to \bigoplus_{i_1} \mathcal M_{<p}(S_{i_1})\to \bigoplus_{i_0} \mathcal M_{<p}(S_{i_0})\to V\to 0
\]
where all $S_{i_0}$, $S_{i_1}$ etc.~are extremally disconnected profinite sets. If $V$ is quasiseparated, then applying $\mathcal M_{<p}(S;-)$ to this sequence gives a long exact sequence
\[
\ldots\to \bigoplus_{i_1} \mathcal M_{<p}(S\times S_{i_1})\to \bigoplus_{i_0} \mathcal M_{<p}(S\times S_{i_0})\to \mathcal M_{<p}(S;V)\to 0.
\]
But the complex
\[
\ldots\to \bigoplus_{i_1} \mathcal M_{<p}(S\times S_{i_1})\to \bigoplus_{i_0} \mathcal M_{<p}(S\times S_{i_0})
\]
computes $\mathcal M_{<p}(S)\dotimes_{\mathbb R_{<p}} V$. This discussion implies that $\mathcal M_{<p}(S)\dotimes_{\mathbb R_{<p}} V$ is concentrated in degree $0$, and given by $\mathcal M_{<p}(S;V)$, as desired.

Finally, for any $p$-liquid condensed $\mathbb R$-vector space $V$, we can find a surjection $V'\to V$ with $V'$ quasiseparated, in which case also $V''=\ker(V'\to V)\subset V'$ is quasiseparated. Taking derived tensor products of
\[
0\to V''\to V'\to V\to 0
\]
with $\mathcal M_{<p}(S)$, we see that $\mathcal M_{<p}(S)\dotimes_{\mathbb R_{<p}} V$ is concentrated in degree $0$ if and only if $\mathcal M_{<p}(S;V'')\to \mathcal M_{<p}(S;V')$ is injective. But this is item (1) above.
\end{proof}\newpage

\section*{Appendix to Lecture III: Remarks on projective objects}

The goal of this appendix is to discuss some positive and some negative results on projective objects. The results of this appendix are not used in this course, and we mostly include them here to record them.

The starting point is the following positive result. Here, $\omega_1$ is the first uncountable ordinal, and hence $\operatorname{Prof}_{\omega_1}$ is the category of profinite sets that are countable limits of finite sets; equivalently, this is the category of metrizable profinite sets. The pullback functor $\operatorname{CondSet}_{\omega_1}\to \operatorname{CondSet}$ is fully faithful; objects in the image are called $\omega_1$-condensed. Equivalently, a condensed set $X$ is $\omega_1$-condensed if the functor
\[
\operatorname{Prof}^{\mathrm{op}}\to \mathrm{Set}
\]
commutes with $\omega_1$-filtered colimits, i.e.
\[
X(\varprojlim_i S_i) = \varinjlim_i X(S_i)
\]
whenever $S_i\in\operatorname{Prof}$, and the index category is $\omega_1$-filtered (i.e.~any countable subdiagram can be extended with a cone point).

\begin{proposition}\label{prop:lightlyprojective} Let $M$ be a $\omega_1$-condensed abelian group. Then for all $i>0$,
\[
\underline{\operatorname{Ext}}^i(\mathbb Z[\mathbb N\cup\{\infty\}],M)=0.
\]
\end{proposition}

In other words, within $\omega_1$-condensed abelian group, $\mathbb Z[\mathbb N\cup \{\infty\}]$ is internally projective. (While within all condensed abelian groups, it is not even projective! And within all condensed abelian groups, Proposition~\ref{prop:tensorproductcompactprojective} shows that there are essentially no internally projective objects.)

\begin{proof} Let $S\in \operatorname{Extr}$. We have to see that
\[
H^i(S\times (\mathbb N\cup \{\infty\}),M)=0
\]
for $i>0$. Consider the pushout diagram
\[\xymatrix{
\partial \beta(S\times \mathbb N)\ar[r]\ar[d] & \beta(S\times \mathbb N)\ar[d]\\
S\times \{\infty\}\ar[r] & S\times (\mathbb N\cup \{\infty\}).
}\]
This induces a Mayer--Vietoris type sequence
\[
\ldots \to H^i(S\times (\mathbb N\cup \{\infty\}),M)\to H^i(S\times \{\infty\},M)\oplus H^i(\beta(S\times \mathbb N),M)\to H^i(\partial\beta(S\times \mathbb N),M)\to \ldots .
\]
Here, $S=S\times \{\infty\}$ and $\beta(S\times \mathbb N)$ are extremally disconnected profinite sets, so their $H^i(-,M)$ vanishes for $i>0$. Thus, it suffices to show that the restriction maps
\[
H^i(\beta(S\times \mathbb N),M)\to H^i(\partial\beta(S\times \mathbb N),M)
\]
are surjective for all $i\geq 0$. This is a special case of the next lemma.
\end{proof}

\begin{lemma}\label{lem:restrictionsurjectivelight} Let $M$ be an $\omega_1$-condensed abelian group and let $S\subset T$ be an injection of profinite sets. Then for all $i\geq 0$ the map $H^i(T,M)\to H^i(S,M)$ is surjective.
\end{lemma}

\begin{proof} As $M$ is $\omega_1$-condensed, the functor $S\mapsto H^i(S,M)$ takes $\omega_1$-cofiltered limits to filtered colimits. Writing $T$ as an $\omega_1$-cofiltered limit of its metrizable quotients (and $S$ as the corresponding limit of its images), we can thus assume that $T$ and $S$ are metrizable. In that case, $S$ is injective in the category of profinite sets (by writing it as a sequential limit of finite sets along surjective maps, and noting that injections of profinite sets admit the left lifting property against surjective maps of finite sets). Thus, $S\hookrightarrow T$ splits, and consequently $H^i(T,M)\to H^i(S,M)$ is surjective.
\end{proof}

One might wonder whether further profinite sets $S$ have the property that $H^i(S,M)=0$ for all $M\in \operatorname{CondAb}_{\omega_1}$ and $i>0$. Regarding this property, we have the following proposition.

\begin{proposition}\label{prop:lightlyprojectiveprofinite} Let $S\in \operatorname{Prof}$. Consider the following conditions.
\begin{enumerate}
\item For all $M\in \operatorname{CondAb}_{\omega_1}$ and all $i>0$, one has $H^i(S,M)=0$.
\item For all $M\in \mathrm{Liquid}_p\cap \operatorname{CondAb}_{\omega_1}$ and all $i>0$, one has $H^i(S,M)=0$.
\item For all Smith spaces $W$ whose unit ball is metrizable, one has $H^1(S,W)=0$.
\item The real Banach space $C(S)$ is separably injective in the sense of \cite[Definition 2.1]{SepInjBanach}.
\end{enumerate}
Then (1) $\Rightarrow$ (2) $\Rightarrow$ (3) $\Rightarrow$ (4).
\end{proposition}

In the proof, we use the yoga of Banach and Smith spaces, for which we refer to \cite[Lecture IV]{Analytic}.

\begin{proof} It is clear that (1) $\Rightarrow$ (2) $\Rightarrow$ (3). To show that this implies (4), let $V\subset V'$ be a closed immersion of separable Banach spaces and $V\to C(S)$ a map of Banach spaces. We wish to show that it extends to a map $V'\to C(S)$. We note that the exact sequence $0\to V\to V'\to V''\to 0$ of Banach spaces dualizes to an exact sequence $0\to W''\to W'\to W\to 0$ of Smith spaces, and the map $V\to C(S)$ dualizes to a map $\mathcal M_1(S)\to W$ of Smith spaces, which is in fact determined by the map $S\to W$. The condition that the $V$'s are separable is equivalent to the condition that the $W$'s have metrizable unit balls. By (3), one has $H^1(S,W'')=0$, and hence the map $S\to W$ lifts to a map $S\to W'$, which then extends uniquely to a map $\mathcal M_1(S)\to W'$ of Smith spaces, necessarily lifting $\mathcal M_1(S)\to W$. Dualizing again, we arrive at the desired extension $V'\to C(S)$.
\end{proof}

One negative result is a theorem of Amir:

\begin{theorem}[{\cite[Theorem 3]{AmirSepInj}}] For a metrizable profinite set $S$, the Banach space $C(S)$ is separably injective if and only if $C(S)\cong c_0(\mathbb N)$. In particular, $C(S)$ is not separably injective for the Cantor set $S$.
\end{theorem}

The book \cite{SepInjBanach} raises the question whether the non-metrizable profinite set $S=\beta \mathbb N\times \beta\mathbb N$ has the property that $C(S)$ is separably injective, see \cite[Section 6.4.1]{SepInjBanach}. We answer this in the negative:

\begin{theorem}\label{thm:productsnotsepinj} The Banach space $C(\beta \mathbb N\times \beta \mathbb N)$ is not separably injective.
\end{theorem}

In particular, by Proposition~\ref{prop:lightlyprojectiveprofinite}, this implies that not all $H^i(S,M)$ vanish for $i>0$ and $M\in \operatorname{CondAb}_{\omega_1}$, which in particular gives Proposition~\ref{prop:tensorproductcompactprojective}. (In fact, if one only wishes to show Proposition~\ref{prop:tensorproductcompactprojective}, it suffices to show that $C(\beta \mathbb N\times \beta \mathbb N)$ is not an injective Banach space. This follows from it admitting $c_0(\mathbb N)$ as a direct summand by \cite{Cembranos}, and $c_0(\mathbb N)$ not being injective, \cite[Theorem 1.25]{SepInjBanach}.)

\begin{remark} The proof below can also be read line-by-line with all Banach spaces $C(S)$ replaced by free condensed abelian groups $\mathbb Z[S]$, leading to a proof that $\mathbb Z[\beta \mathbb N\times \beta \mathbb N]$ has nontrivial $\mathrm{Ext}^1$-groups against $\omega_1$-condensed abelian groups, without unnecessarily involving Banach spaces. We formulate the proof in terms of Banach spaces in order to cover the question asked in \cite{SepInjBanach}.
\end{remark}

Given a profinite set $S$, one can naturally endow the set of all closed subsets of $S$ with the structure of a profinite set $\mathrm{Sub}(S)$. One way to do this is to observe that for finite sets $S$, $\mathrm{Sub}(S)$ is itself a finite set, and the association $S\mapsto\mathrm{Sub}(S)$ admits a covariant functoriality (by taking images of subsets). For general profinite $S=\varprojlim_i S_i$, one can then define $\mathrm{Sub}(S)=\varprojlim_i \mathrm{Sub}(S_i)$ as a profinite set. Note that indeed the underlying set of $\mathrm{Sub}(S)$ is given by the set of closed subsets of $S$: Any closed subset $Z\subset S$ is itself the limit of its images $Z_i\subset S_i$ (and conversely any compatible collection of subsets $Z_i\subset S_i$ defines a closed subset $Z\subset S$ in the inverse limit). Also note that if $S$ is metrizable, then so is $\mathrm{Sub}(S)$. One can give a ``moduli description'' of $\mathrm{Sub}(S)$: For any other profinite set $T$, maps from $T$ to $\mathrm{Sub}(S)$ are in bijection with those closed subsets $Z\subset S\times T$ such that for all finite quotients $S\to S_i$, the image $Z_i\subset S_i\times T$ (of $Z\subset S\times T$) is locally on $T$ constant.

We note that the subset $\mathrm{Sub}^{\mathrm{oc}}(S)\subset \mathrm{Sub}(S)$ of open and closed subsets defines a dense subset, which is also countable in case $S$ is metrizable.

If $S$ is metrizable, one can find a ``universal'' family of metrizable covers of $S$. Indeed, any such $T$ admits a closed immersion $T\hookrightarrow \{0,1\}^{\mathbb N}\times S$. Let $T_n\subset \{0,1\}^n\times S$ be the image of $T$. Then $T_n=\bigsqcup_{\alpha_n\in \{0,1\}^n} T_{n,\alpha_n}$, where each $T_{n,\alpha_n}\subset S$ is a closed subset. Such $T_n$'s are then parametrized by $\mathrm{Sub}(S)^{\{0,1\}^n}$, and the condition that $T_n\to S$ is a cover defines a closed subset $\mathrm{Sub}(S)^{\{0,1\}^n}_{\mathrm{cov}}\subset \mathrm{Sub}(S)^{\{0,1\}^n}$. Varying $n$, there are transition maps
\[
\mathrm{Sub}(S)^{\{0,1\}^{n+1}}_{\mathrm{cov}}\to \mathrm{Sub}(S)^{\{0,1\}^n}_{\mathrm{cov}}
\]
taking a collection
\[
(T_{\alpha_{n+1}})_{\alpha_{n+1}\in \{0,1\}^{n+1}}
\]
of closed subsets of $S$ to the collection $(T_{n,\alpha_n})_{\alpha_n}$ given by $T_{\alpha_n}=T_{(\alpha_n,0)}\cup T_{(\alpha_n,1)}$. Let
\[
\mathrm{Cov}(S) = \varprojlim_n \mathrm{Sub}(S)^{\{0,1\}^n}_{\mathrm{cov}};
\]
this is a metrizable profinite set whose points parametrize metrizable covers $T\to S$ with a fixed embedding $T\hookrightarrow \{0,1\}^{\mathbb N}\times S$.

There is a countable dense subset $\mathrm{Cov}^{\mathrm{oc}}(S)$ of $\mathrm{Cov}(S)$ consisting of those $T$ such that all $T_{n,\alpha_n}$ are open and closed subsets of $S$, and for large $n$ all inclusions $T_{n+1,\alpha_{n+1}}\subset T_{n,\alpha_n}$ are equalities.

Now we can give a reformulation of the separable injectivity of $C(\beta \mathbb N\times \beta \mathbb N)$.

\begin{proposition}\label{prop:equivcondition} The Banach space $C(\beta \mathbb N\times \beta \mathbb N)$ is separably injective if and only if for all metrizable profinite sets $S_1$ and $S_2$ with a metrizable cover $T\to S_1\times S_2$, there are metrizable covers $S_1'\to S_1$ and $S_2'\to S_2$ and a commutative diagram
\[\xymatrix{
C(S_1\times S_2)\ar[r]\ar[dr] & C(T)\ar@{-->}[d]\\
& C(S_1'\times S_2').
}\]
\end{proposition}

\begin{proof} We prove only the direction of relevance to us, namely that the separable injectivity of $C(\beta \mathbb N\times \beta \mathbb N)$ implies this statement. As $T\to S_1\times S_2$ is a cover, the map $C(S_1\times S_2)\to C(T)$ is a closed immersion of separable Banach spaces. Choose surjections $\beta \mathbb N\to S_1$, $\beta \mathbb N\to S_2$. Then the map $C(S_1\times S_2)\to C(\beta \mathbb N\times \beta \mathbb N)$ extends to $C(T)$ by separable injectivity of $C(\beta \mathbb N\times \beta \mathbb N)$. But
\[
C(\beta \mathbb N\times \beta \mathbb N) = \varinjlim_{\beta \mathbb N\to S_1', \beta \mathbb N\to S_2'} C(S_1'\times S_2')
\]
is the $\omega_1$-filtered colimit of Banach spaces, where the index category is over all metrizable quotients $S_1'$, $S_2'$ of $\beta \mathbb N$; we may assume that these quotients map to $S_1$ and $S_2$. Thus, any map $C(T)\to C(\beta \mathbb N\times \beta \mathbb N)$ from the separable Banach space $C(T)$ factors over some $C(S_1'\times S_2')$, giving the desired result.
\end{proof}

We will show that if this property holds true, then in a certain ``universal'' case it cannot be necessary to include the cover of $S_1$. Intuitively, the point is that given $S_1$, there is a universal case for $S_2$ (given by $\mathrm{Cov}(S_1)$) and in this universal case, any one metrizable cover $S_1'\to S_1$ does not really simplify the situation.

As preparation, we recall the following.

\begin{proposition}\label{prop:cantorsetsplits} Let $S$ be the Cantor set, and let $f: T\to S$ be a metrizable cover. Then there is a closed subset $Z\subset S$ homeomorphic to the Cantor set on which $f$ splits.
\end{proposition}

\begin{proof} This is well-known; one can for example write $T=\varprojlim_n T_n$ as above, and inductively for each $n$ choose $2^n$ pairwise disjoint open and closed subsets $S_{\alpha_n}$ with splittings $S_{\alpha_n}\to T_n$; then $\varprojlim_n \bigsqcup_{\alpha_n} S_{\alpha_n}$ is a closed subset of $S$ on which $f$ splits, and which contains a Cantor set (as it surjects onto the Cantor set, and this surjection admits a splitting as all metrizable profinite sets are injective in the category of profinite sets).
\end{proof}

Now let $S$ be the Cantor set and let $J=\mathrm{Cov}^{\mathrm{oc}}(S)$. We get a surjective map $\beta J\to \mathrm{Cov}(S)$.

\begin{proposition}\label{prop:consequence} Assume that $C(\beta \mathbb N\times \beta \mathbb N)$ is separably injective. Then there is a commutative diagram
\[\xymatrix{
C(S\times \mathrm{Cov}(S))\ar[r]\ar[dr] & C(T)\ar@{-->}[d]\\
& C(S\times \beta J)
}\]
where $T\to S\times \mathrm{Cov}(S)$ is the universal cover.
\end{proposition}

\begin{proof} Using Proposition~\ref{prop:equivcondition}, we see that there is some cover $f: S'\to S$ and a commutative diagram
\[\xymatrix{
C(S\times \mathrm{Cov}(S))\ar[r]\ar[dr] & C(T)\ar@{-->}[d]\\
& C(S'\times \beta J).
}\]
We need to get rid of the cover $f: S'\to S$. For this, we pick a copy of the Cantor set $Z\subset S$ on which $f$ splits.

Fix a retraction $r: S\to Z$. This induces a map $\mathrm{Cov}(S)\to \mathrm{Cov}(Z)$ (as for any surjective map of metrizable profinite sets), and letting $T_Z\to Z\times \mathrm{Cov}(Z)$ be the universal family, there is a map $T\to T_Z$ over $S\times \mathrm{Cov}(S)\to Z\times \mathrm{Cov}(Z)$. In particular, we get a commutative diagram
\[\xymatrix{
C(Z\times \mathrm{Cov}(Z))\ar[r]\ar[d] & C(T_Z)\ar[d]\\
C(S\times \mathrm{Cov}(S))\ar[r] & C(T).
}\]
By assumption, we get a further map $C(T)\to C(S'\times \beta J)$.

Let $J_Z=\mathrm{Cov}^{\mathrm{oc}}(Z)$. Using the retraction $r: S\to Z$, one can use contravariant functoriality of $\mathrm{Cov}^{\mathrm{oc}}(-)$ to get a map $J_Z\to J=\mathrm{Cov}^{\mathrm{oc}}(S)$, inducing a map $\beta J_Z\to \beta J\to \mathrm{Cov}(S)$. The composite $\beta J_Z\to \mathrm{Cov}(S)\to \mathrm{Cov}(Z)$ is given by the unique extension of the map $J_Z\to \mathrm{Cov}^{\mathrm{oc}}(Z)\subset \mathrm{Cov}(Z)$.

In particular, we get the map $C(T)\to C(S'\times \beta J)\to C(Z\times \beta J_Z)$ by restriction along $Z\times \beta J_Z\to S'\times \beta J$. We claim that the composite
\[
C(Z\times \mathrm{Cov}(Z))\to C(T_Z)\to C(T)\to C(S'\times \beta J)\to C(Z\times \beta J_Z)
\]
is pullback along the map $Z\times \beta J_Z\to Z\times \mathrm{Cov}(Z)$. But the first three maps compose to pullback along the map $S'\times \beta J\to Z\times \mathrm{Cov}(Z)$ that is given by $S'\to S\xrightarrow{r} Z$ in the first factor and by $\beta J\to \mathrm{Cov}(S)\to \mathrm{Cov}(Z)$ in the second factor. Restricting to $Z\times \beta J_Z$, the claim follows.

Thus, we get the desired splitting (for the Cantor set $Z$ in place of $S$, but they are isomorphic).
\end{proof}

Back to our Cantor set $S$, choose a metrizable cover $S'\to S$ such that $C(S)\to C(S')$ does not split. In fact, let us choose one such cover such that even the map $C(U)\subset C(S)\to C(S')$ does not split for any open and closed subset $U\subset S$, for example by choosing one such cover for each open and closed subset of $S$ and then taking the product of all such covers (noting that there are only countably many open and closed subsets of $S$). We can find a point $x\in \mathrm{Cov}(S)$ giving rise to this cover $S'\to S$. Picking a point of $\beta J$ mapping to $x$, Proposition~\ref{prop:consequence} implies that if $C(\beta \mathbb N\times \beta \mathbb N)$ is separably injective, then there is a commutative diagram
\[\xymatrix{
C(S\times \mathrm{Cov}(S))\ar[r]\ar[dr] & C(T)\ar@{-->}[d]\\
& C(S)
}\]
where the diagonal arrow is evaluation at $x\in \mathrm{Cov}(S)$. Now take the infimum $C$, over all possible open and closed subsets $U\subset S$ and all maps $f: C(T)\to C(U)$ making the diagram
\[\xymatrix{
C(S\times \mathrm{Cov}(S))\ar[r]\ar[d] & C(T)\ar@{-->}[d]\\
C(S)\ar[r] & C(U)
}\]
commute, of the norm of $f$. (Clearly, $C\geq 1$.) Choose some such $U$ and $f$ for which $||f||<C+\tfrac 12$. Note that we may always replace $f$ by $f$ precomposed with the projection to the preimage of an open and closed subset of $\mathrm{Cov}(S)$ containing $x$. Choosing a retraction $r: T\to S'$ of $T$ to $S'=T\times_{\mathrm{Cov}(S)} \{x\}$, we find that $||f-f\circ r||<1$ (otherwise we could make $||f||$ smaller at least on an open and closed subset of $U$ by precomposing with such a projection to an open and closed subset of $\mathrm{Cov}(S)$). But then $f\circ r: C(S')\to C(U)$ is almost an inverse to the pullback $g: C(U)\subset C(S)\to C(S')$, more precisely $||1-g\circ f\circ r||<1$, but then a geometric series produces an actual inverse to $g$.
\newpage

\section{Lecture IV: Liquid tensor product calculations}

Let us recap.  For every $0<p\leq 1$, we have an abelian full subcategory $\operatorname{Liquid}_p$ of $\operatorname{CondAb}$ (or of condensed $\mathbb{R}$-modules), closed under limits and colimits, generated by the compact projective objects
$$\mathcal{M}_{<p}(S),$$
the spaces of $<p$-summable measures on $S\in \operatorname{Prof}$.  There is also a tensor product $-\otimes_{\mathbb{R}_{<p}}-$ on $\operatorname{Liquid}_p$, the \emph{liquid tensor product}, which:
\begin{enumerate}
\item commutes with colimits in each variable;
\item satisfies $\mathcal{M}_{<p}(S)\otimes_{\mathbb{R}_{<p}}\mathcal{M}_{<p}(T) = \mathcal{M}_{<p}(S\times T)$ for $S,T\in\operatorname{Prof}$.
\end{enumerate}
Moreover, for $V,W\in\operatorname{Liquid}_p$, the $p$-liquid vector space $V\otimes_{\mathbb{R}_{<p}} W$ corepresents the functor of bilinear maps of condensed abelian groups (or condensed $\mathbb{R}$-modules) out of $V\times W$ with $p$-liquid target.

We also saw that the basic objects, $\mathcal{M}_{<p}(S)$ for $S\in\operatorname{Prof}$, are flat.  In this lecture, we want to carry this further and demonstrate a non-trivial calculation of the liquid tensor product, in the context of spaces of holomorphic functions on disks in the complex plane.  We will also illustrate, in this case, some important principles which we'll revisit from a more thorough perspective when we discuss \emph{nuclear modules}.

All these spaces of holomorphic functions for varying disks are isomorphic by translation and scaling, so we may as well just consider the unit disk $\mathbb{D}$, with
$$\mathcal{O}(\mathbb{D}) = \{\sum_{n\geq 0} c_nT^n \mid c_nr^n\rightarrow 0\ \forall r<1\}\subset \mathbb{C}[[T]].$$
We will explain the structure of condensed $\mathbb{C}$-module on $\mathcal{O}(\mathbb{D})$ and prove the following.

\begin{theorem}
For each $0<p\leq 1$, the space $\mathcal{O}(\mathbb{D})$ is $p$-liquid, and is flat with respect to the $p$-liquid tensor product.  Moreover, for $k\geq 0$, the $k$-fold $p$-liquid tensor product
$$\mathcal{O}(\mathbb{D})^{\otimes_{\mathbb{C}_{<p}} k}$$
identifies with the space of holomorphic functions on the unit polydisk $\mathbb{D}^k\subset\mathbb{C}^k$,
$$\mathcal{O}(\mathbb{D}^k) = \{\sum_{n_1,\ldots,n_k\geq 0} c_{n_1,\ldots,n_k}T_1^{n_1}\ldots T_k^{n_k} \mid c_{n_1,\ldots,n_k}r_1^{n_1}\ldots r_k^{n_k}\rightarrow 0 \forall r_1,\ldots,r_k<1\}\subset \mathbb{C}[[T_1,\ldots,T_n]].$$
\end{theorem}

\begin{remark} Although this theorem looks very relevant to our purposes, it turns out that we (probably) won't actually use it in the remainder of the course!  We're including it here because another purpose of this class is to present some basic information about the liquid category and how to work in it. \end{remark}

Note that, because of the $\forall r<1$ in the definition, $\mathcal{O}(\mathbb{D})$ is an intersection, or inverse limit, of more basic spaces of sequences with some given convergence criterion. (We will be more explicit about this later.)  The essential content of the above theorem is that we can commute these inverse limits past the liquid tensor product.  This illustrates that the $p$-liquid tensor product is ``complete enough'': even though it's only designed to commute with colimits in each variable, it actually also commutes with many inverse limits arising in practice, though for highly non-formal reasons.

Another notable aspect of this theorem is that the $p$-liquid tensor product is, in this example, independent of $p$, and simply gives the natural multi-variable analog of the given space of functions.  This is also far from formal, and in fact it fails for spaces of continuous functions.  Indeed, if $S$ and $T$ are infinite profinite sets, then
$$C(S;\mathbb{R})\otimes_{\mathbb{R}_{<p}}C(T;\mathbb{R}) \neq C(S\times T;\mathbb{R}),$$
for all $p$, and the left-hand side depends on $p$.  The problem is that the topology on spaces of continuous functions comes from the sup norm, which morally corresponds to taking $p=\infty$, so to have the desired K\"{u}nneth result we would need to take a tensor product over ``$\mathbb{R}_\infty$''.  As discussed in the previous lecture, this doesn't exist as an analytic ring. However, replacing $\mathbb{R}$ by $\mathbb{Q}_2$ for an arbitrary prime $2$, such a theory with $p=\infty$ does exist: the solid theory.  And for the solid tensor product, the K\"{u}nneth formula for continuous functions on profinite sets does hold.

In principle, to perform the tensor product calculation in the theorem, we should resolve $\mathcal{O}(\mathbb{D})$ by the basic liquid spaces $\mathcal{M}_{<p}(S)$ and then use the calculational properties (1),(2) of the liquid tensor product to compute.  But in practice, it's much easier to follow a more indirect route. In fact, we will start with something which looks completely unrelated: spaces of measures on \emph{locally profinite} topological spaces.

Before we explain what we mean by this, let's explain what we don't mean.  For a profinite set $S$, the space of measures $\mathcal{M}_{<p}(S)$ can alternatively be described as the $p$-liquidification of $\mathbb{Z}[S]$ since, by definition, $\mathcal{M}_{<p}(S)$ is the free $p$-liquid space on $S$.  But now for any condensed set $X$, we can also take the $p$-liquidification of $\mathbb{Z}[X]$ and consider this as a kind of space of measures on $X$.

But these aren't the kind of measures we'll want.  For example, if $X$ is the filtered union of its profinite subsets, then
$$\mathbb{Z}[X]^{\mathrm{liq}_p} = \bigcup_{S\subset X, S\in\operatorname{Prof}} \mathcal{M}_{<p}(S),$$
and this corresponds to not arbitrary $<p$-summable measures on $X$ (whatever that means), but to \emph{compactly supported} $<p$-summable measures on $X$.

However, if $X$ is \emph{locally profinite} --- equivalently, $X$ is Hausdorff and the compact open subsets of $X$ form a basis for the topology --- then we can directly define a space of (non-compactly supported) measures on $X$ by a simple generalization of the case of profinite $X$.  Namely, a measure on $X$ is a function
$$\mu:\{U\subset X\mid U \text{ compact open }\}\rightarrow\mathbb{R}$$
such that $\mu(\emptyset)=0$ and $\mu(U\sqcup V)=\mu(U)+\mu(V)$ for disjoint $U$ and $V$.  For $0<p\leq 1$, such a $\mu$ is called a $p$-measure, or a measure of bounded $p$-variation, or a $p$-summable measure, if there is a $C>0$ such that
$$\sum_{i\in I}|\mu(U_i)|^p\leq C$$
whenever $(U_i)_{i\in I}$ is a finite set of disjoint compact open subsets of $X$.

The set of $p$-measures is denoted $\mathcal{M}_p(X)$, and we give it the structure of a condensed set (and in fact, condensed $\mathbb{R}$-module) by
$$\mathcal{M}_p(X)=\bigcup_{C>0}\mathcal{M}(X)_{\ell^p\leq C},$$
where $\mathcal{M}(X)_{\ell^p\leq C}$ is the closed (hence compact Hausdorff) subset of $\prod_{U \text{ compact open }} [-C^{1/p},C^{1/p}]$ specified by the above conditions of finite additivity and having $p$-variation $\leq C$.  Finally, we set
$$\mathcal{M}_{<p}(X)=\bigcup_{q<p}\mathcal{M}_q(X).$$

\begin{warning} These spaces of measures $\mathcal{M}_{<p}(X)$ are not covariantly functorial for arbitrary continuous maps $X\rightarrow Y$, but only for \emph{proper} maps.\end{warning}

Generalizing our discussion for profinite $X$, we have the following.

\begin{proposition}\label{locprofflat}
Let $X$ be a locally profinite space and $0<p\leq 1$.  Then:
\begin{enumerate}
\item $\mathcal{M}_{<p}(X)$ is $p$-liquid;
\item $\mathcal{M}_{<p}(X)$ is flat with respect to the $p$-liquid tensor product;
\item If $X,Y$ are both locally profinite, then there is a natural identification
$\mathcal{M}_{<p}(X)\otimes_{\mathbb{R}_{<p}}\mathcal{M}_{<p}(Y)=\mathcal{M}_{<p}(X\times Y)$.
\end{enumerate}
\end{proposition}

The proof will use the following two lemmas.  The first lemma says that we can express these measures on locally profinite sets simply in terms of measures on profinite sets, by means of the one-point compactification.

\begin{lemma}
Let $X$ be locally profinite, and let $X\cup\infty\in\operatorname{Prof}$ be the one-point compactification of $X$.  Then there are natural identifications
$$\mathcal{M}_{<p}(X) = \mathcal{M}_{<p}(X\cup\infty)/\mathcal{M}_{<p}(\infty) = \operatorname{ker}\left(\mathcal{M}_{<p}(X\cup\infty)\rightarrow \mathcal{M}_{<p}(\ast)\right).$$
\end{lemma}
\begin{proof}
First note that by functoriality, the composition
$$\mathbb{R}=\mathcal{M}_{<p}(\infty)\rightarrow \mathcal{M}_{<p}(X\cup\infty)\rightarrow\mathcal{M}_{<p}(\ast)=\mathbb{R}$$
is the identity, so we get a direct summand decomposition of $\mathcal{M}_{<p}(X\cup\infty)$, and in particular the last two expressions in the statement of the lemma are equal.  Now we will identify the two outer expressions.

The clopen subsets of $X\cup\infty$ can be categorized into two types: \begin{enumerate}
\item those which lie in $X$: these correspond bijectively to compact open subsets of $X$;
\item those which contain $\infty$: these correspond (by intersecting with $X$) bijectively to the \emph{complements} of compact open subsets of $X$.  
\end{enumerate}
Thus, to promote a measure on $X$ to a measure on $X\cup\infty$ amounts to additionally specifying values $\mu(V)$ for subsets $V\subset X$ which are complements of compact opens, such that additivity and bounded $p$-variation hold.

However, if the measure is to lie in $\operatorname{ker}\left(\mathcal{M}_{<p}(X\cup\infty)\rightarrow \mathcal{M}_{<p}(\ast)\right)$, then the total measure has to be $0$, which means that the value on the complement of a compact open $U$ must be equal to $-\mu(U)$.  This sets up a bijection between finitely additive measures on compact open subsets of $X$ and finitely additive measures on clopen subsets of $X\cup\infty$ with total measure $0$.  Moreover it is easy to see with the triangle inequality that this bijection preserves $\ell^q$-bound up to multiplication by $2$; and the resulting bijections on sets with bounds are continuous with respect to the compact Hausdorff topologies.  This gives the claim.
\end{proof}

The next lemma says that the choice of compactification is immaterial in the first identification above.

\begin{lemma}
Let $f:S\twoheadrightarrow T$ be a surjection in $\operatorname{Prof}$, let $\partial T\subset T$ be a closed subset, and set $\partial S = f^{-1}\partial T$.   If $f:S\setminus\partial S\overset{\sim}{\rightarrow} T\setminus \partial T$ (so that $S$ and $T$ can be viewed as two different compactifications of the same locally profinite space $S\setminus \partial S$), then
$$\mathcal{M}_{<p}(S)/\mathcal{M}_{<p}(\partial S)\overset{\sim}{\rightarrow} \mathcal{M}_{<p}(T)/\mathcal{M}_{<p}(\partial T).$$
\end{lemma}
\begin{proof}
Comparing universal properties of mapping out, it suffices to show that
$$S\sqcup_{\partial S}\partial T = T$$
in the category of condensed sets.\footnote{In fact, this pushout also holds in condensed \emph{anima}, a fact which was implicitly used in the proof of \ref{prop:lightlyprojective}.  To deduce this from the same fact in condensed sets, it suffices to see that the pushout $S\sqcup_{\partial S}\partial T$ in condensed anima is a condensed set.  In general, the pushout in condensed anima is given by sheafifying the presheaf pushout.  Sheafification of anima sends sets to sets because it's built from filtered colimits and limits, so it suffices to see that the pushout $S\sqcup_{\partial S}\partial T$ in presheaves of anima is a presheaf of sets.  But this follows because the map $\partial S\rightarrow S$ is injective.}  But indeed the equivalence relation $S\times_T S\subset S\times S$ giving $T$ as a quotient of $S$ is covered by the diagonal $S\rightarrow S\times_T S$  and $\partial S \times_{\partial T}\partial S$ because of the hypothesis $f:S\setminus\partial S\overset{\sim}{\rightarrow} T\setminus \partial T$, and this gives the claim by comparing universal properties of mapping out.
\end{proof}

Now we prove Proposition \ref{locprofflat}.

\begin{proof}
$\mathcal{M}_{<p}(X)$ is $p$-liquid and flat because the first lemma exhibits it as a summand of $\mathcal{M}_{<p}(X\cup\infty)$, for which we know these facts by the previous lecture.  If we calculate $\mathcal{M}_{<p}(X)\otimes_{\mathbb{R}_{<p}}\mathcal{M}_{<p}(Y)$ using the presentations
$$\mathcal{M}_{<p}(X)=\mathcal{M}_{<p}(X\cup \infty)/\mathcal{M}_{<p}(\infty),$$
$$\mathcal{M}_{<p}(Y)=\mathcal{M}_{<p}(Y\cup \infty)/\mathcal{M}_{<p}(\infty)$$
coming from the first lemma, we obtain
$$\mathcal{M}_{<p}(X)\otimes_{\mathbb{R}_{<p}}\mathcal{M}_{<p}(Y) = \mathcal{M}_{<p}((X\cup \infty)\times (Y\cup\infty))/\mathcal{M}_{<p}(X\times\infty \cup \infty\times Y),$$
but this is the same as $\mathcal{M}_{<p}((X\times Y)\cup \infty)/\mathcal{M}_{<p}(\infty)$ by the second lemma.\end{proof}

Let us specialize to our case of main interest.  Let $I$ be a countable set, which we view as a locally profinite space with the discrete topology.  In this case, compact open subsets are the same as finite subsets, and a finitely additive measure on these is uniquely specified by its values on the singeton sets, which can be arbitrary.  Thus we find that
$$\mathcal{M}(I)_{\ell^p\leq C} = \{(x_i)_{i\in I}\mid x_i\in\mathbb{R}, \sum_i |x_i|^p\leq C\},$$
as a closed subset of $\prod_I [-C^{1/p},C^{1/p}]$.   Thus $\mathcal{M}_p(I)$ is simply the set of absolutely $p$-summable sequences of real numbers, equipped with a condensed structure where the ``closed unit balls'' for the $p$-norms are given their natural compact Hausdorff topology.

In fact we also have analogous ``sequence spaces'' $\mathcal{M}_{p}(I)
$ for any $0<p\leq \infty$, not necessarily $\leq 1$, defined by the same description, producing a whole series:

$$\left(\ldots \subset \mathcal{M}_{1/3}(I) \subset \mathcal{M}_{1}(I)\subset \mathcal{M}_{2}(I)\subset \ldots \subset \mathcal{M}_{\infty}(I)\right)\subset \prod_I \mathbb{R},$$
where the largest one is $\mathcal{M}_{\infty}(I) = \bigcup_{C>0}\prod_I [-C,C]$.

But there is more, because each of these $\mathcal{M}_p(I)$ also has a ($p$-)Banach space analog $\ell_p(I)$, which has the same underlying set
$$\ell_p(I) = \{(x_i)_{i\in I}\mid x_i\in\mathbb{R}, \sum_i |x_i|^p< \infty\},$$
but with topology, hence condensed structure, where a basis of neighborhoods of $0$ is given by the open balls
$$\ell_p(I)_{<\epsilon} = \{(x_i)_{i\in I}\mid x_i\in\mathbb{R}, \sum_i |x_i|^p< \epsilon\}$$
for $\epsilon>0$.  This produces another series of sequence spaces
$$\left(\ldots \subset \ell_{1/3}(I) \subset \ell_{1}(I)\subset \ell_{2}(I)\subset \ldots \subset \ell_\infty(I)\right)\subset \prod_I \mathbb{R},$$
which includes termwise in the previous one by maps which are bijections on underlying sets, but not isomorphisms of condensed sets.\footnote{Actually, in this Banach picture, it might be preferable to replace $\ell_\infty(I)$ by its subspace $c_0(I)$, the set of null-sequences, again with topology coming from the sup norm.  The reason has to do with duality; see the exercises.}

The fact that all of these sequence spaces are pairwise non-isomorphic as condensed $\mathbb{R}$-modules encapsulates much of the subtlety of linear algebra in infinite dimensions: first, there are different convexity types parametrized by $p$, and second, there is also a difference between taking compact subsets as fundamental and taking open subsets as fundamental.  However, as noticed and studied in detail by Grothendieck, there are certain maps which completely erase these subtle distinctions.  These are called \emph{nuclear},  \emph{trace-class}, or \emph{($p$-)summable} maps.  The simplest example is the following.  Suppose given a sequence of real numbers $(\lambda_i)_{i\in I}$ which is $p$-summable, i.e.\ $\sum_i |\lambda_i|^p<\infty$.  Then for \emph{any} of the above sequence spaces $V$, we get an endomorphism
$$\cdot\lambda : V\rightarrow V,$$
given by multiplying the $i^{th}$ entry by the scalar $\lambda_i$.  The simple but important phenomenon is the following.

\begin{lemma}\label{basictraceclass}
Let $I$ be a countable infinite set, and let $\lambda = (\lambda_i)_{i\in I}$ be a $p$-summable sequence of real numbers.  Then the map of condensed $\mathbb{R}$-modules
$$\cdot\lambda: \mathcal{M}_{\infty}(I)\rightarrow \mathcal{M}_{\infty}(I)$$
lands inside $\ell^p(I)\subset \mathcal{M}_{\infty}(I)$.
\end{lemma}
\begin{proof}
On underlying sets, this is clear: if we multiply a uniformly bounded sequence by a $p$-summable sequence, we still get an $p$-summable sequence.  To see that the factoring holds as condensed sets, we need that for all $C>0$, the map
$$\prod_I [-C;C]\rightarrow \ell^p(I),$$
$$(x_i) \mapsto (\lambda_ix_i)$$
is continuous.  But if we want to guarantee that
$$\sum_i |\lambda_i x_i- \lambda_i y_i|^p= \sum_i|\lambda_i|^p|x_i-y_i|^p<\epsilon,$$
we can arrange this by taking $x_i$ and $y_i$ to be close enough for $i$ in some finite subset $I_0\subset I$, since the remainder
$$\sum_{i\not\in I_0} |\lambda_i|^p|x_i-y_i|^p \leq 2^pC^p\sum_{i\not\in I_0} |\lambda_i|^p$$
goes to $0$ as $I_0$ grows.\footnote{Compare with the fact that the identity map $\mathcal{M}_p(I)\rightarrow \ell^p(I)$ is \emph{not} continuous: if all we know is that $\sum_i |x_i|^p\leq C$, then we cannot guarantee $\sum_i |x_i|^p<\epsilon$ with just a condition on finitely many entries.}
\end{proof}

Thus this map of multiplication by a $p$-summable $\lambda$ takes us all the way from the largest measure sequence space to the $p$-Banach one.  In fact, we have run into a similar phenomenon before, when discussing the space of holomorphic functions on the disk.  We defined this as

$$\mathcal{O}(\mathbb{D}) = \{\sum_{n\geq 0} c_nT^n \mid c_nr^n\rightarrow 0 \forall r<1\}\subset \mathbb{C}[[T]].$$

We could replace the condition $c_nr^n\rightarrow 0$ by any number of other convergence conditions on the sequence $c_nr^n$ without affecting the definition: we could say $c_nr^n$ is bounded, is $1/3$-summable, is $2$-summable, etc.  For fixed $r<1$, there is of course a difference between these convergence criteria; but since we ask it for all $r<1$, they are equivalent.  Indeed, suppose we take the weakest one, where we just ask that the $c_nr^n$ be uniformly bounded.  If this holds for all $r<1$, then in particular it holds for some $r'>r$.  Thus $c_nr^n$ is equal to a uniformly bounded sequence $c_nr'^n$ multiplied by the exponentially decaying sequence $(r/r')^n$, hence it has whatever other summability property we could ask.

This leads us to consider the following situation.  Suppose given, for every $n\in\mathbb{N}$, an $I$-sequence
$$\lambda^{(n)} = (\lambda^{(n)}_i)_{i\in I}$$
of real numbers which is $q$-summable for some $q<p$, with $0<p\leq 1$, and consider the tower of condensed $\mathbb{R}$-modules
$$\ldots\mathcal{M}_{<p}(I)\overset{\cdot \lambda^{(n)}}{\rightarrow}\mathcal{M}_{<p}(I)\overset{\cdot \lambda^{(n-1)}}{\rightarrow}\ldots\overset{\cdot \lambda^{(1)}}{\rightarrow} \mathcal{M}_{<p}(I).$$
For a reason which will become clear in a moment, we will also ask that all $\lambda^{(n)}_i$ be nonzero.  We are interested in studying the inverse limit (=intersection, as the maps are injective)
$$V = \varprojlim_n \mathcal{M}_{<p}(I)$$
of this tower.  (This $V$ depends on the $\lambda^{(n)}$, but we suppress this from the notation.)  We make a series of remarks about this situation:

\begin{enumerate}
\item Up to a pro-isomorphism of the tower, we can assume that all the $\lambda^{(n)}$ are $q$-summable, for \emph{any} given fixed $q<p$.  Indeed, this can be arranged by composing maps in the tower, for the following reason: all the $\lambda^{(n)}$ are $p$-summable, but the Cauchy-Schwartz inequality implies that if $\lambda$ and $\lambda'$ are $p$-summable, then the termwise product $\lambda\cdot\lambda'$ is $p/2$-summable.
\item Again up to a pro-isomoprhism of the tower, we can replace $\mathcal{M}_{<p}(I)$ by any of the above discussed sequence spaces, either of measure-type or of Banach-type.  This follows from the previous remark and Lemma \ref{basictraceclass}.
\item We have $\varprojlim_n^1 \mathcal{M}_{<p}(I)=0$, so that $V=R\varprojlim_n\mathcal{M}_{<p}(I)$.\footnote{Here it is important that infinite products in $\operatorname{CondAb}$ are exact, so the standard Milnor sequence does calculate the $R\varprojlim_n$'s.} Indeed, by the previous remark, we can replace the $\mathcal{M}_{<p}(I)$ by Banach sequence spaces; and moreover the transition maps have dense image, because they contain the basis vectors for these sequence spaces due to the hypothesis that all $\lambda^{(n)}_i\neq 0$.  Thus this $\varprojlim_n^1$ claim follows from the Mittag-Leffler lemma below.
\item For any $p$, there is a short exact sequence of condensed $\mathbb{R}$-modules
$$0\rightarrow V\rightarrow \prod_\mathbb{N} \mathcal{M}_p(I)\overset{\sigma - id}{\longrightarrow}  \prod_\mathbb{N} \mathcal{M}_p(I)\rightarrow 0,$$
where the map $\sigma$ shifts the components using the given transition maps $\cdot \lambda^{(n)}$.  This is a rephrasing of the vanishing of $\varprojlim^1$ for the tower with terms the $\mathcal{M}_p(I)$, which is the same as the previous claim because all the different sequence spaces yield isomorphic pro-towers.
\item There is a short exact sequence of liquid $p$-modules
$$0\rightarrow V\rightarrow \bigcup_{q<p}\prod_\mathbb{N} \mathcal{M}_q(I)\overset{\sigma - id}{\longrightarrow} \bigcup_{q<p} \prod_\mathbb{N} \mathcal{M}_q(I)\rightarrow 0.$$
This follows by taking filtered colimits in the previous.
\end{enumerate}

We used the following Mittag-Leffler lemma above:

\begin{lemma}
Suppose given a tower $\ldots \rightarrow V_n\rightarrow V_{n-1}\rightarrow \ldots\rightarrow V_1$ of Banach spaces such that the transition maps have dense image.  Then
$$R^1\varprojlim_n V_n=0$$
in $\operatorname{CondAb}$.
\end{lemma}
\begin{proof}
We have to see that for all $S\in\operatorname{ExtrDisc}$, we have
$$R^1\varprojlim_n V_n(S)=0$$
in the category of abelian groups.  When $S=\ast$, this follows from the standard topological Mittag-Leffler result, \cite{grothendieck1961elements} 13.2.4.  But in fact the general case reduces to that one: each $V_n(S)$ also has a Banach space structure with respect to the sup norm, and each transition map $V_n(S)\rightarrow V_{n-1}(S)$ is still dense, as the locally constant functions are dense in the continuous functions $S\rightarrow V_{n-1}$, and locally constant functions lift up to $\epsilon$ by our original density hypothesis.
\end{proof}

It is this short exact sequence in (5) we will use to analyze $p$-liquid tensor products with $V$.  In fact, we have the following.

\begin{lemma}
Let $A$ and $B$ be two countable sets, and let $(X_a)_{a\in A}$ and $(Y_b)_{b\in B}$ be locally profinite spaces parametrized by $A$ and $B$ respectively.  Then
$$\left(\bigcup_{q<p} \prod_{a\in A} \mathcal{M}_q(X_a) \right)\otimes_{\mathbb{R}_{<p}}\left(\bigcup_{q<p} \prod_{b\in B} \mathcal{M}_q(Y_b) \right) = \bigcup_{q<p} \prod_{(a,b)\in A\times B} \mathcal{M}_q(X_a \times Y_b),$$
and all of these $p$-liquid spaces are flat.
\end{lemma}
\begin{proof}
We will prove this by writing all the objects as filtered unions of measure spaces on locally profinite sets.  First note that
$$\mathcal{M}_{<p}(\bigsqcup_{a\in A}X_a)\subset \prod_{a\in A}\mathcal{M}_{<p}(X_a)$$
because by additivity a measure on $\bigsqcup_{a\in A}X_a$ is determined by its values on those compact opens which are contained in some $X_a$.  Now, given an $A$-sequence $C=(C_a)_{a\in A}$ of positive real numbers, let us write $$C\cdot \mathcal{M}_{<p}(\bigsqcup_{a\in A}X_a)\subset \prod_{a\in A}\mathcal{M}_{<p}(X_a)$$
for the isomorphic copy of $\mathcal{M}_{<p}(\bigsqcup_{a\in A}X_a)$ where we multiply the $a$-component by $C_a$.  Then we claim that
$$\bigcup_{q<p} \prod_{a\in A} \mathcal{M}_q(X_a) = \bigcup_C C\cdot \mathcal{M}_{<p}(\bigsqcup_{a\in A}X_a),$$
where the union is over all $A$-sequences of positive real numbers, which is a filtered poset under $C\leq C' \Leftrightarrow C_a\leq C'_a \forall a\in A$.

If we know this claim, then first of all we get $p$-liquidity and flatness of $\bigcup_{q<p} \prod_{a\in A} \mathcal{M}_q(X_a)$, because it is a filtered colimit of modules which we already know to be flat $p$-liquid.  But second, using the K\"{u}nneth property for measure spaces on locally profinite sets proved above and commutation of tensor products with colimits, we find that the tensor product formula we're trying to prove reduces to the following fact: any $A\times B$-sequence $E=(E_{a,b})$ of positive real numbers is termwise $\leq$ one of the form $C_a\cdot D_b$ where $C$ is an $A$-sequence and $D$ is a $B$-sequence.  To prove that, choose bijections $A,B\simeq \mathbb{N}$ and let
$$C_n = D_n = \max\{\sup_{i,j\leq n} E_{i,j},1\}.$$

Thus we've reduced to proving
$$\bigcup_{q<p} \prod_{a\in A} \mathcal{M}_q(X_a) = \bigcup_C C\cdot \mathcal{M}_{<p}(\bigsqcup_{a\in A}X_a).$$
By taking the union over $q$, it suffices to show that
$$\prod_{a\in A} \mathcal{M}_q(X_a) = \bigcup_C C\cdot \mathcal{M}_q(\bigsqcup_{a\in A}X_a),$$
or in other words, for any given $q$-summable measures $\mu_a$ on $X_a$ for all $a\in A$, there is an $A$-sequence $\epsilon=(\epsilon_a)$ of positive constants such that
$$\mu(U) := \epsilon_a\cdot \mu_a(U)$$
for $U\subset X_a$ compact open specifies a $q$-summable measure on $\bigsqcup_{a\in A} X_a$.  But indeed, if $\mu_a$ is of $q$-variation $\leq C_a$, then we can arrange this by taking $\epsilon_a = \lambda_a\cdot C_a^{-1}$ where $\lambda_a$ is any $q$-summable $A$-sequence.
\end{proof}

\begin{remark} The countability hypothesis on $A$ and $B$ is crucial here; see the exercises.  \end{remark}

From this we deduce the following.

\begin{theorem}
If $V=\varprojlim_n \mathcal{M}_{<p}(I)$ and $V'=\varprojlim_n \mathcal{M}_{<p}(I')$ are both presented as inverse limits along towers where each transition map is given by termwise multiplication by some $q$-summable $I$-sequence of nonzero real numbers for some $q<p$, then $V$ and $V'$ are flat $p$-liquid modules, and we can pull the inverse limit out of the tensor product calculation, getting
$$V\otimes_{\mathbb{R}_{<p}} V' = \varprojlim_n \mathcal{M}_{<p}(I\times I'),$$
where the $n^{th}$ transition map on the right is given by termwise multiplication by the $I\times I'$-sequence which is the product of sequences giving the $n^{th}$ transition maps for $V$ and $V'$.
\end{theorem}
\begin{proof}
Remark (5) above writes $V$ as the kernel of a map between two $p$-liquid modules which are flat by the above lemma, thus giving the flatness claim.  Then to calculate the tensor product, we can use remark (5) and the tensor product calculation given by the lemma.
\end{proof}

\begin{remark} More generally, the same conclusion (that the liquid tensor product commutes with certain sequential inverse limits) applies in the more general context of nuclear Fr\'echet spaces (cf.~also Lecture VIII).  Suppose $V$ is a nuclear Fr\'echet space; thus $V$ can be written as a sequential inverse limit of Hilbert spaces along trace-class transition maps with dense image.  Each trace-class map between Hilbert spaces factors through a map on the standard $\ell^2$-sequence space which can given by termwise multiplication by a $1$-summable sequence. Then as in the above discussion we deduce that the original tower of Hilbert spaces is pro-isomorphic to a tower of $\mathcal{M}_{p}$-spaces such that all the transition maps factor through $\mathcal{M}_{p/2}\subset \mathcal{M}_{p}$. Given this, the same argument as above shows that nuclear Fr\'echet spaces are flat for the liquid tensor product, and that in taking the liquid tensor product of two nuclear Fr\'echet spaces, we can commute the defining inverse limits past the tensor product; thus, it agrees with the usual complete tensor product of nuclear Fr\'echet spaces.
\end{remark}

We can now specialize this to the case of holomorphic functions on disks.  Well, there is one small wrinkle: this discussion was over $\mathbb{R}$, whereas the tensor product of holomorphic functions we want to take is over $\mathbb{C}$.  But this is not serious: if $\mathcal{O}(\mathbb{D}^k)_\mathbb{R}$ denotes the $\mathbb{R}$-subspace of $\mathcal{O}(\mathbb{D}^k)$ where the coefficients are required to lie in $\mathbb{R}$, then we can write
$$\mathcal{O}(\mathbb{D}^k) = \mathcal{O}(\mathbb{D}^k)_\mathbb{R} \otimes_\mathbb{R}\mathbb{C},$$
so we can just as well do the calculation over $\mathbb{R}$ instead.  And then it falls into the purview of the above theorem.  Indeed, choosing any increasing sequence $(r_n)$ of real numbers $0<r_n<1$ such that $r_n\rightarrow 1$ as $n\rightarrow\infty$, we can write
$$\mathcal{O}(\mathbb{D}^k)_\mathbb{R} = \varprojlim_n \mathcal{M}_{<q}(\mathbb{N}^k),$$
where the $n^{th}$ transition map is multiplication by the exponentially decaying $\mathbb{N}^k$-sequence
$$(d_1,\ldots d_k)\mapsto (r_{n-1}/r_n)^{d_1+\ldots + d_k}.$$
We deduce the desired claims: $\mathcal{O}(\mathbb{D}^k)$ is flat as a $p$-liquid space for all $k$ and $p$, and
$$\mathcal{O}(\mathbb{D}^k)\otimes_{\mathbb{C}_{<p}}\mathcal{O}(\mathbb{D}^l) = \mathcal{O}(\mathbb{D}^{k+l}),$$
independent of $p$.\\

\textbf{Exercise 1.}  We used sequence spaces to produce the condensed structure on $\mathcal{O}(\mathbb{D})$.  Prove the following more direct description: for a profinite set $S$, we have
$$\mathcal{O}(\mathbb{D})(S) = \{f:S\times\mathbb{D}\rightarrow\mathbb{C}\mid f \text{ is continuous}, f(s,-)\text{ is holomorphic } \forall s\in S\}.$$
\\
\textbf{Exercise 2.}   Let $0<p\leq \infty$ and $I$ a countable set, and consider the internal hom in condensed $\mathbb{R}$-modules
$$\mathcal{M}_p(I)^\vee=\underline{\Hom}(\mathcal{M}_p(I),\mathbb{R}).$$
Show that:
\begin{enumerate}
\item If $p\leq 1$, then $\mathcal{M}_p(I)^\vee=c_0(I),$ the Banach space of null-sequences with sup norm.
\item If $p>1$, then $\mathcal{M}_p(I)^\vee = \ell_q(I)$ where $1\leq q< \infty$ is such that $\frac{1}{p}+\frac{1}{q}=1$.\\
\end{enumerate}

\textbf{Exercise 3.}   Let $0<p\leq \infty$ and $I$ a countable set, and consider the internal hom in condensed $\mathbb{R}$-modules
$$\ell_p(I)^\vee=\underline{\Hom}(\ell_p(I),\mathbb{R}).$$
Show that:
\begin{enumerate}
\item If $p\leq 1$, then $\ell_p(I)^\vee=\mathcal{M}_\infty(I)$.
\item If $p\geq 1$, then $\ell_p(I)^\vee = \mathcal{M}_q(I)$ where $1\leq q\leq \infty$ is such that $\frac{1}{p}+\frac{1}{q}=1$.\\
\end{enumerate}

For the next two exercises, let's define a qs condensed $\mathbb{R}$-module $V$ to be \emph{countably $p$-liquid} if for any countable collection $K_1,K_2,\ldots$ of qc subsets of $V$ and any $q<p$, there is a qc $q$-convex $K\subset V$ such that $K_n\subset \mathbb{R}\cdot K$ for all $n$.  This is a strengthening of the condition saying that $V$ is $p$-liquid.\\

\textbf{Exercise 4.}  Show that any $p$-Banach space is countably $p$-liquid, the measure space $\mathcal{M}_p(S)$ for profinite $S$ is countably $p$-liquid, and any countable inverse limit of countably $p$-liquid spaces is countably $p$-liquid.  Show also that $\mathcal{M}_{<p}(S)$ is not countably $p$-liquid, and the countably $p$-liquid spaces are not closed under countable direct sum.\\

\textbf{Exercise 5.}  Show that if $W$ is qs and countably $p$-liquid, then for any $V=\varprojlim_n \mathcal{M}_{<p}(I)$ with transition maps given by nowhere zero $q$-summable $\lambda^{(n)}$ as in the lecture, then we can pull out the inverse limit from the tensor product with $W$:
$$W\otimes_{\mathbb{R}_{<p}} V = \varprojlim_n W\otimes_{\mathbb{R}_{<p}} \mathcal{M}_{<p}(I).$$
Moreover the $\varprojlim^1$ term vanishes. In fact, the same statement applies more generally to nuclear Fr\'echet spaces $V$.\\

\textbf{Exercise 6.}  Show that if $I,J$ are infinite sets with $I$ uncountable, then
$$\left( \prod_I \mathbb{R}\right)\otimes_{\mathbb{R}_{<p}} \left( \prod_J \mathbb{R}\right)\neq \prod_{I\times J}\mathbb{R},$$
i.e.\ the natural map is not an isomorphism.\\

\textbf{Exercise 7.}  Let $V$ be qs $p$-liquid.  Show that
$$V\otimes_{\mathbb{R}_{<p}} \mathcal{O}(\mathbb{D}) = \bigcup_{K\subset V \text{ qc}} \bigcap_{r<1} \bigcup_{C>0} \prod_{n\geq 0} \frac{C}{r^n} K\cdot T^n \subset V[[T]].$$\newpage

\section{Lecture V: From formal nonsense to holomorphic functions}

The goal of this lecture is to prove Theorem~\ref{thm:holomorphicdefinition}:

\begin{theorem}\label{thm:holomorphicdefinitionredux} There is a (necessarily unique) sheaf $\mathcal O$ (of $p$-liquid $\mathbb C$-algebras) on the topological space $\mathbb C$, such that for any open disc
\[
D=D(x,r)=\{z\in \mathbb C\mid |z-x|<r\}\subset \mathbb C,
\]
one has
\[
\mathcal O(D)=\{\sum_{n=0}^\infty a_n (T-x)^n\mid a_n\in \mathbb C, \forall r'<r, a_n r'^n\to 0\},
\]
with obvious transition maps $\mathcal O(D)\to \mathcal O(D')$ for $D'\subset D$.

Moreover, for any such disc $D$, the sheaf cohomology groups $H^i(D,\mathcal O)=0$ for $i>0$.
\end{theorem}

The proof of this theorem will actually not require the delicate commutation of liquid tensor products with infinite products. But we still need the idea that multiplication by rapidly decaying sequences erases the distinction between all different kinds of sequence spaces.

Let us start with the formal nonsense. It will later be applied to $D(\mathrm{Liq}_p(\mathbb C[T]))$, the derived category of $p$-liquid $\mathbb C[T]$-modules.\footnote{This abstract discussion is closely related to the work of Ben-Bassat--Kremnizer, who emphasized the importance of idempotent algebras (in their language ``homotopy epimorphisms'') in derived analytic geometry, see for example \cite{BenBassatKremnizer}, \cite{BambozziBenBassatKremnizer}. In fact, already in 1972 Taylor \cite{Taylor} suggested very similar ideas; one could say that in this lecture we simply rephrase his ideas in modern language. We believe that one could also carry out the proof of Theorem~\ref{thm:holomorphicdefinition} in the language of Ben-Bassat--Kremnizer, i.e.~without liquid modules, but instead using their (stable $\infty$-category) $\mathcal D(\mathrm{Ind}(\mathrm{Banach}))$ freely generated by Banach spaces, and their site of homotopy epimorphisms, but we have not investigated the details.}

\begin{construction} Let $C$ be a closed symmetric monoidal stable $\infty$-category with all colimits. (The following construction only depends on its homotopy category, which is a closed symmetric monoidal triangulated category with all direct sums, and for most of the discussion such data would be good enough; but some of the later things work better with stable $\infty$-categories.) Consider the collection of idempotent algebra objects in $C$; equivalently, these are pairs of an object $A\in C$ together with a map $1\to A$ such that the induced map $A\to A\otimes A$ is an isomorphism. (Indeed, in that case $A$ admits a unique and functorial $E_\infty$-algebra structure.) For any such $A$ and $A'$, there is at most one map $A\to A'$ commuting with the maps from $1$. Indeed, if there is such a map $f: A\to A'$, then $1\to A\xrightarrow{f} A'\to A\otimes A'$ (where the last map is obtained by tensoring $1\to A$ with $A'$) induces via tensoring with $A'$ maps
\[
A'\to A\otimes A'\to A'\otimes A'=A'\to A\otimes A'\otimes A'=A\otimes A',
\]
where the composite of the first two (resp.~the last two) maps is an isomorphism. It follows that all maps are isomorphisms, and hence $A'$ maps isomorphically to $A\otimes A'$, and the map $f$ is given by the canonical map $A\to A\otimes A'\cong A'$. In particular, the $\infty$-category of such idempotent algebras is naturally a poset.

Actually, we get a ``locale''. Recall that a locale is a ``pointless topological space'': One directly axiomatizes the collection of closed subsets, subject to the condition that one can form finite unions and arbitrary intersections, satisfying some simple conditions (notably infinite intersections distribute over unions with a fixed closed subset). (Usually, one rather axiomatizes the collection of open subsets, but they are in canonical bijection, and in our example it feels more natural to consider the closed subsets.)

\begin{proposition} The above construction defines a locale $\mathcal S(C)$ whose closed subsets $Z\subset \mathcal S(C)$ correspond to idempotent algebras $A$, so that
\begin{enumerate}
\item $Z\cap Z'$ corresponds to $A\otimes A'$;
\item $Z\subset Z'$ if and only if $A\otimes A'=A$;
\item $Z\cup Z' = Z\sqcup_{Z\cap Z'} Z'$ corresponds to $[A\oplus A'\to A\otimes A']$;
\item $\bigcap_i Z_i$ corresponds to $\varinjlim_i A_i$.
\end{enumerate}
\end{proposition}

\begin{remark} This construction goes back at least to the work of Balmer--Favi \cite{BalmerFavi}, and has been investigated in detail by Balmer--Krause--Stevenson \cite{BalmerKrauseStevenson}. These works however generally assume that all compact objects are dualizable, which will be very false in our situation of interest.
\end{remark}

\begin{proof} All assertions reduce to direct computations. The infinite distributive law follows from the commutation of finite limits with arbitrary colimits in a stable $\infty$-category (as finite limits are also finite colimits).
\end{proof}

Now to any closed $Z\subset \mathcal S(C)$, corresponding to an idempotent algebra $A\in C$, we can associate a full sub-$\infty$-category
\[
C(Z) = \mathrm{Mod}_A(C) = \{X\in C\mid X\cong X\otimes A\}\subset C.
\]
(Indeed, if $X$ is an $A$-module, then $X\to X\otimes A$ is an isomorphism as $A$ is idempotent; and conversely, if $X\to X\otimes A$ is an isomorphism then $X$ acquires a unique structure of $A$-module from $X\otimes A$.) We write
\[
i_{Z\ast}: C(Z)\to C
\]
for this full inclusion. It admits a left adjoint
\[
i_Z^\ast: C\to C(Z): X\mapsto X\otimes A
\]
and a right adjoint
\[
i_Z^!: C\to C(Z): X\mapsto \underline{\operatorname{RHom}}_C(A,X).
\]

Formally denoting by $U$ the ``complementary open'' of $Z$, we define the Verdier quotient
\[
C(U) = C / C(Z).
\]
This comes with a localization functor denoted
\[
j_U^\ast: C\to C(U)
\]
which admits a fully faithful left adjoint
\[
j_{U!}: C(U)\to C
\]
determined by
\[
j_{U!} j_U^\ast X = [X\to X\otimes A]
\]
and a fully faithful right adjoint
\[
j_{U\ast}: C(U)\to C
\]
determined by
\[
j_{U\ast} j_U^\ast X = \underline{\operatorname{RHom}}_C([1\to A],X).
\]

In particular, there are functorial distinguished triangles
\[
j_{U!} j_U^\ast X\to X\to i_{Z\ast} i_Z^\ast X
\]
and
\[
i_{Z\ast} i_Z^! X\to X\to j_{U\ast} j_U^\ast X.
\]
\end{construction}

In particular, this formal nonsense produces a structure sheaf on $\mathcal S(C)$:

\begin{proposition}\label{prop:structuresheafformalnonsense} The functor $U\mapsto C(U)$ defines a sheaf of $\infty$-categories on the locale $\mathcal S(C)$. In particular, for any $X\in C$, the functor
\[
U\mapsto X(U):=j_{U\ast} j_U^\ast X\in C
\]
defines a sheaf on $\mathcal S(C)$ with values in $C$. In particular, applied to $X=1$ the unit object of $C$, this gives a ``structure sheaf'' on $\mathcal S(C)$.
\end{proposition}

\begin{proof} The sheaf property can be stated in two pieces. First, for any two open $U$, $U'$, one gets a cartesian diagram
\[\xymatrix{
C(U\cup U')\ar[r]\ar[d] & C(U)\ar[d]\\
C(U')\ar[r] & C(U\cap U').
}\]
To show this, consider the corresponding idempotent algebras $A$ and $A'$. For fully faithfulness, one needs to see that for all $X\in C(U\cup U')$ (considered as a full subcategory of $C$ via $j_{U\cup U'\ast}$), the square
\[\xymatrix{
X\ar[r]\ar[d] & \underline{\operatorname{RHom}}([1\to A],X)\ar[d]\\
\underline{\operatorname{RHom}}([1\to A'],X)\ar[r] & \underline{\operatorname{RHom}}([1\to A\otimes A'],X)
}\]
is cartesian. But this reduces to $X=\underline{\operatorname{RHom}}([1\to B],X)$ where $B=[A\oplus A'\to A\otimes A']$ corresponds to $U\cup U'$. Essential surjectivity reduces to a similar straightforward computation.

The other thing to check is that for a filtered union $U=\bigcup_i U_i$, the functor
\[
C(U)\to \varprojlim_i C(U_i)
\]
is an isomorphism (where in fact all transition functors are full inclusions, so the limit is an intersection). Again, fully faithfulness reduces to $X\to \varprojlim_i \underline{\operatorname{RHom}}([1\to A_i],X)$ being an isomorphism (where $U_i$ corresponds to $A_i$). But the right-hand side is $\underline{\operatorname{RHom}}([1\to A],X)$ for $A=\varinjlim_i A_i$, which corresponds to $U$. As both categories $C(U)$ and $\varprojlim_i C(U_i)$ are full subcategories of $C$, essential surjectivity reduces to the assertion that if $X\in C$ is such that $X=\underline{\operatorname{RHom}}([1\to A_i],X)$ for all $i$, then also $X=\underline{\operatorname{RHom}}([1\to A],X)$, which follows by passage to limits.
\end{proof}

Now we want to apply this formal nonsense to $C=\mathcal D(\operatorname{Liq}_p(\mathbb C[T]))$. To get things off the ground, we need examples of idempotent algebras, which ought to correspond to the closed subsets $\{|T|\leq 1\}$ and $\{|T|\geq 1\}$. For this, we take the corresponding ``rings of overconvergent holomorphic functions''. Concretely, let
\[
A(\{|T|\leq 1\}) := \bigcup_{r>1} \{\sum_{n=0}^\infty a_n T^n\mid a_n\in \mathbb C, a_n r^n\to 0\} = \bigcup_{r>1} \mathcal M_{<p}(\{T^n/r^n\}_{n\geq 0})
\]
and
\[
A(\{|T|\geq 1\}) := \bigcup_{r<1} \{\sum_{n<<\infty} a_n T^n\mid a_n\in \mathbb C, a_n r^n\to 0\} = \bigcup_{r<1} \bigcup_{m\geq 0} \mathcal M_{<p}(\{T^n/r^n\}_{n\leq m}).
\]
The following is the only real computation we need to do.

\begin{proposition}\label{prop:computationunitdisc} The liquid $\mathbb C[T]$-algebras $A(\{|T|\leq 1\})$ and $A(\{|T|\geq 1\})$ are idempotent. The corresponding localization of $\mathbb C[T]$ is given by
\[
\mathcal O(\{|T|>1\}) = \bigcup_{m\geq 0} \bigcap_{r>1} \{\sum_{n=m}^{-\infty} a_n T^n\mid a_n\in \mathbb C, a_n r^n\to 0\} = \bigcup_{m\geq 0} \bigcap_{r>1} \mathcal M_{<p}(\{T^n/r^n\}_{n\leq m})
\]
resp.
\[
\mathcal O(\{|T|<1\}) = \bigcap_{r<1} \{\sum_{n=0}^\infty a_n T^n\mid a_n\in \mathbb C, a_n r^n\to 0\} = \bigcap_{r<1} \mathcal M_{<p}(\{T^n/r^n\}_{n\geq 0}) = \mathcal O(D(0,1)).
\]
\end{proposition}

We note that in the first case, the localization gives only those holomorphic functions that are meromorphic at $\infty$ --- generally, the algebraicity at the boundary is remembered (which is what makes GAGA theorems in the affine case possible).

\begin{proof} We first show that $A(\{|T|\leq 1\})$ is idempotent. This implies that $A(\{|T|\geq 1\})$ is idempotent as well, as it is in fact a $\mathbb C[T^{\pm 1}]$-algebra obtained from the first by base change along $T\mapsto T^{-1}$. Now, as $p$-liquid $\mathbb C$-vector spaces, the tensor product of $A=A(\{|T|\leq 1\})$ with itself is
\[
\bigcup_{r_1,r_2>1} \mathcal M_{<p}(\{T^n U^m/r_1^n r_2^m\}_{n,m\geq 0}) = \bigcup_{r>1} \{\sum_{n,m=0}^\infty a_{n,m} T^n U^m\mid a_{n,m}\in \mathbb C, a_{n,m} r^{n+m}\to 0\}
\]
where the equality again follows as all transition maps multiply by rapidly decaying sequences. Now, to compute the tensor product over $\mathbb C[T]$ instead, we need to mod out by the ideal $U-T$ (where $U-T$ is evidently a nonzerodivisor). It is now a simple exercise that
\[
\bigcup_{r>1} \{\sum_{n,m=0}^\infty a_{n,m} T^n U^m\mid a_{n,m}\in \mathbb C, a_{n,m} r^{n+m}\to 0\} / (U-T)\to \bigcup_{r>1} \{\sum_{n=0}^\infty a_n T^n\mid a_n\in \mathbb C, a_n r^n\to 0\}
\]
is an isomorphism. Indeed, for all elements in the kernel, one can explicitly write down division by $U-T$, and checks a very simple convergence condition.

It remains to compute the corresponding localizations. Here, it is enough to do it in the second case: Indeed, the first localization must be a $\mathbb C[T^{\pm 1}]$-algebra, which is necessarily obtained from the second localization via base change along $T\mapsto T^{-1}$. Thus, we need to compute
\[
\underline{\operatorname{RHom}}_{\mathbb C[T]}([\mathbb C[T]\to A(\{|T|\geq 1\})],\mathbb C[T]).
\]
Note that $\mathbb C[T]$ injects into $A(\{|T|\geq 1\})$, with quotient
\[
\bigcup_{r<1} \mathcal M_{<p}(\{T^n/r^n\}_{n<0}).
\]
Commuting the union with the dual, we get
\[
R\varprojlim_{r<1} \underline{\operatorname{RHom}}_{\mathbb C[T]}(\mathcal M_{<p}(\{T^n/r^n\}_{n<0}),\mathbb C[T])[1].
\]
To compute the dual, we first compute the dual over $\mathbb C$ (instead of $\mathbb C[T]$). Then, as $\mathcal M_{<p}(\mathbb N)$ is internally compact projective in $\mathrm{Liq}_p\cap \mathrm{CondAb}_{\omega_1}$, the internal RHom sits in degree $0$, and commutes with the infinite direct sum $\mathbb C[T] = \bigoplus_{m\geq 0}\mathbb C\cdot T^m$. Renaming $T$ by $U$ temporarily, we get
\[
\underline{\operatorname{RHom}}_{\mathbb C}(\mathcal M_{<p}(\{T^n/r^n\}_{n<0}),\mathbb C[U]) = \bigoplus_{m\geq 0} C_0(\{T^n/r^n\}_{n<0},\mathbb C)\cdot U^m.
\]
The space $C_0(\{T^n/r^n\}_{n<0},\mathbb C)$ of continuous maps vanishing at $\infty$ is naturally an $\ell_0$-sequence space with basis $(T^n/r^n)_{n>0}$. To compute the internal RHom over $\mathbb C[T]$ instead, we need to quotient by $T-U$. This will reintroduce the $T^0$-term, and in the limit over $r$ we get
\[
R\varprojlim_{r<1} C_0(\{T^n/r^n\}_{n\geq 0},\mathbb C).
\]
By the last lecture, the derived limit is concentrated in degree $0$, and the precise choice of sequence space does not matter, so we get
\[
\mathcal O(\{|T|<1\}) = \bigcap_{r<1} \{\sum_{n=0}^\infty a_n T^n\mid a_n\in \mathbb C, a_n r^n\to 0\} = \bigcap_{r<1} \mathcal M_{<p}(\{T^n/r^n\}_{n\geq 0}).\qedhere
\]
\end{proof}

In order to apply this, we need to get some understanding of $\mathcal S(\mathbb C[T]) := \mathcal S(\mathcal D(\mathrm{Liq}_p(\mathbb C[T])))$. It is impossible to describe this whole locale, as there are way too many idempotent algebras: One can give many different growth conditions on the coefficients $a_n$ that define idempotent algebras. (This makes it possible to localize to very precise growth conditions, and hence analyze algebras of power series with growth conditions that are not convergent for any value of $T$ in $\mathbb C$! In other words, the formulas still make sense, even if they do not define functions.) We will restrict to the part of the locale that can be probed with obvious variants of the above idempotent algebras.

In fact, we slightly generalize the setting, and consider any (abstract, i.e.~noncondensed) $\mathbb C$-algebra $R$, and the corresponding locale $\mathcal S(R) = \mathcal S(\mathcal D(\mathrm{Liq}_p(R)))$.\footnote{We could also consider liquid $\mathbb C$-algebras, but for the moment the whole discussion will only use the underlying abstract $\mathbb C$-algebra.} For any $f\in R$, we get the map $\mathbb C[T]\to R$ sending $T$ to $f$. This defines a map
\[
\mathcal S(R)\to \mathcal S(\mathbb C[T])
\]
(in terms of idempotent algebras, this amounts to base changing them from $\mathbb C[T]$ to $R$). In particular, we can define closed subsets
\[
\{|f|\leq 1\},\{|f|\geq 1\}\subset \mathcal S(R)
\]
as the preimages of
\[
\{|T|\leq 1\},\{|T|\geq 1\}\subset \mathcal S(\mathbb C[T]).
\]
As further notation, for any $r>0$, we write
\[
\{|f|\leq r\} := \{|f/r|\leq 1\}\subset \mathcal S(R),\ \{|f|\geq r\} := \{|f/r|\geq 1\}\subset \mathcal S(R).
\]

A priori, it is unclear how these different subsets of $\mathcal S(R)$ interact; each such question is a question about the behaviour of the corresponding idempotent algebras. There are some obvious expected properties that are simple to verify:

\begin{proposition}\label{prop:propertiessubsets} Let $f,g\in R$, $\alpha\in \mathbb C$ and $r,s>0$.
\begin{enumerate}
\item $\{|f|\leq 1\}=\bigcap_{r>1} \{|f|\leq r\}$ and $\{|f|\geq 1\}=\bigcap_{r<1} \{|f|\geq r\}$;
\item $\{|f|\leq 1\}\cup \{|f|\geq 1\}=\mathcal S(R)$;
\item for $r<1$, one has $\{|f|\leq r\}\cap \{|f|\geq 1\}=\emptyset$;
\item $\{|f|\leq 1\}\cap \{|g|\leq 1\}\subset \{|fg|\leq 1\}$;
\item $\{|f|\geq 1\}\cap \{|g|\geq 1\}\subset \{|fg|\geq 1\}$;
\item $\{|\alpha|\leq 1\} = \mathcal S(R)$ if $|\alpha|\leq 1$;
\item $\{|\alpha|\geq 1\} = \mathcal S(R)$ if $|\alpha|\geq 1$;
\item $\{|f|\leq r\}\cap \{|g|\leq s\}\subset \{|f+g|\leq r+s\}$.
\end{enumerate}
\end{proposition}

\begin{proof} Each of these follows by unraveling definitions, and is a nice exercise. Let us do some to give the spirit. For (1), it suffices to do the first in the universal case $\mathbb C[T]$ with $f=T$, and then recall that intersections of closed subsets correspond to filtered colimits of idempotent algebras, and then this follows directly from the definition of the idempotent algebras $A(\{|T|\leq 1\})$ and $A(\{|T|\geq 1\})$. For (2), again we can assume $f=T\in R=\mathbb C[T]$. The union $\{|T|\leq 1\}\cup \{|T|\geq 1\}$ corresponds to the idempotent algebra
\[
[A(\{|T|\leq 1\})\oplus A(\{|T|\geq 1\})\to A(\{|T|\leq 1\})\otimes_{\mathbb C[T]_{<p}} A(\{|T|\geq 1\})].
\]
Now one first computes that
\[
A(\{|T|\leq 1\})\otimes_{\mathbb C[T]_{<p}} A(\{|T|\geq 1\}) = \bigcup_{r>1,s<1} \{\sum_{n=-\infty}^\infty a_n T^n\mid a_n s^n\xrightarrow{n\to -\infty} 0, a_n r^n\xrightarrow{n\to \infty} 0\}
\]
are the overconvergent holomorphic functions on $\{|T|=1\}$. This is a simple analogue of the idempotence of $A(\{|T|\leq 1\})$. Splitting this according to negative and nonnegative powers of $T$, one sees that the map
\[
A(\{|T|\leq 1\})\oplus A(\{|T|\geq 1\})\to A(\{|T|\leq 1\})\otimes_{\mathbb C[T]_{<p}} A(\{|T|\geq 1\})
\]
is surjective, and that the only elements in the kernel are $\mathbb C[T]$, which gives the idempotent algebra $\mathbb C[T]$ corresponding to all of $\mathcal S(\mathbb C[T])$.

We omit the proofs of (3)--(7). For (8), it suffices to consider the universal case $\mathbb C[f,g]$, i.e.~$\mathbb C[T,U]$ with $f=T$ and $g=U$. The left-hand side corresponds to the algebra
\[
\bigcup_{r'>r,s'>s} \{\sum_{n,m\geq 0} a_{n,m} T^n U^m\mid a_{n,m} r'^n s'^m\to 0\},
\]
the right-hand side to the algebra
\[
\mathbb C[T,U]\otimes_{\mathbb C[V]} \bigcup_{t>r+s} \{\sum_n a_n V^n\mid a_n t^n\to 0\}
\]
where the map $\mathbb C[V]\to \mathbb C[T,U]$ sends $V$ to $T+U$. It is enough to construct a map
\[
\bigcup_{t>r+s} \{\sum_n a_n V^n\mid a_n t^n\to 0\}\to \bigcup_{r'>r,s'>s} \{\sum_{n,m\geq 0} a_{n,m} T^n U^m\mid a_{n,m} r'^n s'^m\to 0\}
\]
linear over $\mathbb C[V]\to \mathbb C[T,U]: V\mapsto T+U$. But this sends $\sum_n a_n V^n$ to
\[
\sum_{n\geq 0} a_n (T+U)^n = \sum_{n\geq 0} \sum_{m=0}^n a_n \binom{n}{m} T^m U^{n-m} = \sum_{n,m\geq 0} a_{n+m} \binom{n+m}{n} T^n U^m.
\]
Now if $a_n t^n\to 0$ for some $t>r+s$ then if we choose $r'>r$ and $s'>s$ so that $t\geq r'+s'$ then also
\[
|a_{n+m}| \binom{n+m}{n} r'^n s'^m\leq |a_{n+m}| (r'+s')^{n+m}\leq |a_{n+m}| t^{n+m}
\]
converges to $0$, giving the desired map.
\end{proof}

Let us unravel what this means for $\mathcal S(\mathbb C[T])$.

\begin{corollary} On the open subset $U=\bigcup_{r>0} \{|T|<r\}\subset \mathcal S(\mathbb C[T])$, one gets a unique continuous map
\[
U\xrightarrow{|T|} \mathbb R_{\geq 0}
\]
so that the preimage of $[0,r]$ is $\{|T|\leq r\}$ and the preimage of $[r,\infty)$ is $\{|T|\geq r\}\cap U$.
\end{corollary}

Indeed, this follows from Proposition~\ref{prop:propertiessubsets} and the following general lemma.

\begin{lemma}\label{lem:localemaptoRge0} Let $\mathcal S$ be any locale. Then maps $\mathcal S\to \mathbb R_{\geq 0}$ are equivalent to collections of closed subsets $\mathcal S(\leq r),\mathcal S(\geq r)\subset \mathcal S$ for all $r>0$, subject to the following conditions:
\begin{enumerate}
\item For all $r>0$, one has $\mathcal S(\leq r)=\bigcap_{s>r} \mathcal S(\leq s)$ and $\mathcal S(\geq r)=\bigcap_{s<r} \mathcal S(\geq s)$.
\item For $r<s$, one has $\mathcal S(\leq r)\subset \mathcal S(\leq s)$, $\mathcal S(\geq s)\subset \mathcal S(\geq r)$, $\mathcal S(\leq r)\cap \mathcal S(\geq s)=\emptyset$ and $\mathcal S(\geq r)\cup \mathcal S(\leq s)=\mathcal S$.
\item One has $\bigcap_r \mathcal S(\geq r)=\emptyset$ as $r$ gets large.
\end{enumerate}
Here, a map $f: \mathcal S\to \mathbb R_{\geq 0}$ is mapped to $\mathcal S(\leq r)=f^{-1}([0,r])$ and $\mathcal S(\geq r)=f^{-1}([r,\infty))$.
\end{lemma}

\begin{proof} It is clear that $f$ is uniquely determined by this collection of closed subsets, and they satisfy these condition. We need to see that conversely, any such collection of closed subsets determines a map $f: \mathcal S\to \mathbb R_{\geq 0}$. Using ((2) and) (3), we can write $\mathcal S$ as the (increasing) union of the open sublocales $\mathcal S\setminus \mathcal S(\geq r)$, and it suffices to construct the map on each of those; we can thus assume that $\mathcal S(\geq r)=\emptyset$ for $r$ large. We can then replace $\mathbb R_{\geq 0}$ by some interval $[0,c]$. Now any closed subset $Z$ of $[0,c]$ can be written as a cofiltered intersection of closed subsets that are finite unions of closed intervals, and whose interior contains $Z$; and any two such cofiltered systems are cofinal. One already knows where to send intervals, hence their finite unions (which for disjoint intervals are disjoint by (2)) and hence their cofiltered intersections (as maps of locales commute with those). It is easy to see that this determines a map of locales. One needs to see that it recovers the given closed subsets, and for this one uses condition (1).
\end{proof}

In fact, we get the following more refined corollary. Here, we consider the Berkovich spectrum
\[
\mathcal M^{\mathrm{Berk}}(R) = \{|\cdot|: R\to \mathbb R_{\geq 0}\mid |fg|=|f||g|, |f+g|\leq |f|+|g|, |\cdot| |_{\mathbb C} = |\cdot|_{\mathbb C}\}
\]
of bounded multiplicative seminorms, equipped with the closed subspace topology from $\prod_{f\in R} \mathbb R_{\geq 0}$ (cf.~\cite{BerkovichSpectral}).

\begin{corollary}\label{cor:maptoberkovich} The map
\[
U\xrightarrow{|f|_{f\in \mathbb C[T]}} \prod_{f\in \mathbb C[T]} \mathbb R_{\geq 0}
\]
factors over the Berkovich spectrum $\mathcal M^{\mathrm{Berk}}(\mathbb C[T])$.
\end{corollary}

\begin{proof} First note that $U\subset \bigcup_{r>0} \{|f|<r\}$ for all $f\in \mathbb C[T]$, using the triangle inequality. This means that the map is well-defined. The condition to factor over the Berkovich spectrum is a collection of many conditions, each of which pertains only to at most two $|f|,|g|$: Namely, for any $f,g\in \mathbb C[T]$, we need to see that
\[
U\xrightarrow{|f|,|g|,|f+g|,|fg|} \mathbb R_{\geq 0}^4
\]
factors over the closed subset $\{(x,y,s,p)\mid xy=p,s\leq x+y\}\subset \mathbb R_{\geq 0}^4$. This follows easily from the conditions of Proposition~\ref{prop:propertiessubsets}.
\end{proof}

So far, everything would have worked in exactly the same way over a nonarchimedean field like $\mathbb Q_2$ in place of $\mathbb C$. The crucial difference is in the identification of the Berkovich spectrum, which is extremely simple for $\mathbb C$-algebras.

\begin{theorem}[{Ostrowski, \cite{Ostrowski}}]\label{thm:ostrowski} Let $R$ be any $\mathbb C$-algebra. Then any multiplicative seminorm $|\cdot|\in \mathcal M^{\mathrm{Berk}}(R)$ is given as a composite
\[
R\xrightarrow{f} \mathbb C\xrightarrow{|\cdot|_{\mathbb C}} \mathbb R_{\geq 0}
\]
for a unique $\mathbb C$-algebra homomorphism $R\to \mathbb C$, inducing a homeomorphism
\[
\mathcal M^{\mathrm{Berk}}(R) = \mathrm{Hom}_{\mathbb C}(R,\mathbb C).
\]
\end{theorem}

Before giving the proof, we note that this finishes the proof of Theorem~\ref{thm:holomorphicdefinitionredux}: We get a map
\[
\mathcal S(\mathbb C[T])\supset U\to \mathcal M^{\mathrm{Berk}}(\mathbb C[T])\cong \mathrm{Hom}_{\mathbb C}(\mathbb C[T],\mathbb C)=\mathbb C
\]
and the pushforward of (the restriction to $U$ of) the structure sheaf of $\mathcal S(\mathbb C[T])$ to $\mathbb C$ is the desired sheaf, by the simple computation of Proposition~\ref{prop:computationunitdisc}.

\begin{proof}[Proof of Theorem~\ref{thm:ostrowski}] We note that this is usually regarded as a consequence of the Gelfand--Mazur theorem, which in turn is usually proved by complex analysis. To avoid any vicious circles, we use the following elementary argument going back to Ostrowski.

The kernel of $|\cdot|$ is an ideal; quotienting it out, we can assume that $|x|=0$ only if $x=0$. Our goal is then to show that $R=\mathbb C$. Take any $x\in R$. By the triangle inequality, the function $\mathbb C\to \mathbb R_{\geq 0}: z\mapsto |x-z|$ is continuous and gets large for large $|z|$, so attains a minimum. Replacing $x$ by some $x-z$, we can thus assume that for all $z\in \mathbb C$, one has $|x-z|\geq |x|$. Our goal is then to show that $x=0$, for which it suffices that $|x|=0$. If not, we can rescale $x$ by an element of $\mathbb R_{>0}$ so that $|x|=2$. For any integer $n$, consider the equation
\[
x^n-1=\prod_{i=0}^{n-1} (x-\zeta_n^i).
\]
Taking absolute values, we find
\[
2^n+1\geq |x^n-1|=\prod_{i=0}^{n-1} |x-\zeta_n^i|\geq 2^{n-1} |x-1|\geq 2^n
\]
as all $|x-\zeta_n^i|\geq |x|=2$. Taking the limit $n\to \infty$, we find $|x-1|=2$. Thus, by induction
\[
2=|x|=|x-1|=|x-2|=\ldots,
\]
which is impossible (for example as $|x-5|\geq 5-|x|=3$).
\end{proof}

We note that it is exactly in this final line that we used that we are doing archimedean geometry as opposed to non-archimedean geometry! In non-archimedean geometry, everything works in the same way, but there is no analogue of Ostrowski's theorem; there are many other multiplicative seminorms in nonarchimedean geometry. From our point of view, the natural underlying space for analytic spaces is something like the Berkovich spectrum, and it is only a weird coincidence that in complex geometry, this happens to agree with the ``classical points''.\\

\textbf{Exercise 1.} Show that any idempotent algebra in $\mathcal D(\mathbb Z)$ is a localization $\mathbb Z[1/S]$ of $\mathbb Z$ (for some multiplicative set $S$). Deduce that the locale $\mathcal S(\mathcal D(\mathbb Z))$ is given by the opposite of $\operatorname{Spec}(\mathbb Z)$ (i.e., the same constructible topology, but with the order of specializations reversed). Describe the sheaf of $\infty$-categories on $\mathcal S(\mathcal D(\mathbb Z))$.\\

\textbf{Exercise 2.} Show that for any noetherian ring $R$, the locale $\mathcal S(\mathcal D(R))$ is given by the opposite of $\operatorname{Spec}(R)$.\\

\textbf{Exercise 3.} Give an example of a commutative ring $R$ and an idempotent algebra $A\in \mathcal D(R)$ that is not accounted for by the map $\mathcal S(\mathcal D(R))\to \operatorname{Spec}(R)^{\mathrm{op}}$. (Hint: Almost mathematics.)\newpage

\section*{Appendix to Lecture V: Spectral Theory}

In this short appendix, we show how the techniques of the lecture also give a proof of the basic result on spectral theory:

\begin{theorem}\label{thm:spectraltheory} Let $V$ be a nonzero complex quasi-Banach space, and let $T: V\to V$ be an endomorphism. Then the set of all $x\in \mathbb C$ for which $T-x$ is not invertible is a nonempty compact subset of $\mathbb C$.
\end{theorem}

The set of all such $x$ is usually called the spectrum of $T$. The key statement here is the really the nonemptiness of the spectrum.

We will deduce this as a consequence of two propositions.

\begin{proposition}\label{prop:banachmodulelarger} Let $V$ be a complex quasi-Banach space and let $T: V\to V$ be an endomorphism. Then the $\mathbb C[T]$-module structure on $V$ extends to a module structure over $\mathcal O(\{|T|\leq r\})$ for some large enough $r>0$.
\end{proposition}

\begin{proof} Let $r$ be the norm of $T$. Then any sum $\sum_{n\geq 0} a_n T^n$ with $a_n r'^n\to 0$ for some $r'>r$ gives a well-defined endomorphism of $V$.
\end{proof}

Thus, it suffices to prove the following slight generalization of Theorem~\ref{thm:spectraltheory}:

\begin{proposition}\label{prop:spectraltheorygen} Let $V$ be a nonzero $p$-liquid $\mathbb C$-vector space such that any endomorphism $U: V\to V$ has the property that the $\mathbb C[U]$-module structure on $V$ extends to a module structure over $\mathcal O(\{|U|\leq r\})$ for some large enough $r>0$. Let $T: V\to V$ be an endomorphism.

Then the set of all $x\in \mathbb C$ for which $T-x$ is not an isomorphism agrees with the minimal closed subset $Z\subset \mathbb C$ such that $V\in C=\mathcal D(\mathrm{Liq}_p(\mathbb C[T]))$ is contained in $C(Z)$. Moreover, $Z$ is a nonempty compact subset of $\mathbb C$.
\end{proposition}

\begin{proof} Clearly if $V$ is contained in some $C(Z)$, then $T-x$ is invertible for all $x\not\in Z$. For the converse, we note that if $U=(T-x)^{-1}$ exists, then $V$ is also a module over $\mathcal O(\{|U|\leq r'\})$ for some $r'>0$. Reinterpreting in terms of $T=U^{-1}-x$, this means that $V$ is a module over $\mathcal O(\{|T-x|\geq r'^{-1}\})$. This implies that the spectrum $Z$ is closed, and $V$ is contained in $C(Z)$. Also, the spectrum is bounded as $V$ is a module over $\mathcal O(\{|T|\leq r\})$ for some $r>0$, so the spectrum is actually compact.

But as $V$ is a nonzero object contained in $C(Z)$, it follows that $Z$ must be nonempty!
\end{proof}

The following is a standard corollary. We note that we did not use any of this before. (We only used the existence of infinitely many roots of unity; but for example all $2$-power roots of unity can be written down explicitly by radicals.)

\begin{corollary}[Fundamental Theorem of Algebra; Gelfand--Mazur Theorem] The field of complex numbers $\mathbb C$ is algebraically closed. In fact, any (quasi-)Banach field over $\mathbb C$ is equal to $\mathbb C$ itself.
\end{corollary}

\begin{proof} Let $K$ be a quasi-Banach field over $\mathbb C$. Take any $T\in K\setminus \mathbb C$. By spectral theory, there is some $x\in \mathbb C$ for which $T-x$ is not invertible. But in a field, this means $T-x=0$, i.e.~$T=x$.
\end{proof}\newpage

\section{Lecture VI: Analytification and GAGA}

Let's recap what happened last time.  To any cocomplete closed symmetric monoidal stable $\infty$-category $C$, we assigned a locale $\mathcal{S}(C)$ whose poset of ``closed subsets'' identifies, by definition, with the opposite of the poset of idempotent algebras in $C$, i.e.\ if there is a (necessarily unique) homomorphism $A\rightarrow A'$ of idempotent algebras, then the corresponding closed subsets have the inclusion relation $Z\supset Z'$ (roughly, think of $A$ as functions on $Z$).

We also showed that $C$ localizes on this locale: to an open $U\subset \mathcal{S}(C)$ we assign the quotient category
$$C(U):=C/\operatorname{Mod}_A(C)$$
where $A$ is the idempotent algebra corresponding the closed complement of $U$, and we had the result that $U\mapsto C(U)$ is a sheaf of (cocomplete closed symmetric monoidal) $\infty$-categories on $\mathcal{S}(C)$.  In particular, by passing to endomorphisms of the unit objects, we get a ``structure sheaf'' on $\mathcal{S}(C)$.

We then specialized this abstract nonsense to the case $C=\mathcal D(\mathrm{Liq}_p(\mathbb C[T]))$, getting the locale $\mathcal{S}(\mathbb C[T]):= \mathcal{S}(\mathcal D(\mathrm{Liq}_p(\mathbb C[T])))$. For every $r>0$, we produced closed subsets
$$\{|T|\leq r\}, \{|T|\geq r\} \subset \mathcal{S}(\mathbb C[T]),$$
given by idempotent algebras of holomorphic functions which overconverge on the disk of radius $r$ (in the first case), or which overconverge on the exterior of the disk of radius $r$ and are meromorphic at $\infty$ (in the second case).

For any (discrete commutative) $\mathbb{C}$-algebra $R$ and any $f\in R$, we can then consider the closed subsets
$$\{|f|\leq r\}, \{|f|\geq r\} \subset \mathcal{S}(R)$$
obtained by base-changing the previous ones along the map $\mathbb{C}[T]\rightarrow R$, $T\mapsto f$.  We also showed that these subsets satisfy relations which reflect the triangle inequality, multiplicativity, and behavior on complex scalars expected of such a formal expression ``$|f|$''.  For example there is an inclusion of closed subsets
$$\{|f|\leq r\}\cap \{|g|\leq r'\}\subset \{|f+g|\leq r+r'\},$$
for all $f,g\in R$ and $r,r'>0$, which reflects the triangle inequality.

Next we consider, for $f\in R$, the open subset
$$\mathcal{S}(R,f) = \bigcup_{r>0} \{|f|<r\}\subset \mathcal{S}(R).$$
We can think of this as the open subset of $\mathcal{S}(R)$ on which $f$ behaves like an analytic variable as opposed to an algebraic variable.  In particular, in the ring $\mathcal{O}(\mathcal{S}(R,f))$, expressions of the form $\varphi(f)$, for any entire holomorphic function $\varphi$, have well-defined meaning, while in $\mathcal{O}(\mathcal{S}(R))=R$ we can only evaluate $f$ on polynomials.

Thanks to the multiplicativity and triangle inequality, we have
$$\mathcal{S}(R,fg), \mathcal{S}(R,f+g)\supset \mathcal{S}(R,f)\cap \mathcal{S}(R,g).$$
It follows that if $R$ is finite type, generated by $X_1,\ldots X_n$, then the ``subset"
$$\mathcal{S}(R,R):=\bigcap_{f\in R} \mathcal{S}(R,f)$$
is is in fact an open subset of $\mathcal{S}(R)$, being equal to the finite intersection
$$\bigcap_{i=1}^n \mathcal{S}(R,X_i).$$
We can think of this $\mathcal{S}(R,R)$ as some kind of ``analytification'' of $\operatorname{Spec}(R)$: we have forced all the functions $f\in R$ to be analytic.

To make an explicit connection to the usual analytification, recall also from the previous lecture that there is a natural map of locales
$$\mathcal{S}(R,R)\rightarrow \mathcal{M}^{\mathrm{Berk}}(R) \cong \operatorname{Spec}(R)(\mathbb{C}) =V(I)\subset \mathbb{C}^n,$$
where  we take $R=\mathbb{C}[X_1,\ldots,X_n]/I$ finite type. This map of locales was uniquely characterized by the fact that for $f\in R$ and $r>0$, this sends the honest open subsets
$$\{|f|<r\},\{|f|>r\}\subset V(I)$$
to the formally defined open subsets of the locale $\mathcal{S}(R,R)$ which we gave the same notation $\{|f|<r\}$, $\{|f|>r\}$, in some sense justifying this notation.

By pushforward along this map of locales $\mathcal{S}(R,R)\rightarrow \operatorname{Spec}(R)(\mathbb{C})$, we deduce a sheaf of (cocomplete closed symmetric stable monoidal) $\infty$-categories on $\operatorname{Spec}(R)(\mathbb{C})$.  This in particular yields a ``structure sheaf'' on $\operatorname{Spec}(R)(\mathbb{C})$ with values in commutative algebra objects of $\mathcal D(\mathrm{Liq}_p)$, by taking endomorphisms of the unit in these categories.  The last thing we saw in the previous lecture was the calculation of this structure sheaf when $R=\mathbb{C}[T]$: namely, its value on an open disk is the usual algebra of convergent power series on that disk, as expected from the analytification of the affine line.

The first thing we want to do here is generalize this calculation to arbitrary finite type $\mathbb{C}$-algebras $R$.  In the previous lecture we did the calculation on $R=\mathbb{C}[T]$ by using the formal fact that the localization to the open subset corresponding to an idempotent algebra $A$ in $C$ is given by $\underline{\operatorname{RHom}}([1\rightarrow A],-)$.  But this method is more complicated to carry out for $R=\mathbb{C}[X_1,\ldots,X_n]$: the corresponding idempotent algebra $A$ cutting out $\mathcal{S}(R,R)$ lives in cohomological degrees $[0,n]$, calculated by iterating the operation $A_1,A_2\mapsto A_1\times_{A_1\otimes A_2} A_2$ describing unions of closed subsets, and it becomes annoying to perform this preliminary calculation, let alone the resulting $\underline{\operatorname{RHom}}$ calculation.

Fortunately, there is another method, which is no more difficult for $\mathbb{C}[X_1,\ldots,X_n]$ than for $\mathbb{C}[T]$.  The idea is very simple: intuitively speaking, we should have
$$\{|T|<1\} = \bigcup_{r<1} \{|T|\leq r\}.$$
This is a cover of an open subset by closed subsets, but since it is refined by the open cover
$$\{|T|<1\} = \bigcup_{r<1} \{|T|< r\},$$
it should be just as good for calculating the structure sheaf.  Taking the idea that the structure sheaf on the closed subset $\{|T|\leq r\}$ should be the corresponding idempotent algebra $\mathcal{O}^{hol}(\{|T|\leq r\})$ of overconvergent holomorphic functions, this gives
$$\mathcal{O}(\{|T|<1\}) = R\varprojlim_{r<1} \mathcal{O}^{hol}(\{|T|\leq r\}),$$
yielding the desired result $\mathcal{O}(\{|T|<1\})=\mathcal{O}^{hol}(\{|T|<1\})$ (using pro-isomorphisms and Mittag-Leffler, as previously discussed).

The problem is that in the locale picture, open subsets and closed subsets are just formal gadgets in formal bijection with one another, and we haven't said what it means for one to be contained in another in a way which justifies the above intuitive calculation. But we will introduce another formal framework which does this, and deduce the following general result.

\begin{theorem}\label{explicitanalytification}
Let $R=\mathbb{C}[X_1,\ldots,X_n]/I$ be a finite type $\mathbb{C}$-algebra.  Take $z=(z_1,\ldots, z_n)\in V(R)\subset \mathbb{C}^n$ and $r=(r_1,\ldots, r_n)$ with $r_i>0$ for all $i$, parametrizing the open polydisk $D(z,r)$ of polyradius $r$ centered at $z$ (implicitly, intersected with $V(I)$).  Then:
\begin{enumerate}
\item For the sheaf of categories on $\operatorname{Spec}(R)(\mathbb{C})=V(I)$ described above, we have
$$C(D(z,r)) = \varprojlim_{r'<r} \operatorname{Mod}_{\mathcal{O}^{hol}(\overline{D}(z,r'))/I}(\mathcal D(\mathrm{Liq}_p(\mathbb C))).$$

\item For the structure sheaf we get
$$\mathcal{O}(D(z,r)) = \mathcal{O}^{hol}(D(z,r))/I.$$
\end{enumerate}

Moreover, for the global sections we similarly have
$$C(V(I)) = \varprojlim_r \operatorname{Mod}_{\mathcal{O}^{hol}(\overline{D}(0,r))/I}(\mathcal D(\mathrm{Liq}_p(\mathbb C)))$$
and
$$\mathcal{O}(V(I)) = \mathcal{O}^{hol}(\mathbb{C}^n)/I.$$
\end{theorem}

Here $\mathcal{O}^{hol}(\overline{D}(z,r'))$ (resp.\ $\mathcal{O}^{hol}(D(z,r))$) is ring of power series in $n$ variables centered at $z$ which overconverge on the polydisk of radius $r'$, (resp.\ converge on the open polydisk of radius $r$), and $ \mathcal{O}^{hol}(\mathbb{C}^n)$ is the ring of power series which converge on all of $\mathbb{C}^n$.

This theorem shows that we can reduce the study of our sheaf of categories to the study of the ``purely algebraic'' categories of modules over various rings (of overconvergent functions, in the base liquid category). This is an important complement to the formal flexibility provided by the locale picture and its associated descent.

\begin{remark}
Note that the value $C(D(z,r))$ differs from the ``naive guess'' of
$$\operatorname{Mod}_{\mathcal{O}^{hol}(D(z,r))/I}(\mathcal D(\mathrm{Liq}_p(\mathbb C))),$$
the category of modules over the value of the structure sheaf on $D(z,r)$.  In fact, $C(D(z,r))$ is smaller than this naive guess: the natural functor
$$\operatorname{Mod}_{\mathcal{O}^{hol}(D(z,r))/I}(\mathcal D(\mathrm{Liq}_p(\mathbb C)))\rightarrow C(D(z,r))$$
is a nontrivial localization.

However, the sheaf of categories $C(-)$ on $V(I)$ is nonetheless canonically determined by just the structure sheaf (of commutative algebras in $\mathcal D(\mathrm{Liq}_p(\mathbb C))$): for example, we can describe $C(-)$ as the \emph{sheafification} of this naive guess
$$D(z,r)\mapsto \operatorname{Mod}_{\mathcal{O}^{hol}(D(z,r))/I}(\mathcal D(\mathrm{Liq}_p(\mathbb C))).$$
(Another description will be given in the exercises.)  This is not a general feature of the formalism, but it follows in this case from the statement of the theorem.
\end{remark}

To prove this theorem, as indicated above, we will find a framework in which we can work with ``open subsets'' and ``closed subsets'' on equal footing.  We will take the following maximalist approach.\footnote{In the lecture, we restricted to a more minimalist context of ``immersions'', but it seems this restriction is actually irrelevant and distracting, so we take the maximalist approach here. In hindsight, one can say that the discussion of this lecture was initiating the path towards the theory of Gestalten developed later with Stefanich, \cite{Gestalten}.}

\begin{definition}
Let $\operatorname{Sym}$ denote the (very large) $\infty$-category of cocomplete closed symmetric monoidal stable $\infty$-categories, with morphisms the symmetric monoidal cocontinuous functors between them.  We denote such functors by formal expressions such as
$$f^*:C\rightarrow D,$$
and we will then write simply $f$ for the same morphism, considered as a morphism in the opposite category $\operatorname{Sym}^{\mathrm{op}}$ instead.
\end{definition}

As desired, we can view both closed subsets and open subsets in this common framework, and the interactions between them are as expected from naive considerations:

\begin{lemma}
Let $C\in\operatorname{Sym}$.
\begin{enumerate}
\item There is a fully faithful finite-limit-preserving functor from the poset of closed subsets of $\mathcal{S}(C)$ to
$(\operatorname{Sym}^{\mathrm{op}})_{/C}$, sending the closed subset $Z$ to the localization
$$i_Z^*:C\rightarrow \operatorname{Mod}_A(C),$$
where $A$ is the corresponding idempotent algebra.
\item There is a fully faithful finite-limit-preserving functor from the poset of open subsets of $\mathcal{S}(C)$ to $(\operatorname{Sym}^{\mathrm{op}})_{/C}$, sending the open subset $U$ to the localization
$$j_U^*:C\rightarrow C/\operatorname{Mod}_A(C)$$
where $A$ is the idempotent algebra corresponding to $U^c$.
\item Let $Z$ be a closed subset of $\mathcal{S}(C)$ and $U$ an open subset of $\mathcal{S}(C)$.  Then
$$ i_Z^* \text{ factors through }
j_U^* \Leftrightarrow Z^c \cup U = \mathcal{S}(C),\text{ i.e. }``Z\subset U''.$$
and
$$j_U^*\text{ factors through } i_Z^*\Leftrightarrow Z^c \cap U = \emptyset , \text{ i.e. }``U\subset Z''.$$
\item Let $Z$ be a closed subset of $\mathcal{S}(C)$ with corresponding idempotent algebra $A$, and let $f^*:C\rightarrow D$ be a map in $\operatorname{Sym}$.  Then the pushout of $i_Z^*$ along $f^*$ exists in $\operatorname{Sym}$ and is given by $i_{Z'}^*$ where $Z'$ corresponds to the idempotent algebra $f^* A$ in $D$.  Similarly, if $U$ is the complement of $Z$, then the pushout of $j_U^*$ along $f^*$ exists and is given by $j_{U'}^*$ where $U'$ is the complement of $Z'$.
\end{enumerate} 
\end{lemma}
\begin{proof}
These claims are all straightforward.  For (1), since $i_Z^*$ is a localization, $i_Z$ is a monomorphism in $(\operatorname{Sym}^{\mathrm{op}})_{/C}$, hence the full subcategory of $(\operatorname{Sym}^{\mathrm{op}})_{/C}$ spanned by all these $i_Z$ is indeed equivalent to a poset.  Then the well-definedness and full faithfulness of the functor amounts to the assertion that 
$$\operatorname{Mod}_A(C)\subset \operatorname{Mod}_{A'}(C) \Leftrightarrow Z\subset Z' (\Leftrightarrow A=A'\otimes A),$$
where $A$ corresponds to $Z$ and $A'$ to $Z'$, and this is simple.  The functor in (2) is well-defined and fully faithful for the exact same reason.  

The claim about finite limits in (1) and (2) is a simple consequence of the base-change result (4), as is the claim in (3).  Finally, the base-change result (4) follows from the description of the kernel of $i_Z^*$ (resp.\ $j_U^*$) as the tensor ideal generated by $[1\rightarrow A]$ (resp.\ $A$).
\end{proof}

A map in $\operatorname{Sym}^{\mathrm{op}}$ isomorphic to one of the form $i_Z$ as in (1) will be called a \emph{closed immersion}, and a map in $\operatorname{Sym}^{\mathrm{op}}$ isomorphic to one of the form $j_U$ as in (2) will be called an \emph{open immersion}.  Thus we have seen that the poset of open immersions into a fixed $C$ identifies with the poset of open subsets of $\mathcal{S}(C)$, and the poset of closed immersions into $C$ identifies with the poset of closed subsets of $\mathcal{S}(C)$.  Also, closed and open subsets interact ``as expected'' in the larger poset of subobjects of $C$ in $\operatorname{Sym}^{\mathrm{op}}$.  Moreover, both open and closed immersions are closed under base-change, and these base-changes are calculated in the naive way in terms of corresponding algebras.

There is the following abstract characterization of open and closed immersions.

\begin{proposition}
Let $f^*:C\rightarrow D$ in $\operatorname{Sym}$.  Then:
\begin{enumerate}
\item $f$ is a closed immersion if and only if $f^*$ has a fully faithful right adjoint $f_*$ which preserves colimits and satisfies the projection formula
$$c \otimes f_*d \overset{\sim}{\rightarrow} f_*(f^*c \otimes d).$$
\item $f$ is an open immersion if and only if $f^*$ has a fully faithful left adjoint $f_\natural$ which preserves colimits (this is of course automatic for a left adjoint) and satisfies the projection formula
$$f_\natural(d\otimes f^*c)\overset{\sim}{\rightarrow} c\otimes f_\natural d.$$
\end{enumerate}
Moreover, in each of these cases the specified adjoint commutes with arbitrary base-change.
\end{proposition}
\begin{proof}
For (1), if $f$ is a closed immersion corresponding to an idempotent algebra $A$, then we have that $f_\ast$ is the forgetful functor $\operatorname{Mod}_A(C)\rightarrow C$, and the required properties are clear.  Conversely, if $f^*$ has a colimit-preserving right adjoint $f_*$ satisfying the projection formula, then first of all we have that $1\rightarrow f_*1$ makes $f_*1$ into an idempotent algebra, by the projection formula applied to $c=f_*1$ and $d=1$.  Then to identify $f^*$ with the closed immersion corresponding to $f_*1$, as both are localizations it suffices to see
$$f_*D = \operatorname{Mod}_{f_*1}(C).$$
as full subcategories of $C$.  The inclusion $\subset$ is formal from $f_*$ being right adjoint to the symmetric monoidal $f^*$.  For $\supset$, note that $X\otimes f_*1 = f_*f^*X$ by the projection formula.

For (2), if $f$ is an open immersion corresponding to an idempotent algebra $A$, then we can identify $D$ with the full subcategory of $C$ consisting of those $X$ with $X\otimes A=0$, and in these terms $f^*$ becomes $-\otimes [1\rightarrow A]$ and $f_\natural$ becomes the inclusion.  Then the projection formula is clear.  Conversely, suppose $f^*$ has the left adjoint $f_\natural$ satisfying projection formula.  Let $A$ be the cofiber of $f_\natural 1\rightarrow 1$.  Then $A$ is an idempotent algebra by the projection formula, and we need to see that
$$\operatorname{ker}(f^*) = \operatorname{Mod}_A(C).$$
as full subcategories of $C$.   Certainly $A$, and hence $\operatorname{Mod}_A(C)$, is in the kernel because $f^*f_\natural 1=1$ by fully faithfulness of $f_\natural$.  Conversely, if $X\in \operatorname{ker}(f^*)$, then $X\otimes f_\natural 1=0$ by the projection formula, whence $X\in \operatorname{Mod}_A(C)$.

Finally, the commutation with base-change follows from (4) in the previous lemma and the explicit description of $i_*i^* = -\otimes A$ and $j_\natural j^* = -\otimes [1\rightarrow A]$.
\end{proof}

Next, we deduce the expected permanence properties related to composition.

\begin{corollary}
In $\operatorname{Sym}^{\mathrm{op}}$, we have the following.
\begin{enumerate}
\item A composition of closed immersions is a closed immersion, and a composition of open immersions is an open immersion.
\item Let $f^*:C\rightarrow D$ and $g^*:D\rightarrow D'$ be composable maps such that $f$ is a monomorphism (which holds if it is either a closed or an open immersion).  If $f\circ g$ is an open immersion, then $g$ is an open immersion, and if $f\circ g$ is a closed immersion, then $g$ is a closed immersion.
\end{enumerate}
\end{corollary}
\begin{proof}
The claim in (1) follows from the characterization in terms of the projection formula.  The claim in (2) follows immediately from base-change stability of open and closed immersions.
\end{proof}

From these properties it follows for example that if $j_U^*:C\rightarrow D$ is the pullback functor corresponding to an open subset $U\subset \mathcal{S}(C)$, then $\mathcal{S}(D)$ identifies with the locale of open subsets of $U$: a nice sanity check.

Furthermore we get the following, letting us formalize descent with respect to not-necessarily-open covers:

\begin{theorem}
\begin{enumerate}
\item There is a Grothendieck topology on $\operatorname{Sym}^{\mathrm{op}}$ where the covering sieves over a given $C$ are those which contain some set of open immersions whose corresponding open subsets cover $\mathcal{S}(C)$.
\item The identity functor $(\operatorname{Sym}^{\mathrm{op}})^{\mathrm{op}}\rightarrow \operatorname{Sym}$ is a sheaf with respect to this Grothendieck topology.
\item The poset of open immersions also satisfies descent with respect to this Grothendieck topology; same for the poset of closed immersions.
\end{enumerate}
\end{theorem}
\begin{proof}
Part (1) follows immediately from the compatibility under composition and base-change properties discussed above.  Part (2) then follows from the descent on the locale proved in the previous lecture.  Part (3) follows formally from part (2) and the description of base-changes of open and closed immersions in terms of idempotent algebras.
\end{proof}

This concludes our detour into abstract nonsense.  Now we turn to the proof of Theorem \ref{explicitanalytification}, which now almost writes itself: note that for $r'<r$, we have
$$D(z,r')\subset \overline{D}(z,r')\subset D(z,r),$$
so the $\overline{D}(z,r')$ for $r'<r$ give a collection of closed immersions into $D(z,r)$ which is refined by the cover by the open immersions $D(z,r')$.  Thus we can calculate $C(D(z,r))$ by descent, getting
$$C(D(z,r)) = \varprojlim_{r'<r} \operatorname{Mod}_{\mathcal{O}(\overline{D}(z,r'))}(\mathcal{D}(\mathrm{Liq}_p(\mathbb C))),$$
where $\mathcal{O}(\overline{D}(z,r'))$ is the idempotent algebra in liquid $R$-modules which corresponds to the given closed subset $\overline{D}(z,r') = \bigcap_{i=1}^n\{ |X_i - z_i| \leq r'_i\}$.  When $R=\mathbb{C}[X_1,\ldots,X_n]$, a straightforward tensor product calculation as in the previous lecture gives that
$$\mathcal{O}(\overline{D}(z,r')) = \mathcal{O}^{hol}(\overline{D}(z,r'))$$
is indeed the usual ring of overconvergent functions on the given closed disk in $\mathbb{C}^n$, and we deduce the claims of the theorem in that case.  For general $R=\mathbb{C}[X_1,\ldots,X_n]/I$, we get that $\mathcal{O}(\overline{D}(z,r'))$ is given by the (derived) base-change
$$\mathcal{O}^{hol}(\overline{D}(z,r'))\otimes_{\mathbb{C}[X_1,\ldots,X_n]}  R$$
in liquid modules.  For that we have the following, which thereby finishes the proof of Theorem \ref{explicitanalytification}:

\begin{lemma}\label{analyticaffinelemma}
Let $I\subset \mathbb{C}[X_1,\ldots,X_n]$ be an ideal with quotient $R$, let $z=(z_1,\ldots,z_n)\in \mathbb{C}^n$, and let $r=(r_1,\ldots,r_n)\in \mathbb{R}_{>0}^n$.  Then in $\mathcal{D}(\operatorname{Liq}_p)$, we have
$$\mathcal{O}^{hol}(\overline{D}(z,r))\otimes_{\mathbb{C}[X_1,\ldots,X_n]} R = \mathcal{O}^{hol}(\overline{D}(z,r))/I,$$
meaning the derived liquid tensor product on the left is concentrated in degree zero and given by the naive modding out on the right, which is the cokernel of the map
$$\mathcal{O}^{hol}(\overline{D}(z,r))^{\oplus m}\rightarrow \mathcal{O}^{hol}(\overline{D}(z,r))$$
given by $m$ chosen generators of $I$.

Similarly,
$$\mathcal{O}^{hol}(D(z,r))\otimes_{\mathbb{C}[X_1,\ldots,X_n]} R = \mathcal{O}^{hol}(D(z,r))/I,$$
and this is also equal to $\varprojlim_{r'<r} \mathcal{O}^{hol}(\overline{D}(z,r'))/I$.
\end{lemma}
\begin{proof}
We claim that the condensed ring $\mathcal{O}^{hol}(\overline{D}(z,r))$ is flat as a module over the \emph{discrete} ring $\mathbb{C}[X_1,\ldots,X_n](\ast)$, i.e.\ for any discrete module $M$ over $\mathbb{C}[X_1,\ldots,X_n](\ast)$ the derived tensor product
$$\mathcal{O}^{hol}(\overline{D}(z,r))\otimes_{\mathbb{C}[X_1,\ldots,X_n](\ast)}M$$
is concentrated in degree zero.\footnote{However, beware that the map of liquid rings $\mathbb{C}[X_1,\ldots,X_n]\rightarrow \mathcal{O}^{hol}(\overline{D}(z,r))$ is \emph{not} flat! Indeed, if $\overline{D'}$ is a closed polydisk disjoint from $\overline{D}$, then $\mathbb{C}[X_1,\ldots,X_n]\rightarrow \mathcal{O}^{hol}(\overline{D'}(z,r))$ is an injection which base changes to the non-injective map $\mathcal{O}^{hol}(\overline{D}(z,r))\rightarrow 0$.} We will include the argument later when we discuss the finiteness properties of such rings.

For the remainder of the claims, note that since $\mathbb{C}[X_1,\ldots,X_n](\ast)$ is a regular noetherian ring, a finitely generated module such as $R$ has a finite resolution by finite free modules, so the derived tensor product commutes with arbitrary (derived) limits, where in this setting ``derived'' is optional due to Mittag-Leffler.
\end{proof}

Our next topic is globalization and GAGA. To recap, for any finite type $\mathbb{C}$-algebra $R$, we have assigned the open subset
$$\mathcal{S}(R,R)\subset \mathcal{S}(R),$$
with corresponding localization functor in $\operatorname{Sym}$
$$C(R)\rightarrow C(R,R).$$
The source $C(R)$ is the purely algebraic category of $R$-modules in derived liquid vector spaces, while, as we have seen, the target is analytic: it localizes on the topological space $\operatorname{Spec}(R)(\mathbb{C})$ and is simply described in terms of usual rings of convergent power series on disks.

Now, both assignments $R\mapsto C(R)$ and $R\mapsto C(R,R)$ form sheaves with respect to the Zariski topology on finite type $\mathbb{C}$-schemes (we will review an argument for this argument shortly).  Hence they globalize by descent to assignments
$$X\mapsto C(X), C(X,X)$$
for arbitrary finite type $\mathbb{C}$-schemes $X$.  Naturally, the comparison functor also globalizes to a map in $\operatorname{Sym}$
$$f^*:C(X)\rightarrow C(X,X).$$

Then we have the following result.

\begin{theorem}
Suppose $X$ is a finite type $\mathbb{C}$-scheme and consider the above algebraic-to-analytic comparison functor
$$f^*:C(X)\rightarrow C(X,X).$$
If $X$ is proper, then $f$ is an isomorphism.
\end{theorem}

As first sight, this seems a bit shocking: usually GAGA only holds for categories of coherent sheaves, but here we're saying it holds for these rather huge categories of (what amount to) \emph{quasicoherent} liquid sheaves.  But in fact, the notion of quasi-coherent sheaf in analytic geometry has been somewhat missing until now.  Moreover it takes some extra work, which we'll do, to identify the usual categories of coherent sheaves inside our big categories of quasi-coherent sheaves so as to see that the above theorem does imply the classical GAGA.

We will prove this theorem in a more general axiomatic framework.  Suppose given the following data:

\begin{enumerate}
\item A commutative ring $k$. (In our example, $k=\mathbb{C})$
\item A $k$-linear object $C$ of $\operatorname{Sym}$, i.e.\ a map in $\operatorname{Sym}$ of the form $\mathcal{D}(k)\rightarrow C$.  (In our example, $C = \mathcal{D}(\operatorname{Liq}_p(\mathbb{C}))$.)
\item An open subset $U\subset \mathcal{S}(\operatorname{Mod}_{k[T]}(C))$ (in our example, $U=\bigcup_r \{ |T|< r\}$.)
\end{enumerate}

This gives rise to the following notation: for any $k$-algebra $R$ and any $f\in R$, let
$$\mathcal{S}(R,f) \subset \mathcal{S}(R):=\mathcal{S}(\operatorname{Mod}_R(C))$$
denote the preimage of $U$ under the map $k[T]\rightarrow R$ sending $T$ to $f$; this is the ``subset of $\mathcal{S}(R)$ on which $f$ is analytic''.

We impose the following axioms:

\begin{enumerate}
\item For any $f\in k$, we have $\mathcal{S}(k,f)=\mathcal{S}(k)$, i.e.\ ``constants are analytic''.
\item For any $R$ finite type over $k$ and $f,g\in R$, we have
$$\mathcal{S}(R,f)\cap\mathcal{S}(R,g)\subset \mathcal{S}(R,fg),\mathcal{S}(R,f+g),$$
i.e.\ ``the sum and product of analytic functions are analytic''.
\item For any $R$ finite type over $k$ and $f\in R$, the open subset
$$\mathcal{S}(R[1/f],1/f)\subset \mathcal{S}(R[1/f])$$
is even open in $\mathcal{S}(R)$, meaning this composition of an open immersion and closed immersion happens to also be an open immersion.  Denote this open subset of $\mathcal{S}(R)$ by $D^{\mathrm{an}}(f)$; it corresponds to ``inverting $f$ in the analytic sense''.
\item For any $R$ finite type over $k$ and $f\in R$, we have
$$\mathcal{S}(R) = \mathcal{S}(R,f)\cup D^{\mathrm{an}}(f).$$
\item For any $R$ finite type over $k$ and $f,g\in R$, we have
$$D^{\mathrm{an}}(f+g)\subset D^{\mathrm{an}}(f)\cup D^{\mathrm{an}}(g).$$
\end{enumerate}

Given this data and axioms, for $R$ finite type over $k$ we can define
$$\mathcal{S}(R,R)\subset \mathcal{S}(R)$$
as the intersection of the $\mathcal{S}(R,x_i)$ for any finite generating set $x_i$; by axiom (2) above this is independent of the generating set, and $\mathcal{S}(R,R)$ is contained in $\mathcal{S}(R,f)$ for all $f\in R$.  We use $C(-)$ to denote the canonical sheaf of categories on $\mathcal{S}(R)$.  Then we have the following abstract GAGA theorem:

\begin{theorem}
Both assignments
$$R\mapsto C(R),C(R,R)$$
satisfy descent for the Zariski topology, hence they globalize by descent to assignments
$$X\mapsto C(X), C(X,X)$$
for arbitrary finite type $k$-schemes $X$. For the induced comparison functor
$$f^*:C(X)\rightarrow C(X,X),$$
we have that $f$ is an isomorphism if $X$ is proper.
\end{theorem}

Before proving this, let's make sure it applies in our given example situation.  Axiom (1), that constants are analytic, follows from the previous lecture, as does axiom (2), using the triangle inequality and multiplicativity.  For axiom (3), by base change it suffices to consider $T\in \mathbb{C}[T]$, and then we claim that the desired open subset $D^{\mathrm{an}}(T)$ is given by
$$D^{\mathrm{an}}(T)=\bigcup_r \{|T|>r\}.$$
Indeed, the corresponding idempotent algebra is the ring of germs of holomorphic functions at $0$; every $T$-torsion $\mathbb{C}[T]$-module is a module over this ring, so $T$ is inverted in the structure sheaf on this open subset $\bigcup_r \{|T|>r\}$, hence we can also view this as an open in $\mathcal{S}(\mathbb{C}[T,T^{-1}])$.  But essentially by definition, when $T$ is inverted the subset $\{|T|>r\}$ is the same as $\{|T^{-1}|<r^{-1}\}$, and when we take the union over $r$ we exactly get $\mathcal{S}(\mathbb{C}[T,T^{-1}],T^{-1})$ as desired.  For axiom (4), we can again consider just $T\in\mathbb{C}[T]$, and then we need to see that
$$\bigcup_r \{|T|< r\}  \cup \bigcup_r \{|T| > r\} = \mathcal{S}(\mathbb{C}[T]).$$
But already $\{|T|< 2\}\cup \{|T|> 1\}= \mathcal{S}(\mathbb{C}[T])$ by the calculations in the previous lecture.  Finally, axiom (5) follows from the triangle inquality, e.g.
$$\{|f+g|>r\}\subset \{|f|>r/2\}\cup \{|g|>r/2\},$$
and taking the union over $r$ gives the claim.

Having verified that the axiomatics hold in our given situation, let's turn to the proof of the abstract GAGA theorem.  We start with the following simple lemma, which is essentially the affine case of GAGA:

\begin{lemma}\label{integralanalytic}
Let $R$ be a $k$-algebra and $f,c_0,\ldots c_{n-1}\in R$ such that
$$f^n + c_{n-1}f^{n-1}+\ldots + c_0 = 0.$$
Then
$$\bigcap_i \mathcal{S}(R,c_i)\subset \mathcal{S}(R,f),$$
i.e.\ ``if $f$ satisfies a monic polynomial with analytic coefficients, then $f$ is analytic''.
\end{lemma}
\begin{proof}
By axiom (4), to prove $f$ is analytic, it suffices to show that $f$ is analytic once we analytically invert $f$.  But
$$f = -c_{n-1} - \ldots - c_0 f^{-n+1},$$
so this follows from axiom (2).
\end{proof}

It follows from this (and axioms (1) and (2)) that we have a well-defined open subset
$$\mathcal{S}(R,R^+)\subset\mathcal{S}(R)$$
whenever $R^+\subset R$ is the integral closure of a finitely generated $k$-subalgebra of $R$.  Using this, we can plug into the formalism of \emph{discrete Huber pairs}, which we briefly review here, though see \cite[Lecture IX]{Condensed} for more background.

To any pair $(R,R^+)$ of a commutative ring $R$ and integrally closed subring $R^+$, Huber assigns the set $\operatorname{Spa}(R,R^+)$ defined as
$$\{v:R\rightarrow \Gamma\cup\{+\infty\}\mid v(fg)=v(f)+v(g), v(f+g)\geq \min\{v(f),v(g)\}, v(0)=+\infty, v(R^+)\geq 0\}/\sim,$$
the set of equivalence classes of valuations on $R$ which are non-negative on $R^+$.  Here $\Gamma$ is a totally ordered abelian group.  Given such a valuation, the set of $f\in R$ with $v(f)=+\infty$ is always a prime ideal, and in fact $\operatorname{Spa}(R,R^+)$ is an enlargement of the usual spectrum of the ring $R$: besides the prime ideal, we also get an induced valuation on the residue field, giving a corresponding subring of elements $\geq 0$ and maximal ideal of elements $>0$.  In terms of intuition, one can think of such a valuation $v$ as measuring some kind of ``order of vanishing/pole'' of a function at some fictional point.

Moreover, the set $\operatorname{Spa}(R,R^+)$ has the structure of a spectral topological space, with basis of quasi-compact opens given by the \emph{rational open subsets}
$$U(\frac{g_1,\ldots,g_n}{f}) = \{v\mid v(f)\neq +\infty, v(g_i)\geq v(f)\forall i\}.$$
These subsets are parametrized by arbitrary elements $f, g_1,\ldots g_n$ of $R$.  Different choices of these elements can give rise to the same $U(\frac{g_1,\ldots,g_n}{f})$; however, the discrete Huber pair
$$(R[1/f],\widetilde{R^+[g_1/f,\ldots,g_n/f]}),$$
where the tilde denotes integral closure, \emph{is} canonically determined by the rational open subset $U(\frac{g_1,\ldots,g_n}{f})$; and conversely
$$U(\frac{g_1,\ldots,g_n}{f})=\operatorname{Spa}(R[1/f],\widetilde{R^+[g_1/f,\ldots,g_n/f]}),$$
showing that the rational open is likewise determined by the corresponding discrete Huber pair.

More precisely, the poset of rational open subsets of $\operatorname{Spa}(R,R^+)$ is canonically equivalent to the opposite of the category of \emph{finitary localizations} of $(R,R^+)$, meaning those maps discrete Huber pairs
$$(R,R^+)\rightarrow (S,S^+)$$
under $(R,R^+)$ for which $S$ is a finitary localization of $R$ and $S^+$ is generated, as an integrally closed subring of $S$ containing $R^+$, by finitely many elements.  Moreover, this equivalence of posets preserves finite intersections.

We can therefore transfer the Grothendieck topology on the poset of basic quasi-compact open subsets of $\operatorname{Spa}(R,R^+)$ over to this category of finitary localizations of $(R,R^+)$.  By definition, a covering is a collection of maps which is surjective on $\operatorname{Spa}(-,-)$.  However, there is also the following finitary algebraic description:

\begin{lemma}\label{Hubertopology}
Let $(R,R^+)$ be a discrete Huber pair.  Then the Grothendieck topology on the opposite category of finitary localizations of $(R,R^+)$ is generated by covers of the following type:
\begin{enumerate}
\item For $(S,S^+)$ a finitary localization of $(R,R^+)$ and $f\in B$, the maps
$$(S,S^+)\rightarrow (S,\widetilde{S^+[f]})$$
and
$$(S,S^+)\rightarrow (S[1/f],\widetilde{S^+[1/f]})$$
form a cover of $(S,S^+)$.
\item For $(S,S^+)$ a finitary localization of $(R,R^+)$ and $f_1,\ldots,f_n\in S$ with $f_1+\ldots+ f_n=1$, the maps
$$\{(S,S^+)\rightarrow (S[1/f_i],\widetilde{S^+[1/f_i]})\}_{i\in I}$$
form a cover of $(S,S^+)$.
\end{enumerate}
\end{lemma}
\begin{proof}
First note that each of these are covers.  For (1), this is because of the tautology that for any valuation $v$ we either have $v(f)\geq 0$ or $v(f)<0$ (in which case $v$ extends to a valuation on $R[1/f]$ with $v(1/f)\geq 0$, indeed $<0$).  For (2), for any valuation $v$ we must have $v(f_i)\leq 0$ for some $i$ by the ultrametric triangle inquality and the fact $v(1)=0$; then the valuation extends to $R[1/f_i]$ with $v(1/f_i)\geq 0$.

For the converse, first recall (see \cite[Lecture X]{Condensed} for the argument) that the Grothendieck topology is generated by the covers in (1) together with the algebraic Zariski covers, i.e.\ those of the following type: for $(S,S^+)$ a finitary localization of $(R;R^+)$ and $f_1,\ldots,f_n\in S$ generating the unit ideal, take the cover
$$\{(S,S^+)\rightarrow (S[1/f_i],\widetilde{S^+})\}_{i\in I}.$$
Thus it suffices to see that any such algebraic Zariski cover is refined by a cover as in (2).  However, we get $g_i$ such that $f_1g_1+\ldots f_ng_n=1$; then replacing $f_i$ by $f_ig_i$ we see that the cover (2) provides a refinement.
\end{proof}

By this lemma, given a presheaf on the opposite category of finitary localizations of $(R,R^+)$ which satisfies descent with respect to (1) and (2), we get a unique sheaf on the topological space $\operatorname{Spa}(R,R^+)$ whose value on a rational open is given by the value of our original presheaf on the corresponding finitary localization.

Now we apply this in our axiomatic GAGA situation.  Let $(R,R^+)$ be a discrete Huber pair which is of finite type over $(k,k)$, i.e.\ $R$ is a finite type $k$-algebra and $R^+$ is generated, as an integrally closed $k$-subalgebra of $R$, by finitely many elements.  To this data we can naturally assign the object
$$C(R,R^+)\in \operatorname{Sym},$$
which, we recall, is the the localization of $C(R)=\operatorname{Mod}_R(C)$ corresponding to the open subset $\mathcal{S}(R,R^+)\subset \mathcal{S}(R)$.  Given a finitary localization
$$(R,R^+)\rightarrow (S,S^+),$$
we have the induced map in $\operatorname{Sym}$
$$f^*:C(R,R^+)\rightarrow C(S,S^+).$$
Then $f$ is a monomorphism in $\operatorname{Sym}^{\mathrm{op}}$, being the composition of an open inclusion (forcing the finitely many generating elements of $S^+$ to be analytic) and a closed inclusion (algebraically inverting some $f\in R$ to go from $R$ to $S$).  Moreover this functor from the opposite category of finitary localizations of $(R,R^+)$ to subobjects of $C(R,R^+)$ in $\operatorname{Sym}^{\mathrm{op}}$ preserves finite intersections, by the same reasoning and base-change properties for open and closed immersions.  Finally arrive at the crucial:

\begin{proposition}
Suppose we are in the axiomatic GAGA set-up.  Then the functor
$$(R,R^+)\mapsto C(R,R^+)$$
from finite type discrete Huber pairs over $(k,k)$ to $\operatorname{Sym}$ gives a morphism of sites from the Huber topology on finitary localizations to the topology of open covers in $\operatorname{Sym}^{\mathrm{op}}$ described earlier in the lecture.
\end{proposition}

\begin{proof}
We have already seen that the functor preserves pullbacks, so it suffices to see that it sends our generating covers (1), (2) in Lemma~\ref{Hubertopology} to open covers in $\operatorname{Sym}^{\mathrm{op}}$.  For the first type (1) of generating covers, this is exactly axiom (4) in our set-up.  For the second type, it follows from axiom (5) together with the equality $D^{\mathrm{an}}(1)=\mathcal{S}(R)$ which follows from the definitions and axiom (1).
\end{proof}

\begin{corollary}
For each discrete Huber pair $(R,R^+)$ of finite type over $(k,k)$, there is a unique sheaf $C(-)$ on $\operatorname{Spa}(R,R^+)$ with values in $\operatorname{Sym}$ whose value on the basic quasicompact open $U(\frac{g_1,\ldots,g_n}{f})$ is given by
$$C\left(U(\frac{g_1,\ldots,g_n}{f})\right) = C(R[1/f],R^+[g_1/f,\ldots g_n/f]).$$
\end{corollary}

Now we are almost done with the proof of GAGA.  Both assignments
$$R\mapsto \operatorname{Spa}(R,k)$$
and
$$R\mapsto \operatorname{Spa}(R,R)$$
glue with respect to the Zariski topology; indeed it is elementary to check in both cases that Zariski covers of $R$ go to covers of rational open subsets.  Thus these assignments globalize to
$$X\mapsto \operatorname{Spa}(X,k)$$
and
$$X\mapsto \operatorname{Spa}(X,X)$$
for any finite type $k$-scheme $X$, hence likewise the sheaf $C(-)$ in the corollary uniquely extends to these global constructions by descent.  Moreover, the natural open inclusion $\operatorname{Spa}(R,R)\subset  \operatorname{Spa}(R,k)$ glues to a map
$$\operatorname{Spa}(X,X)\rightarrow \operatorname{Spa}(X,k)$$
of topological spaces, and the sheaf of categories on $\operatorname{Spa}(X,X)$ is the pullback of categories on $\operatorname{Spa}(X,k)$. The map of topological spaces is by construction locally an open inclusion on the source, but the valuative criterion for properness exactly says that it is a bijection on underlying sets when $X$ is proper.  Thus we deduce that it is a homeomorphism when $X$ is proper.  This proves the abstract GAGA theorem.\\

\textbf{Exercise 1.} Show that for any open subset $U\subset \mathbb{C}^n$, the category $C(U)$ can be described in terms of the structure sheaf (of liquid rings) $\mathcal{O}(-)$ on $\mathbb{C}^n$ as follows: $C(U)$ identifies with the full subcategory of $\mathcal{O}$-module sheaves (in derived liquid $\mathbb{C}$-vector spaces) consisting of those $\mathcal{M}$ for which
$$\mathcal{M}(D)\otimes_{\mathcal{O}(D)} \mathcal{O}\vert_D\overset{\sim}{\rightarrow} \mathcal{M}\vert_D$$
for every open polydisk $D\subset U$, i.e.\ in some sense $C(\mathbb{C}^n)$ is the category of \emph{quasicoherent} liquid sheaves on the analytic space $\mathbb{C}^n$.\\

\textbf{Exercise 2.} Show that in fact the quasicoherence condition in Exercise 1 is automatic, i.e.~it holds for any sheaf of $\mathcal O$-modules in $\mathcal D(\mathrm{Liq}_p(\mathbb C))$. More generally, if $X$ is locally compact Hausdorff and $C\in \operatorname{Sym}$ with a map $\mathcal S(C)\to X$, inducing (by pushforward) a structure sheaf $\mathcal O_C$ on $X$, any sheaf of $\mathcal O_C$-modules on $X$ is quasicoherent, and their category is equivalent to $C$. (Hint: Given any sheaf $\mathcal M$ of $\mathcal O_C$-modules on $X$, we want to see that $\mathcal M(X)\otimes_C \mathcal O_C\to \mathcal M$ is an isomorphism. It suffices to check after evaluating on compact subsets $Z\subset X$ (i.e., the filtered colimit of sections on open neighborhoods $U_n\subset X$ of $Z$). Now for compact $i: Z\subset X$ with open complement $j: U\to X$, use the triangle $j_! j^\ast \mathcal M\to \mathcal M\to i_\ast i^\ast \mathcal M$ to reduce the question to $i_\ast i^\ast \mathcal M$ or $j_! j^\ast \mathcal M$. For $i_\ast i^\ast \mathcal M$, the claim follows from idempotence of $\mathcal O_C(Z)$, as $\Gamma(X,i_\ast i^\ast \mathcal M) = \Gamma(Z,i^\ast \mathcal M)$ is a module over $\mathcal O_C(Z)$. For $j_! j^\ast \mathcal M$, write this further as a colimit of $i_{n\ast} i_n^! \mathcal M$ for closed subsets $Z_n\subset U$ exhausting $U$, and then use that $\Gamma(X,i_{n\ast} i_n^! \mathcal M) = \Gamma(Z_n,i_n^! \mathcal M)$ is a module over $\mathcal O_C(Z_n)$, and $\mathcal O_C(Z_n)\otimes_C \mathcal O_C(Z)=0$ as $Z_n\cap Z=\emptyset\in \mathcal S(C)$.)

\textbf{Exercise 3.} Show that if $f$ is a closed immersion and $g$ is an open immersion, then $f\circ g$ can be rewritten as $g'\circ f'$ where $g'$ is an open immersion and $f'$ is a closed immersion.\\

\textbf{Exercise 4.} Explicitly prove GAGA for $\mathbb{P}^n_{\mathbb{C}}$ by refining a Zariski cover by an analytic cover.\newpage

\section{Lecture VII: Analytification and GAGA, redux}

As a lot happened in the last two lectures, let us reorganize the material using the following (insanely general, but still useful) definition, a bit analogous to the notion of a ringed space.

\begin{definition}\label{def:categorifiedlocale} A \emph{categorified locale} is a pair $(X,C)$ consisting of a locale $X$ and a cocomplete closed symmetric monoidal stable $\infty$-category $C$ with a map $f: \mathcal S(C)\to X$.
\end{definition}

In particular, this gives rise to a sheaf of (cocomplete closed symmetric monoidal stable) $\infty$-categories on $X$, taking any $U$ to $C(U):=C(f^{-1}(U))$. In particular, looking at the endomorphisms of the unit object, one gets a sheaf of $E_\infty$-algebras on $X$, that one may call the structure sheaf.

In the last lecture, several different functors from finite type $\mathbb C$-algebras $R$ to categorified locales were defined. In all cases, these sent $R=\mathbb C$ to $(\ast,\mathcal D(\mathrm{Liq}_p(\mathbb C)))$, and so in general one gets categorified locales over $(\ast,\mathcal D(\mathrm{Liq}_p(\mathbb C)))$.

Recall that to any $R$ we can associate the derived $\infty$-category $\mathcal D(\mathrm{Liq}_p(R)) = \mathrm{Mod}_R(\mathcal D(\mathrm{Liq}_p))$ of $p$-liquid $R$-modules. Inside $\mathcal S(R):=\mathcal S(\mathcal D(\mathrm{Liq}_p(R)))$, we defined open subsets $\{|f|<r\}$ and $\{|f|>r\}$ for any $f\in R$ and $r>0$. In particular, for any $f\in R$, we can consider the open subset
\[
\mathcal S(R,f)=\{|f|<\infty\}=\bigcup_{r>0} \{|f|<r\}\subset \mathcal S(R),
\]
and then set
\[
\mathcal S(R,R)=\bigcap_{f\in R} \{|f|<\infty\}\subset \mathcal S(R).
\]
As it is enough to take the intersection over a set of algebra generators, this is in fact an open subset of $\mathcal S(R)$. Then the arguments from Lecture V defined a natural continuous map
\[
\mathcal S(R,R)\to \mathcal M^{\mathrm{Berk}}(R) = \Spec(R)(\mathbb C)
\]
to the classical points of $\Spec(R)$ (with their complex topology), so that the loci $\{|f|<r\}$ and $\{|f|>r\}$ are obtained as preimages of the evident loci in $\Spec(R)(\mathbb C)$.

This construction glues for the Zariski topology, so for any scheme $X$ locally of finite type over $C$, one can define a categorified locale
\[
(X(\mathbb C),C^{\mathrm{an}}(X))
\]
which we refer to the as the analytification of $X$. (In fact, we will later identify the usual category of complex-analytic spaces as a full subcategory of categorified locales over $(\ast,\mathcal D(\mathrm{Liq}_p(\mathbb C)))$, and under this identification, this corresponds to the usual complex-analytic space associated to $X$. Most of the work here was done in the last lecture, where $C^{\mathrm{an}}(X)$ was described explicitly.)

One main goal of the last lecture was to give a proof of GAGA. This was in fact understood not merely as an identification of the category of coherent sheaves, but of all of $C^{\mathrm{an}}(X)$, in case $X$ is proper.

Namely, one can also pass to a very algebraic ``categorified locale''. To any $R$, we associate all of $\mathcal D(\mathrm{Liq}_p(R))$, and we only use the map
\[
\mathcal S(R)\to \mathrm{Spec}(R)^{\mathrm{op}}
\]
coming from the usual Zariski localizations (which define idempotent algebras, and thus closed subsets of $\mathcal S(R)$). One has to be slightly careful, as now Zariski localizations give closed embeddings, but one can still glue (finitely many) categorified locales along closed subsets. In particular, if $X$ is separated and of finite type over $\mathbb C$, one can define a categorified locale
\[
(X,C^{\mathrm{alg}}(X)).
\]
This construction did not use any of the ``topology'' of $\mathbb C$. In fact, it could have been carried not with liquid $\mathbb C$-vector spaces, but with all condensed $\mathbb C$-vector spaces; the current version would then simply be a base change to the liquid situation.

Now the GAGA theorem admits the following formulation.

\begin{theorem}\label{thm:GAGAabstract} For proper $X$, there is a natural equivalence of (cocomplete closed symmetric monoidal $\mathcal D(\mathrm{Liq}_p(\mathbb C))$-linear stable) $\infty$-categories
\[
C^{\mathrm{an}}(X)\cong C^{\mathrm{alg}}(X).
\]
\end{theorem}

In fact, the proof will proceed by comparing them as categorified locales; but for this, $C^{\mathrm{alg}}(X)$ first needs a finer ``structural map'':

\begin{theorem}\label{thm:localizeonSpa} For a finite type $\mathbb C$-algebra $R$, there is a natural map of locales
\[
\mathcal S(R)\to \mathrm{Spa}(R,\mathbb C)'.
\]
Here, $\mathrm{Spa}(R,\mathbb C)'$ is the set of valuations $|\cdot|: R\to \Gamma\cup\{0\}$ with values in some totally ordered abelian group $\Gamma$ (satisfying $|fg|=|f||g|$, $|f+g|\leq \mathrm{max}(|f|,|g|)$, $|0|=0$, $|\mathbb C^\times|=1$). We denote the inequalities in $\Gamma$ by the symbol $\ll$, and $\mathrm{Spa}(R,\mathbb C)'$ is given the topology for which the subsets $\{|f|\ll |g|\}$ form a generating family of closed subsets. The preimage of $\{|g|\neq 0\}$ is $\mathcal S(R[\tfrac 1g])$, and the preimage of $\{|f|\ll |g|\}\subset \{|g|\neq 0\}$ is
\[
\bigcap_{r>0} \{|\tfrac fg|\leq r\}\subset \mathcal S(R[\tfrac 1g])\subset \mathcal S(R).
\]
\end{theorem}

Something slightly weird was done to the topology of $\mathrm{Spa}(R,\mathbb C)$: Usually, in adic spectra, $\{|f|\neq 0\}$ and $\{|f|\leq 1\}$ define open subsets. For us, the first is closed, while the second is open. Intuitively, the valuations considered here keep track only of ``infinite'' values of $|f|$. Namely, $|f|\leq 1$ actually means $|f|=O(1)$, i.e.~$|f|<\infty$ in usual notation. This is the reason we also write $|f|\ll |g|$ in place of $|f|<|g|$.\footnote{Beware however of a trap: While this reads well in the direction $|f|$ is much smaller than $|g|$, which is always true when $f=0$, the condition $0\ll |g|$ might also be read as $|g|\gg 0$, i.e.~$g$ is much larger than zero. It doesn't mean that, $0\ll |g|$ only means $|g|\neq 0$.} For the same reason, we will also write $|f|\leq O(|g|)$ in place of $|f|\leq |g|$. In other words, this map to the adic spectrum is not keeping track of precise absolute values, but only of asymptotics (for any pair of functions in $R$!). This also explains why one gets the strong triangle inequality -- if both $|f_1|,|f_2|\leq O(|g|)$, then also $|f_1+f_2|\leq O(|g|)$.

Another point worth making is that the valuations are not required to be continuous in any way -- we treat $R$ as a discrete $\mathbb C$-algebra in the above definition of $\mathrm{Spa}(R,\mathbb C)$. In fact, we ask $|\mathbb C^\times|=1$ while $|0|=0$.

\begin{proof} The proof is similar to the construction of the map $\mathcal S(R,R)\to \mathcal M^{\mathrm{Berk}}(R)$: One can describe the adic spectrum in terms of possible binary relations $\ll$ on $R$, subject to some simple axioms, see \cite[proof of Proposition 2.2]{HuberContVal}. One then has to see that the closed subsets of $\mathcal S(R)$ we assigned by hand to $\{|f|\ll |g|\}$ satisfy the correct intersection properties:
\begin{enumerate}
\item $\{|0|\ll |1|\}$ is everything;
\item $\{|f|\ll |g|\}\cap \{|g|\ll |f|\}$ is empty;
\item $\{|f|\ll |h|\}\subset \{|f|\ll |g|\}\cup \{|g|\ll |h|\}$;
\item $\{|f|\ll |g+h|\}\subset \{|f|\ll |g|\}\cup \{|f|\ll |h|\}$;
\item $\{|f|\ll |g|\}\cap \{0\ll |h|\} = \{|fh|\ll |gh|\}$.
\end{enumerate}
Properties (1), (2), and (5) are straightforward. For (3), note that $h$ is invertible on the left-hand side, so we can assume that $h$ is invertible, and then replace $f$ by $\tfrac fh$ and $g$ by $\tfrac gh$ to assume that $h=1$. Now we can cover the whole space by the loci $\{|g|\geq 1\}$ and $\{|g|\leq 1\}$ as defined in Lecture V. On the first, $\{|f|\ll 1\}$ is contained in $\{|f|\ll |g|\}$, while on the second $\{|g|\ll 1\}$ is everything and hence $\{|f|\ll 1\}\subset \{|g|\ll 1\}$. It remains to prove (4). On its left-hand side, $g+h$ is invertible, so we can assume $g+h$ is invertible; dividing by it, we can assume $g+h=1$. Then the whole space is covered by the loci $\{|g|\geq \tfrac 12\}$ and $\{|h|\geq \tfrac 12\}$ by the arguments from Lecture V. On the first, $\{|f|\ll 1\}\subset \{|f|\ll |g|\}$, and on the second $\{|f|\ll 1\}\subset \{|f|\ll |h|\}$.
\end{proof}

There is a natural map from $\mathrm{Spa}(R,\mathbb C)'$ to $\Spec(R)^{\mathrm{op}}$. Namely, there is a map
\[
\mathrm{Spa}(R,\mathbb C)'\to \mathrm{Spec}(R)^{\mathrm{op}}
\]
such that the preimage of $\{f\neq 0\}$ is $\{|f|\neq 0\}$. In particular, the categorified locale
\[
(\mathrm{Spec}(R)^{\mathrm{op}},C^{\mathrm{alg}}(R))
\]
naturally refines to a categorified locale
\[
(\mathrm{Spa}(R,\mathbb C)',C^{\mathrm{alg}}(R)).
\]
One can also glue this construction: for any separated finite type $\mathbb C$-scheme $X$, one can glue $\mathrm{Spa}(R,\mathbb C)'$ to a space $X^{\mathrm{ad}'/\mathbb C}$; beware again that the gluing happens along closed subsets. This gives a categorified locale
\[
(X^{\mathrm{ad}'/\mathbb C},C^{\mathrm{alg}}(X))
\]
refining $(X,C^{\mathrm{alg}}(X))$.

On the other hand, recall that for any integrally closed subalgebra $R^+\subset R$, one defines
\[
\mathrm{Spa}(R,R^+)'\subset \mathrm{Spa}(R,\mathbb C)'
\]
as the open subset where $|f|\leq O(1)$ for all $f\in R^+$. (It is enough to demand this condition for a finite subset of $R^+$, generating it as an integrally closed subalgebra of $R$.) Recall also that $\{|f|\leq O(1)\}$ translates in $\mathcal S(R)$ into the condition $\{|f|<\infty\}$, so we find that
\[
\mathcal S(R,R)\subset \mathcal S(R)
\]
is the preimage of
\[
\mathrm{Spa}(R,R)'\subset \mathrm{Spa}(R,\mathbb C)'.
\]

In particular, there is an open immersion of categorified locales
\[
(\mathrm{Spa}(R,R)',C^{\mathrm{an}}(A))\subset (\mathrm{Spa}(R,\mathbb C)',C^{\mathrm{alg}}(A)).
\]
Now we glue this information. Gluing now along open subsets, one gets another adic space $X^{\mathrm{ad}}$ associated to $X$, glued from $\mathrm{Spa}(R,R)$'s. If $X$ is separated finite type, there is a natural open immersion $X^{\mathrm{ad}}\hookrightarrow X^{\mathrm{ad}'/\mathbb C}$ glued from the similar open immersions on affine pieces. Moreover, the valuative criterion of properness says that $X$ is proper if and only if this map is bijective, i.e.~a homeomorphism. In other words, for any separated $X$, we get an open immersion of categorified locales
\[
(X^{\mathrm{ad}},C^{\mathrm{an}}(X))\subset (X^{\mathrm{ad}'/\mathbb C},C^{\mathrm{alg}}(X)).
\]
If $X$ is proper, the two spaces are the same, and thus the categories are the same, giving Theorem~\ref{thm:GAGAabstract}.

In fact, there is probably one small addition one should make, about how $C^{\mathrm{an}}(X)$ is glued. In the previous step, we glued it on $X^{\mathrm{ad}}$, while previously we glued it on $X(\mathbb C)$. What is the relation?

There is a natural continuous map
\[
\Spec(R)(\mathbb C)\to \Spec(R),
\]
where the target is equipped with the Zariski topology, and there is also a natural continuous map
\[
\Spa(R,R)\to \Spec(R),
\]
where the preimage of $\{f\neq 0\}$ is $\{|f|\geq O(1)\}$ (which translates to $\bigcup_{r>0} \{|f|>r\}$ in $\mathcal S(R,R)$). Over these comparison maps, one gets maps of categorified locales
\[
(\Spec(R)(\mathbb C),C^{\mathrm{an}}(R))\to (\Spec(R),C^{\mathrm{an}}(R))\leftarrow (\Spa(R,R),C^{\mathrm{an}}(R))
\]
which glues to comparison maps
\[
(X(\mathbb C),C^{\mathrm{an}}(X))\to (X,C^{\mathrm{an}}(X))\leftarrow (X^{\mathrm{ad}},C^{\mathrm{an}}(X)).
\]
In other words, to glue $C^{\mathrm{an}}(X)$, Zariski covers suffice, and then one can refine the localizations either to $X(\mathbb C)$ or to $X^{\mathrm{ad}}$.

What is striking about this proof of GAGA is its formality -- the key is to construct the map to the adic spectrum, which reduces to some formal calculations with idempotent algebras. Once one has that, general nonsense reduces this version of GAGA to the purely valuation-theoretic assertion that the map $X^{\mathrm{ad}}\to X^{\mathrm{ad}'/\mathbb C}$ is a bijection.

We also note that $\mathcal S(A)$ has the striking property that it is globally algebraic -- the global functions are precisely $R$ -- while it is locally analytic -- the functions on the open subset $\mathcal S(R,R)$ are analytic functions. Using the map
\[
\mathcal S(R)\to \mathrm{Spa}(R,\mathbb C)'
\]
one can also localize near the ``boundary'' of $\mathcal S(R)$ using specific growth conditions on functions. This makes it possible to construct algebraic functions, or also coherent sheaves etc., by gluing them after such ``analytic'' localizations. This is somewhat reminiscent of the ``affine GAGA'' of Bakker--Brunebarbe--Tsimerman \cite{BakkerBrunebarbeTsimerman}, built on o-minimal techniques.

\section{Lecture VIII: Nuclear modules}

In the previous lectures, we proved a GAGA theorem of the following type: if $X$ is a proper $\mathbb{C}$-scheme, then two a priori different $\infty$-categories of ``$p$-liquid derived quasicoherent sheaves'' on $X$ actually agree.  The first is Zariski-glued from the categories
$$C(R)=\operatorname{Mod}_R(\mathcal{D}(\operatorname{Liq}_p))$$
of derived $R$-modules in $p$-liquid vector spaces, while the second is Zariski-glued from the analytic localizations of these,
$$C(R,R),$$
which localize on the underlying topological space $\operatorname{Spec}(R)(\mathbb{C})$ and therefore also Zariski-glue.

The usual GAGA theorem is a similar comparison theorem, but for coherent sheaves instead.  The above categories of quasi-coherent sheaves are much bigger than the categories of coherent sheaves: first of all, there's no finiteness condition on the $R$-modules in definition of $C(R)$, and second of all, our base category is $\operatorname{Liq}_p$ and not the usual algebraic category of $\mathbb{C}$-vector spaces.  In this lecture we want to work towards explaining how to recover the usual categories of coherent sheaves inside these larger categories, so as to deduce the usual GAGA from the one we have just proved.

The general mechanism for this fits into the familiar ansatz that ``finite = compact + discrete'': we will recover the category of coherent sheaves as the intersection of two different full subcategories of our large categories, one which corresponds to a kind of compactness condition and the other to a kind of discreteness condition.  In this lecture, we talk about the appropriate discreteness condition.

The first thing to emphasize is that there is a naive notion of discreteness for $p$-liquid vector spaces.  Indeed, the forgetful functor
$$\operatorname{Liq}_p\rightarrow\operatorname{Mod}_\mathbb{C},$$
$$X\mapsto X(\ast),$$
admits an exact fully faithful left adjoint $M\mapsto M^\delta := M\otimes_{\mathbb{C}(\ast)}\mathbb{C}$; more concretely, if $M$ is a direct sum of copies of $\mathbb{C}$ in the category of abstract vector spaces, then $M^\delta$ is the ``same'' direct sum of copies of the object $\mathbb{C}$ in the category of $p$-liquid spaces.  (Note however that $M^\delta$ is not discrete as an underlying condensed set; it is only discrete ``relative to $\mathbb{C}$'').

We can call a $p$-liquid vector space ``discrete'' if it is in the essential image of this functor.  But this category of discrete $p$-liquid vector spaces, which is equivalent to the category of abstract $\mathbb{C}$-vector spaces, is not robust enough for our purposes.  The reason is that the most relevant liquid vector spaces for us are the rings of overconvergent holomorphic functions on closed polydisks, and those are not discrete.

But it turns out that there is an enlargement of the category of discrete $p$-liquid vector spaces, which contains these rings of overconvergent holomorphic functions on closed polydisks (and rings of convergent holomoprhic functions on open polydisks), but is still constrained enough that it can be used as a replacement for discreteness in the ansatz ``finite = compact + discrete''.  This is the category of \emph{nuclear modules} which is the subject of this lecture.\footnote{The notion of nuclearity is due to Grothendieck and arose from his study of tensor products of locally convex topological vector spaces, \cite{GrothendieckTensor}.  It appears for us here in a slightly different guise, but the fundamental ideas are the same.}

The definition of nuclear module is based on the notion of nuclear (or trace-class) map, which we can give in the context of an arbitrary closed symmetric monoidal $\infty$-category $C$.  First, some notation: for $x,y\in C$, we will denote by
$$y^x$$
the internal hom object from $x$ to $y$.  We will also set
$$x^\vee := 1^x$$
where $1$ denotes the unit object of $C$, and
$$x(\ast):= Hom(1,x),$$
the mapping anima from the unit to $x$.

Note for all $x,y\in C$ that there is a natural map
$$x^\vee\otimes y\rightarrow y^x,$$
and in particular a natural map
$$(x^\vee\otimes y)(\ast)\rightarrow \operatorname{Hom}(x,y).$$

\begin{definition} Let $f:x\rightarrow y$ be a map in $C$.  We say that $f$ is \emph{trace-class} or \emph{nuclear} if $f$ lies in the image of the natural map
$$(x^\vee\otimes y)(\ast)\rightarrow \operatorname{Hom}(x,y).$$
\end{definition}

Note that this is a $\pi_0$ condition, i.e.\ we are asking that the component of the anima $\operatorname{Hom}(x,y)$ containing $f$ is hit by the map from $(x^\vee\otimes y)(\ast)$.  The following lemma can be proved by some diagram chasing.

\begin{lemma}\label{traceclassproperties}
Let $C$ be a closed symmetric monoidal $\infty$-category.
\begin{enumerate}
\item If $f:x\rightarrow y$ is trace-class and $g:y\rightarrow y'$ and $h:x'\rightarrow x$ are arbitrary, then $g\circ f\circ h$ is trace-class.
\item If $f:x\rightarrow y$ and $f':x'\rightarrow y'$ are trace-class, then $f\otimes f': x\otimes x'\rightarrow y\otimes y'$ is trace-class.
\item If $f:x\rightarrow y$ is trace-class and $c\in C$ is arbitrary, then the commutative square
\[\xymatrix{
c^\vee\otimes x \ar[r]\ar[d] & c^\vee\otimes y\ar[d]\\
x^c \ar[r] & y^c
}\]
admits a diagonal map $x^c\rightarrow c^\vee\otimes y$ making both triangles commute.
\end{enumerate}
\end{lemma}

To continue, we add some hypotheses to our closed symmetric monoidal $\infty$-category $C$, namely we require:

\begin{enumerate}
\item $C$ is stable and compactly generated;
\item The unit $1\in C$ is a compact object.
\end{enumerate}

Recall that an object $x\in C$ is called compact if $\operatorname{Hom}(x,-):C\rightarrow \operatorname{An}$ commutes with filtered colimits, and $C$ is \emph{compactly generated} if it has all colimits, and there is a collection of compact objects $C_0\subset C$ such that $C$ is the smallest cocomplete full subcategory of $C$ containing $C_0$.  We can close $C_0$ up under finite colimits, and then we even have the canonical expression
$$x=\varinjlim_{c_0\in C_0, c_0\rightarrow x} c_0$$
of $x$ as a filtered colimit of objects in $C_0$ for all $x\in C$, and indeed $C\simeq \operatorname{Ind}(C_0)$.

\begin{remark}
Note that we do \emph{not} require that the collection of compact objects be closed under tensor product.  We would love to be able to impose this very convenient hypothesis, but unfortunately it is (probably) not satisfied in our main example of $C=\mathcal{D}(\operatorname{Liq}_p)$.  Indeed, an object of $\mathcal{D}(\operatorname{Liq}_p)$ is compact if and only if it can be represented by a finite complex of compact projective objects of $\operatorname{Liq}_p$.  As recorded in the appendix to Lecture III, the collection of compact projective objects is not closed under tensor products; and there seems to be no reason why the tensor product of two compact projective objects should be represented by a finite complex of compact projective objects, either.

The main difficulty caused by this lack of closure under tensor products is the following: a compact object $c\in C$ need not be ``compact in the sense of internal hom'', i.e.\ the functor $x\mapsto x^c:C\rightarrow C$ need not commute with filtered colimits.
\end{remark}

In the above abstract context we have the following simple lemma.

\begin{lemma}\label{factorthroughcompact}
Let $C$ be a compactly generated closed symmetric monoidal $\infty$-category with compact unit object.  For any trace-class map $f:x\rightarrow y$ in $C$, there is a compact object $c\in C$ and a factoring of $f$ as
$$x\rightarrow c\rightarrow y$$
where $x\rightarrow c$ is also trace-class.
\end{lemma}
\begin{proof}
By definition, the map $f$ comes from a class in $\pi_0( (x^\vee\otimes y)(\ast))$.  As the functor $x^\vee\otimes -$ and the functor $(-)(\ast)$ both commute with filtered colimits, and $y$ is a filtered colimit of compact objects, we deduce that $f$  lifts to a class in $\pi_0( (x^\vee\otimes c)(\ast))$ for some compact $c$ mapping to $y$, giving the claim.
\end{proof}

\begin{definition}
Let $C$ be a stable compactly generated closed symmetric monoidal $\infty$-category with compact unit object.
\begin{enumerate}
\item For $x\in C$, we say that $x$ is \emph{nuclear} if for all compact objects $c\in C$, the natural map
$$(c^\vee\otimes x)(\ast)\rightarrow \operatorname{Hom}(c,x)$$
is an isomorphism.
\item For $x\in C$, we say that $x$ is \emph{basic nuclear} it is isomorphic to the colimit of a sequence
$$x_0\rightarrow x_1\rightarrow\ldots$$
of trace-class maps between objects of $C$.
\end{enumerate}
\end{definition}

Note that by replacing $c$ by its shifts, we see that condition (1) is insensitive to whether we use mapping spectra or mapping anima.  It's more convenient to use mapping spectra because then we stay in the stable context, so we'll do that from now on.

\begin{theorem}
Let $C$ be a stable compactly generated closed symmetric monoidal $\infty$-category with compact unit object.  Then:
\begin{enumerate}
\item The full subcategory $\operatorname{Nuc}$ of nuclear objects in $C$ is stable, closed under colimits, and closed under tensor products,
\item The stable $\infty$-category $\operatorname{Nuc}$ is $\aleph_1$-compactly generated, and the $\aleph_1$-compact objects are exactly the basic nuclear modules.
\end{enumerate}
\end{theorem}
\begin{proof}
Stability and cocompleteness of $\operatorname{Nuc}$ follow from the fact that both functors $(c^\vee\otimes -)(\ast)$ and $\operatorname{Hom}(c,-)$ (with values in spectra) commute with all colimits.  For closure under tensor products, this follows from (2), since basic nuclears are clearly closed under tensor products.

So we just need to show (2).  First note that basic nuclears are nuclear: indeed, if $x$ is basic nuclear given by $\varinjlim x_n$ as in the definition, then for $c\in C$, the two ind-systems
$$x_0^c \rightarrow x_1^c\rightarrow\ldots$$
and
$$c^\vee \otimes x_0\rightarrow c^\vee\otimes x_1\rightarrow\ldots$$
are isomorphic as a consequence of Lemma \ref{traceclassproperties}.  When $c$ is compact, we can take maps from the unit object and then pull in the colimits.  We deduce that $x$ is nuclear.

Next, note that any basic nuclear is a sequential colimit of \emph{compact} objects under trace-class maps.  Indeed, if $x=\varinjlim x_n$ along trace-class maps, by Lemma \ref{factorthroughcompact} we can factor each $x_n\rightarrow x_{n+1}$ as $x_n\rightarrow c_n\rightarrow x_{n+1}$ where $c_n$ is compact and $x_n\rightarrow c_n$ is trace-class.  Then the $c_n$ with transition maps given by the compositions $c_n\rightarrow x_{n+1}\rightarrow c_{n+1}$ provide an equivalent ind-system where the terms are compact and the maps are trace-class, as desired.

It follows that any basic nuclear is $\aleph_1$-compact in $C$, hence in particular in $\operatorname{Nuc}$.  Next we claim that the basic nuclear objects form a stable subcategory closed under countable colimits.  It suffices to show it's closed under cones and countable direct sums.  For cones, the argument is in \cite[Lecture XIII]{Analytic}: briefly, given a map between sequential colimits of compact objects under trace-class maps, we can assume the map is given term-wise, and then we term-wise pass to the cofiber; the transition maps in the cofiber isn't necessarily trace-class, but the composition of any two consecutive transition maps is, and that's enough.  For countable direct sums, we can choose a representative of each term as a sequential colimit along trace-class maps, then rewrite the countable direct sum as a sequential colimit of finite direct sums and pass to the diagonal in the sequential colimit of sequential colimits.

Now we can prove (2).  In the previous two paragraphs, we saw that every basic nuclear is $\aleph_1$-compact, and that the basic nuclears are stable under countable colimits.  Thus it suffices to show that for any nuclear $x$, if $\operatorname{Hom}(y,x)=0$ for all basic nuclear $y$, then $x=0$. Thus suppose $\operatorname{Hom}(y,x)=0$, and let $c$ be compact.  We will show that any map $c\rightarrow x$ is $0$, which will imply $x=0$ because $C$ is compactly generated.  But indeed, by repeatedly applying Lemma \ref{factorthroughcompact}, we deduce that $c\rightarrow x$ factors through a basic nuclear; hence by hypothesis it is the zero map.\end{proof}

By this, every nuclear object is an $\aleph_1$-filtered colimit of basic nuclear objects.  Thus, to understand nuclear objects, to a large degree we reduce to understanding basic nuclear objects.

Now we turn to our case of interest, which is
$$C=\mathcal{D}(\operatorname{Liq}_p).$$
We would like to gain an understanding of the basic nuclear objects in $C$.  By the above, every basic nuclear is a sequential colimit of compact objects with trace-class transition maps.  Actually, it suffices to only look at those compact objects which can be represented by finite complexes where each term is of the form $\mathcal{M}_{<p}(S)$ for $S$ profinite, as these already contain a generating family of compact objects closed under shifts and finite colimits.

Using this we can make the following further reduction.

\begin{lemma}\label{basicnuclearcomplex}
Every basic nuclear in $\mathcal{D}(\operatorname{Liq}_p)$ is represented by a complex which is a sequential colimit of complexes
$$\varinjlim_n C_\bullet^{(n)}$$
such that:
\begin{enumerate}
\item Each term $C_k^{(n)}$ of each complex $C_\bullet^{(n)}$ is a measure space $\mathcal{M}_{<p}(S)$ with $S\in\operatorname{Prof}$;
\item each map $C_\bullet^{(n)}\rightarrow C_\bullet^{(n+1)}$ is term-wise given by a trace-class map between measure spaces.
\end{enumerate}
\end{lemma}
\begin{proof}
We will need to use the following non-trivial facts about these $\mathcal{M}_{<p}(S)$: first, the (derived) dual
$$\mathcal{M}_{<p}(S)^\vee = \underline{\operatorname{RHom}}(\mathcal{M}_{<p}(S),\mathbb{R})= \underline{\operatorname{RHom}}(\mathbb{Z}[S],\mathbb{R})$$
sits in degree $0$; and second, $\mathcal{M}_{<p}(S)$ is flat.  The first was proved in \cite[Lecture III]{Condensed} and the second was proved in Lecture III of these notes.

To prove the claim, we need to see that every trace-class map in the derived category between finite complexes $M_\bullet,N_\bullet$ with terms of the form $\mathcal{M}_{<p}(S)$ can be represented by an honest map of complexes which is trace-class in each term.  But indeed, a trace-class map in the derived category comes from a class in
$$H_0( (M_\bullet)^\vee \otimes^L N_\bullet)(\ast)),$$
where we a priori use the derived dual and derived tensor product.  But by the facts we just recalled, these are the same as the underived dual and underived tensor product.  Thus we see that classes in $H_0( (M_\bullet)^\vee \otimes^L N_\bullet)(\ast))$ always come from degree $0$ cycles in the tensor product of complexes
$$(\operatorname{Hom}(M_\bullet,\mathbb{R})\otimes N_\bullet)(\ast).$$
Unwinding, such a degree $0$ cycle exactly gives a map of complexes $M_\bullet\rightarrow N_\bullet$ which is trace-class in each term and induces our original map in the derived category.\end{proof}

Next we analyze these trace-class maps $\mathcal{M}_{<p}(S)\rightarrow \mathcal{M}_{<p}(T)$.  It will be convenient to allow $S$ and $T$ also to be locally profinite as in Lecture IV, because we'll soon want to switch to sequence spaces.

\begin{proposition}\label{explicittraceclass}
Let $S$ and $T$ be locally profinite spaces.  Then a map $f:\mathcal{M}_{<p}(S)\rightarrow \mathcal{M}_{<p}(T)$ in $\mathcal{D}(\operatorname{Liq}_p)$ is trace-class if and only there is a $q<p$ and a $q$-summable sequence of real numbers $(\lambda_n)$ such that $f$ factors as the following kind of composite:
$$\mathcal{M}_{<p}(S)\rightarrow c_0(\mathbb{N})\overset{\cdot\lambda}{\rightarrow}\ell_q(\mathbb{N})\subset \mathcal{M}_{<p}(\mathbb{N})\rightarrow \mathcal{M}_{<p}(T),$$
where the two outer maps are arbitrary and the displayed map is termwise multiplication by our sequence $(\lambda_n)$ (see Lemma \ref{basictraceclass}).

\end{proposition}
\begin{proof}
Trace-class maps $\mathcal{M}_{<p}(S)\rightarrow \mathcal{M}_{<p}(T)$ come from classes in
$$(C_0(S;\mathbb{R})\otimes_{\mathbb{R}_{<p}} \mathcal{M}_{<p}(T))(\ast),$$
where here again we use that the derived dual of $\mathcal{M}_{<p}(S)$ lives in degree $0$, given by the Banach space $C_0(S;\mathbb{R})$ of continuous functions $S\rightarrow\mathbb{R}$ vanishing at $\infty$, and that $\mathcal{M}_{<p}(T)$ is flat.  By Lemma \ref{compactnullsequence} from Lecture II, $C_0(S;\mathbb{R})$ admits a surjection from a direct sum of copies of $\mathcal{M}_{<p}(\mathbb{N})$, corresponding to nullsequences in the Banach space $C_0(S;\mathbb{R})$.  On the other hand, we have
$$\mathcal{M}_{<p}(\mathbb{N})\otimes_{\mathbb{R}_{<p}}\mathcal{M}_{<p}(T)=\mathcal{M}_{<p}(\mathbb{N}\times T)$$
by Lemma \ref{locprofflat}, and this is the subset of
$$\prod_\mathbb{N} \mathcal{M}_{<p}(T)$$
consisting of those sequences of $<p$-measures on $T$ which are of the form $$(\lambda_n\mu_n)_{n\in\mathbb{N}}$$
where $(\lambda_n)$ is a sequence of real numbers which is $<p$-summable and the $\mu_n$ are uniformly bounded, i.e.\ there is a $q<p$ and $C>0$ such that $\mu_n \in  \mathcal{M}(T)_{\ell^q\leq C}$ for all $n$.

In total, we deduce that trace-class maps $f:\mathcal{M}_{<p}(S)\rightarrow \mathcal{M}_{<p}(T)$ are exactly those of the following form: take the following data:
\begin{enumerate}
\item A null-sequence $\varphi_1,\varphi_2,\ldots$ in $C_0(S;\mathbb{R})$;
\item A $<p$-summable sequence $\lambda_1,\lambda_2,\ldots$ of real numbers;
\item A uniformly bounded sequence $\mu_1,\mu_2,\ldots$ of $<p$-measures on $T$;
\end{enumerate}
and produce from them the map $f:\mathcal{M}_{<p}(S)\rightarrow \mathcal{M}_{<p}(T)$ given by
$$f\vert_S(s)=\sum_{n=1}^\infty \lambda_n \varphi_n(s) \mu_n.$$
Up to replacing ``uniformly bounded sequence $\mu_1,\mu_2,\ldots$ of $<p$-measures on $T$'' by ``nullsequence $\mu_1,\mu_2,\ldots$ of $<p$-measures on $T$'', this corresponds exactly to writing $f$ as a composite of three maps as in the statement.  So to finish it suffices to show that we can arrange it so that $\mu_1,\mu_2,\ldots$ is a nullsequence.  For this we can write the nullsequence of norms of the $\varphi_1,\varphi_2,\ldots$ as a product of two null-sequences and move one of them over as coefficients on the $\mu_n$. \end{proof}

It follows that any nuclear object in $\mathcal{D}(\operatorname{Liq}_p)$ lies in the stable cocomplete subcategory generated by $\mathcal{M}_{<p}(\mathbb{N})$, and in particular any nuclear object is $\omega_1$-condensed.  However, by Proposition \ref{prop:lightlyprojective}, $\mathcal{M}_{<p}(\mathbb{N})$ is a compact projective object in $\omega_1$-condensed $p$-liquid modules.  It follows that we can run the above arguments inside the smaller stable cocomplete category generated by $\mathcal{M}_{<p}(\mathbb{N})$ to deduce that, in the description of basic nuclears given by Lemma \ref{basicnuclearcomplex}, we can use $\mathcal{M}_{<p}(\mathbb{N})$'s instead of $\mathcal{M}_{<p}(S)$'s as the terms in our complex.

But now that we're in this sequence space context, we have the following further refinement:

\begin{proposition}\label{basicnuclearcomplexinjective}
Every basic nuclear in $\mathcal{D}(\operatorname{Liq}_p)$ is represented by a complex which is a sequential colimit of complexes
$$\varinjlim_n C_\bullet^{(n)}$$
such that:
\begin{enumerate}
\item each $C_k^{(n)}$ is equal to $\mathcal{M}_{<p}(\mathbb{N})$;
\item each map $C_\bullet^{(n)}\rightarrow C_\bullet^{(n+1)}$ is term-wise given by an \emph{injective} trace-class map between $\mathcal{M}_{<p}(\mathbb{N})$'s.
\end{enumerate}
\end{proposition}
\begin{proof}
The new information is that we can arrange the transition maps to be termwise injective.  Working our way along inductively, it suffices to show that if $f:C_\bullet\rightarrow D_\bullet$ is term-wise trace class map of finite complexes with each term equal to $\mathcal{M}_{<p}(\mathbb{N})$,  then by direct summing an acyclic complex to $D_\bullet$ we can arrange that $f$ be term-wise injective, but still trace-class.  For that, choose arbitrarily an injective trace-class map $g_i:C_i\rightarrow D_i$ for all $i$ (e.g.\ take $g_i$ to be termwise multiplication by a $<p$-summable sequence of nonzero real numbers), and replace $D_\bullet$ by
$$D'_\bullet = D_\bullet\oplus \bigoplus_i \left[ D_i\overset{id}{\rightarrow} D_i\right][i]$$
and take the new map $f':C_\bullet\rightarrow D'_\bullet$ to be the old map $f$ on the $D_\bullet$ component, and on the $i$-component to be the map which is $g_i$ in degree $i$ and $g_i\circ d$ in degree $i+1$.
\end{proof}

In particular, every basic nuclear is represented by a complex where each term is a sequential colimit of sequence spaces $\mathcal{M}_{<p}(\mathbb{N})$ under injective trace-class maps.  Next we will identify this class of spaces with the classical \emph{dual nuclear Frechet} spaces.  But first, we start by recalling the definition and some background on these.

The most convenient form of the definition of a DNF (dual nuclear Frechet) space involves Hilbert spaces, which we may as well take to be $\ell_2(\mathbb{N})$.  In the Hilbert space context, there is a somewhat more direct analog of the notion of trace-class map, based on the \emph{singular value decomposition}, which we start with.

A map $f:\ell_2(\mathbb{N})\rightarrow \ell_2(\mathbb{N})$ is said to be a \emph{compact operator} if $f$ factors through the inclusion $\ell_2(\mathbb{N})\subset \mathcal{M}_2(\mathbb{N})$.  If $f$ is a compact operator, then attached to $f$ is a unique non-increasing nullsequence of non-negative real numbers
$$\sigma_1\geq \sigma_2\geq\ldots,$$
called the \emph{singular values} of $f$, such that after orthonormal basis change in both source and target (independently), $f$ is represented by the diagonal matrix with entries the $\sigma_n$.  This is called the singular value decomposition of $f$.

The proof is not difficult: the function $\vert f\vert: \mathcal{M}(\mathbb{N})_{\ell^2\leq 1}\rightarrow\mathbb{R}_{\geq 0}$ attains a maximum value which we call $\sigma_1$, and it follows that there are $e_0,e'_0\in \ell_2(\mathbb{N})$ of norm $1$ such that $f(e_0) = \sigma_1e'_0$.  Then we pass to the orthogonal complement of $e_0$ in the source $\ell_2(\mathbb{N})$ and $e'_0$ in the target $\ell_2(\mathbb{N})$, and continue inductively.

This leads to the following definition, see \cite{schatten2013norm} for this and the unjustified claims below:

\begin{definition}
Let $0<p<\infty$.  A map $f:\ell_2(\mathbb{N})\rightarrow \ell_2(\mathbb{N})$ is said to be \emph{$p$-Schatten} if it is compact, and its singular values satisfy
$$\sum_n \sigma_n^p<\infty.$$
\end{definition}

We have the following properties:

\begin{enumerate}
\item  If $f$ is $p$-Schatten and $g,h$ are arbitrary, then $g\circ f\circ h$ is $p$-Schatten.
\item  If $f$ is $p$-Schatten and $q>p$, then $f$ is $q$-Schatten.
\item  If $f$ is $p$-Schatten and $g$ is $q$-Schatten, then $f\circ g$ is $r$-Schatten where $\frac{1}{r}=\frac{1}{p}+\frac{1}{q}$.  (In particular, the composition of two $p$-Schatten maps is $p/2$-Schatten.)
\end{enumerate}

The first two are simple to prove, while the third is a bit tricky due to the change of basis involved in putting $f$ and $g$ into diagonal form.  But in the end it does reduce to the H\"older inequality just as for compositions of diagonal matrices.

The case of $2$-Schatten maps is somehow nicest from a Hilbert space perspective, and has a separate name: a $2$-Schatten map is also called a Hilbert-Schmidt map.  In terms of infinite matrices, an operator is Hilbert-Schmidt if and only if the total $\ell^2$-norm of the matrix is finite.

Here, then, is the definition of a dual nuclear Frechet space.

\begin{definition}
A condensed abelian group $V$ is called a DNF space if it is a sequential filtered colimit of Hilbert spaces along injective $2$-Schatten transition maps.
\end{definition}

By the above properties, we get an equivalent definition if we replace $2$ by $p$ in this definition, for any $0<p<\infty$.  Now let us make the connection with our previous discussion.  Recall that we showed that every basic nuclear in $\mathcal{D}(\operatorname{Liq}_p)$ can be represented by a complex where each term is a sequential colimit of sequence spaces $\mathcal{M}_{<p}(\mathbb{N})$ under injective trace-class transition maps.

\begin{lemma}
A $p$-liquid $V$ is a DNF space if and only if it is isomorphic to a sequential colimit of sequence spaces  $\mathcal{M}_{<p}(\mathbb{N})$ under injective trace-class transition maps.
\end{lemma}
\begin{proof}
First, assume $V$ is a DNF space, and write it as a union
$$V=\varinjlim_n H_n$$
of Hilbert spaces under injective transition maps, such that the singular values of each transition map can be written as a product
$$\sigma_k = \alpha_k \lambda_k \beta_k,$$
where $(\alpha_k)$ and $(\beta_n)$ are nullsequences and $(\lambda_k)$ is a $p/2$-summable sequence.  Then by choosing suitable orthonormal bases, we can factor each transition map $H_n\rightarrow H_{n+1}$ as
$$H_n\overset{\cdot\alpha}{\rightarrow} c_0(\mathbb{N})\overset{\cdot\lambda}{\rightarrow} \ell_{p/2}(\mathbb{N})\subset \mathcal{M}_{<p}(\mathbb{N}) \overset{\beta}{\rightarrow} H_{n+1},$$
where we use termwise multilication by the two null-sequences on the outside maps and the $p/2$-summable sequence $(\lambda_n)$ in the middle.

Note that all these maps are injective because each $\sigma_n\neq 0$ due to injectivity of our original transition maps.   Then if we consider just the $\mathcal{M}_{<p}(\mathbb{N})$ terms and the induced transition maps between them, then these transition maps are injective, but also trace class by the criterion of Lemma \ref{explicittraceclass}.  And we see that the ind-system of Hilbert spaces is isomorphic to the ind-system of $\mathcal{M}_{<p}(\mathbb{N})$'s, in particular proving that $V$ is a sequential union of $\mathcal{M}_{<p}(\mathbb{N})$'s under injective trace class maps.

Conversely, suppose $V$ is a sequential colimit of $\mathcal{M}_{<p}(\mathbb{N})$'s along injective trace-class transition maps.  Factor the $n^{th}$ transition map as
$$\mathcal{M}_{<p}(\mathbb{N})\rightarrow c_0(\mathbb{N})\overset{\cdot\lambda}{\rightarrow} \ell_q(\mathbb{N})\subset \mathcal{M}_{<p}(\mathbb{N})\rightarrow \mathcal{M}_{<p}(\mathbb{N})$$
according to Lemma \ref{explicittraceclass}, and then further factor $ c_0(\mathbb{N})\overset{\cdot\lambda}{\rightarrow} \ell_q(\mathbb{N})$ as
$$ c_0(\mathbb{N})\overset{\cdot\lambda^{\frac{q}{2}}}{\rightarrow} \ell_2(\mathbb{N})\overset{\cdot\lambda^{\frac{2-q}{2}}}{\rightarrow} \ell_q(\mathbb{N}).$$
(Implicitly, $q<p$ and $\lambda$ depend on $n$, but this is not relevant for the argument.)  Then for the ind-system formed by these $ \ell_2(\mathbb{N})$ spaces, each transition map factors as
$$ \ell_2(\mathbb{N})\rightarrow  c_0(\mathbb{N})\overset{\cdot \alpha}{\rightarrow} \ell_2(\mathbb{N})$$
where the second map is given by termwise multiplication with a $2$-summable sequence $(\alpha_j)$ ($= \lambda^{\frac{q}{2}}$).  Now, the first map corresponds to a sequence of elements of the dual space of $ \ell_2(\mathbb{N})$ which is in particular uniformly bounded.  However, the dual space of $ \ell_2(\mathbb{N})$ identifies (on underlying Banach spaces) with $\ell_2(\mathbb{N})$ again via the inner product.   Thus we get a uniformly bounded sequence of elements $x_i$ in $\ell_2(\mathbb{N})$ such that our transition map is given by the matrix whose $(i,j)$-entry is
$$\alpha_j\cdot \langle \delta_i,x_j\rangle.$$
This matrix is clearly Hilbert-Schmidt, i.e.\ the transition maps are $2$-Schatten.  However, they need not be injective.  Nonetheless, our original $V$ is a filtered union of qs spaces hence and hence is also qs, so we can conclude that $V$ is a DNF space thanks to the following lemma.
\end{proof}

\begin{lemma}
Let $0<p<\infty$.  Suppose $V=\varinjlim_n H_n$ is a sequential colimit of Hilbert spaces with $p$-Schatten transition maps.  If $V$ is qs, then $V$ is also a sequential colimit of Hilbert spaces with injective $p$-Schatten transition maps, meaning $V$ is a DNF space.
\end{lemma}

\begin{proof}
Let $H'_n\subset V$ denote the image of $H_n$ in $V$.  Then $H_n'$ is also a Hilbert space!  Indeed, the inclusion $\operatorname{ker}(H_n\rightarrow V)\subset H_n$ is a quasi-compact map since $V$ is qs; thus, since Hilbert spaces are compactly generated, this kernel corresponds to a closed subspace of the Hilbert space $H_n$.  But every closed subspace of a Hilbert space is a Hilbert space, and moreover the inclusion has a complement given by the orthogonal subspace.  It follows that $H'_n$, which is the quotient of $H_n$ by this kernel, is also a Hilbert space, since it identifies with this orthogonal complement.  Moreover, the inclusion $H'_n\subset H'_{n+1}$ then factors through $H_n\rightarrow H_{n+1}\rightarrow H'_{n+1}$, so it is also $p$-Schatten.
\end{proof}

By similar arguments we prove the following permanence properties of DNF spaces, due to Grothendieck.

\begin{lemma}\label{DNFlemmas}\leavevmode
\begin{enumerate}
\item Let $V$ be a DNF space, and $W$ an arbitrary qs condensed $\mathbb{R}$-module.  For any homomorphism $f:V\rightarrow W$, we have that $\operatorname{ker}(f)$ and $\operatorname{im}(f)$ are DNF spaces.
\item Any extension of DNF spaces is a DNF space.
\end{enumerate}
\end{lemma}
\begin{proof}
The proof of (1) is very similar to the argument in the previous lemma; we skip it.  For (2), write $V=\varinjlim_n V_n$ as a sequential colimit of sequence spaces $\mathcal{M}_{<p}(\mathbb{N})$ with injective trace-class transition maps.  Any extension $0\rightarrow W\rightarrow \widetilde{V}\rightarrow V\rightarrow 0$ is split over each $\mathcal{M}_{<p}(\mathbb{N})$, because these $\mathcal{M}_{<p}(\mathbb{N})$ are compact projective in $\omega_1$-condensed $p$-liquid spaces.  Thus, choosing such a splitting for each $n$ with no mind for compatibility, we can write
$$\widetilde{V} = \varinjlim_n (V_n\oplus W),$$
where the transition map
$$V_n\oplus W\rightarrow V_{n+1}\oplus W$$
is given by block upper triangular matrix: the map $V_n\rightarrow V_{n+1}$ is the original transition map, the map $W\rightarrow W$ is the identity, the map $W\rightarrow V_{n+1}$ is $0$, and the map $V_n\rightarrow W$ is just some arbitrary map.  However, any map $V_n\rightarrow W$ is trace class and moreover factors through a trace class map to some $W_{\varphi(n)}$ with respect to an analogous presentation of $W$ by $\mathcal{M}_{<p}(\mathbb{N})$'s.  We can assume $\varphi:\mathbb{N}\rightarrow\mathbb{N}$ to be an increasing function, and then we can write
$$\widetilde{V} = \varinjlim_n (V_n\oplus W_{\varphi(n)}),$$
where now this is an expression which proves that $\widetilde{V}$ has the same form as $V$ and $W$, meaning it's DNF.
\end{proof}

Now we can finally give the desired characterization of basic nuclear liquid modules.

\begin{theorem}\label{basicnuclearcharacterization}
Let $X \in \mathcal{D}(\operatorname{Liq}_p)$.  The following are equivalent:
\begin{enumerate}
\item $X$ is basic nuclear, i.e.\ $X$ is a sequential colimit along trace-class maps;
\item $X$ can be represented by a complex of DNF spaces;
\item Each homology group $H_n(X)$ is isomorphic to a quotient of DNF spaces;
\item $X$ lies in the smallest countably cocomplete stable subcategory generated by the DNF spaces.
\end{enumerate}
\end{theorem}
\begin{proof}
We have already seen that (1) $\Rightarrow$ (2).  The claim (2) $\Rightarrow$ (3) follows from the first part of the previous lemma.  For the claim (3) $\Rightarrow$ (4), by considering canonical truncations and shifting we can assume $X$ is concentrated in homologically non-negative degrees.  It suffices to show that there's a DNF space $V$ with a map $V\rightarrow X$ surjective on $H_0$, and that the cone of this map also satisfies the same condition (3).  For then we can continue inductively and get resolution of $X$ by DNF spaces, proving (4).  However, by assumption we can find a DNF $V$ with a surjective map $V\rightarrow H_0X$ whose kernel is a DNF.  This map lifts to $V\rightarrow X$ because the obstructions to this lifting all live in $\operatorname{Ext}^i(V;-)$ for $i\geq 2$ which vanishes because $V$ is a sequential colimit of compact projectives.  To see that the cone also satisfies (3), we need to show that an extension of a DNF space $V$ by a quotient $W/W'$ of DNF spaces is still a quotient of DNF spaces.  However, the same $\operatorname{Ext}^2$-vanishing used above shows that any extension of $V$ by $W/W'$ lifts to an extension of $V$ by $W$.  Such an extension $\widetilde{V}$ is a DNF space by the previous lemma, and then our original extension is $\widetilde{V}/W'$, a quotient of DNF spaces as desired.  Finally, for (4) $\Rightarrow$ (1), we have already seen that DNF spaces are basic nuclear and that the full subcategory of basic nuclears is stable and closed under countable colimits.
\end{proof}

We have a series of remarkable corollaries; the proofs are immediate.

\begin{corollary}
An object $X \in \mathcal{D}(\operatorname{Liq}_p)$ is basic nuclear if and only if each homology group $H_n(X)[0]$ is basic nuclear.  Likewise, an object $X \in \mathcal{D}(\operatorname{Liq}_p)$ is nuclear if and only if each homology group $H_n(X)[0]$ is nuclear.
\end{corollary}

\begin{corollary}
The collection of basic nuclears living in degree $0$ is an abelian subcategory of $\operatorname{Liq}_p$, closed under extensions and countable colimits.  It consists exactly of the quotients of DNF spaces.

The collection of nuclear objects living in degree $0$ is likewise an abelian subcategory of $\operatorname{Liq}_p$, closed under extensions and arbitrary colimits.  It consists exactly of the ($\aleph_1$-filtered) colimits of quotients of DNF spaces.
\end{corollary}

\begin{corollary}
The collections of nuclear and basic nuclear modules in $\mathcal{D}(\operatorname{Liq}_p)$ are independent of $p$, and also independent of which uncountable strong limit cut-off cardinal one chooses in the foundations for condensed sets.  The derived liquid tensor product is likewise independent of $p$ on the class of nuclear modules.
\end{corollary}

We finish by discussing a technical issue which will arise in the following lecture.  Recall that the definition of nuclear was that
$$(c^\vee\otimes x)(\ast)\rightarrow \operatorname{Hom}(c,x)$$
should be an isomorphism for all compact objects $c$.  One could ask whether the a priori stronger condition
$$c^\vee\otimes x \overset{\sim}{\rightarrow} x^c$$
then automatically holds.  In the abstract set-up we see no reason why this should be true.  But now we can nonetheless prove it for liquid modules.  In fact we have the more general:

\begin{proposition}\label{prop:internallynuclear}
Let $V$ be a nuclear object in $\mathcal{D}(\operatorname{Liq}_p)$, and let $S$ be a compact Hausdorff space.  Then
$$\underline{\operatorname{RHom}}(\mathbb{Z}[S],\mathbb{R})\otimes^L_{<p} V\overset{\sim}{\rightarrow} \underline{\operatorname{RHom}}(\mathbb{Z}[S], V).$$

\end{proposition}
\begin{proof}
First note that $\underline{\operatorname{RHom}}(\mathbb{Z}[S],\mathbb{R})=C(S;\mathbb{R})$ is concentrated in degree $0$, by \cite[Lecture III]{Condensed}.  Now, the functor $\underline{\operatorname{RHom}}(\mathbb{Z}[S], -)$ commutes with $\aleph_1$-filtered colimits, because $\mathbb{Z}[S]$ is a countable colimit of compact projective $\mathbb{Z}[S_i]$'s via taking a hypercover of $S$ by extremally disconnected profinite sets.  Thus we reduce to proving the claim for basic nuclear $V$.  Taking colimits of canonical truncations, we can also assume $V$ is homologically concentrated in nonnegative degrees.  Now we claim that on both sides, we can pull out the limit of canonical truncations of $V$.  On the right this is clear as $\underline{\operatorname{RHom}}(\mathbb{Z}[S], -)$ commutes with all limits.  On the left, this is because each homology group of $V$ is a quotient of DNF's, hence a quotient of flat objects, hence has uniformly bounded tor-dimension, whereas we have already seen that $\underline{\operatorname{RHom}}(\mathbb{Z}[S],\mathbb{R})$ lives in degree $0$.  Thus we can reduce to the case of $V$ a quotient of DNF's, which reduces us to the case of $V$ a DNF.  To handle that case, by the intertwining argument and Lemma \ref{traceclassproperties} part 3 it suffices to show that $\underline{\operatorname{RHom}}(\mathbb{Z}[S], -)$ commutes with the sequential colimit presenting the DNF $V$. But as $\mathbb Z[S]$ and in fact all $\mathbb Z[S\times S']$ for profinite $S'$ are pseudocoherent, $\underline{\operatorname{RHom}}(\mathbb{Z}[S], -)$ commutes with filtered colimits of coconnective objects.
\end{proof}

\begin{corollary}\label{cor:relativenuclear} Let $A$ be a commutative algebra in $\mathcal D(\mathrm{Liq}_p)$ such that the underlying object of $A$ is nuclear in $D(\mathrm{Liq}_p)$. Then an object $M\in C=\mathrm{Mod}_A(\mathcal D(\mathrm{Liq}_p))$ is nuclear (in the sense of $C$) if and only if the underlying object of $\mathcal D(\mathrm{Liq}_p)$ is nuclear (in the sense of $\mathcal D(\mathrm{Liq}_p)$).
\end{corollary}

\begin{proof} For $M\in \mathrm{Mod}_A(\mathcal D(\mathrm{Liq}_p))$, nuclearity over $A$ means that for all extremally disconnected $S$, the map
\[
(\underline{\operatorname{RHom}}(\mathbb Z[S],A)\otimes^L_{<p,A} M)(*)\to M(S)
\]
is an isomorphism. But the left-hand side agrees with $(\underline{\operatorname{RHom}}(\mathbb Z[S],\mathbb C)\otimes^L_{<p} M)(*)$, leading to the definition of nuclearity of $M$ in $\mathcal D(\mathrm{Liq}_p)$.
\end{proof}
\textbf{Exercise 1.}  Recall that the category $\operatorname{Solid}_\mathbb{Z}$ of solid abelian groups has compact projective generators of the form $\prod_I\mathbb{Z}$, with
$$\operatorname{Hom}(\prod_I\mathbb{Z},\prod_J\mathbb{Z})=\prod_J\bigoplus_I\mathbb{Z},$$
and with a tensor product such that
$$\prod_I\mathbb{Z} \otimes \prod_J\mathbb{Z}=\prod_{I\times J}\mathbb{Z}.$$
Show that an object in $\mathcal{D}(\operatorname{Solid}_\mathbb{Z})$ is nuclear if and only if it is discrete, i.e.\ lies in the cocomplete stable full subcategory generated by the unit $\mathbb{Z}$.\\

\textbf{Exercise 2.}  Show that the definition of compactness of an operator $f:\ell_2(\mathbb{N})\rightarrow \ell_2(\mathbb{N})$ given in the lecture, namely $f$ is compact if $f$ factors through $\mathcal{M}_2(\mathbb{N})$, is equivalent to the usual definition: $f$ is compact if there is an open neighborhood $U\subset \ell_2(\mathbb{N})$ of the origin and a compact subset $K\subset  \ell_2(\mathbb{N})$ such that $f(U)\subset K$.\newpage

\section{Lecture IX: Coherent Sheaves, I}

The goal of this lecture and the next is to develop the theory of coherent sheaves on complex-analytic spaces, by isolating them inside the very large (derived) category of ``liquid quasicoherent sheaves''. In this lecture, our goal is to isolate the pseudocoherent complexes of (abstract) modules inside the category of all liquid modules, and show that this category localizes on $\mathbb C^n$. In the next lecture, we will go from the level of complexes to the abelian level.

Let us start with the general nonsense. Take any closed symmetric monoidal compactly generated stable $\infty$-category $C$, such that $1\in C$ is compact. If $A=\mathrm{End}(1)$ is the endomorphism algebra of $1$ (a ``commutative algebra in $\infty$-category land'', known as an $E_\infty$-algebra), then one gets a full embedding
\[
\mathcal D(A)\hookrightarrow C
\]
of the category of $A$-modules in spectra into $C$. If $A$ happens to be concentrated in degree $0$, then $\mathcal D(A)$ is just the usual derived $\infty$-category of (abstract) $A$-modules. For example, if $C=\mathcal D(\mathrm{Liq}_p)$, one gets a full embedding
\[
\mathcal D(\mathbb R)\hookrightarrow \mathcal D(\mathrm{Liq}_p)
\]
from the derived $\infty$-category of (abstract) $\mathbb R$-vector spaces into that of $p$-liquid $\mathbb R$-vector spaces.

\begin{definition}\label{def:discreteobject} An object $V\in C$ is discrete if it lies in the essential iamge of $\mathcal D(A)\hookrightarrow C$; equivalently, if it lies in the subcategory generated under shift and colimits by $1\in C$.
\end{definition}

As the subcategory of discrete objects is just the category of abstract modules over $A$, it is very easy to understand. In particular, if one intersects it with the class of compact objects, one just gets perfect complexes of $A$-modules:

\begin{proposition}\label{prop:discretecompact} An object $V\in C$ is compact and discrete if and only if it lies in the subcategory generated under retracts, shifts, and finite colimits by $1\in C$, i.e.~corresponds to a perfect complex of $A$-modules.
\end{proposition}

\begin{proof} Recall that $\mathcal D(A)$ is compactly generated, with compact objects the perfect complexes of $A$-modules. As compact objects of $C$ are stable under retracts, shifts, and finite colimits, all perfect complexes of $A$-modules give compact (and discrete) objects of $C$. Conversely, if $V\in C$ is discrete and compact, then it is also compact as an object of $\mathcal D(A)$ (as $\mathcal D(A)\subset C$ is stable under all colimits), and hence corresponds to a perfect complex of $A$-modules.
\end{proof}

On the other hand, it is a very fragile category -- many natural operations leave the category of discrete objects. For this reason, the more general class of nuclear objects was introduced in the last lecture. It turns out that the previous proposition has a natural analogue when replacing discrete by nuclear, in which case it singles out the important class of dualizable objects of $C$. Recall that $V\in C$ is dualizable if there exists some $V'\in C$ together with maps $\alpha: 1\to V'\otimes V$, $\beta: V\otimes V'\to 1$ such that the composites
\[
V\xrightarrow{1\otimes \alpha} V\otimes V'\otimes V\xrightarrow{\beta\otimes 1} V
\]
and
\[
V'\xrightarrow{\alpha\otimes 1} V'\otimes V\otimes V'\xrightarrow{1\otimes \beta} V'
\]
are the identity. In that case, for any $W$, one has a canonical isomorphism $W^V=V'\otimes W$; in particular $V'=V^\vee$ is the dual of $V$ (making $\beta$ the evaluation map).

\begin{proposition}\label{prop:nuclearcompact} An object $V\in C$ is compact and nuclear if and only if it is dualizable.
\end{proposition}

\begin{proof} Assume first that $V$ is dualizable. Then $\mathrm{Hom}(V,-)=(V'\otimes -)(*)$ commutes with all colimits, so $V$ is compact. Also, $\mathrm{Hom}(V,V)=(V'\otimes V)(*)$, and in particular any map $f: V\to V$ is trace-class, for example the identity. Then $V$ is the sequential colimit of $V$, with transition maps the identity, which is trace-class; so $V$ is basic nuclear.

Conversely, assume $V$ is compact and nuclear. Let $V'=V^\vee$, with $\beta$ the evaluation map. We want to produce a coevaluation map $\alpha: 1\to V^\vee\otimes V$ in order to show that $V$ is dualizable. Such an $\alpha$ is an element of $(V^\vee\otimes V)(*)$. As $V$ is compact and $V$ is nuclear, this agrees with $\mathrm{Hom}(V,V)$. But here we have the identity of $V$, giving the desired $\alpha$. It is a simple exercise to check that the diagrams commute.
\end{proof}

We note that this gives a way to construct dualizable objects.

\begin{proposition}\label{prop:coneoffredholmdualizable} Let $V\in C$ be compact and let $f: V\to V$ be a trace-class map. Then the cone $[V\xrightarrow{1-f} V]$ of the endomorphism $1-f$ on $V$ is a dualizable object of $C$. Conversely, any dualizable object of $C$ is a retract of an object of this form.
\end{proposition}

\begin{proof} Clearly, the cone of $1-f$ is compact (as a finite colimit of compact $V$'s). Now note that the map $f: V\to V$ induces an isomorphism on $[V\xrightarrow{1-f} V]$ (as on the cone, $f$ becomes equal to $1$). Thus, passing to the sequential colimit over multiplication by $f$, one can write the cone also as the cone of $[W\xrightarrow{1-f} W]$ where $W$ is the sequential colimit of $V\xrightarrow{f} V\xrightarrow{f} V\ldots$. Then $W$ is (basic) nuclear, and hence so is the cone.

For the converse, note that if $V$ is dualizable, then it is compact, and the identity of $V$ is trace-class, so $[V\xrightarrow{0} V]$ belongs to the class of dualizable objects produced in this way. But $V$ is a retract of $[V\xrightarrow{0} V]$.
\end{proof}

\begin{remark} The proof shows more generally that if $V\in C$ is compact and $f: V\to V$ is an endomorphism such that some power $f^n: V\to V$ is trace-class, then the cone $[V\xrightarrow{1-f} V]$ is dualizable.
\end{remark}

\begin{example} Let us give an example of a dualizable object that is not discrete. Consider $A=\mathbb C[T]$, and $C=\mathcal D(\mathrm{Liq}_p(A))$. Let $V$ be the DNF space of sequences of at most polynomial growth, and turn it into a $\mathbb C[T]$-module by letting $T$ act via diagonal multiplication by $(1,2,3,\ldots)$. We claim that $V$ is dualizable. Indeed, if $V_0=\mathcal M_{<p}(\beta \mathbb N)\otimes \mathbb C[T]$ is the compact projective generator, and we consider the endomorphism $f: V_0\to V_0$ given by diagonal multiplication by $(T,\tfrac T2,\tfrac T3,\ldots)$, then $f^N$ is trace-class for sufficiently large $N$ (depending on $p$). Moreover, to compute the cone
\[
V_0\xrightarrow{1-(T,\tfrac T2,\tfrac T3,\ldots)} V_0
\]
we can first take the colimit along diagonal multiplication by $(1,\tfrac 12,\tfrac 13,\ldots)$, leading to the space $V[T]$ of polynomials of sequences of at most polynomial growth. Then we get the cone
\[
V[T]\xrightarrow{1-(T,\tfrac T2,\tfrac T3,\ldots)} V[T]
\]
which can also be written as the cone
\[
V[T]\xrightarrow{(1,2,3,\ldots)-T} V[T]
\]
and hence is given by $V$ with $T$ acting by multiplication by $(1,2,3,\ldots)$.
\end{example}

In many cases, it will turn out that all dualizable objects of $C$ are in fact discrete, i.e.~are perfect complexes of $A$-modules. By Proposition~\ref{prop:coneoffredholmdualizable} this is the case if and only if cones $[V\xrightarrow{1-f} V]$ are discrete. Here, $1-f$ is a perturbation of the isomorphism $1: V\to V$ by a trace-class operator. Classically, if $V$ is say a Hilbert space, such things are Fredholm operators, and Fredholm operators have finite-dimensional kernel and cokernel, leading to the desired discreteness. Abstractly, we make the following definition:

\begin{definition}\label{def:fredholmcategory} Let $C$ be a closed symmetric monoidal compactly generated stable $\infty$-category with compact unit, as above. Then $C$ is Fredholm if all dualizable objects of $C$ are discrete; equivalently, if for all compact $V\in C$ and trace-class endomorphisms $f: V\to V$, the cone $[V\xrightarrow{1-f} V]$ is discrete.
\end{definition}

We will be interested in verifying this condition in our situations of interest. Thus, assume from now on that $C=\mathcal D(\mathcal A)$ for some closed symmetric monoidal abelian category $\mathcal A$ generated by compact projective objects $\mathcal A^{cp}\subset \mathcal A$, and with $1\in \mathcal A$ compact projective.\footnote{In this situation, $C$ can also be described as the finite-product preserving functors from $(\mathcal A^{cp})^{\mathrm{op}}$ to spectra. One could more generally allow such functor categories from additive $\infty$-categories. This generalization is relevant when one allows animated rings.} In this situation, $A=\mathrm{End}(1)$ is concentrated in degree $0$, i.e.~is a usual commutative ring.

We make the following auxiliary hypothesis:

\begin{assumption}\label{ass:dualstatic} For all compact projective $T\in \mathcal A$, the dual $T^\vee\in\mathcal D(\mathcal A)$ is concentrated in degree $0$. Equivalently, for all compact projective $T,T'\in \mathcal A$, one has $\mathrm{Ext}^i_{\mathcal A}(T_1\otimes T_2,A)=0$ for $i>0$, where $A=\mathrm{End}(1)$ as before.
\end{assumption}

It turns out that this is a pretty mild assumption. Under this assumption, we note that for $T\in \mathcal A$ compact projective, trace-class endomorphisms come from $(T^\vee\otimes T)(*)$, where all operations can be performed in $\mathcal A$ (as opposed to $\mathcal D(\mathcal A)$), making things more concrete.

When $C=\mathcal D(\mathcal A)$, one can consider a generalization of the notion of compact objects. Recall that $X\in C$ is compact if and only if it can be represented by a finite complex of compact projective objects. Often, when building such projective resolutions, one can ensure finiteness in each degree, but not necessarily termination. This leads to the notion of pseudocompact objects.

\begin{definition}\label{def:pseudocompact} An object $X\in \mathcal D(\mathcal A)$ is pseudocompact if it can be represented by a complex of compact projectives that is bounded to the right. This is equivalent to $\mathrm{Hom}(X,-)$ commuting with all direct sums of objects in $\mathcal D_{\leq n}(\mathcal A)$, for all $n$.
\end{definition}

In other words, a pseudocompact is one that can be approximated arbitrarily well (with respect to the $t$-structure) by compact objects. It turns out that under Assumption~\ref{ass:dualstatic}, a similar assertion holds true for compact objects replaced by dualizable objects.

\begin{proposition}\label{prop:almostdualizable} Assume Assumption~\ref{ass:dualstatic}. Then:
\begin{enumerate}
\item The subcategory of dualizable objects of $C$ is generated under shifts, cones, and retracts by the cone $[T\xrightarrow{1-f} T]$ for compact projective $T\in \mathcal A$ with a trace-class endomorphism $f$.
\item Any nuclear and pseudocompact object $X$ of $C$ can be written as a sequential colimit $X=\varinjlim_n X_n$ where each $X_n$ is dualizable, and the cone of $X_n\to X$ lives in $\mathcal D_{\geq n}(\mathcal A)$.
\end{enumerate}
In particular, by (1), $C$ is Fredholm if and only if for all compact projective $T\in \mathcal A$ with trace-class endomorphism $f: T\to T$, the cone $[T\xrightarrow{1-f} T]$ is discrete. In that case, by (2) all nuclear and pseudocoherent objects of $C$ are discrete, and hence are given by $\mathcal D_{pc}(A)\subset C$.
\end{proposition}

\begin{proof} We show that any nuclear and pseudocompact object $X$ of $C$ can be written as a sequential colimit $X=\varinjlim_n X_n$ where each $X_n$ lies in the subcategory generated under shifts and cones by objects $[T\xrightarrow{1-f} T]$ as in (1), and the cone of $X_n\to X$ lies in $\mathcal D_{\geq n}(\mathcal A)$. This clearly gives part (2). But in (1), the object $X$ is compact, and hence for some $n$ it is a retract of $X_n$, giving also (1).

To prove this claim, assume we already have $X_n$ so that $Y_n:=[X_n\to X]\in \mathcal D_{\geq n}(\mathcal A)$. We will find some $T_n$ and $f_n$ with a map $[T_n\xrightarrow{1-f_n} T_n][n]\to Y_n$ whose cone $Y_{n+1}$ lies in $\mathcal D_{\geq n+1}(\mathcal A)$. In that case, the fibre $X_{n+1}$ of the composite $X\to Y_n\to Y_{n+1}$ is an extension of $X_n$ and a shift of $[T_n\xrightarrow{1-f_n} T_n]$, giving the desired induction step.

Shifting, we can assume now that $n=0$. Our task is to show that given $X\in \mathcal D_{\geq 0}(\mathcal A)$ that is pseudocompact and nuclear, there is some compact projective $T_0$ and some trace-class endomorphism $f_0: T_0\to T_0$ such that there is a map $[T_0\xrightarrow{1-f_0} T_0]\to X_0$ that is surjective in homological degree $0$. By pseudocompactness, we can represent $X$ by a complex
\[
\ldots\to T_2\to T_1\to T_0\to 0
\]
of compact projective objects $T_i$ in nonnegative homological degrees. In particular, we get a map $g: T_0\to X$. Note that by nuclearity of $X$, this comes from some class in $(T_0^\vee\otimes X)(*)$. But as $T_0^\vee$ sits in degree $0$ and $g: T_0\to X$ is surjective in degree $0$, the map
\[
(T_0^\vee\otimes T_0)(*)\to (T_0^\vee\otimes X)(*)
\]
induced by $g: T_0\to X$ is surjective. In other words, we can find a trace-class endorphism $f_0: T_0\to T_0$ such that $gf_0=g$ in $D(\mathcal A)$. This means that the map $g: T_0\to X$ can be factored over a map $[T_0\xrightarrow{1-f_0} T_0]\to X$. This is necessarily surjective in degree $0$ as $g$ is.
\end{proof}

This finishes the abstract nonsense. In order to apply this, we need to verify the Fredholm property.

\begin{proposition}\label{prop:banachalgebrafredholm} Let $A$ be a ($p$-)Banach algebra over $\mathbb R$. Then $C=\mathcal D(\mathrm{Liq}_p(A))$ satisfies Assumption~\ref{ass:dualstatic}, and is Fredholm. More generally, this holds true if $A$ is a filtered colimit of $p$-Banach algebras.
\end{proposition}

\begin{proof} First, we reduce the case of filtered colimits to the case of $p$-Banach algebras, so assume $A=\varinjlim_i A_i$ is a filtered colimit of Banach algebras. Any compact projective $T$ in $\mathrm{Liq}_p(A)$ is the base change of a compact projective $T_i$ in $\mathrm{Liq}_p(A_i)$ (as any idempotent endomorphism of some $\mathcal M_{<p}(S)\otimes_{<p} A$ is already defined over some $A_i$). Its dual is then
\[
\underline{\mathrm{RHom}}_{A_i}(T_i,A) = \varinjlim_{j\geq i} \underline{\mathrm{RHom}}_{A_i}(T_i,A_j),
\]
which is the filtered colimit of the duals of $T_i\otimes^L_{<p,A_i} A_j$. (The justification for the commutation with the colimit is as in Proposition~\ref{prop:internallynuclear}.) Also, any trace-class endomorphism $f: T\to T$ is also the base change of some trace-class endomorphism $f_i: T_i\to T_i$ for large enough $i$, in which case $[T\xrightarrow{1-f} T]$ is the base change of $[T_i\xrightarrow{1-f_i} T_i]$. Thus, we can assume that $A$ is a $p$-Banach algebra.

For any $p$-Banach vector space $V$, one has $\underline{\mathrm{RHom}}(\mathbb Z[S],V) = C(S,V)$, concentrated in degree $0$, for any profinite $S$; this gives Assumption~\ref{ass:dualstatic} (as any other compact projective is a retract of a base change of $\mathbb Z[S]$). Now we have to verify the key assertion that for all compact projective $T\in \mathrm{Liq}_p(A)$ with a trace-class endomorphism $f: T\to T$, the cone $[T\xrightarrow{1-f} T]$ is discrete. Note that if $f: T\to T$ factors as $T\xrightarrow{g} T'\xrightarrow{h} T$, then letting $f'=gh: T'\to T'$, the map $g$ induces an isomorphism $[T\xrightarrow{1-f} T]\cong [T'\xrightarrow{1-f'} T']$ (with inverse induced by $h$). In particular, we can assume $T=\mathcal M_{<p}(S)\otimes_{<p} A$ is a standard generator, with $S$ extremally disconnected. The trace-class maps are then induced by elements of
\[
C(S,A)\otimes_{<p} \mathcal M_{<p}(S).
\]
By the arguments in the proof of Proposition~\ref{explicittraceclass}, any trace-class map factors over $\mathcal M_{<p}(\mathbb N)\otimes_{<p} A$. Thus, we can also consider a trace-class endomorphism of $\mathcal M_{<p}(\mathbb N)\otimes_{<p} A$, which is induced by an element of
\[
c_0(\mathbb N,A)\otimes_{<p} \mathcal M_{<p}(\mathbb N).
\]
This can be thought of as a matrix, whose rows are given by $\lambda_0 v_0,\lambda_1 v_1,\ldots$, for a uniformly bounded sequence of nullsequences $v_0,v_1,\ldots\in c_0(\mathbb N,A)$ and a $<p$-summable sequence $\lambda_0,\lambda_1,\ldots$. We can assume that all $v_i$ have supremum norm $\leq 1$. Choose some $N$ so that $s=\lambda_N+\lambda_{N+1}+\ldots<1$, and write the matrix in block form
\[
\left(\begin{array}{cc} E & F\\ G & H\end{array}\right)
\]
where $E$ is an $N\times N$-matrix containing the coefficients $\{0,\ldots,N-1\}^2$. Then $1-H$ is invertible, as the geometric series $1+H+H^2+\ldots$ converges: The $i$-th row of $H^j$ is bounded in supremum norm by $\lambda_{N+i} s^{j-1}$ (this is clear by definition for $j=1$, and then follows by an easy induction), and the geometric series $1+s+s^2+\ldots$ converges as $s<1$. (It is at this step that we use that $A$ is a Banach algebra -- we need to guarantee the existence of this infinite sum.)

Now by easy matrix manipulations, this means that one can write
\[
\left(\begin{array}{cc} 1-E & -F\\ -G & 1-H\end{array}\right) = \left(\begin{array}{cc} 1 & X\\ 0 & 1\end{array}\right) \left(\begin{array}{cc} E' & 0\\ 0 & 1-H\end{array}\right) \left(\begin{array}{cc} 1 & 0\\ Y & 1\end{array}\right),
\]
where the first and last factor are invertible, as well as the lower part $1-H$ of the middle matrix. This means that the cone of $1-f$ on $\mathcal M_{<p}(\mathbb N)\otimes_{<p} A$ can be rewritten as the cone of $E'$ on the finite free $A$-module $\mathcal M_{<p}(\{0,\ldots,N-1\})\otimes_{<p} A = A^N$, which is thus discrete.
\end{proof}

Finally, we apply this machinery to complex analysis. We will use the following definition.\footnote{It agrees with the similar notion that exists in the literature.}

\begin{definition}\label{def:steincompact} A compact subset $Z\subset \mathbb C^n$ is Stein if the corresponding idempotent $\mathbb C[X_1,\ldots,X_n]$-algebra $\mathcal O(Z)\in \mathcal D(\mathrm{Liq}_p(\mathbb C[X_1,\ldots,X_n]))$ is concentrated in degree $0$.
\end{definition}

Note that, by the cofinality of open and closed neighborhoods, one can write
\[
\mathcal O(Z) = \varinjlim_{U\supset Z} R\Gamma(U,\mathcal O)
\]
and by the computation of the structure sheaf, this sits in cohomological degrees $\geq 0$. Thus, the condition here is really that $\mathcal O(Z)$ sits in cohomological degrees $\leq 0$, i.e.~nonnegative homological degrees.

Any polydisc is Stein, and if $Z_1$ and $Z_2$ are Stein, then their intersection $Z_1\cap Z_2$ is also Stein, as it corresponds to the idempotent algebra $\mathcal O(Z_1\cap Z_2)=\mathcal O(Z_1)\otimes_{<p,\mathbb C[X_1,\ldots,X_n]}^L \mathcal O(Z_2)$, which sits in nonnegative homological degrees. Using a similar argument with filtered colimits, it follows that any intersection of Stein compact subsets is again Stein.

\begin{proposition}\label{prop:steincompactnuclear} For any compact subset $Z\subset \mathbb C^n$, the idempotent algebra $\mathcal O(Z)$ is basic nuclear. Moreover, $H^0(\mathcal O(Z))$ is a sequential colimit of Banach algebras.
\end{proposition}

\begin{proof} One can write $Z$ as a sequential intersection of subsets that are finite unions of polydiscs. The claim is stable under sequential intersections, so one can assume $Z$ is a finite union of polydiscs. By the formula for the idempotent algebras for finite unions, this reduces the claim of nuclearity to intersections of polydiscs, and thus to polydiscs (for the claim about $H^0$, one can directly reduce to the polydiscs, as closed subalgebras of Banach algebras are Banach algebras). But for polydiscs, the presentation
\[
\mathcal O(\{|X_1|,\ldots,|X_n|\leq 1\}) = \varinjlim_{r>1} \mathcal M_{<p}(\{X_1^{i_1}\cdots X_n^{i_n}/r^{i_1+\ldots+i_n}\}_{\{i_1,\ldots,i_n\geq 0\}})
\]
gives a presentation as a sequential colimit along trace-class maps, showing that it is basic nuclear. One can also write down a similar presentation as a sequential colimit of Banach algebras.
\end{proof}

Recall also the following result that follows from the abstract nonsense of Lecture V. Here, we endow the category of compact subsets of $\mathbb C^n$ with the topology of finite covers (so $\{Z_i\subset Z\}$ form a cover if a finite family of $Z_i$'s cover $Z$).

\begin{proposition}\label{prop:liquidquasicoherent} The association taking any compact subset $Z\subset \mathbb C^n$ to $\mathrm{Mod}_{\mathcal O(Z)}(\mathcal D(\mathrm{Liq}_p))$ defines a sheaf of $\infty$-categories.
\end{proposition}

In the following, we restrict to the site of compact Stein subsets (where this result then also applies).

The main theorem of this lecture is the following.

\begin{theorem}\label{thm:pseudocoherentglues} The association taking any compact Stein subset $Z\subset \mathbb C^n$ to $\mathcal D_{pc}(\mathcal O(Z))$ defines a sheaf of $\infty$-categories.
\end{theorem}

Here, $\mathcal D_{pc}(\mathcal O(Z))\subset \mathcal D(\mathcal O(Z))$ is the full subcategory of the derived $\infty$-category of (abstract!) $\mathcal O(Z)$-modules that can be represented by a complex of finite projective $\mathcal O(Z)$-modules that is bounded to the right.

The theorem will be the consequence of three individual results. First:

\begin{proposition}\label{prop:nuclearglues} The association taking any compact subset $Z\subset \mathbb C^n$ to
\[
\mathrm{Mod}_{\mathcal O(Z)}(\mathrm{Nuc})\subset \mathrm{Mod}_{\mathcal O(Z)}(\mathcal D(\mathrm{Liq}_p))
\]
defines a sheaf of $\infty$-categories. Moreover, $\mathrm{Mod}_{\mathcal O(Z)}(\mathrm{Nuc})$ agrees with the nuclear objects in $\mathrm{Mod}_{\mathcal O(Z)}(\mathcal D(\mathrm{Liq}_p))$.
\end{proposition}

\begin{proof} The final sentence comes from Corollary~\ref{cor:relativenuclear}.

First, note that it defines a presheaf, as the base change $-\otimes^L_{<p,\mathcal O(Z)} \mathcal O(Z')$ is given by a colimit of
\[
\ldots\to -\otimes^L_{<p} \mathcal O(Z)\otimes^L_{<p} \mathcal O(Z')\to -\otimes^L_{<p} \mathcal O(Z')
\]
and all occuring operations preserve nuclear $\mathbb C$-vector spaces. The fully faithfulness part of the sheaf axiom follows from Proposition~\ref{prop:liquidquasicoherent}. For essential surjectivity, we have to see that if $V\in \mathrm{Mod}_{\mathcal O(Z)}(\mathcal D(\mathrm{Liq}_p))$ is such that for a finite cover of $Z$ by $Z_i$'s, all $V_i=V\otimes_{<p,\mathcal O(Z)}^L \mathcal O(Z_i)$ are nuclear, then $V$ is nuclear. But $V$ can be recovered via a finite Cech complex from $V_i$ and the further base changes to all finite intersections, and all of these are nuclear.
\end{proof}

Note that if $Z$ is Stein, then
\[
\mathrm{Mod}_{\mathcal O(Z)}(\mathcal D(\mathrm{Liq}_p)) = \mathcal D(\mathrm{Liq}_p(\mathcal O(Z)))
\]
is the derived category of $p$-liquid $\mathcal O(Z)$-modules. In that case, we denote the nuclear objects by $\mathcal D_{nuc}(\mathrm{Liq}_p(\mathcal O(Z)))$.

The second input:

\begin{proposition}\label{prop:pseudocompactglues} The association taking any compact Stein subset $Z\subset \mathbb C^n$ to
\[
\mathcal D_{pc}(\mathrm{Liq}_p(\mathcal O(Z)))\subset \mathcal D(\mathrm{Liq}_p(\mathcal O(Z)))=\mathrm{Mod}_{\mathcal O(Z)}(\mathcal D(\mathrm{Liq}_p)))
\]
defines a sheaf of $\infty$-categories.
\end{proposition}

\begin{remark} The same result holds for actual compactness, in which case the argument is even easier. The variant for pseudocompact objects is however sometimes slightly more useful.
\end{remark}

\begin{proof} The key observation is that for $Z'\subset Z$, the functor $-\otimes_{<p,\mathcal O(Z)}^L \mathcal O(Z')$ is of bounded Tor-dimension. By idempotence, this can also be written as the composite of the (exact) forgetful functor to $\mathcal D(\mathrm{Liq}_p(\mathbb C[X_1,\ldots,X_n]))$, and the base change of $Z'$. The latter can also be written as the filtered colimit over the base change to all open neighborhoods $U$ of $Z'$. But if $U$ corresponds to a complementary closed $W$, then the base change to $U$ is given by $\underline{\mathrm{RHom}}([1\to \mathcal O(W)],-)$. Thus, it suffices to bound the amplitude of $\mathcal O(W)$ to the right. But $\mathcal O(W)=R\Gamma(W,\mathcal O)$, and the sheaf $\mathcal O$ is concentrated in degree $0$. As $W\subset \mathbb C^n$ has finite cohomological dimension, the result follows.

Again, the association clearly defines a presheaf of $\infty$-categories that satisfies the fully faithfulness part of the sheaf axiom. For essential surjectivity, we need to see that if $V\in \mathcal D(\mathrm{Liq}_p(\mathcal O(Z))$ is such that $V_i=V\otimes_{<p,\mathcal O(Z)}^L \mathcal O(Z_i)$ is pseudocompact for some finite cover of $Z$ by $Z_i$'s, then $V$ is pseudocompact. But one can compute $\mathrm{Hom}_{\mathcal O(Z)}(V,-)$ as a finite Cech-type limit of $\mathrm{Hom}_{\mathcal O(Z_J)}(V_J,-\otimes_{<p,\mathcal O(Z)}^L \mathcal O(Z_J))$ over all nonempty finite subsets $J$ of $I$. By the observation on finite Tor-dimension, any direct sum of objects in $\mathcal D_{\leq n}(\mathcal O(Z))$ stays of a similar form after base change to any $Z_J$ (for some possibly different $n$), so the result follows.
\end{proof}

Thus Theorem~\ref{thm:pseudocoherentglues} reduces to:

\begin{proposition}\label{prop:nuclearpseudocompactdiscrete} For any compact Stein subset $Z\subset \mathbb C^n$, one has an equality of full subcategories:
\[
\mathcal D_{pc}(\mathcal O(Z)) = \mathcal D_{nuc}(\mathrm{Liq}_p(\mathcal O(Z)))\cap \mathcal D_{pc}(\mathrm{Liq}_p(\mathcal O(Z)))\subset \mathcal D(\mathrm{Liq}_p(\mathcal O(Z))).
\]
\end{proposition}

\begin{proof} By Proposition~\ref{prop:almostdualizable}, it suffices to see that $\mathcal D(\mathrm{Liq}_p(\mathcal O(Z)))$ is Fredholm. This follows from Proposition~\ref{prop:steincompactnuclear} and Proposition~\ref{prop:banachalgebrafredholm}.
\end{proof}\newpage

\section{Lecture X: Coherent Sheaves, II}

In the last lecture, we proved the following theorem.

\begin{theorem} The association taking any compact Stein subset $Z\subset \mathbb C^n$ to $\mathcal D_{pc}(\mathcal O(Z))$ defines a sheaf of $\infty$-categories (with respect to the finite cover topology).
\end{theorem}

Here, a compact subset $Z\subset \mathbb C^n$ was defined to be Stein if $H^i(Z,\mathcal O)=0$ for $i>0$, and $\mathcal D_{pc}(\mathcal O(Z))\subset \mathcal D(\mathcal O(Z))$ was the full subcategory of objects that admit a bounded to the right (but not necessarily finite) representative by a complex consisting of finite projective $\mathcal O(Z)$-modules.

In this lecture, we will refine this result by proving a descent result on the abelian level. As a first preparation, we note that with an essentially formal argument, one can control the amplitude of the complexes to the right.

\begin{corollary}\label{cor:animateddescends} The association taking any compact Stein subset $Z\subset \mathbb C^n$ to the $\infty$-category $\mathcal D_{pc,\geq 0}(\mathcal O(Z))$ of pseudocoherent complexes in nonnegative homological degrees defines a sheaf of $\infty$-categories (with respect to the finite cover topology).
\end{corollary}

This is some version of the Cartan--Oka Theorem on the vanishing of positive degree cohomology of coherent sheaves on Stein spaces.

\begin{proof} The pullback functors are well-defined, and by the previous theorem, it satisfies the fully faithfulness part of the sheaf axiom. It remains to see that if $V\in \mathcal D_{pc}(\mathcal O(Z))$ is such that $V_i=V\otimes_{<p,\mathcal O(Z)}^L \mathcal O(Z_i)$ (for some finite cover of $Z$ by $Z_i$'s) lies in $\mathcal D_{\geq 0}$ for all $i$, then $V\in \mathcal D_{\geq 0}$. This follows from the next lemma applied to the lowest nonzero homology group of $V$.
\end{proof}

\begin{lemma}\label{lem:fingenmodulelocallyzero} Let $Z\subset \mathbb C^n$ be compact Stein, and let $M$ be a finitely generated $\mathcal O(Z)$-module. Assume that for all $z\in Z$, the (non-derived!) base change $M\otimes_{<p,\mathcal O(Z)} \mathcal O_z=0$ vanishes. Then $M=0$.
\end{lemma}

This lemma holds true in very large generality for ``analytic rings''. We note that, crucially, the tensor products here are underived.

\begin{proof} Assume $M\neq 0$. The hypothesis passes to any quotient of $M$, so we can assume that $M$ is generated by one element, so $M=\mathcal O(Z)/I$ for some ideal $I\subset \mathcal O(Z)$. In that case, $M$ acquires the structure of an $\mathcal O(Z)$-algebra. Thus, $M\otimes_{<p,\mathcal O(Z)}^L \mathcal O(Z')$ is a derived algebra, and in particular is $0$ as soon as $1=0$ in its $H_0$. By assumption, this happens in all the stalks, and thus in some neighborhood. Thus, $M\otimes_{<p,\mathcal O(Z)}^L \mathcal O(Z_i)=0$ for some finite cover of $Z$ by $Z_i$'s. But $M$ can be recovered via a Cech complex from these base changes, so $M=0$.
\end{proof}

\begin{remark} Our arguments so far were of a very formal nature; we used no finiteness results of the algebras $\mathcal O(Z)$, and one could obtain results for many variants. For example, \cite{Andreychev} argues in a very similar way to obtain such descent results on analytic adic spaces, without any noetherian hypotheses. In such situations, it may also happen that the localization maps like $\mathcal O(Z)\to \mathcal O(Z')$ are not flat. In such a large generality, getting descent for $\mathcal D_{pc,\geq 0}$ is the best one can hope for; basically, one is forced to replace modules by animated modules, and then the results say that for pseudocoherent animated modules, descent works as expected.
\end{remark}

It turns out that to go from Corollary~\ref{cor:animateddescends} to the usual theory of coherent sheaves, one needs to prove no further results about modules, but instead about the rings $\mathcal O(Z)$. Here, we have the following theorem; it is a strong form of Oka's (first) coherence theorem \cite{OkaVII}.

\begin{theorem}\label{thm:compactsteincoherent} Let $Z\subset \mathbb C^n$ be compact Stein.
\begin{enumerate}
\item The ring $\mathcal O(Z)$ is coherent, i.e.~any finitely generated ideal is finitely presented.
\item The maximal ideals of $\mathcal O(Z)$ correspond to the points of $Z$. The complete local rings of $\mathcal O(Z)$ at closed points are formal power series algebras, and are flat over $\mathcal O(Z)$. The ring $\mathcal O(Z)$ is flat over $\mathbb C[X_1,\ldots,X_n]$.
\item If $Z'\subset Z$ is compact Stein, the restriction map $\mathcal O(Z)\to \mathcal O(Z')$ is flat.
\item If $Z$ is contained in the interior of some compact $Z'$ (not necessarily Stein), any ideal of $\mathcal O(Z')$ becomes finitely generated after base change to $\mathcal O(Z)$.
\item If $Z$ is a point, then $\mathcal O(Z)$ is noetherian.
\item If $\mathcal O(Z)$ is noetherian, then it is regular and excellent.
\end{enumerate}
\end{theorem}

\begin{remark}\label{rem:siunoetherian} It is not in general true that $\mathcal O(Z)$ is noetherian, as for example if $Z$ is totally disconnected (and infinite), when there are infinitely many idempotents. In some sense, this is the only obstruction: A theorem of Siu \cite{SiuNoetherian} says that $\mathcal O(Z)$ is noetherian as soon as any closed analytic subspace has only finitely many connected components. We will come back to the question of noetherianity of $\mathcal O(Z)$ later.
\end{remark}

Before going into the proof of Theorem~\ref{thm:compactsteincoherent}, we deduce the following corollary.

\begin{corollary}\label{cor:coherentsheavesdescend} Assume that parts (1) and (3) of Theorem~\ref{thm:compactsteincoherent} hold in dimension $n$. Then the functor taking any compact Stein $Z\subset \mathbb C^n$ to the category of finitely presented $\mathcal O(Z)$-modules defines a sheaf of categories (for the finite cover topology). It is naturally a subsheaf of $Z\mapsto \mathcal D_{pc,\geq 0}(\mathcal O(Z))$.
\end{corollary}

\begin{proof} Over a coherent ring, any finitely presented module admits a resolution that is finite free in each degree, thus giving a full embedding of $\mathrm{Coh}(\mathcal O(Z))$ into $\mathcal D_{pc,\geq 0}(\mathcal O(Z))$. As $\mathcal O(Z)\to \mathcal O(Z')$ is flat whenever $Z'\subset Z$ is compact Stein, this defines a subpresheaf. It automatically satisfies the fully faithfulness part of the sheaf axiom. For essential surjectivity, note that if one glues objects in degree $0$, the resulting glued object lives in homological degrees $\leq 0$ (as it is computed by a Cech complex). But by Corollary~\ref{cor:animateddescends}, it also lives in degrees $\geq 0$, so it still lives in degree $0$. But the objects of $\mathcal D_{pc,\geq 0}(\mathcal O(Z))$ that are concentrated in degree $0$ are precisely the finitely presented $\mathcal O(Z)$-modules.
\end{proof}

\begin{remark} On a coherent scheme, the usual way to think about coherent sheaves is that on affines, they correspond to finitely presented modules, and in general one glues such. The previous corollary says that one can also think of coherent sheaves in complex-analytic geometry in this way, if one regards compact Stein subsets as the analogues of affine subsets. Together with Remark~\ref{rem:siunoetherian}, one can even pretend to be on a noetherian scheme.
\end{remark}

Before we start the proof, we recall the Weierstra\ss\ Preparation Theorem.

\begin{theorem}[Weierstra\ss\ Preparation]\label{thm:weierstrassprep} Let $\mathcal O_0=\mathcal O(\{(0,\ldots,0\})$ be the ring of germs of holomorphic functions at the origin $(0,\ldots,0)\in \mathbb C^n$. Let $f\in \mathcal O_0$ be such that
\[
f(X_1,0,\ldots,0)\neq 0.
\]
Let $k$ be the order of vanishing of $f(X_1,0,\ldots,0)$ at $X_1=0$. Then one can uniquely write
\[
f=gu
\]
where $u\in \mathcal O_0$ is invertible and $g\in \mathcal O_0$ is a monic polynomial in $X_1$ of degree $k$ whose coefficients are germs of holomorphic functions in $X_2,\ldots,X_n$ vanishing at the origin.
\end{theorem}

\begin{proof} Inside formal power series, it is easy to solve for the coefficients of $g$ and $u$ inductively, arguing modulo powers of $(X_2,\ldots,X_n)$. More precisely, we will find inductively $g_i$ and $u_i$ so that $f-g_iu_i\in (X_2,\ldots,X_n)^i \mathcal O_0$. We may start with $g_1=X_1^k$ and $u_1=f(X_1,0,\ldots,0)/X_1^k$. Given $g_i$ and $u_i$, we first compute
\[
(f-g_iu_i)u_1^{-1}\in (X_2,\ldots,X_n)^i\mathcal O_0,
\]
let $h_i$ be the part of the power series where $X_1$ appears to a power less than $k$, and then set $g_{i+1}=g_i+h_i$. Then modulo $X_1^k$, we have
\[
f - g_{i+1} u_i = f - g_iu_i - (f-g_iu_i)u_1^{-1}u_i = (f - g_iu_i)(1 - u_1^{-1} u_i)\in (X_2,\ldots,X_n)^{i+1}.
\]
Thus modulo $(X_2,\ldots,X_n)^{i+1}$, we have
\[
f - g_{i+1} u_i\equiv X_1^k v_i
\]
for some $v_i\in (X_2,\ldots,X_n)^i\mathcal O_0$, and we set $u_{i+1} = u_i + v_i$. Then modulo $(X_2,\ldots,X_n)^{i+1}$, we have
\[
f - g_{i+1} u_{i+1} = f - g_{i+1} u_i - g_{i+1} v_i\equiv X_1^k v_i - g_{i+1} v_i\equiv (X_1^k-g_{i+1})v_n\equiv 0,
\]
giving the desired induction step. Uniqueness can also easily be shown.

One possibility to conclude is now to analyze this algorithm and deduce bounds on the coefficients. Let us give an abstract argument. We argue by induction on $n$. Using this induction, we already know that $f(X_1,\ldots,X_{n-1},0)$ admits the desired factorization, or equivalently that $g(X_1,\ldots,X_{n-1},0)$ and $u(X_1,\ldots,X_{n-1},0)$ have the desired growth of coefficients. Thus, the map
\[
\bigoplus_{j=0}^{k-1} \mathcal O_0' \cdot X_1^j\to \mathcal O_0/f
\]
is an isomorphism after quotienting by $X_n$. But this can be approximated by a similar map of Banach algebras, where again it is an isomorphism after quotienting by $X_n$. But by Banach's Open Mapping Theorem, this means that the inverse of $X_n$ has bounded norm, and in particular the map becomes an isomorphism after localizing to the locus where $|X_n|$ is sufficiently small. But this gives the desired result.\footnote{This argument is nicely presented in a paper of Ullrich, \cite{UllrichWeierstrassPreparation}.}
\end{proof}

Now we can prove the finiteness results on rings of holomorphic functions.

\begin{proof}[Proof of Theorem~\ref{thm:compactsteincoherent}] We argue by induction on $n$. For $n=0$, there is nothing to prove.

Now assume the result in dimension $n-1$. Let $Z'$ be any compact subset of $\mathbb C^n$ that contains $Z$ in its interior, and let $I'$ be any ideal of $\mathcal O(Z')$. Let $I=I'\mathcal O(Z)$ and $M=\mathcal O(Z)/I$. We claim that there is a finite cover of $Z$ by $Z_i$ such that the (non-derived!) base changes $M_i=M\otimes_{\mathcal O(Z)} \mathcal O(Z_i)$ lie in $\mathcal D_{pc,\geq 0}(\mathcal O(Z_i))$, and have the property that all further base changes $M_i\otimes^L_{<p,\mathcal O(Z_i)} \mathcal O(Z_i')$ are concentrated in degree $0$. To show this, it suffices to find a cover of $Z$ by $Z_i$ such that each $M_i$ admits a (possibly infinite) resolution by finite free $\mathcal O(Z_i)$-modules and has the property that $M_i\otimes^L_{<p,\mathcal O(Z_i)} \mathcal O_z$ is concentrated in degree $0$ for all $z\in Z_i$. To find such a cover, we can argue locally around a point $z\in Z$. We can assume $z=0\in \mathbb C^n$ is the origin. We can also assume that both $Z$ and $Z'$ are polydiscs. If $I=0$, there is nothing to show. Otherwise, we find some nonzero $f\in I$, and after shrinking $Z$ and changing coordinates, we can apply Theorem~\ref{thm:weierstrassprep} to assume that it is a monic polynomial in $X_1$. In that case $\mathcal O(Z)/f$ is finite free over a similar polydisc $W$ one dimension lower. By the theorem in one dimension lower, it follows that $\mathcal O(Z)/I$ is pseudocoherent over $\mathcal O(W)$, hence over $\mathcal O(W)[X_1]$, and then also over $\mathcal O(Z)$ (as any resolution over $\mathcal O(W)[X_1]$ base changes to a resolution over the idempotent liquid $\mathcal O(W)[X_1]$-algebra $\mathcal O(Z)$). Moreover, all further base changes along compact Stein subsets $W'\subset W$ stay in degree $0$, and in particular this is true for all stalks. But as $\mathcal O(Z)/f$ is finite over $\mathcal O(W)$, this also recovers (up to direct summands) all the stalks over $Z$.

Now given such a cover of $Z$ by $Z_i$'s, the $M_i$ actually define a descent datum in $\mathcal D_{pc,\geq 0}$, and hence glue back to some $M'\in \mathcal D_{pc,\geq 0}(\mathcal O(Z))$, that is necessarily concentrated in degree $0$, i.e.~just a $\mathcal O(Z)$-module. Being pseudocoherent, it is in particular finitely presented. Also, $M_i=M'\otimes_{\mathcal O(Z)}^L \mathcal O(Z_i)$ is also the derived base change, as $M'$ is the result of gluing the $M_i$ in the setting of $\mathcal D_{pc,\geq 0}$. Now the cokernel of $M\to M'$ is a finitely generated $\mathcal O(Z)$-module that vanishes locally, so vanishes by Lemma~\ref{lem:fingenmodulelocallyzero}. Thus, $M\to M'$ is surjective. Let $K\subset M$ be the kernel. This is again a finitely generated $\mathcal O(Z)$-module (as the kernel of a map from a finitely generated module to a finitely presented module), and using that $M'\otimes_{\mathcal O(Z)}^L \mathcal O(Z_i)$ is concentrated in degree $0$, one also sees that it vanishes locally. Thus, also $K=0$, and $M\to M'$ is an isomorphism. But $M'$ is pseudocoherent, i.e.~has a resolution by finite free $\mathcal O(Z)$-modules, hence so is $M=\mathcal O(Z)/I$. This implies in particular that $I$ is finitely presented, and hence that $\mathcal O(Z)$ is coherent. Moreover, $M'$ gives rise to a sheaf that is concentrated in degree $0$, hence so does $M$. Thus, for all $Z'\subset Z$ compact Stein, the base change $M\otimes_{\mathcal O(Z)}^L \mathcal O(Z')$ is concentrated in degree $0$. As this is true for all finitely presented $M=\mathcal O(Z)/I$, we see that $\mathcal O(Z)\to \mathcal O(Z')$ is flat.

At this point, we have proved parts (1), (3) and (4). Part (5) follows directly from Weierstra\ss preparation and induction on $n$. For part (2), consider a maximal ideal $I\subset \mathcal O(Z)$, and let $A=\mathcal O(Z)/I$. By Lemma~\ref{lem:fingenmodulelocallyzero}, there is some $z\in Z$ such that the base change $A_z=\mathcal O_z/I\mathcal O_z$ is nonzero. But $\mathcal O_z$ is local with maximal ideal the kernel of the evaluation at $z$. Thus, there is a further quotient $A_z\to \mathbb C$; as $I$ was maximal, this means it is the kernel of evaluation at $z$. It is clear that the complete local ring is a power series ring; in fact, the derived base change
\[
\mathbb C\otimes_{\mathbb C[X_1,\ldots,X_n]}^L \mathcal O(Z)=\mathbb C
\]
as $\mathbb C$ is a module over $\mathcal O(Z)$ and $\mathcal O(Z)$ is idempotent. This means that the regular sequence $(X_1-z_1,\ldots,X_n-z_n)$ generates the maximal ideal, and then one easily computes the completion. To show that the completion is flat, we factor it as $\mathcal O(Z)\to \mathcal O_z\to \mathbb C[[X_1-z_1,\ldots,X_n-z_n]]$. The first map is flat by (3), so it suffices to show that the second map is flat, where it follows from (5). To show that $\mathcal O(Z)$ is flat over $\mathbb C[X_1,\ldots,X_n]$ for all $Z$, note that for any ideal $I\subset \mathbb C[X_1,\ldots,X_n]$, we get a sheaf $Z\mapsto \mathcal O(Z)\otimes_{\mathbb C[X_1,\ldots,X_n]}^L \mathbb C[X_1,\ldots,X_n]/I$, and to show that this is concentrated in degree $0$, we can check on stalks. These are given by $\mathcal O_z\otimes_{\mathbb C[X_1,\ldots,X_n]}^L \mathbb C[X_1,\ldots,X_n]/I$. As $\mathcal O_z$ is noetherian, the flatness can now be checked on complete local rings, where it is clear (as it agrees with the complete local rings of the polynomial algebra).

Finally, part (6) follows from (2) and a criterion of Matsumura \cite[Theorem 102]{Matsumura}, using the derivations $\frac{\partial}{\partial X_i}$ for $i=1,\ldots,n$, which are ``orthogonal'' to a system of parameters at each closed point, by (2).\footnote{Excellence was proved by Bingener \cite[Theorem 1.10]{Bingener} and Scheja--Storch \cite[Satz 8.10]{SchejaStorch}.}
\end{proof}\newpage

\section*{Appendix to Lecture X: Noetherianity}

In this appendix, we discuss some results on the Noetherianity of $\mathcal O(Z)$ for compact Stein $Z$. The goal is to prove the following theorem of Frisch \cite{Frisch}.\footnote{A simple proof of this was given by Langmann in \cite{LangmannFrisch}, but we believe that it is incorrect, as the proof of \cite[Lemma 1]{LangmannFrisch} seems to work only in the complex-analytic case, but not in the real-analytic case where it is applied later. The proof we give below fixes Langmann's argument.}

\begin{theorem}\label{thm:noetherian} Assume that $Z\subset \mathbb C^n$ is a compact Stein subset given by finitely many inequalities of the form $|f(z)|\leq 1$ for real-analytic $f$. Then $\mathcal O(Z)$ is noetherian.
\end{theorem}

\begin{remark} The theorem implies, via excellence, Oka's Second and Third Coherence Theorem (as passing to the reduced quotient, and to normalizations, now passes to complete local rings, and hence also to localizations $\mathcal O(Z)\to \mathcal O(Z')$).
\end{remark}

We note that a consequence of this is that $Z$ can have only finitely many connected components. Such basic facts about (closed) semi-analytic sets, as established in \cite{Lojasiewicz}, usually enter as ingredients into the proof of Theorem~\ref{thm:noetherian}. The argument below avoids appeal to \cite{Lojasiewicz} (but still starts by proving some simple facts about semi-analytic sets by hand).

First, we note that we can restrict to the real-analytic case. More precisely, we can embed $Z\subset \mathbb C^n\cong \mathbb R^{2n}$ into $\mathbb C^{2n}$. Then, given that $\mathcal O(Z\subset \mathbb C^{2n})$ is noetherian, the map $\mathcal O(Z\subset \mathbb C^n)\to \mathcal O(Z\subset \mathbb C^{2n})$ is faithfully flat,\footnote{By the classification of maximal ideals, it suffices to check flatness; but this can be done on local rings, and then on complete local rings as the target is noetherian.} and hence $\mathcal O(Z\subset \mathbb C^n)$ is noetherian.

Thus, assume from now on that $Z\subset \mathbb R^n$. Using a Zariski closed immersion, we can then reduce to a cube $Z=[-1,1]^n\subset \mathbb R^n$. We will now argue by induction on $n$, and work with the algebra $\mathcal O^{\mathbb R}(Z)$ of real-analytic functions on $Z$ (so $\mathcal O(Z) = \mathcal O^{\mathbb R}(Z)\otimes_{\mathbb R}\mathbb C$).

We will use the analogue of Weierstra\ss\ Preparation in this situation, so we need to discuss finite algebras over $\mathcal O^{\mathbb R}([-1,1]^n)$; let $A$ be such an algebra. The first result we need is the following lemma on connected components.

\begin{lemma}\label{lem:locallyfinitepi0} Let $A$ be an $\mathcal O^{\mathbb R}([-1,1]^n)$-algebra that is finitely presented as a module. Let $0<c<1$. Then the image of the map
\[
\pi_0(A(\mathbb R)\times_{[-1,1]^n} [-c,c]^n)\to \pi_0(A(\mathbb R))
\]
is finite.
\end{lemma}

(Theorem~\ref{thm:noetherian} in fact implies that the target is finite.)

\begin{proof} We need to see that $A(\mathbb R)$ has only finitely many connected components whose image in $[-1,1]^n$ meets a given compact subset contained in the interior $(-1,1)^n$ of the cube. This can be checked locally around points of the interior. We argue by induction on $n$. After passing to a quotient by a nilpotent ideal, there is some $f\in \mathcal O^{\mathbb R}([-1,1]^n)$ such that
\[
\mathcal O^{\mathbb R}([-1,1]^n)[f^{-1}]\to A[f^{-1}]
\]
is finite \'etale. In particular, over $\{f\neq 0\}\subset [-1,1]^n$, the map $A(\mathbb R)\to [-1,1]^n$ becomes a finite covering space. As $\{f\neq 0\}$ has only finitely many connected components by Lemma~\ref{lem:nonvanishingfinitepi0} below, there are only finitely many connected components of $A(\mathbb R)$ meeting the locus where $f\neq 0$. Thus, it suffices to prove the similar result for $A/f$ instead. But $A/f$ locally admits a finitely presented projection to an $n-1$-dimensional cube, so we conclude by induction.
\end{proof}

\begin{lemma}\label{lem:nonvanishingfinitepi0} Let $f\in \mathcal O^{\mathbb R}([-1,1]^n)$. Then
\[
\{f\neq 0\}\subset [-1,1]^n
\]
has only finitely many connected components.
\end{lemma}

\begin{proof} By induction, we can assume that this is true for the intersections with all the faces. It is then enough to prove that $\{f\neq 0\}\cap (-1,1)^n$ has only finitely many connected components. But $f$ extends to $[-C,C]^n$ for some $C>1$, and if we multiply $f$ by the hyperplane equations for all the faces, it is enough to prove that for any $g\in \mathcal O^{\mathbb R}([-C,C]^n)$, the locus $\{g\neq 0\}$ has only finitely many connected components that meet $[-1,1]^n$. But then we can again locally find a projection to an $n-1$-dimensional cube so that the locus $g=0$ is finite over it. Arguing as in the previous proof, there is some nonzero function $h$ on the $n-1$-dimensional cube such that over the nonvanishing locus $h\neq 0$, the locus $g=0$ defines a finite topological cover. In that case, the locus $\{g\neq 0,h\neq 0\}\subset \mathbb R\times \{h\neq 0\}$ mapping to $\{h\neq 0\}$ is, over each connected component, a finite union of open intervals, and hence has also only finitely many connected components. But by openness, no connected component of $\{g\neq 0\}$ can map into $\{h=0\}$.
\end{proof}

Coming back to the finite $\mathcal O^{\mathbb R}([-1,1]^n)$-algebra $A$, we note that inside the corresponding compact Hausdorff space $A(\mathbb C)$ of $\mathbb C$-valued points, we have the closed subspace $A(\mathbb R)\subset A(\mathbb C)$ of real points, giving rise to an idempotent $p$-liquid $A$-algebra $A^{\mathbb R}$ of real-analytic functions on $A(\mathbb R)$. Note that $A^{\mathbb R}$ is in general far from a finite $\mathcal O^{\mathbb R}([-1,1]^n)$-algebra.

The key step is the following proposition, showing a local finite generation result in the interior of the cube.

\begin{proposition}\label{prop:locallyfingenreal0} Assume that $\mathcal O^{\mathbb R}([-1,1]^n)$ is noetherian. Let $A$ be a finite $\mathcal O^{\mathbb R}([-1,1]^n)$-algebra and let $I\subset A^{\mathbb R}$ be an ideal. Then there is some $\epsilon>0$ such that, denoting
\[
A_\epsilon^{\mathbb R}=A^{\mathbb R}\otimes^{\mathrm{liq}}_{\mathcal O^{\mathbb R}([-1,1]^n)} \mathcal O^{\mathbb R}([-\epsilon,\epsilon]^n),
\]
the ideal $IA_\epsilon^{\mathbb R}$ is finitely generated.
\end{proposition}

\begin{proof} We argue by induction on $n$ (noting that the noetherianity assumption for some $n$ implies it for smaller $n$, by passing to quotients). Passing to a smaller cube around the origin, we can assume that $A(\mathbb C)$ has a unique point above $0^n\in [-1,1]^n$, and that this point is real. In particular,
\[
A_0^{\mathbb R} = A^{\mathbb R}\otimes_{\mathcal O^{\mathbb R}([-1,1]^n)} \mathcal O^{\mathbb R}(0^n) = A\otimes_{\mathcal O^{\mathbb R}([-1,1]^n)} \mathcal O^{\mathbb R}(0^n)
\]
is the stalk at this point, and is a finite $\mathcal O^{\mathbb R}(0^n)$-algebra, where the latter is noetherian, regular and excellent. The ideal $IA_0^{\mathbb R}$ is therefore finitely generated, and at the expense of replacing $A$ by a quotient by a finitely generated ideal, we can assume that $IA_0^{\mathbb R}=0$. The goal is then to show that $IA_\epsilon^{\mathbb R}=0$ for some $\epsilon>0$. Note that as $\mathcal O^{\mathbb R}([-1,1]^n)$ is noetherian, we can replace $A$ by its reduction (arguing modulo powers of the nilradical).

Let $\tilde{A}$ be the normalization of $A$, which is finite by excellence of $\mathcal O^{\mathbb R}([-1,1]^n)$. For any $0<\epsilon<1$, the map
\[
\pi_0(\Spec(A_\epsilon^{\mathbb R}\otimes_A \tilde{A}))\to \pi_0(\Spec(A^{\mathbb R}\otimes_A \tilde{A}))
\]
has finite image by Lemma~\ref{lem:locallyfinitepi0} (writing $\tilde{A}$, as a finite $A$-algebra, by adjoining $2$ real variables each algebraic variable (as its real and imaginary part), which are also implicitly added into the coordinates of the base cube, which can be taken large in these directions).\footnote{Here, the connected components of the spectrum can also be understood in terms of connected components of the maximal spectrum, with its archimedean topology.} Taking $\epsilon$ small enough, it then follows that the image agrees with the image of
\[
\pi_0(\Spec(A_0^{\mathbb R}\otimes_A \tilde{A}))\to \pi_0(\Spec(A^{\mathbb R}\otimes_A \tilde{A})).
\]
But now the kernel of $A^{\mathbb R}\to A_0^{\mathbb R}$ is contained in the kernel of $A^{\mathbb R}\otimes_A \tilde{A}\to A_0^{\mathbb R}\otimes_A \tilde{A}$, and as this map is flat and the source is normal (as normality passes to analytic localizations, as it can (by excellence) be checked on complete local rings), any element in the kernel of this map must vanish on the whole connected components containing $\Spec(A_0^{\mathbb R}\otimes_A \tilde{A})$. But then it also vanishes on the image of $\Spec(A_\epsilon^{\mathbb R}\otimes_A \tilde{A})$, and hence the kernel of $A^{\mathbb R}\otimes_A \tilde{A}\to A_0^{\mathbb R}\otimes_A \tilde{A}$ is equal to the kernel of $A^{\mathbb R}\otimes_A \tilde{A}\to A_\epsilon^{\mathbb R}\otimes_A \tilde{A}$. This implies that $IA_\epsilon^{\mathbb R}\otimes_A \tilde{A}=0$ and hence also $IA_\epsilon^{\mathbb R}=0$.
\end{proof}

Finally, we can prove \cite[Lemma 1]{LangmannFrisch}, at least assuming the inductive hypothesis.

\begin{lemma}\label{lem:locallyfingenreal} Assume that $\mathcal O^{\mathbb R}([-1,1]^n)$ is noetherian. Let $I\subset \mathcal O^{\mathbb R}([-1,1]^{n+1})$ be any ideal. Then for $0<c<1$,  the ideal $I\mathcal O^{\mathbb R}([-c,c]^{n+1})$ of $\mathcal O^{\mathbb R}([-c,c]^{n+1})$ is finitely generated.
\end{lemma}

\begin{proof} Using Lemma~\ref{lem:fingenmodulelocallyzero}, it is enough to prove that any point in the interior of $(-1,1)^{n+1}$ admits a closed neighborhood $Z$ such that $I\mathcal O^{\mathbb R}(Z)$ is finitely generated. If $I=0$, there is nothing to prove, so take some nonzero $f\in I$. Then locally around any point, we can find a finite projection to an $n$-dimensional cube, reducing us to Proposition~\ref{prop:locallyfingenreal0}.
\end{proof}

Now the rest of the proof of Theorem~\ref{thm:noetherian} is as in \cite{LangmannFrisch}. Namely, assume by induction that $\mathcal O^{\mathbb R}([-1,1]^n)$ is noetherian, and consider any sequence $f_0,f_1,\ldots,\in \mathcal O^{\mathbb R}([-1,1]^{n+1})$ generating an ideal $I$. By induction, for every face $Z$ of the cube, $\mathcal O^{\mathbb R}(Z)$ is noetherian: Indeed, if the face is given by say $x_1=1$, then its $x_1-1$-adic completion is $\mathcal O^{\mathbb R}([-1,1]^n)[[x_1-1]]$, which is noetherian by induction, and faithfully flat over $\mathcal O^{\mathbb R}(Z)$. In particular, there is some $N$ such that $(f_0,\ldots,f_N)\mathcal O^{\mathbb R}(Z) = I\mathcal O^{\mathbb R}(Z)$ for all faces $Z$ of the cube. For every $N'>N$, the quotient
\[
(f_0,\ldots,f_{N'})/(f_0,\ldots,f_N)\subset \mathcal O^{\mathbb R}([-1,1]^{n+1})/(f_0,\ldots,f_N)
\]
defines a coherent sheaf whose support is a compact subset of $(-1,1)^n$. Note that
\[
\mathcal O^{\mathbb R}([-1,1]^{n+1})/(f_0,\ldots,f_N)
\]
extends to a coherent sheaf on a small neighborhood $[-1-\epsilon,1+\epsilon]^{n+1}$ of $[-1,1]^n$ (as all $f_0,\ldots,f_N$ extend to a small neighborhood), and so do these coherent subsheaves $(f_0,\ldots,f_{N'})/(f_0,\ldots,f_N)$ (via extending by zero outside $[-1,1]^{n+1}$). Each of them thus defines a coherent submodule of $\mathcal O^{\mathbb R}([-1-\epsilon,1+\epsilon]^{n+1})/(f_0,\ldots,f_N)$, and in particular an ideal of $\mathcal O^{\mathbb R}([-1-\epsilon,1+\epsilon]^{n+1})$ via taking the preimage. Let $\tilde{I}\subset \mathcal O^{\mathbb R}([-1-\epsilon,1+\epsilon]^{n+1})$ be their union. By construction, one has $I=\tilde{I} \mathcal O^{\mathbb R}([-1,1]^{n+1})$. We conclude with Lemma~\ref{lem:locallyfingenreal}.
\newpage

\section{Lecture XI: Complex analytic spaces}

The goal in this lecture will be to formally introduce a category of complex analytic spaces in which we will work.  But first we want to add some small things to the discussion of the previous lecture.  Recall from the previous lecture the following:

\begin{enumerate}
\item For every compact Stein $K\subset\mathbb{C}^n$, the ring $\mathcal{O}(K)(\ast)$ is coherent, and in particular an object in $\mathcal{D}(\operatorname{Mod}_{\mathcal{O}(K)(\ast)})$ is pseudocompact if and only if it is homologically bounded below and each homology group is finitely presented as an $\mathcal{O}(K)(\ast)$-module.  (In fact, in the appendix we also saw that for enough $K$, including closed polydisks, the ring $\mathcal{O}(K)(\ast)$ is even noetherian, so for these $K$ we can replace ``finitely presented'' by ``finitely generated''.)
\item The association $K\mapsto \mathcal{D}_{pc}(\mathcal{O}(K)(\ast))$ satisfies descent for the topology of finite covers of compact Steins.
\item For $K'\subset K$, the restriction map $\mathcal{O}(K)(\ast)\rightarrow \mathcal{O}(K')(\ast)$ is flat.
\end{enumerate}

It follows that assigning to every $K$ the category of finitely presented $\mathcal{O}(K)(\ast)$-modules is a sheaf of abelian categories with respect to the topology of finite covers.  We would now like prove the following ``hygienic'' properties of this sheaf of abelian categories.

\begin{proposition}\label{coherentisDNF}
Suppose $K$ is a compact Stein in $\mathbb{C}^n$.  Then:
\begin{enumerate}
\item For any finitely presented $\mathcal{O}(K)$-module $M$, we have that $M$, viewed as an underlying object of $\operatorname{Liq}_p$, is a DNF space (see Lecture VIII).   In particular $M$ is quasi-separated.
\item For any $\mathcal{O}(K)(\ast)$-module $M$, we have that $\mathcal{O}(K)\otimes^L_{\mathcal{O}(K)(\ast)} M\in\mathcal{D}(\operatorname{Liq}_p)$ lives in degree $0$.
\end{enumerate}
\end{proposition}

The claim in (2) is a bit subtle: we know that $\mathcal{D}(\mathcal{O}(K)(\ast))$ embeds fully faithfully into $\mathcal{D}(\operatorname{Mod}_{\mathcal{O}(K)}(\operatorname{Liq}_p))$ via $M\mapsto \mathcal{O}(K)\otimes^L_{\mathcal{O}(K)(\ast)} M$ for formal reasons (the unit object in $\mathcal{D}(\operatorname{Mod}_{\mathcal{O}(K)}(\operatorname{Liq}_p))$ is compact and has endomorphism ring $\mathcal{O}(K)(\ast)$), but there's nothing formal guaranteeing that this functor is $t$-exact.  Nonetheless, that's what (2) claims.  Note that this property, combined with flatness of $\mathbb{C}[x_1,\ldots,x_n](\ast)\rightarrow \mathcal{O}(K)(\ast)$, was used above in Lemma \ref{analyticaffinelemma} in the calculation of the structure sheaf of the analytification of a finite type $\mathbb{C}$-algebra.

\begin{proof}
For (1), note that the DNF spaces are exactly the basic nuclear objects concentrated in degree $0$ which are quasi-separated.  (Indeed, without being quasi-separated, we know from Theorem \ref{basicnuclearcharacterization} that we have a quotient of DNF spaces; but by Lemma \ref{DNFlemmas} a quasiseparated quotient of DNF spaces is DNF.)  So it suffices to show that $M$ is basic nuclear and quasi-separated. For $M=\mathcal O(K)$, basic nuclearity is Proposition~\ref{prop:steincompactnuclear}, and general $M$'s are a quotient of a finite direct sum of such, which is still basic nuclear. Thus, it remains to see that $M$ is quasiseparated. We can embed $M$ into a product of copies of the local rings at $x\in K$. It suffices to see that these are qs, as a product of qs spaces is qs, and a subobject of a qs space is qs. Thus, it suffices to show that any finitely generated module $M$ over the local ring $\mathcal{O}_{\mathbb{C}^n,x}$ is qs, or equivalently DNF.  Because an extension of DNF spaces is DNF (Lemma \ref{DNFlemmas}), we can use induction to reduce to the case of one generator.  Then $M = \mathcal{O}_{\mathbb{C}^n,x}/I$ for an ideal $I$. If $I=0$, we're done by the explicit formula for the local ring; otherwise there is $0\neq f\in I$, and after a linear change of coordinates we can use Weierstra\ss\ preparation to conclude that $\mathcal{O}_{\mathbb{C}^n,x}/f$ is finite free over a local ring in one dimension lower.  Then $\mathcal{O}_{\mathbb{C}^n,x}/I$ is finitely generated over a local ring in one dimension lower, so we conclude by induction on the dimension.

For (2), by writing every $\mathcal{O}(K)(\ast)$-module as a filtered colimit of finitely presented ones, we reduce to the case when $M$ is finitely presented.  Then by coherence of $\mathcal{O}(K)(\ast)$, we can resolve $M$ by a complex of finite free $\mathcal{O}(K)(\ast)$-modules.  It follows that $\mathcal{O}(K)\otimes^L_{\mathcal{O}(K)(\ast)} M$ is calculated by a complex where each term is a finite direct sum of copies of $\mathcal{O}(K)$.  These terms are then DNF spaces, and by construction this complex has the property that, after applying $(-)(\ast)$, it is acyclic in positive degrees.  We want to see that the same holds also before applying $(-)(\ast)$.  But this is a consequence of the following lemma.
\end{proof}

\begin{lemma}
Suppose $C_\bullet$ is a complex of DNF spaces.  If $C_\bullet(\ast)$ is exact, then $C_\bullet$ is exact.
\end{lemma}
\begin{proof}
Considering the kernels of the differentials, which are also DNF spaces by Lemma \ref{DNFlemmas}, it suffices to show that if $f:V\rightarrow W$ is a map of DNF spaces which is surjective on applying $-(\ast)$, then $f$ is surjective.  For that, we will only use that $W=\bigcup_i W_i$ is a filtered union of Banach spaces and $V=\bigcup_n K_n$ is a countable union of compact Hausdorff spaces.  Indeed, for $n$ and $m$, consider the closed subspace $f(K_n)\cap W_m \subset W_m$.  Fixing $m$, the union over all $n$ of these subsets is set-theoretically equal to $W_m$ by hypothesis.  By the Baire category theorem, it follows that one of these subsets has non-empty interior.  Thus there exists an open ball in $W_m$ which is contained in some $f(K_n)$.  Translating and scaling, we see that every open ball in $W_m$ is contained in the image of a compact subset of $V$.  Thus, for all $m$, every compact subset of $W_m$ is contained in the image of a compact subset of $V$.  Every compact subset of $W$ is contained in $W_m$ for some $m$, so this gives the conclusion.\end{proof}

\begin{remark}
Recall that the basic nuclear objects in $\mathcal{D}(\operatorname{Liq}_p)$ are exactly those which can be represented by a complex of DNF spaces.  Thus, another interpretation of the above lemma is that the functor
$$X\mapsto X(\ast)$$
from basic nuclears to $\mathcal{D}(\mathbb{C})$ is conservative, i.e.\@ detects isomorphisms.
\end{remark}

Now we turn to our main task, a definition of complex analytic space.  Classically, the definition is based on the formalism of locally ringed spaces.  More precisely, a complex analytic space is a locally ringed space $(X,\mathcal{O}_X)$ which is locally isomorphic to one of the following form: take the following data:
\begin{enumerate}
\item An open polydisk $\mathbb{D}^n\subset\mathbb{C}^n$, which may as well be the unit polydisk;
\item Finitely many holomorphic functions $f_1,\ldots,f_m\in\mathcal{O}(\mathbb{D}^n)$,
\end{enumerate}
and produce the following locally ringed space:
\begin{enumerate}
\item $X = Z(f_1,\ldots,f_m)\subset \mathbb{D}^n$;
\item $\mathcal{O}_X = \mathcal{O}_{\mathbb{D}^n}/(f_1,\ldots, f_m)$, a sheaf of rings on $\mathbb{D}^n$ supported on $X$ which therefore corresponds to a sheaf of rings on $X$.
\end{enumerate}

The definition we give will be more general, in two different directions:

\begin{enumerate}
\item We want to allow (certain) compact Stein subsets $K\subset\mathbb{C}^n$ also to count as complex analytic spaces, meaning in some sense we want to allow our complex analytic spaces to have ``boundaries''.  This makes for more convenient reduction to algebra, as the rings of functions $\mathcal{O}(K)$ have nice finiteness properties like those proved in the previous lectures.
\item We want to allow derived structure sheaves as well, in order to have clean formulations of base-change theorems.
\end{enumerate}

Except for this derived nonsense and the fact that we are in the complex analytic context, we essentially follow the set-up of Grosse-Kl\"onne's theory of \emph{dagger spaces}, \cite{grosse2000rigid}, in terms of which local models we allow.

\begin{definition}
A commutative algebra $A$ in $\mathcal{D}_{\geq 0}(\operatorname{Liq}_p)$ is called \emph{affinoid} if there exists an $n\geq 0$ and a homomorphism
$$\mathcal{O}(\overline{\mathbb{D}}^n)\rightarrow A$$
of commutative algebras in $\mathcal{D}_{\geq 0}(\operatorname{Liq}_p)$, such that $A \in \mathcal{D}_{pc}(\mathcal{O}(\overline{\mathbb{D}}^n))$ as an underlying $\mathcal{O}(\overline{\mathbb{D}}^n)$-module in $\mathcal{D}(\operatorname{Liq}_p)$.
\end{definition}

In other words, by the noetherianness of $\mathcal{O}(\overline{\mathbb{D}}^n)$ proved in the appendix to the previous lecture, an affinoid algebra is a connective algebra over the ring $\mathcal{O}(\overline{\mathbb{D}}^n)$ of holomoprhic fuctions on some closed polydisk, such that each homology group is finitely generated as a module over $\mathcal{O}(\overline{\mathbb{D}}^n)$.

The following will be useful for reducing some questions about $A$ to $H_0A$.

\begin{lemma}\label{replacebyH0}
Let $d\geq 0$, let $B$ be an idempotent $\mathbb{C}[x_1,\ldots,x_d]$-algebra in $\mathcal{D}_{\geq 0}(\operatorname{Liq}_p)$, and let $A$ be a commutative algebra object in $\mathcal{D}_{\geq 0}(\operatorname{Liq}_p)$.  Then the map
$$\operatorname{Hom}(B,A)\rightarrow \operatorname{Hom}(B, H_0A)$$
is a bijection on $\pi_0$.  Here we take homomorphisms of commutative algebra objects in $\mathcal{D}_{\geq 0}(\operatorname{Liq}_p)$.
\end{lemma}
\begin{proof}
Because of idempotency, $\operatorname{Hom}(\mathcal{O}(K),-)$ is a full sub-anima of $\operatorname{Hom}(\mathbb{C}[x_1,\ldots,x_d],-)$.  But the latter identifies with $(-)^d$ since $\mathbb{C}[x_1,\ldots,x_d]$ is free on $d$ generators.\footnote{Here we implicitly use that $\mathbb{C}$ has characteristic zero: otherwise the free commutative algebra in $\mathcal{D}_{\geq 0}(\operatorname{Liq}_p)$ would have contributions from the derived functors of $\Sigma_n$-coinvariants.}  As $A^d\rightarrow H_0A^d$ is evidently a bijection on $\pi_0$, we reduce to showing the following: given a map $f:\mathbb{C}[x_1,\ldots,x_d]\rightarrow A$, we have that $f$ factors through $B$ if and only if the composition $\mathbb{C}[x_1,\ldots,x_d]\rightarrow A\rightarrow H_0A$ factors through $B$.  ``Only if'' is obvious.  For ``if'', by idempotency it suffices to show that $A$, viewed as an $\mathbb{C}[x_1,\ldots,x_d]$-module, lies in the full subcategory of $B$-modules.  Again by idempotency the collection of $B$-modules is closed under limits and colimits, so we reduce to the claim that $H_iA$ is an $B$-module for any $i$, which follows from the hypothesis as $H_iA$ is a module over $H_0A$.
\end{proof}

For our definition of complex analytic space, instead of using the standard formalism of locally ringed space, we will use the formalism of \emph{categorified locale} discussed in Lecture VII.  Recall that a categorified locale is a triple
$$(X, C, f:\mathcal{S}(C)\rightarrow X)$$
where $X$ is a locale, $C$ is a cocomplete closed symmetric monoidal stable $\infty$-category, and $f$ is a map of locales.  Essentially, the idea is this: the locale $\mathcal{S}(C)$ is much too huge to work with in its entirety, so we use a well-chosen map $f$ to a simpler locale $X$ in order to single out part of the ``geometry'' of $C$.

Our next task will be to assign a canonical categorified locale to any affinoid algebra $A$, where the category $C$ will be
$$C_A=\operatorname{Mod}_A(\mathcal{D}(\operatorname{Liq}_p)),$$
and the locale will be the topological space
$$X_A = \operatorname{Hom}(A(\ast),\mathbb{C}) = \mathcal{M}^{\mathrm{Berk}}((H_0A)(\ast)),$$
following the pattern from Lecture V.

We will denote the locale $\mathcal{S}(C_A)$ also just by $\mathcal{S}(A)$.  Recall from Lecture V that for a liquid ring $A$ and $f\in A(\ast)$, there is a native notion of ``$f$ being analytic'': it means that
$$\mathcal{S}(A)=\bigcup_r \{ |f|<r\},$$
where the open subsets $\{ |f|<r\}$ are base-changed from defining idempotent algebras over $\mathbb{C}[T]$ via the homomorphism $\mathbb{C}[T]\rightarrow A$ classifying $f$.  It is actually equivalent to say that there exists an $c>0$ for which
$$\mathcal{S}(A)= \{ |f|\leq c\}.$$
Indeed, since the unit object in $C_A$ is compact, the locale $\mathcal{S}(A)$ is quasicompact.

For affinoid rings, it turns out that every element is analytic.

\begin{lemma}\label{affinoidsareanalytic}
Let $A$ be an affinoid algebra, and let $f\in A(\ast)$.  Then
$$\mathcal{S}(A)=\bigcup_r \{ |f|<r\},$$
or equivalently there exists $c>0$ such that
$$\mathcal{S}(A)= \{ |f|\leq c\}.$$
\end{lemma}
\begin{proof}
By Lemma \ref{replacebyH0}, we can assume $A=H_0A$.  Then by definition $A$ is a finite algebra over some $\mathcal{O}(\overline{\mathbb{D}}^n)$.  In particular, by the noetherian property, any $f\in A$ is the root of a monic polynomial with coefficients in $\mathcal{O}(\overline{\mathbb{D}}^n)$.  By Lemma \ref{integralanalytic}, this reduces us to $A=\mathcal{O}(\overline{\mathbb{D}}^n)$ itself.  But if we let
$$f=\sum a_{i_1,\ldots,i_n} z_1^{i_1}\ldots z_n^{i_n},$$
then taking $c = \sum |a_{i_1,\ldots,i_n}|$ we can factor
$$\mathbb{C}[T]\overset{T\mapsto f}{\longrightarrow} \mathcal{O}(\overline{\mathbb{D}}^n)$$
through $\mathcal{O}(c\cdot \overline{\mathbb{D}})$ by substitution of power series.
\end{proof}

As a corollary, we note the following, showing in particular that affinoid algebras in degree $0$ are also the same thing as quotients of rings of the form $\mathcal{O}(\overline{\mathbb{D}}^n)$ by ideals:

\begin{lemma}\label{affinoidsarezariskiclosed}
Suppose $A$ is affinoid.  Then there exists an $n\geq 0$ and a homomorphism
$$\mathcal{O}(\overline{\mathbb{D}}^n)\rightarrow A$$
such that $A \in \mathcal{D}_{pc}(\mathcal{O}(\overline{\mathbb{D}}^n))$ and the induced map
$$\mathcal{O}(\overline{\mathbb{D}}^n)\twoheadrightarrow H_0A$$
is surjective.
\end{lemma}
\begin{proof}
We start with an arbitrary $\mathcal{O}(\overline{\mathbb{D}}^n)\rightarrow A$ such that $A \in \mathcal{D}_{pc}(\mathcal{O}(\overline{\mathbb{D}}^n))$.  Then $H_0A$ is finite over $\mathcal{O}(\overline{\mathbb{D}}^n)$, so we can choose finitely many generators $f_1,\ldots,f_k$.  By the previous lemma, each of these generators is analytic, so we can add them to the presentation, getting $\mathcal{O}(\overline{\mathbb{D}}^{n+k})\rightarrow A$ surjective on $H_0$, as desired.
\end{proof}

Now we state the result giving the categorified locale structure on an affinoid.

\begin{theorem}\label{affinoidgiveslocale}
Suppose $A$ is an affinoid algebra.  Then:
\begin{enumerate}
\item  There is a unique map of locales
$$\pi_A:\mathcal{S}(A)\rightarrow X_A:=\operatorname{Hom}_\mathbb{C}(A(\ast),\mathbb{C})= \operatorname{Hom}_\mathbb{C}(H_0A(\ast),\mathbb{C})\subset \prod_{f\in (H_0A)(\ast)} \mathbb{C}$$
such that for $f\in (H_0A)(\ast)$, the map to the $f$-factor in the target is given by
$$\mathcal{S}(A)\rightarrow \mathcal{S}(\mathbb{C}[T],\mathbb{C}[T])\rightarrow\mathbb{C},$$
where the first map is induced by the homomorphism $\mathbb{C}[T]\rightarrow A$ classifying $f$ and the second map is from Lecture V.
\item If we write $H_0A=\mathcal{O}(\overline{\mathbb{D}}^n)/I$ for an ideal $I$, then $X_A=Z(I)\subset \overline{\mathbb{D}}^n$, the common zero set of $I$ in the closed unit polydisk with its standard Euclidean topology.  In particular, $X_A$ is a compact Hausdorff space.
\item The map $\pi_A$ is ``surjective'' in the sense that for two closed subsets $Z,W$ of $X_A$, we have $Z\subset W\Leftrightarrow \pi_A^{-1}Z\subset \pi_A^{-1}W$.
\end{enumerate}
\end{theorem}
\begin{proof}
Consider part (1).  First, the map $\mathcal{S}(A)\rightarrow \mathcal{S}(\mathbb{C}[T])$ induced by $f$ does indeed land in  $\mathcal{S}(\mathbb{C}[T],\mathbb{C}[T])$: this is a restatement of Lemma \ref{affinoidsareanalytic}.  Then the rest follows as in Lecture V: we need to see that the map $\mathcal{S}(A)\rightarrow \prod_{f\in (H_0A)(\ast)} \mathbb{C}$ lands inside the closed subset $\operatorname{Hom}(A(\ast),\mathbb{C})$, but all the required relations come from polynomial algebras mapping to $A$, and then it follows from the material in Lecture V (plus Lemma \ref{replacebyH0} to reduce to $H_0A$).

For part (2), the claim set-theoretically follows from the description of maximal ideals in $\mathcal{O}(\overline{\mathbb{D}}^n)$ in the previous lecture.  This resulting bijection $X_A\rightarrow Z(I)$ is continuous because it corresponds to projection onto the $\mathbb{C}$-factors given by the coordinate functions $x_1,\ldots,x_n\in H_0A(\ast)$.  The inverse is also continuous because every $f\in (H_0A)(\ast)$ does indeed evaluate to a continuous function $Z(I)\rightarrow \mathbb{C}$.

For part (3), the direction $\Rightarrow$ holds by definition.  For $\Leftarrow$, let $\mathcal{O}(Z)$ be the idempotent algebra in $C_A$ corresponding to $Z$, and $\mathcal{O}(W)$ the idempotent algebra corresponding to $W$.  Thus our assumption is that there is a map $\mathcal{O}(W)\rightarrow \mathcal{O}(Z)$.  If $x\not\in W$, then $\mathcal{O}(\{x\})\otimes_A \mathcal{O}(W)=0$, hence $\mathcal{O}(\{x\})\otimes_A \mathcal{O}(Z)=0$.  But if $x\in Z$ then we have a homomorphism $\mathcal{O}(Z)\rightarrow \mathcal{O}(\{x\})$.  To get a contradiction and thereby finish the proof, we therefore need to check that $\mathcal{O}(\{x\})\neq 0$ for any $x\in X_A$.  But indeed it has a homomorphism to $\mathbb{C}$ given by evaluation at $x$.
\end{proof}

We have just assigned a categorical locale $(X_A,C_A,\pi_A:\mathcal{S}(A)\rightarrow X_A)$ to each affinoid $A$.  In fact, we will generally not view this as a bare categorified locale, but rather as a categorified locale over the base categorified locale
$$(\ast,\mathcal{D}(\operatorname{Liq}_p)).$$
With this extra structure, we have:

\begin{corollary}
The assignment $A\mapsto (X_A,C_A,\pi_A)$ defines a fully faithful contravariant functor from the $\infty$-category of affinoid algebras to the $\infty$-category of categorified locales over $(\ast, \mathcal{D}(\operatorname{Liq}_p))$.
\end{corollary}
\begin{proof}
The first bit of data in a map of such categorified locales $(X_B,C_B,\pi_B)\rightarrow (X_A,C_A,\pi_A)$ is a symmetric monoidal colimit-preserving functor
$$C_A\rightarrow C_B.$$
Taking into account our base categorified locale, this map should be linear over the base category $\mathcal{D}(\operatorname{Liq}_p)$.  Certainly, given a homomorphism $A\rightarrow B$ we get such a functor by relative tensor product.  Conversely, considering endomorphisms of the unit objects we get a homomorphism $A\rightarrow B$ from any such functor $C_A\rightarrow C_B$.  These two associations are mutually inverse by \cite[Corollary 4.8.5.21]{lurie2017higher}.

Moreover, the construction $A\mapsto (X_A,C_A,\pi_A)$ is evidently functorial.  So to conclude it suffices to show that the map of locales $X_B\rightarrow X_A$ completing the data of a map of categorified locales is uniquely determined by the functor $C_A\rightarrow C_B$.  But this follows from part (3) of the theorem.
\end{proof}

\begin{remark}
It seems that in the categorified locale picture, in contrast to the locally ringed space picture, there is automatically a tighter relation between the algebra of the category and the topology of the locale.  This means that there is no analog of the usual technical annoyance that one needs to impose an at first sight slightly obscure ``locality'' condition on the straightforward notion of a map of ringed spaces when defining a map of schemes.
\end{remark}

\begin{example}
Given two closed polydisks $\overline{\mathbb{D}}^n$ and $\overline{\mathbb{D}}^m$, maps of the corresponding categorified locales are the same as holomorphic maps $\overline{\mathbb{D}}^n\rightarrow \overline{\mathbb{D}}^m$ in the usual sense, i.e.\ we give an $m$-tuple of elements of $\mathcal{O}(\overline{\mathbb{D}}^n)$ such that the induced map
$$\overline{\mathbb{D}}^n\rightarrow \mathbb{C}^m$$
lands inside $\overline{\mathbb{D}}^m$.
\end{example}

Let us also treat fiber products.

\begin{proposition}
If $B\leftarrow A\rightarrow C$ are maps of affinoid algebras as indicated, then $B\otimes_AC$ is also affinoid, where we use the derived relative tensor product in $p$-liquid vector spaces over $\mathbb{C}$.  For the associated topological spaces we have
$$X_{B\otimes_AC} = X_B\times_{X_A}X_{C},$$
and we also get a fiber product in the $\infty$-category of categorified locales.
\end{proposition}
\begin{proof}
Let's first consider the claim about affinoids being closed under relative tensor products.  Writing $A$ as pseudocompact over some $\mathcal{O}(\overline{\mathbb{D}}^n)$ and using that a composition of pseudocompact maps is pseudocompact, we reduce to $A=\mathcal{O}(\overline{\mathbb{D}}^n)$.  Then by idempotency it suffices to show the analogous claim with $A=\mathbb{C}[x_1,\ldots,x_n]$.  But then if we choose presentations for $B$ and $C$ as in Lemma \ref{affinoidsarezariskiclosed}, meaning we get pseudocompact maps $\mathcal{O}(\overline{\mathbb{D}}^m)\rightarrow B$ and $\mathcal{O}(\overline{\mathbb{D}}^p)\rightarrow B$ which are surjective on $H_0$, then we see that the maps $\mathbb{C}[x_1,\ldots,x_n]\rightarrow B$ and $\mathbb{C}[x_1,\ldots,x_n]\rightarrow C$ lift to maps $\mathbb{C}[x_1,\ldots,x_n]\rightarrow \mathcal{O}(\overline{\mathbb{D}}^m)$ and $\mathbb{C}[x_1,\ldots,x_n]\rightarrow \mathcal{O}(\overline{\mathbb{D}}^p)$, letting us reduce to showing that
$$\mathcal{O}(\overline{\mathbb{D}}^m)\otimes_{\mathbb{C}[x_1,\ldots,x_n]}\mathcal{O}(\overline{\mathbb{D}}^p)$$
is affinoid.  But we can rewrite this as
$$\left(\mathcal{O}(\overline{\mathbb{D}}^m) \otimes \mathcal{O}(\overline{\mathbb{D}}^p)\right)\otimes_{\mathbb{C}[x_1,\ldots,x_n;x'_1,\ldots x'_n]} \mathbb{C}[x_1,\ldots,x_n],$$
and then we conclude using pseudocompactness (in fact compactness) of the surjection
$$\mathbb{C}[x_1,\ldots,x_n;x'_1,\ldots x'_n]\rightarrow \mathbb{C}[x_1,\ldots,x_n].$$
For the claim about topological spaces, recall that
$$X_A=\operatorname{Hom}(H_0A(\ast),\mathbb{C})$$
with a natural topology which we proved to be compact Hausdorff in Theorem \ref{affinoidgiveslocale}.  Because a bijective continuous map of compact Hausdorff spaces is a homeomorphism, we can check the claim set-theoretically.  But then using the universal property of relative tensor product, it suffices to show that
$$\operatorname{Hom}(H_0A(\ast),\mathbb{C})=\operatorname{Hom}(H_0A,\mathbb{C}),$$
i.e.\ the homomorphisms of condensed $\mathbb{C}$-algebras $H_0A\rightarrow \mathbb{C}$ are in bijection with the homomorphisms of discrete $\mathbb{C}$-algebras $H_0A(\ast)\rightarrow \mathbb{C}$ via applying $(-)(\ast)$.  The map is injective because $H_0A$ is quasiseparated (Lemma \ref{coherentisDNF}), and the map is surjective because we classified all the homomorphisms $H_0A(\ast)\rightarrow \mathbb{C}$: they all come from evaluations at points of $Z(I)$, and these evaluations also define homomorphisms $H_0A\rightarrow\mathbb{C}$. 
The final claim about pullbacks of categorical locales follows by combining the two previous claims.
\end{proof}

Using this one can see that the notion of affinoid is rather robust.  In fact, though the definition of affinoid is based only on one kind of compact Stein, namely closed polydisks, there are many other compact Steins which also turn out to be affinoid.  For example, if $A$ is affinoid and $f\in A$, then for all $r>0$ the closed sub-categorified locales of $X_A$ defined by
$$\{|f|\leq r\},\{|f|\geq r\}$$
are also (the categorified locales associated to) affinoids.  Indeed, we can get to $\{|f|\leq r\}$ by first taking the product with $r\cdot\overline{\mathbb{D}}$ (i.e., on the level of affinoids, adjoining a variable $T$ with $|T|\leq r$), and then passing to the ``Zariski-closed subset'' given on the level of affinoids by modding out $T-f$; and for the second one we can take the product with $r^{-1}\cdot\overline{\mathbb{D}}$ and then mod out by $Tf-1$.   Then also, if $g$ is invertible we can impose
$$|f|\leq |g|$$
and still remain affinoid, by imposing $|f/g|\leq 1$.  By the above result on fiber products, we can also take finite intersections of affinoids and we still get affinoids, so we can also simultaneously impose finitely many such inequalities.  Then when we do this, we get more functions to play with, and so on.

\begin{warning}
Be careful to distinguish two ways of ``setting $f$ equal to $0$'', given an affinoid $A$ and $f\in A$.  One is by considering the closed subset $\{|f|\leq 0\}$ in the locale picture, and the other is by modding out by $f$ on the level of affinoids, passing from $A$ to $A/(f)$.  The first corresponds to the idempotent $A$-algebra given by ``germs of functions in neighborhood of $Z(f)$''.  The second does \emph{not} correspond to an idempotent algebra: although $A/(f)$ is of course idempotent over $A$ with the underived tensor product, it is not idempotent with the derived tensor product, which is the one we have to use to get the locale picture to work.

We've already gotten used to the idea that, in this locale picture, Zariski opens are actually closed; now we have to get used to the idea that Zariski closed subsets are not really closed subsets (nor are they open subsets).
\end{warning}

Finally, let us give the definition of complex analytic space we'll use.

\begin{definition}
A (generalized) complex analytic space is a categorified locale over $(\ast,\mathcal{D}(\operatorname{Liq}_p))$ which is locally (in the sense of open subsets) isomorphic to an open subset of the categorified locale over  $(\ast,\mathcal{D}(\operatorname{Liq}_p))$ associated to an affinoid $A$.
\end{definition}

The affinoids in degree zero are exactly those of the form
$$A=\mathcal{O}(\overline{\mathbb{D}}^n)/I$$
for some ideal $I$ (automatically finitely generated, by noetherianness).  For an example of an open subset of the corresponding categorified locale, we can take the intersection with the \emph{open} polydisk.  If we only consider those generalized complex analytic spaces which are locally isomorphic to such a local model, we get an equivalent category to the usual category of complex analytic spaces.  (This local model is slightly less general than the usual local model for complex analytic spaces, because we require that the generators of our ideal actually come by restriction from holomorphic functions on the \emph{closed} polydisk.  But once we globalize this doesn't matter, because any  open polydisk $D$ is covered by smaller open polydisks whose closures still lie inside $D$.)\\

\textbf{Exercise 1.}  Let $A$ be an affinoid and $x\in X_A$ be a point of the interior.  Show that, potentially after shrinking $A$ to an affinoid neighborhood of $x$, we can find an $n\geq 0$ and a pseudocoherent homomorphism $\varphi:\mathcal{O}(\overline{\mathbb{D}}^n)\rightarrow A$ (as in the definition of affinoid) such that $\varphi$ is \emph{injective}.  (This is an analog of the Noether normalization theorem.)\\

\textbf{Exercise 2.} Let $A$ be an affinoid and $x\in X_A$ be a point of the interior.  Using the previous exercise, show that the topological dimension of $X_A$ at $x$ identifies with the Krull dimension of the ring $H_0(A_x)(\ast)$, where $A_x$ is the idempotent algebra corresponding to the closed subset $\{x\}\subset X_A$.\newpage

\section{Lecture XII: Proper pushforward}

Picking up from the previous lecture, we discuss the notion of (generalized) complex analytic space.  Recall the definition: a complex analytic space is a categorified locale over $(\ast,\mathcal{D}(\operatorname{Liq}_p))$ which is locally isomorphic to an open subset of the categorified locale $\mathcal M(A)$ over  $(\ast,\mathcal{D}(\operatorname{Liq}_p))$ associated to an affinoid $A$.  Recall that this categorified locale is the triple
$$(\mathcal M(A),C_A,\pi_A)$$
where $C_A=\operatorname{Mod}_A(\mathcal{D}(\operatorname{Liq}_p))$, $\operatorname{Spa}(A) = \operatorname{Hom}_{\mathbb{C}}(A(\ast),\mathbb{C})$ ($=Z(I)\subset \overline{\mathbb{D}^n}$ if $H_0A = \mathcal{O}(\overline{\mathbb{D}^n})/I$), and
$$\pi_A:\mathcal{S}(C_A)\rightarrow \mathcal M(A)$$
is determined by the fact that if $f\in A(\ast)$, then the open subsets $\{|f|< 1\},\{|f|>1\}\subset \mathcal M(A)$ pull back to the corresponding open subsets of $\mathcal{S}(C_A)$, defined by base-change from the open subsets $\{|T|< 1\},\{|T|>1\}$ in $\mathcal{S}(\mathbb{C}[T])$ constructed in Lecture V.

We will generally abuse notation and refer to a categorified locale $(X,C,\pi)$ just by $X$, with the rest of the data being implicit.  We will also use the phrase ``affinoid'' to refer to a categorified locale over $(\ast,\mathcal{D}(\operatorname{Liq}_p))$ isomorphic to one of the form $\mathcal M(A)$.  

In general, we can study complex analytic spaces by reduction to affinoids via a two step process:

\begin{enumerate}
\item First, reduce a general complex analytic space to a quasi-affinoid complex analytic space, i.e.\ one which is isomorphic to an open subset of an affinoid.  This is done by working locally.

\item Second, given a quasi-affinoid $U$, we can reduce to affinoids in either of two ways: first, choose an ambient affinoid $X$ in which $U$ is an open subset; or second, use the fact that a quasi-affinoid $U$ is indeed locally affinoid in the following sense: there is a closed cover of $U$ by affinoids which is refined by an open cover.  (This follows from the fact that closed polydisks give a neighborhood basis of any point $x\in \operatorname{Spa}(A)$.)
\end{enumerate}

Let's give an example of this kind of reduction in action.

\begin{proposition}
Let $X\overset{f}{\rightarrow}Z\overset{g}{\leftarrow}Y$ be maps of complex analytic spaces as indicated.  Then the pullback $X\times_ZY$ exists in the $\infty$-category of categorified locales and is itself a complex analytic space, in particular giving a pullback in the $\infty$-category of complex analytic spaces.  The forgetful functor to topological spaces also preserves this pullback.  
\end{proposition}
\begin{proof}
In the case where $X,Y,Z$ are all affinoid, we saw this in the previous lecture; in that case $X\times_ZY$ is also affinoid.  Now suppose $X,Y,Z$ are all quasi-affinoid.  Putting $Z$ inside an affinoid, we can reduce to where $Z$ is an affinoid.  Finding a cover of $X,Y$ by closed sub-affinoids which is refined by an open cover, we reduce to where $X,Y$ are open subsets of affinoids $X',Y'$ which still map to $Z$.  Then the pullback $X\times_ZY$ is given by the obvious open subset of $X'\times_Z Y'$.
In the general case, we work locally on $Z$, then on $X$ and $Y$ to reduce to the quasi-affinoid case.
\end{proof}

\begin{warning}
It is special to this complex analytic situation that the pullback in complex analytic spaces is also a pullback in categorified locales (and on underlying topological spaces).  In nonarchimedean geometry or scheme theory, this does not hold: the product topological space of the spectrum of two $k$-algebras is generally not fine enough to give the spectrum of the tensor product.  The essential reason for this special behavior of the complex analytic situation is the Gelfand-Mazur theorem, showing that ``points'' are the same as homomorphisms to $\mathbb{C}$.
\end{warning}

We can also give an abstract formulation of this kind of descent procedure for reducing arbitrary complex analytic spaces to affinoids.

\begin{proposition}
Consider the following two sites:
\begin{enumerate}
\item The $\infty$-category of all complex analytic spaces, equipped with the topology of open covers;
\item The $\infty$-category of affinoid complex analytic spaces, equipped with the topology of (finite) closed covers which can be refined by open covers.
\end{enumerate}
Then the associated $\infty$-toposes are equivalent, via restriction from sheaves on the first site to sheaves on the second site.

In other words, giving a presheaf on the category of all complex analytic spaces which satisfies descent for open covers is equivalent to giving a presheaf on the category of all affinoid complex analytic spaces which satisfies descent for those closed covers which can be refined by open covers.
\end{proposition}
\begin{proof}
Consider the intermediate site of all quasiaffinoids with the open cover topology.  By descent, the restriction of sheaves from the site (1) to this site is an equivalence of $\infty$-toposes.  On the other hand, inside the site of quasi-affinoids, the affinoids form a basis closed under fiber products, so the restriction functor on sheaves is an equivalence.
\end{proof}

A basic example is the following.  By definition, to every complex analytic space $X$ is attached its large $\infty$-category $C_X$ of ``derived liquid quasicoherent sheaves''.  For the formal reasons explained in Lecture V, this assignment $X\mapsto C_X$ is a sheaf of $\infty$-categories for the open cover topology.  It follows that this assignment $X\mapsto C_X$ is determined by its restriction to affinoids $\mathcal{M}(A)$, where we have
$$C_{\mathcal{M}(A)} = C_A = \operatorname{Mod}_A(\mathcal{D}(\operatorname{Liq}_p)),$$
with pullback functoriality coming from relative tensor product along maps $A\rightarrow B$.  Thus quasicoherent sheaves are determined on affinoids, where they are simply given by modules over the ring of functions.

We can also use this descent procedure to single out a nice full subcategory of $C_X$ consisting of ``derived coherent'' objects, which on affinoids gives the $\infty$-category $\mathcal{D}_{pc}(A)$ of bounded below complexes of $A(\ast)$-modules which can be resolved by finite free modules, or equivalently by noetherianness, where each homology group is finitely generated.  Recall from Lecture X that this assignment
$$A \mapsto \mathcal{D}_{pc}(A)\subset \operatorname{Mod}_A(\mathcal{D}(\operatorname{Liq}_p))$$
satisfies descent for the topology of finite covers of affinoids; in particular it satisfies descent for the topology of finite covers refined by open covers, so it globalizes to
$$X \mapsto \mathcal{D}_{pc}(X)$$
for complex analytic spaces $X$.  To be specific, $\mathcal{D}_{pc}(X)$ is the full subcategory of $C_X$ consisting of those $\mathcal{F}\in C_X$ such that $f^\ast \mathcal{F} \in \mathcal{D}_{pc}(A)$ for any map $f:\mathcal{M}(A)\rightarrow X$ from an affinoid.  (It is enough to check this condition on a collection of maps from affinoids which is refined by an open cover of $X$.)

Moreover, since the transition maps $A(\ast)\rightarrow B(\ast)$ are flat when $\mathcal{M}(B)\subset \mathcal{M}(A)$ is a closed sub-affinoid of the affinoid $X_A$ (\ref{thm:compactsteincoherent}), there is a t-structure on $\mathcal{D}_{pc}(X)$ defined by gluing the local t-structures on $\mathcal{D}_{pc}(A)$, and in particular we have an abelian category of coherent sheaves, given by those $\mathcal{F}\in\mathcal{D}_{pc}(X)$ which lie in degree $0$.

\begin{warning}
Although for $A$ affinoid the big $\infty$-category $C_A$ does have a natural $t$-structure induced from the $t$-structure on the derived category of $p$-liquid vector spaces, the transition maps $C_A\rightarrow C_B$ are \emph{not} $t$-exact in general for $A\rightarrow B$ corresponding to a closed sub-affinoid $\mathcal{M}(B)\subset \mathcal{M}(A)$; see the footnote in the proof of Lemma \ref{analyticaffinelemma} for the counterexample.  In particular, unlike with the situation of the smaller category $\mathcal{D}_{pc}(X)$, there is no way to get a t-structure on $C_X$ by descent from the affinoid case.  However, the functors $C_A\rightarrow C_B$ are $t$-bounded with a bound depending only on the dimension by the argument in the proof of Proposition \ref{prop:pseudocompactglues}, and this is a reasonable substitute for the t-structure in some situations.
\end{warning}

Next we discuss some properties of maps of complex analytic spaces.

\begin{definition}
Let $f:X\rightarrow Y$ be a map of complex analytic spaces.  We say that $f$ is
\begin{enumerate}
\item an \emph{open inclusion} if $f$ identifies $X$ with $(U,C_Y(U),\pi_Y\vert_{\pi_Y^{-1}U})$ for some open subset $U\subset Y$;
\item a \emph{closed inclusion} if $f$ identifies $X$ with $(Z,\operatorname{Mod}_\mathcal{A}(C_Y),\pi_Y\vert_{\pi_Y^{-1}Z})$ for some closed subset $Z\subset Y$ with corresponding idempotent algebra $\mathcal{A}$ in $C_Y$;
\item a \emph{Zariski closed immersion} if for every affinoid $A$ with a map $\mathcal{M}(A)\rightarrow Y$, the pullback $\mathcal{M}(A)\times_YX$ is represented by an affinoid $\mathcal{M}(B)$, and moreover the map $A\rightarrow B$ satisfies the following two conditions: one, it exhibits $B$ as an object in $D_{pc}(A)$, and two, the induced map $H_0A\rightarrow H_0B$ is surjective;
\item \emph{separated} if the map on underlying topological spaces is separated, i.e.\@ the diagonal $X\rightarrow X\times_YX$ is a closed inclusion of topological spaces;
\item \emph{proper} if the map on underlying topological spaces is proper;
\item \emph{smooth} if, locally on both source and target, $f$ is isomorphic to a projection map $Y\times\mathbb{D}^n\rightarrow Y$, where $\mathbb{D}^n$ is the open polydisk (open subset of the affinoid $\overline{\mathbb{D}^n}$ given by $|z_1|,\ldots,|z_n|<1$, where the $z_i$ are the coordinate functions);
\item \emph{boundaryless} if, locally on both source and target, $f$ is isomorphic to a Zariski closed immersion followed by a projection map $Y\times\mathbb{D}^n\rightarrow Y$, where $\mathbb{D}^n$ is the open polydisk.
\end{enumerate}
\end{definition}

Note that a boundaryless map $f:X\rightarrow\ast$ to point is simply the derived analog of a usual complex analytic space.

It is easy to see that all of these properties of maps are closed under base-change and composition, and satisfy descent for the open cover topology on the target.  From this and descent we can deduce the following classification of Zariski closed immersions as ``relative spectra'':

\begin{lemma}
Let $Y$ be a complex analytic space.  The $\infty$-category of Zariski closed immerions $i:X\rightarrow Y$ identifies with the opposite of the $\infty$-category of commutative algebra objects $B\in D_{pc}(Y)_{\geq 0}$ such that the unit map $\mathcal{O}_Y\rightarrow B$ is surjective on $H_0$.
\end{lemma}
\begin{proof}
By descent for $D_{pc}(Y)$ together with its t-structure, we reduce to case $Y=\mathcal{M}(A)$.  Then to prove the claim we only need to see that if $A\rightarrow B$ is such that $B\in \mathcal{D}_{pc}(A)$ and $H_0A\twoheadrightarrow H_0B$, then the same holds after base-change along any map of affinoids $A\rightarrow A'$.  But this is clear.\end{proof}

Next we discuss separated maps.  Recall that separatedness was defined by a condition on the underlying map of topological spaces.  The following lemma shows that this matches with the analog of the usual definition in scheme theory.

\begin{lemma}
Let $f:X\rightarrow Y$ be a map of complex analytic spaces.  Then $f$ is separated if and only if the diagonal $\Delta_f:X\rightarrow X\times_YX$ is a Zariski closed immersion.
\end{lemma}
\begin{proof}
Note that if $i$ is a Zariski closed immersion of complex analytic spaces, then $i$ is a closed inclusion on underlying topological spaces.  Indeed, we can work locally on the target to reduce to the affinoid case, when the claim is clear.  Since fiber products of complex analytic spaces commute with the forgetful functor to topological spaces, this shows that if $\Delta_f$ is Zariski closed then $f$ is separated.  Now suppose $f$ is separated.  To show that $\Delta_f$ is a Zariski closed immersion, we can work locally on $Y$ and therefore assume $Y$ is quasi-affinoid.  In this case, $\Delta_{Y}$ is a Zariski closed immersion, as one easily sees by reduction to an affinoid containing $Y$ as an open subset, then further reducing to $Y=\overline{\mathbb{D}^1}$ where the claim follows from
$$\mathcal{O}(\overline{\mathbb{D}^2})/(z_1-z_2) = \mathcal{O}(\overline{\mathbb{D}^1}).$$
Then to show $\Delta_f$ is a closed immersion, by standard diagrams it is enough to show that $\Delta_{X}$ is a closed immersion.

Therefore, it suffices to show that if the underlying topological space of $X$ is Hausdorff, then $\Delta_X:X\rightarrow X\times X$ is a Zariski closed immersion.  We check this locally on $X\times X$.  Let $(x_1,x_2)\in X\times X$.  If $x_1=x_2=x$, then choose a quasi-affinoid open neighborhood $U$ of $x$ and use $U\times U$ as an open neighborhood of $X\times X$.  The pullback of $\Delta_X$ is $\Delta_U$, which we've already seen to be a Zariski closed immersion in the previous paragraph.  If $x_1\neq x_2$, choose quasi-affinoid $U_1\ni x_1, U_2\ni x_2$ with $U_1\cap U_2=\emptyset$.  Then with $U_1\times U_2$ as a neighborhood of $(x_1,x_2)$, the pullback of $\Delta_X$ is $\emptyset\rightarrow U_1\times U_2$, also clearly Zariski closed.
\end{proof}

Recall that for a general complex analytic space, the reduction to affinoids goes through this somewhat annoying two-step process: work open-locally to reduce to quasi-affinoids, then cover a quasi-affinoid by affinoids with a closed cover (refined by an open cover) to reduce to affinoids.  In particular, the basic affinoids are only locally closed subspaces, not closed subspaces in general.  However, when $X\rightarrow\ast$ is separated (i.e.\@ $X$ is Hausdorff), basic affinoids are closed and we can do the descent in just one step:

\begin{lemma}\label{Hausdorffbasis}
Let $X$ be a Hausdorff complex analytic space.  Then:
\begin{enumerate}
\item The collection of affinoid closed subspaces of $X$ is closed under finite intersection (= fiber product).
\item For any $x\in X$, the affinoid closed neighborhoods of $x$ form a neighborhood basis for the topology at $x$.
\item If $K\subset X$ is a compact subset, then there is another compact subset $K'\subset X$ such that:
\begin{enumerate}
\item $K'$ contains an open neighborhood of $K$, or in symbols $K\Subset K'$;
\item $K'$ admits a finite cover by affinoid closed subspaces refined by an open cover.
\end{enumerate}
In particular such a $K'$ is itself a complex analytic space.
\end{enumerate}
\end{lemma}
\begin{proof}
Since $\Delta_X$ is a closed immersion, it is relatively affinoid, and (1) follows.  For (2), we know for general $X$ that the affinoid closed subsets of a quasi-affinoid open neighborhood of $X$ form a neighborhood basis; but these are also closed in $X$ because they are compact and $X$ is Hausdorff.  And (3) follows immediately from (2).
\end{proof}

Now we turn to our main goal: discussion of pushforward functors on these categories $C_X$ attached to a complex analytic space $X$.  First, a preliminary remark.  Let us denote by $\Gamma:C_X\rightarrow \mathcal{D}(\operatorname{Liq}_p)$ the functor represented by $\mathcal{O}_X\in C_X$.  By descent, we have
$$\Gamma(\mathcal{F}) = \varprojlim_{f:\mathcal{M}(A)\rightarrow X} \Gamma(f^\ast\mathcal{F});$$
where on an affinoid $\mathcal{M}(A)$ this global sections functor $\Gamma$ is just the forgetful functor
$$\operatorname{Mod}_A( \mathcal{D}(\operatorname{Liq}_p))\rightarrow  \mathcal{D}(\operatorname{Liq}_p).$$
In the limit above, it also suffices to restrict to those $f$ given by a closed immersion from an affinoid to some quasi-affinoid open subset of $X$, and if $X$ is Hausdorff, then thanks to the previous lemma it even suffices to restrict to just closed immersions from affinoids.

Now, coming back to general $X$, we can define a functor
$$sh:C_X\rightarrow \operatorname{Sh}(X; \mathcal{D}(\operatorname{Liq}_p))$$
by
$$(sh \mathcal{F})(U) = \Gamma(U;\mathcal{F}\mid_U).$$

\noindent For formal reasons, this promotes to
$$sh:C_X\rightarrow \operatorname{Mod}_{\mathcal{O}_X}(\operatorname{Sh}(X; \mathcal{D}(\operatorname{Liq}_p)),$$

\noindent where $\mathcal{O}_X$ is the image of the unit object under $sh$.  By Exercise 2 in Lecture VI, this functor is actually an equivalence when $X$ is Hausdorff. We will only need the following lemma.

\begin{lemma}\label{pushforwardlemma}
Let $X$ be a Hausdorff complex analytic space and $\mathcal{F}\in C_X$.  Then for a closed inclusion $f:\mathcal{M}(A)\rightarrow X$ from an affinoid, we have
$$\Gamma(f^\ast\mathcal{F}) = \varinjlim_{U\supset \mathcal{M}(A)} \Gamma(\mathcal{F}\mid_U) \in \operatorname{Mod}_A( \mathcal{D}(\operatorname{Liq}_p))$$
where the filtered colimit is over open neighborhoods of $\mathcal{M}(A)$ in $X$.  More generally, this holds with $\mathcal{M}(A)$ replaced by any compact subspace of $X$, replacing $\operatorname{Mod}_A$ by $\operatorname{Mod}_{\Gamma(\mathcal{O}_X)}$.
\end{lemma}
\begin{proof}
We claim that
$$ \varinjlim_{j:U\subset X, U\supset  \mathcal{M}(A)} j_\ast j^\ast \mathcal{F}\rightarrow f_\ast f^\ast \mathcal{F} $$
is an isomorphism in $C_X$.  Here the transition maps and comparison map come from the fact that the poset of locally closed subsets of $X$ embeds fully faithfully contravariantly in the poset of localizations of $C_X$, via $g_\ast g^\ast$ where $g$ is the inclusion of the locally closed subset, by the material in Lecture VI.

If this isomorphism claim is known, then we can deduce the claim of the lemma as follows: using Lemma \ref{Hausdorffbasis}, we can shrink $X$ to assume $X$ is a finite union of affinoids with affinoid intersections.  Then the global sections functor on $C_X$ commutes with colimits, being a finite limit of forgetful functors from module categories.  Whence the claim of the lemma by taking global sections.

Now, to prove the above isomorphism claim, we can replace the filtered colimit over open neighborhoods of $\mathcal{M}(A)$ by a filtered colimit over arbitrary locally closed subsets containing an open neighborhood of $\mathcal{M}(A)$, by cofinality of the former in the latter.  But then by a similar cofinality we can replace this with the filtered colimit over aribitrary compact neighborhoods of $\mathcal{M}(A)$.  The intersection of such compact neighborhoods reduces to $\mathcal{M}(A)$ itself, so for the corresponding idempotent algebras in $C_X$ we get a filtered colimit.  Since for closed subsets the localization $g_\ast g^\ast$ is just given by tensoring with the idempotent algebra, this proves the claim.
\end{proof}

Now let us discuss the pushforward functor for a general map of complex analytic spaces.

\begin{proposition}\label{describepushforward}
Let $f:X\rightarrow Y$ be a map of complex analytic spaces.  Then:
\begin{enumerate}
\item the pullback functor $f^\ast: C_Y\rightarrow C_X$ admits a right adjoint $f_\ast:C_X\rightarrow C_Y$;
\item the formation of $f_\ast$ commutes with base-change along open inclusions;
\item if $Y$ is Hausdorff and $i:\mathcal{M}(A)\rightarrow Y$ is a closed inclusion from an affinoid, then
$$i^\ast f_\ast\mathcal{F} = \varinjlim_{U\supset \mathcal{M}(A)} \Gamma(\mathcal{F}\mid_{f^{-1}U}) \in \operatorname{Mod}_A( \mathcal{D}(\operatorname{Liq}_p)),$$
where $U$ runs over open neighborhoods of $\mathcal{M}(A)$ in $X$.
\end{enumerate}
\end{proposition}
\begin{proof}
Since $f^\ast$ commutes with colimits, the existence of $f_\ast$ is automatic modulo set-theoretic issues.  In particular, if we work in the $\kappa$-condensed context, then $f_\ast$ exists.  Then (2) and (3) would give a formula for $f_\ast$ which in particular shows that $f_\ast$ is independent of $\kappa$ and therefore gives us existence of $f_\ast$ without the $\kappa$-bound.  Thus, for the purposes of proving all three points, we can just as well assume $f_\ast$ exists and prove (2) and (3) hold.

For (2), recall from Lecture VI that for an open inclusion $j$, the pullback $j^\ast$ has a left adjoint $j_\natural$ which commutes with arbitrary base-change.  Passing to right adjoints in this statement proves (2).  Then (3) follows from the previous lemma applied to $f_\ast\mathcal{F}$, plus the base-change claim (2).
\end{proof}

Now we can prove the first version of a proper base change theorem.

\begin{theorem}
Let $f:X\rightarrow Y$ be a proper map of complex analytic spaces.  Then the functor $f_\ast:C_X\rightarrow C_Y$ commutes with colimits and arbitrary base-change, and satisfies the projection formula: for $\mathcal{G}\in C_Y$ and $\mathcal{F}\in C_X$, we have
$$\mathcal{G}\otimes f_\ast\mathcal{F}\overset{\sim}{\rightarrow}f_\ast((f^\ast\mathcal{G})\otimes \mathcal{F}).$$
\end{theorem}
\begin{proof}
Working locally on $Y$, we can assume that $Y$, hence also $X$, is Hausdorff.  Then if $i:\mathcal{M}(A)\rightarrow Y$ is a closed inclusion from an affinoid, the open neighborhoods $f^{-1}U$ of $f^{-1}\mathcal{M}(A)$ in $X$ are cofinal in all open neighborhoods of $f^{-1}\mathcal{M}(A)$, by properness of $f$.  Then combining Lemma \ref{pushforwardlemma} and Proposition \ref{describepushforward} we deduce that $f_\ast$ commutes with $i^\ast$.  This in turn reduces us to the case where $Y=\mathcal{M}(A)$ is affinoid; similarly, for checking the commutation with base-change, we only need to consider base-change along maps of affinoids.

But when $Y$ is affinoid, it is compact, so $X$ is also compact, hence a finite union of affinoids with affinoid intersection by Lemma \ref{Hausdorffbasis}.  Thus the pushforward from $X$ to $Y$ is a finite limit of pushforwards from affinoids mapping to $Y$, letting us reduce to the case where $X$ is also affinoid.  But then our categories are just categories of modules and pushforwards are forgetful functors, and the claims follow from obvious base-change properties in that algebraic context.
\end{proof}

Our next goal will be to ``repair'' the failure of the proper base-change theorem to hold outside the proper case, by introducing a modification to the functor $f_\ast$, denoted $f_!:C_X\rightarrow C_Y$, for which the conclusion of proper base-change does hold (and with $f_!=f_\ast$ for $f$ proper).  Although in principle we should define such a $f_!$ for any separated map $f:X\rightarrow Y$ (or indeed for any map whatsoever, as long as we don't ask for a natural transformation $f_!\rightarrow f_\ast$), for simplicity we will only consider the case where both $X$ and $Y$ are Hausdorff (in which case any map $f:X\rightarrow Y$ is separated).

The construction of $f_!$ will be by formal extension from the full subcategory of \emph{compactly supported} objects $\mathcal{F}\in C_X$.  To define this condition, note that for $\mathcal{F}\in C_X$, the set of open subsets $U\subset X$ for which $\mathcal{F}|_U=0$ is closed under arbitrary unions, by the sheaf property.  Hence there is a maximal $U$ for which $\mathcal{F}|_U=0$.  The complement of this $U$ is a closed subset $Z=\operatorname{Supp}\mathcal{F}$ of $X$ called the \emph{support} of $\mathcal{F}$.  If $i:Z\rightarrow X$ denotes the closed inclusion, then it follows that $\mathcal{F}$ lies in the essential image of the fully faithful functor $i_\ast$.

\begin{definition}
Let $X$ be a Hausdorff complex analytic space.  We say that an $\mathcal{F}\in C_X$ is \emph{compactly supported} if $\operatorname{Supp}\mathcal{F}$ is compact; or equivalently, if $\mathcal{F}$ is pushed forward from a compact subset.  By Lemma \ref{Hausdorffbasis}, this is also equivalent to saying that $\mathcal{F}$ is pushed forward from a compact closed sub-analytic space, a finite union of affinoids.
\end{definition}

The first thing to note is the following.

\begin{lemma}
The base-change and projection formula for $f_\ast$ from the proper base-change theorem hold for arbitrary maps $f:X\rightarrow Y$ of complex analytic spaces, provided we restrict to the full subcategory of $C_X$ consisting of those objects whose support is proper over $Y$.  (In particular, if $Y$ is Hausdorff then they hold on compactly supported objects of $C_X$.)  Similarly, the claim about $f_\ast$ preserving colimits holds if all the terms in the colimit diagram have support in a common subset of $X$ which is proper over $Y$.
\end{lemma}
\begin{proof}
This follows by the same argument used to prove the proper base-change theorem.  Namely, using the tube lemma we reduce to a base affinoid.  Then the support is compact, and we invoke Lemma \ref{Hausdorffbasis} to place our compact subset inside a finite union of affinoids with affinoid intersections to conclude.
\end{proof}

The technical result we have to prove to get a good $f_!$ functor by extension from the compactly supported case is that every $\mathcal{F}\in C_X$ is canonically approximated by compactly supported sheaves.

\begin{lemma}\label{LKE}
Let $\mathcal{F}\in C_X$.  Then the natural map
$$\varinjlim_{\mathcal{G}\rightarrow \mathcal{F}}\mathcal{G}\overset{\sim}{\rightarrow} \mathcal{F}$$
is an isomorphism, where the colimit runs over all compactly supported $\mathcal{G}$ with a map to $ \mathcal{F}$.  Moreover, this colimit is filtered, and the Ind-system giving the diagram on the left is isomorphic both to the Ind-system giving
$$\varinjlim_{i:K\subset X} i_\ast i^!\mathcal{F},$$
where $K$ runs over all compact subsets of $X$, and to the Ind-system giving
$$\varinjlim_{i:V\subset X}\operatorname{Fib}(\mathcal{F}\rightarrow i_\ast i^\ast\mathcal{F}),$$
where $V$ runs over all closed subsets of $X$ whose complement is contained in a compact subset.

All of these Ind-systems are, more precisely, isomorphic objects in the Ind-category of the full subcategory of compactly supported objects of $C_X$.
\end{lemma}
\begin{proof}
The indexing $\infty$-category for the first colimit is filtered because the collection of compactly supported objects is closed under finite colimits.  Note that the indexing poset for the second colimit maps to the indexing $\infty$-category for the first poset by
$$(i:K\subset X) \mapsto (i_\ast i^!\mathcal{F}\rightarrow \mathcal{F}),$$
where this is well-defined and functorial because $i_\ast i^!$ is the co-localization functor associated to $K\subset X$.  To prove the Ind-systems are isomorphic, it suffices to show cofinality for this functor between filtered indexing categories. But indeed, if $\mathcal{G}\rightarrow\mathcal{F}$ with $\mathcal{G}$ compactly supported, then for $K=\operatorname{Supp}\mathcal{G}$, setting $j:X\setminus K \subset X$, we get $j_\ast j^\ast\mathcal{G}=0$, hence the composition
$$\mathcal{G}\rightarrow\mathcal{F}\rightarrow j_\ast j^\ast \mathcal{F}$$
gets a canonical nullhomotopy, hence $\mathcal{G}\rightarrow\mathcal{F}$ factors through
$$i_\ast i^!\mathcal{F} = \operatorname{Fib}(\mathcal{F}\rightarrow j_\ast j^\ast \mathcal{F}),$$
giving the required cofinality.

To show that the last two Ind-systems are isomorphic, note that $i_\ast i^!\mathcal{F} = \operatorname{Fib}(\mathcal{F}\rightarrow j_\ast j^\ast\mathcal{F})$, then use another cofinality argument.

Finally, to show that all of these three colimits give $\mathcal{F}$, looking at the third colimit, it suffices to show that
$$\varinjlim_{i:V\subset X} i_\ast i^\ast \mathcal{F}=0,$$
But the intersection of all such $V$ is $\emptyset$, so for the corresponding idempotent algebras we find that their  colimit is $0$, whence the claim.
\end{proof}

Before we state the theorem constructing the proper pushforward functors, note that if $f:X\rightarrow Y$ is any map of Hausdorff complex analytic spaces, then $f_\ast:C_X\rightarrow C_Y$ sends compactly supported objects to compactly supported objects.  Indeed, this follows from compatibility of $f_\ast$ with pullbacks along open inclusions, together with the fact that the image of a compact subset is compact.  

\begin{theorem}\label{thm:lowershriek}
Let $f:X\rightarrow Y$ be a map of Hausdorff complex analytic spaces.  Then there is a unique colimit-preserving functor
$$f_!:C_X\rightarrow C_Y$$
equipped with an isomorphism of $f_!$ with $f_\ast$ when restricted to the full subcategory of compactly supported objects of $C_X$.  Moreover:
\begin{enumerate}
\item There is a unique natural transformation $f_!\rightarrow f_\ast$ restricting to the given isomorphism on compactly supported objects of $C_X$; and $f_!\mathcal{F}\overset{\sim}{\rightarrow} f_\ast\mathcal{F}$ is an isomorphism more generally when $\operatorname{Supp}\mathcal{F}$ is proper over $Y$.
\item There is a unique functorial isomorphism $\mathcal{G}\otimes f_! \mathcal{F}\simeq f_!((f^\ast\mathcal{G})\otimes\mathcal{F})$ for $\mathcal{G}\in C_Y$, $\mathcal{F}\in C_X$ which restricts to the projection formula isomorphism $\mathcal{G}\otimes f_\ast \mathcal{F}\simeq f_\ast((f^\ast\mathcal{G})\otimes\mathcal{F})$ for compactly supported $\mathcal{F}$.
\item Given another map $g:Y\rightarrow Z$ of Hausdorff complex analytic spaces, there is a unique isomorphism $(g\circ f)_! \simeq g_!\circ f_!$ which, on compactly supported objects of $C_X$, restricts to the natural isomorphism $(g\circ f)_\ast \simeq g_\ast \circ f_\ast$.
\item Given another map $\varphi: Y'\rightarrow Y$ of Hausdorff complex analytic spaces, there is a unique ``base-change'' isomorphism $\varphi^\ast \circ f_!\simeq f'_!\circ \varphi'^\ast$ which, on compactly supported objects, restricts to the natural isomorphism $\varphi^\ast \circ f_\ast\simeq f'_\ast \circ \varphi'^\ast$ from the proper base-change theorem.
\item If $f$ is an open inclusion, then there is a unique isomoprhism $f_!\simeq f_\natural$ compatibile with the natural maps of both of these functors to $f_\ast$.  (Recall $f_\natural$ denotes the left adjoint of $f^\ast$.)
\end{enumerate}
\end{theorem}

Note that there should also be coherence data relating these different isomorphisms of functors.  We will not need more than a small finite number of these coherences, and what we do need can be easily produced by hand using the uniqueness claims.

\begin{proof}
Lemma \ref{LKE} implies, by the objectwise formula for left Kan extensions, that any colimit-preserving functor out of $C_X$ identifies with the left Kan extension of its restriction to the full subcategory of compactly supported objects.  More precisely, the $\infty$-category of colimit preserving functors out of $C_X$ is equivalent to that of functors out of the full subcategory of compactly supported objects whose left Kan extension to $C_X$ commutes with colimits.  Thus to prove the claim giving and characterizing $f_!$, we need to see that the left Kan extension of $f_\ast$ from the compactly supported objects does preserve colimits.

For this we use the more economical presentation of the colimit over all compactly supported objects mapping to $\mathcal{F}$ to calculate the left Kan extension.  For the functor $f_!$ we then get the definition
$$f_!\mathcal{F} :=  \varinjlim_{i:V\subset X} f_\ast \operatorname{Fib}(\mathcal{F}\rightarrow i_\ast i^\ast\mathcal{F}),$$
where $V$ runs over all closed subsets of $X$ whose complement is contained in a compact subset.  To see that this preserves colimits in $\mathcal{F}$, it suffices to fix such a $V$ and show that $f_\ast \operatorname{Fib}(\mathcal{F}\rightarrow i_\ast i^\ast\mathcal{F})$ preserves colimits in $\mathcal{F}$.  But $\operatorname{Fib}(\mathcal{F}\rightarrow i_\ast i^\ast\mathcal{F})$ preserves colimits in $\mathcal{F}$ and is supported on a compact subset independent of $\mathcal{F}$, so the conclusion comes from the proper base change theorem proved earlier.

The abstract nonsense with left Kan extensions then gives the natural transformation in (1), and repeating the argument we just gave proves it's an isomorphism on properly supported objects, proving (1).  Part (2) is clear from left Kan extension, and so is (3) from left Kan extension and the proper base-change theorem.  For (4), by left Kan extension it suffices to show that the natural map $f_\natural\rightarrow f_\ast$ is an iso on compactly supported objects of $C_X$.  But indeed
$$f_\natural \mathcal{F} = \operatorname{Fib}(f_\ast \mathcal{F}\rightarrow i_\ast i^\ast f_\ast \mathcal{F})$$
where $i$ is the inclusion of the closed complement $Z$ of $X\subset Y$, and if $\mathcal{F}$ is supported in the compact $K\subset U$ then we get $i_\ast i^\ast f_\ast \mathcal{F}=0$ from $K\cap Z=\emptyset$.
\end{proof}

Now that we have this well-behaved colimit preserving functor $f_!:C_X\rightarrow C_Y$, we can ask about its right adjoint.  It turns out the right adjoint is also fairly well-behaved when the morphism $f:X\rightarrow Y$ is boundaryless, i.e.\@ locally on $X$ and $Y$ identifies with a Zariski closed immersion followed by projection off an open polydisk.  To phrase this precisely, we need the notion of an object of $C_X$ being ``homologically bounded''.  As we don't have a global t-structure on $C_X$, this needs to be phrased with a bit of care.  But the two possible definitions thankfully coincide:

\begin{lemma}\label{locallyboundeddimension}
Let $X$ be a complex analytic space, $\mathcal{F}\in C_X$, $x\in X$, and $d\in\mathbb{Z}$.  The following conditions are equivalent:
\begin{enumerate}
\item There is a quasi-affinoid open neighborhood $U$ of $x$ such that for all affinoid closed subspaces $\mathcal{M}(A)\subset U$, we have $\Gamma(\mathcal{F}|_{\mathcal{M}(A)})\in \mathcal{D}(\operatorname{Liq}_p)_{\leq d}$.
\item There is a quasi-affinoid open neighborhood $U$ of $x$ such that the $\mathcal{D}(\operatorname{Liq}_p)$-valued sheaf $sh(\mathcal{F})|_U$ on $U$ is $d$-truncated, i.e.~for all $y\in U$ the stalk at $y$ lies in $\mathcal{D}(\operatorname{Liq}_p)_{\leq d}$.
\end{enumerate}

Furthermore, if $X$ is affinoid and $\Gamma(\mathcal{F})\in \mathcal{D}(\operatorname{Liq}_p)_{\leq d}$, then $\mathcal{F}$ satisfies the above conditions at all $x\in X$ with $d$ replaced by $d-N-1$, where $N$ is the cohomological dimension of the topological space $X$.
\end{lemma}
\begin{proof}
Suppose (1) is satisfied.  Then as the affinoid closed subspaces $\mathcal{M}(A)\subset U$ form a neighborhood basis at $x$, the stalk of $sh(\mathcal{F})|_U$ at $x$ is a filtered colimit of the $\Gamma(\mathcal{F}|_{\mathcal{M}(A)})$.  Thus (2) is satisfied.  Now suppose (2) is satisfied.  Then for any affinoid closed subspace $\mathcal{M}(A)\subset U$, $\Gamma(\mathcal{F}|_{\mathcal M(A)})$ is the global sections of a $d$-truncated sheaf and thus itself $d$-truncated. Finally, the last claim about affinoid $A$ follows from the proof of Proposition \ref{prop:pseudocompactglues}.
\end{proof}

If $\mathcal{F}$ satisfies these two equivalent conditions, we say that $\mathcal{F}$ is homologically $d$-bounded at $x$.  If for all $x\in X$ there is a $d\in\mathbb{Z}$ such that $\mathcal{F}$ is homologically $d$-bounded at $x$, we say that $\mathcal{F}$ is homologically bounded.  If a collection $(\mathcal{F}_i)_{i\in I}$ of objects of $C_X$ is such that for all $x\in X$ there is a $d\in\mathbb{Z}$ such that $\mathcal{F}_i$ is homologically $d$-bounded at $x$ for all $i\in I$, we say that $(\mathcal{F}_i)_{i\in I}$ is \emph{uniformly homologically bounded}.

\begin{theorem}\label{boundarylessuppershriek}
Suppose $f:X\rightarrow Y$ is a boundaryless morphism between Hausdorff complex analytic spaces.  If $(\mathcal{F}_i)_{i\in I}$ is a collection of objects of $C_Y$ which is uniformly homologically bounded, then the collection $(f^!\mathcal{F})_{i\in I}$ of objects of $C_X$ is also uniformly homologically bounded, and the natural map
$$\bigoplus_i f^!(\mathcal{F}_i)\overset{\sim}{\rightarrow} f^!(\bigoplus_i \mathcal{F}_i)$$
is an isomorphism.
\end{theorem}
\begin{proof}
Since $j^!=j^\ast$ for open inclusions $j$, we can work locally on $X$ and $Y$, hence reduce to the case where $f$ is the composition of a Zariski closed immersion and projection off an open polydisk.  The collection of maps $f$ satisfying the conclusion is closed under composition, so we separately reduce to the case of a Zariski closed immersion and projection off the one-dimensional unit disk $\mathbb{D}$.

For Zariski closed immersions, by Lemma \ref{locallyboundeddimension} it suffices to show that if $A\rightarrow B$ gives a Zariski closed immersion of affinoids $f:\mathcal{M} B\rightarrow \mathcal{M} A$, then
$$f^!:\operatorname{Mod}_A (\mathcal{D}(\operatorname{Liq}_p))\rightarrow \operatorname{Mod}_B (\mathcal{D}(\operatorname{Liq}_p))$$
sends $\operatorname{Mod}_A (\mathcal{D}(\operatorname{Liq}_p))_{\leq 0}$ to $\operatorname{Mod}_B (\mathcal{D}(\operatorname{Liq}_p))_{\leq 0}$.  But indeed we have
$$f^!M = \underline{\operatorname{RHom}}_A(B,M),$$
so this follows from the fact that $B\in \operatorname{Mod}_A (\mathcal{D}(\operatorname{Liq}_p))_{\geq 0}$.

For projections $f:\mathbb{D}\times Y\rightarrow Y$, we claim that $f^!$ commutes with base-change, satisfies
$$f^\ast\mathcal{G}\otimes f^!(\mathcal{O}_Y)\overset{\sim}{\rightarrow} f^!(\mathcal{G}),$$
and for $Y=\ast$ we have $f^!(\mathcal{O}_{\ast}) = \mathcal{O}_{\mathbb{D}}[1]$.  Given these claims we easily deduce that $f^!=f^\ast [1]$ for general $Y$, which implies the statement of the theorem.  

To prove the claims, by descent, we can reduce to the case $Y=\mathcal{M}(A)$ affinoid.  Then we factor $f$ as the open inclusion $j:\mathbb{D}\times Y\subset \overline{\mathbb{D}}\times Y$ followed by the (proper) projection $\pi$ to $Y$.  The functor $j_!$ is fully faithful, and furthermore
$$C_{\overline{\mathbb{D}}\times Y}=  \operatorname{Mod}_{\mathcal{O}(\overline{\mathbb{D}})\otimes A}(\mathcal{D}(\operatorname{Liq}_p))$$
embeds fully faithfully into $\operatorname{Mod}_{A[T]}(\mathcal{D}(\operatorname{Liq}_p))$ by the forgetful functor, due to idempotency.  The essential image of $j_!$ in this latter category consists exactly of those $N\in \operatorname{Mod}_{A[T]}(\mathcal{D}(\operatorname{Liq}_p))$ which become $0$ on base-change along $\mathbb{C}[T]\rightarrow \mathcal{O}(|T|\geq 1)$.  Moreover, since $\pi$ is a map of affinoids we have that $\pi_!=\pi_\ast$ is just a forgetful functor on module categories.

Now, let $N\in \operatorname{Mod}_{A[T]}(\mathcal{D}(\operatorname{Liq}_p))$ which dies on base-change along $\mathbb{C}[T]\rightarrow \mathcal{O}(|T|\geq 1)$, and let $M\in \operatorname{Mod}_{A}(\mathcal{D}(\operatorname{Liq}_p))$.  Then from the above we get
$$\operatorname{RHom}_{C_{\mathbb{D}\times Y}}(N,f^!M) = \operatorname{RHom}_{A}(N,M).$$
Now we calculate further, using only that $N$ dies on base-change along $\mathbb{C}((T^{-1}))$ so that any $A[T]$-module map from $N$ to an $A\otimes\mathbb{C}((T^{-1}))$-module vanishes:
\begin{eqnarray*}
\operatorname{RHom}_{C_{\mathbb{D}\times Y}}(N,f^!M) &=& \operatorname{RHom}_{A}(N,M)\\
&=& \operatorname{RHom}_{A[T]}(N,\underline{\operatorname{RHom}}_{A}(A[T],M))\\
&=& \operatorname{RHom}_{A[T]}(N,M((T^{-1}))/M[T])\\
&=& \operatorname{RHom}_{A[T]}(N,M[T])[1].
\end{eqnarray*}
This gives $f^!\simeq f^\ast[1]$ via an isomorphism compatible with base-change, proving all the claims.
\end{proof}

Now we can prove the main theorem of this lecture, the finiteness of proper pushforwards.

\begin{theorem}[Grauert's Coherence Theorem]\label{thm:grauert}
Let $f:X\rightarrow Y$ be a boundaryless proper map between complex analytic spaces.  Then the functor $f_\ast:C_X\rightarrow C_Y$ sends $\mathcal{D}_{pc}(X)$ to $\mathcal{D}_{pc}(Y)$.
\end{theorem}

\begin{remark} The same proof applies also for ``families of complex-analytic spaces over other bases'', as in the work of Houzel \cite{Houzel}.
\end{remark}

\begin{proof}
Working locally on $Y$, we can assume $Y=\mathcal{M}(A)$ is affinoid, in particular compact Hausdorff, hence $X$ is also compact Hausdorff.  By Proposition \ref{prop:nuclearpseudocompactdiscrete}, it suffices to show that if $\mathcal{F}\in C_X$ is nuclear (meaning its pullback to every affinoid is nuclear), then so is $f_\ast\mathcal{F}$, and if $\mathcal{F}\in C_X$ is pseudocompact (meaning its pullback to every affinoid is pseudocompact), then so is $f_\ast\mathcal{F}$.

For the preservation of nuclearity, we can cover $X$ by finitely many affinoids with affinoid intersections to reduce to the fact that forgetful functors from module categories over nuclear algebras preserve nuclearity, Lemma \ref{cor:relativenuclear}.  For the preservation of pseudocompactness, suppose $(M_i)_{i\in I}$ is a uniformly homologically bounded collection of elements of $ \operatorname{Mod}_{A}(\mathcal{D}(\operatorname{Liq}_p))$.  First note
$$\operatorname{RHom}_A(f_*\mathcal{F},\bigoplus_i M_i) = \operatorname{RHom}_{C_X}(\mathcal{F},f^!\bigoplus_i M_i)$$
because $f_*=f_!$ by properness.  But by Theorem \ref{boundarylessuppershriek}, $f^!$ commutes with this direct sum, and for any affinoid closed $i:\mathcal{M}(B)\subset X$ the $B$-modules $(i^\ast f^! M_i)_{i\in I}$ are also uniformly homologically bounded.  Applying this to a finite cover of $X$ by closed affinoids with affinoid intersection and using descent, we deduce that
$$\operatorname{RHom}_{C_X}(\mathcal{F},\bigoplus_i f^!M_i) = \bigoplus_i \operatorname{RHom}_{C_X}(\mathcal{F},f^!M_i) = \bigoplus_i \operatorname{RHom}_{A}(f_\ast\mathcal{F},M_i),$$
whence the conclusion.
\end{proof}

Taking $Y=\ast$, this shows in particular that if $X$ is a compact complex analytic space in the classical sense, and $\mathcal{F}$ is a coherent sheaf on $X$, then the total cohomology $\bigoplus_i H^i(X,\mathcal{F})$ is finite dimensional as a $\mathbb{C}$-vector space.  The proof we gave has some kinship with the classical Cartan--Serre method for establishing such finite-dimensionality claims. But it is ``local'' in the sense that the key point is a more general property of the functor $f_!$ for possibly non-proper maps, namely that its right adjoint $f^!$ commutes with (homologically bounded) direct sums.\newpage

\section{Lecture XIII: Serre Duality, and GAGA for coherent sheaves}

In the first lecture, we promised to prove four theorems: Finiteness of Coherent Cohomology; Serre Duality; GAGA; Hirzebruch--Riemann--Roch. The first of those theorems was proved in the last lecture, even in the relative case. Moreover, we already did some of the work towards GAGA, and in the previous lecture also towards Serre Duality. The goal of this lecture is to finish the discussion of Serre Duality and GAGA. Next week, we will discuss the proof of Hirzebruch--Riemann--Roch.

As before, we fix some $0<p\leq 1$. Recall that we defined a (generalized) complex-analytic space to be a categorified locale $(X,C_X,\pi_X: \mathcal S(C_X)\to X)$ over $(\ast,\mathcal D(\mathrm{Liq}_p(\mathbb C)),\pi: \mathcal S(\mathcal D(\mathrm{Liq}_p(\mathbb C)))\to \ast)$ that is locally isomorphic to an open subspace of the categorified locale
\[
\mathcal{M}(A):=(\mathrm{Hom}_{\mathbb C}(A(\ast),\mathbb C),\mathrm{Mod}_A(\mathcal D(\mathrm{Liq}_p)),\pi_A)
\]
associated to some affinoid algebra $A$.  We note that this notion depends on the choice of $p$. But actually for a different $p'\leq p$, there is a base change functor from the notion for $p'$ to the notion for $p$, and this turns out to be an equivalence of $\infty$-categories (Exercise!); thus, the $\infty$-category of (generalized) complex-analytic spaces is independent of the choice of $p$. (But given $X$, the $\infty$-category $C_X$ depends on $p$.)

In the last lecture, several kinds of maps $f: X\to Y$ of complex-analytic spaces were introduced:
\begin{enumerate}
\item The map $f$ is a Zariski closed immersion if for every affinoid $A$ with a map $\mathcal{M}(A)\to Y$, the pullback $X\times_Y \mathcal{M}(A)$ is affinoid, given by some $\mathcal{M}(B)$, and $A\to B$ makes $B\in \mathcal D_{pc,\geq 0}(A)$ with $H_0 A\to H_0 B$ surjective.
\item The map $f$ is separated if the diagonal $\Delta_f: X\to X\times_Y X$ is a Zariski closed immersion; equivalently, if $|X|\to |X\times_Y X|=|X|\times_{|Y|} |X|$ is a closed immersion of topological spaces.
\item The map $f$ is proper if it is separated and quasicompact; equivalently, if $|f|: |Y|\to |X|$ is proper.
\item The map $f$ is smooth if locally on source and target $f$ is isomorphic to a projection $Y\times \mathbb D^n\to Y$ from an open polydisc.
\item The map $f$ is boundaryless if locally on source and target $f$ is isomorphic to a Zariski closed immersion followed by a smooth map.
\end{enumerate}

Moreover, we introduced for any separated map $f: X\to Y$ the functor
\[
f_!: C_X\to C_Y,
\]
the pushforward with compact support, with a natural transformation $f_!\to f_\ast$. Actually, we worked under the small restriction that $Y$ is itself separated over $\ast$, i.e. $|Y|$ is Hausdorff. (Then $f$ is separated if and only if also $|X|$ is Hausdorff.) Then $f_!$ is the colimit-preserving approximation to $f_\ast$. In fact, it agrees with $f_\ast$ on all objects that are compactly supported over $Y$. In particular, if $f$ is proper, then $f_!=f_\ast$ is just the usual pushforward. Moreover, $f_!$ satisfies arbitrary base-change and a projection formula. Writing $X$ as an increasing union of closed subspaces proper over $Y$, one can then write $f_!$ explicitly as a filtered colimit of sections with support in those closed subspaces.

Let
\[
f^!: C_Y\to C_X
\]
be the right adjoint of $f_!$. (It is unique when it exists, and if we fixed a cutoff cardinal $\kappa$, the existence of $f^!$ would follow from the adjoint functor theorem for presentable $\infty$-categories. It is not hard to show that in fact these functors for different sufficiently large $\kappa$ are compatible; in any case, in all situations below where we compute $f^!$, that formula simultaneously shows that it is well-defined.) To prove Serre duality, we need to identify $f^!$ for smooth morphisms. First, note that there is always a natural transformation
\[
f^! \mathcal O_X\otimes_{\mathcal O_Y} f^\ast M\to f^! M
\]
from a twisted version of $f^\ast$ to $f^!$. Indeed, this is adjoint to the natural transformation
\[
f_!(f^! \mathcal O_X\otimes_{\mathcal O_Y} f^\ast M)=f_! f^!\mathcal O_X\otimes_{\mathcal O_X} M\to M
\]
coming from the projection formula and the counit $f_! f^! \mathcal O_X\to \mathcal O_X$.

\begin{proposition}\label{prop:serreduality0} Let $f: X\to Y$ be a smooth morphism of complex-analytic spaces. Then the natural transformation
\[
f^!\mathcal O_X\otimes_{\mathcal O_Y} f^\ast \to f^!
\]
is an isomorphism. Moreover, $f^! \mathcal O_X$ is locally isomorphic to $\mathcal O_Y[d]$ where $d$ is the dimension of $f$, and the formation of $f^!$ commutes with any base change in $Y$.
\end{proposition}

\begin{proof} We can work locally on $X$ and $Y$. One can then reduce to the case $X=Y\times \mathbb D^n$. By induction on $n$, one can assume $n=1$. It suffices to prove all claims when $Y$ is affinoid. (Indeed, when one knows base change compatibility in that case, one can use this to show by descent to affinoid $Y$ that $f^!$ is always given by the desired formula, and commutes with base change to affinoids.) Now the result follows from the computations in the end of the proof in Theorem~\ref{boundarylessuppershriek}.
\end{proof}

To finish the proof of Serre duality, we need to identify $f^! \mathcal O_X$ with $\Omega^d_{Y/X}[d]$. Here, in case $Y$ is not derived, one can define $\Omega^1_{Y/X}$ as usual as $I/I^2$ where $I\subset \mathcal O_{X\times_Y X}$ is the ideal sheaf of the diagonal. When $Y$ is derived, one has to be more careful with the meaning of $I^2$. It can, however, still be defined:

\begin{proposition}\label{prop:powersofideal} There is a functorial definition of a filtered graded $A$-algebra $(I^n)_{n\geq 0}$ (in $\mathcal D(A)$) for any Zariski closed immersion $A\to B$ of affinoid algebras, such that $I^0=A$ and $I^1=I$ is the homotopy fibre of $A\to B$. All $I^n$ lie in $\mathcal D_{pc,\geq 0}(A)$. The filtered graded $A$-algebra $(I^n)_{n\geq 0}$ commutes with any base change in $A$, and if $A$ and $B$ are concentrated in degree $0$ and $A(\ast)\to B(\ast)$ is a local complete intersection, then $(I^n)_{n\geq 0}$ sits in degree $0$ and agrees with the usual filtered $A$-algebra of powers of $I=\mathrm{ker}(A\to B)$.
\end{proposition}

Note that everything here is relatively discrete over $A$, so it suffices to prove the analogous result for the map of abstract derived rings $A(\ast)\to B(\ast)$. More generally, one has the following result; a very related (and more detailed) discussion is in \cite[Section 3]{MaoPD}. (Here, we use the notion of animated commutative rings, which are ``simplicial commutative rings up to homotopy''. In characteristic $0$, this is the same thing as a connective $E_\infty$-algebra, as used previously. In positive or mixed characteristic, the notions diverge. There is still a functor from animated rings to $E_\infty$-rings; for geometric questions, generally animated rings behave better. For example, free animated rings are given by polynomial algebras, while free $E_\infty$-rings are very hard to understand.)

\begin{proposition}[cf.~{\cite[Corollary 3.54]{MaoPD}}]\label{prop:powersofidealanimated} Let $A\to B$ be a map of animated commutative rings that is surjective in degree $0$. Consider the initial animated filtered $A$-algebra $(I^n)_{n\geq 0}$ equipped with a map of animated $A$-algebras $B\to I^0/I^1$. Then the natural maps $A\to I^0$ and $B\to I^0/I^1$ are isomorphisms; in particular $I=I^1$ is the fibre of $A\to B$. Moreover, there is a natural isomorphism $L_{B/A}[-1]\cong I^1/I^2$, and the natural maps $\mathrm{Sym}^n_B(I/I^2)\to I^n/I^{n+1}$ are isomorphisms. The formation of $(I^n)_{n\geq 0}$ commutes with any base change in $A$, and if $A$ and $B$ are concentrated in degree $0$ and $A\to B$ is a local complete intersection, then $(I^n)_{n\geq 0}$ sits in degree $0$ and agrees with the usual filtered $A$-algebra of powers of $I$.
\end{proposition}

We stress that we need this result only in case the base space $Y$ is derived.

\begin{proof} It is clear that the formation of $(I^n)_{n\geq 0}$ commutes with all colimits in $B$, and is compatible with base change. Moreover, there are maps $A\to I^0$, $B\to I^0/I^1$ and $L_{B/A}\to I/I^2[1]$ (as $A/I^2$ is a square-zero extension of $B$ by $I^1/I^2$) and $\mathrm{Sym}^n(I/I^2)\to I^n/I^{n+1}$ (the latter exists for any animated filtered algebra). We claim that these are isomorphisms. As everything commutes with all sifted colimits, one can assume that
\[
B=A/^L(a_1,\ldots,a_n) = A\otimes^L_{\mathbb Z[X_1,\ldots,X_n]} \mathbb Z
\]
for some set of elements $a_1,\ldots,a_n\in A$; indeed, under sifted colimits such animated $A$-algebras generate all animated $A$-algebras $B$ with $H_0 A\to H_0 B$ surjective. By base change compatibility, one can then assume that $A=\mathbb Z[X_1,\ldots,X_n]$ and $B=\mathbb Z$ (with all $X_i$ mapping to $0$). It is then enough to prove the final claim, that for complete intersections, it gives the usual algebra of powers of $I$. By K\"unneth, this can be reduced to the case of one variable, so to $A=\mathbb Z[X]\to B=\mathbb Z$. But in this case one gets the free animated filtered $A$-algebra on a generator $Y$ in filtered degree $1$, modulo the relation $Y=X$. The first step gives $\mathbb Z[X,Y]$ filtered by $Y^n \mathbb Z[X,Y]$, and the second cuts it down to $\mathbb Z[X]$ filtered by $X^n\mathbb Z[X]$, as desired.
\end{proof}

Using this construction, we can define a general deformation to the normal cone. Namely, recall that filtered objects can equivalently be considered as $\mathbb G_m$-equivariant objects over $\mathbb A^1$, by the Rees construction. We can always glue them to $\mathbb G_m$-equivariant objects over $\mathbb P^1$ by using the trivial filtration at $\infty$. In this language, we arrive at the following construction, if we extend the above filtered algebra $(I^n)_{n\geq 0}$ to the full $\mathbb Z$-indexed case by setting $I^n=A$ for $n\leq 0$. (Actually, we ignore the $\mathbb G_m$-action here for simplicity.)

\begin{construction}\label{prop:deformationtonormalcone} Let $f: A\to B$ be a map of animated commutative rings that is surjective on $H_0$, let $I$ be the fibre of $f$, and consider $I^n$ as in Construction~\ref{prop:powersofidealanimated}. Then there is a Zariski closed immersion
\[
\tilde{Y}=\mathbb P^1_B\to \tilde{X}
\]
of affine derived schemes over $\mathbb P^1_{\mathbb Z}$ such that over $\mathbb P^1_{\mathbb Z}\setminus \{0\}$ it is isomorphic to the base change of $\mathrm{Spec}(B)\to \mathrm{Spec}(A)$ from $\Spec(\mathbb Z)$ to $\mathbb P^1_{\mathbb Z}\setminus \{0\}$, and the fibre over $0$ is isomorphic to the embedding
\[
\mathrm{Spec}(B)\to \mathrm{Spec}(\bigoplus_{n\geq 0} I^n/I^{n+1}) = \mathrm{Spec}(\mathrm{Sym}^\bullet(L_{B/A}[-1])),
\]
the (derived) normal cone on $L_{B/A}[-1]$.

The cotangent complex $L_{\tilde{Y}/\tilde{X}}$ is isomorphic to the pullback of $L_{B/A}$ to $\mathbb P^1_B$, twisted by $\mathcal O(-1)$. If $B$ is pseudocoherent over $A$, then also $\tilde{Y}\to \tilde{X}$ is pseudocoherent.
\end{construction}

\begin{proof} One takes $\tilde{X}$ as corresponding to the animated filtered commutative ring $\bigoplus_{n\in \mathbb Z} I^n$, and $\tilde{Y}$ corresponding to $\bigoplus_{n\leq 0} B$. Only the statement that $L_{\tilde{Y}/\tilde{X}}$ is isomorphic to the pullback of $L_{B/A}$ twisted by $\mathcal O(-1)$ must be proved. Note that $D(B)\to D(\mathbb P^1_B)$ is fully faithful, so it suffices to see that $L_{\tilde{Y}/\tilde{X}}(1)$ lies in the essential image. This condition is stable under sifted colimits in $B$, so as in the last construction we can reduce to $A=\mathbb Z[X_1,\ldots,X_n]$ and $B=\mathbb Z$. In that case, one gets the usual deformation to the normal cone, and the result is readily verified.
\end{proof}

\begin{remark} Another consequence of Construction~\ref{prop:powersofidealanimated} is the construction of blow-ups of derived schemes, by taking the projective spectrum of the animated graded $A$-algebra $\bigoplus_{n\geq 0} I^n$. The previous construction can be understood as an open subscheme of the blow-up of $\mathbb P^1_A$ along $\mathrm{Spec}(B)\times \{0\}\to \mathbb P^1_A$.
\end{remark}

To identify the dualizing complex $f^! \mathcal O_Y$ for smooth maps $f: X\to Y$, we note that it is enough to identify the pullback $s^\ast f^! \mathcal O_X$ along a section $s: Y\to X$. Indeed, pulling back $f: X\to Y$ along the map $g: Y'\to Y$ given by $f=g: X=Y'\to Y$, with fibre product
\[
f': X'=X\times_Y Y'\to Y',
\]
and $g': X'\to X$ the other projection, it acquires the universal section $s: Y'\to X'$ given by the diagonal $Y'=X\to X'=X\times_Y X$, and then the base change compatibility of $f^!$ shows that
\[
s^\ast f'^! \mathcal O_{Y'} = s^\ast g'^\ast f^! \mathcal O_X = f^! \mathcal O_Y.
\]
Thus, to compute $f^!\mathcal O_Y$, it suffices to compute $s^\ast f'^! \mathcal O_{Y'}$, as promised.

Thus, assume from now on that the smooth map $f: X\to Y$ comes with a section $s: Y\to X$. To compute $s^\ast f^!\mathcal O_Y$, we deform $Y\to X$ to the normal cone (locally on the corresponding affinoid algebras, and then analytify). This gives a map of complex-analytic spaces
\[
\tilde{f}: \tilde{X}\to \tilde{Y}=\mathbb P^1_Y
\]
with a section $\tilde{s}: \tilde{Y}\to \tilde{X}$, such that away from $0\in \mathbb P^1$ this is isomorphic to the base change of $Y\to X$, and over $0$ it is isomorphic to the normal cone $Y\to N_{Y\subset X}$ over $Y$.

Now consider
\[
\tilde{s}^\ast \tilde{f}^! \mathcal O_{\mathbb P^1_Y}(1)\in C_{\mathbb P^1_Y}.
\]
The pullback functor
\[
\pi^\ast: C_Y\to C_{\mathbb P^1_Y}
\]
is fully faithful -- for this, it suffices to see that $\pi_\ast \pi^\ast$ is the identity, but this is given by tensoring with $\pi_\ast \mathcal O_{\mathbb P^1_Y}$ by the projection formula, and the latter is just $\mathcal O_Y$. We claim that
\[
\tilde{s}^\ast \tilde{f}^! \mathcal O_{\mathbb P^1_Y}(d)\in \pi^\ast C_Y\subset C_{\mathbb P^1_Y}
\]
where $d$ is the dimension of $f$. This can be checked locally on $Y$ and after replacing $X$ by an open neighborhood of the section, so we can assume that $X=Y\times \mathbb D^d$. Then we can assume $Y=\ast$ is a point, where it follows from a direct computation.

But this means that all fibres of $\tilde{s}^\ast \tilde{f}^! \mathcal O_{\mathbb P^1_Y}(d)$ along sections of $\mathbb P^1_Y\to Y$ are canonically isomorphic (as they map isomorphically to $\pi_\ast \tilde{s}^\ast \tilde{f}^! \mathcal O_{\mathbb P^1_Y}(d)$). But the fibre over $1\in \mathbb P^1$ is just $s^\ast f^! \mathcal O_Y$, while the fibre over $0\in\mathbb P^1$ is the similar object for the normal cone $Y\subset N_{Y\subset X}$. The latter evidently only depends on $s^\ast \Omega^1_{X/Y}$. In fact, we can now consider the association that takes any $Y$ equipped with a rank $d$ vector bundle $f: E\to Y$ to $0^\ast f^! \mathcal O_Y$. This defines an invertible object of $C_{[\ast/\mathrm{GL}_d]}$ for the stack $[\ast/\mathrm{GL}_d]$ (for the open cover topology). As it is given by a line bundle shifted into degree $d$, to identify it, it suffices to identify it for $Y=\ast$ a point, with the $\mathrm{GL}_d$-action on it. This can be done easily, finally showing that there is a natural isomorphism
\[
s^\ast f^! \mathcal O_Y\cong \Omega^d_{X/Y}[d].
\]
We have proved the following theorem.

\begin{theorem}[Serre Duality] For a $d$-dimensional smooth map $f: X\to Y$ of (generalized) complex-analytic spaces, there is a natural isomorphism
\[
f^! M\cong f^\ast M\otimes_{\mathcal O_X} \Omega^d_{X/Y}[d]
\]
of functors $C_Y\to C_X$.
\end{theorem}

Let us discuss this theorem in case $Y=\ast$ is a point. If $X$ is proper, it gives the usual Serre duality, extended from coherent sheaves to all ``liquid quasicoherent sheaves''. Indeed, for any $M\in C_X$, we get
\[
\mathrm{Hom}(f_\ast M,\mathbb C)=\mathrm{Hom}(f_! M,\mathbb C)=f_\ast \mathrm{Hom}(M,f^! \mathbb C)=f_\ast(\mathrm{Hom}(M,\mathcal O_X)\otimes \Omega^d_X)[d],
\]
identifying the dual of the cohomology $R\Gamma(X,M) = f_\ast M$ with the cohomology $R\Gamma(X,M^\vee\otimes \Omega^d_X)$ up to a shift by degree $d$. Here $M^\vee := \mathrm{Hom}(M,\mathcal O_X)$.

But we can also apply this result in case $X$ is not proper. Then it is still true that the dual of
\[
R\Gamma_c(X,M) = f_! M\in \mathcal D(\mathrm{Liq}_p(\mathbb C))
\]
is given by $R\Gamma(X,M^\vee\otimes \Omega^d_X)$. We note that if $X\in \mathcal D_{pc}(X)$ is (pseudo)coherent, then $R\Gamma_c(X,M)$ is basic nuclear, as follows from the formula for $f_!$ from the proof of Theorem~\ref{thm:lowershriek}. In particular, all cohomology groups $H^i_c(X,M)$ are quotients of DNF spaces. Note that if $V$ is a DNF space, then by Mittag-Leffler, one has $\mathrm{Ext}^i_{\mathbb C}(V,\mathbb C)=0$ for $i>0$. Thus, for quotients of DNF spaces, one gets vanishing of $\mathrm{Ext}^i$ for $i>1$. The derived duality therefore induces short exact sequences
\[
0\to \mathrm{Ext}^1_{\mathbb C}(H_c^{d-i+1}(X,M),\mathbb C)\to H^i(X,M^\vee\otimes \Omega^d_X)\to \mathrm{Hom}_{\mathbb C}(H^{d-i}_c(X,M),\mathbb C)\to 0.
\]
This gives a statement on the level of individual cohomology groups even if these cohomology groups are not quasiseparated. The algebra here is very similar to the algebra in Poincar\'e duality in case there is torsion in the cohomology. More precisely, if $V$ is any quotient of DNF spaces, there is a maximal quasiseparated quotient $\overline{V}$, which is a DNF space; the kernel of $V\to \overline{V}$ can be identified with the closure $V^0\subset V$ of $0\in V$. We get a short exact sequence
\[
0\to V^0\to V\to \overline{V}\to 0.
\]
Then $\mathrm{Hom}_{\mathbb C}(V,\mathbb C)=\mathrm{Hom}_{\mathbb C}(\overline{V},\mathbb C)$ while $\mathrm{Ext}^1_{\mathbb C}(V,\mathbb C)=\mathrm{Ext}^1(V^0,\mathbb C)$. Using the (underived) duality between DNF and NF spaces, one can show that $V^0\neq 0$ if and only if $\mathrm{Ext}^1_{\mathbb C}(V^0,\mathbb C)\neq 0$, and that in the latter $\{0\}$ is dense. Thus, the above short exact sequence can be rewritten as
\[
0\to \mathrm{Ext}^1_{\mathbb C}(H_c^{d-i+1}(X,M)^0,\mathbb C)\to H^i(X,M^\vee\otimes \Omega^d_X)\to \mathrm{Hom}_{\mathbb C}(\overline{H^{d-i}_c}(X,M),\mathbb C)\to 0,
\]
where the first term is precisely the closure of $0\in H^i(X,M^\vee\otimes \Omega^d_X)$, and the last term is the quasiseparated quotient of $H^i(X,M^\vee\otimes \Omega_X^d)$ (and is a nuclear Fr\'echet space). In particular:

\begin{corollary} If $X$ is a complex manifold of dimension $d$ and $M\in \mathcal D_{pc}(X)$, then $H_c^{d-i+1}(X,M)$ is quasiseparated if and only if $H^i(X,M^\vee\otimes \Omega_X^d)$ is quasiseparated.
\end{corollary}

Finally, let us say a few words about GAGA. Generalizing our previous discussion slightly, we can actually start with any affinoid algebra $A$ and a proper derived scheme $X$ over $A$. It admits an analytification $X^{\mathrm{an}}$ over $A$, and there is an equivalence $C_X\cong C_{X^{\mathrm{an}}}$; the proof is the same as before. We note that $X^{\mathrm{an}}$ is actually proper (over $\mathcal{M}(A)$, or, as that one is proper, equivalently in the absolute sense). It is easy to see that it is separated, using the compatibility of all constructions with fibre products, and with Zariski closed immersions. For quasicompacity, we use the following proposition:

\begin{proposition} For any complex-analytic space $Y$, the map $\pi_Y: \mathcal S(C_Y)\to |Y|$ has the property that a closed subset $Z\subset |Y|$ is quasicompact if and only if the preimage of $Z$ in $\mathcal S(C_Y)$ is quasicompact.
\end{proposition}

\begin{proof} This can be checked on affinoid pieces, where it follows from the construction, which associated to closed subsets of the compact Hausdorff space $|\mathcal{M}(A)|$ idempotent algebras (whose corresponding locale is quasicompact).
\end{proof}

Thus, to see that $|X^{\mathrm{an}}|$ is quasicompact, it suffices to see that $\mathcal S(C_{X^{\mathrm{an}}})$ is quasicompact. But this agrees with $\mathcal S(C_X)$, and the latter is quasicompact for any scheme $X$ that is separated and of finite type. Indeed, it is glued in finitely many steps from $\mathcal S(C_A)$ for liquid algebras $A$, and those are quasicompact.

\begin{remark} This gives a proof that for $X$ a proper scheme over $\mathbb C$, the space $X(\mathbb C)$ is compact Hausdorff that does not use Chow's lemma in order to reduce to projective space, but instead directly uses the valuative criterion.
\end{remark}

Inside $C_X$, we have $\mathcal D_{pc}(X)\subset C_X$, by gluing on affine pieces the derived $\infty$-category of pseudocoherent complexes; these can be characterized as the discrete pseudocompact objects (and using this description one can prove that this glues). The functor $C_X\to C_{X^{\mathrm{an}}}$ takes $\mathcal D_{pc}(X)$ into $\mathcal D_{pc}(X^{\mathrm{an}})$. Finally, we can formulate the ``usual'' form of GAGA.

\begin{theorem}[GAGA] The fully faithful functor $\mathcal D_{pc}(X)\hookrightarrow \mathcal D_{pc}(X^{\mathrm{an}})$ is an equivalence of $\infty$-categories.
\end{theorem}

We note that the mere fully faithfulness already includes the comparison of cohomology.

\begin{proof} It remains to see that for all $M\in \mathcal D_{pc}(X^{\mathrm{an}})$ and any affine $\Spec(B)\subset X$, letting $i^\ast: C_{X^{\mathrm{an}}}\cong C_X\to C_B$ be the pullback functor, the preimage $i^\ast M$ lies in $\mathcal D_{pc}(B)\subset C_B=\mathrm{Mod}_B(\mathcal D(\mathrm{Liq}_p))$. We already know that it is pseudocompact (as $i_\ast$ commutes with all direct sums). Thus, it remains to see that it is discrete. As $B$ itself is discrete, it suffices to see that $R\Gamma(i^\ast M)=R\Gamma(X^{\mathrm{an}},i_\ast i^\ast M)$ is discrete. We note that by blowing up (some ideal sheaf supported at) the boundary $D=X\setminus \Spec(B)$, we can assume that the boundary $D$ is a Cartier divisor. Then $i_\ast i^\ast M = \mathrm{colim}_n M(-nD)$, and
\[
R\Gamma(X^{\mathrm{an}},i_\ast i^\ast M) = \mathrm{colim}_n R\Gamma(X^{\mathrm{an}},M(-nD)).
\]
Here each $M(-nD)\in \mathcal D_{pc}(X^{\mathrm{an}})$, so by Theorem~\ref{thm:grauert}, one has $R\Gamma(X^{\mathrm{an}},M(-nD))\in \mathcal D_{pc}(A)$, which in particular is discrete. Thus, also the colimit over $n$ is discrete, as desired.
\end{proof}

\begin{exercise} A map $f: X\to Y$ is a local complete intersection if it is locally the composite of a Zariski closed immersion that is a (derived) base change of $\ast\to \mathbb D^n$, and a smooth map. Show that the natural transformation
\[
f^! \mathcal O_Y\otimes_{\mathcal O_X} f^\ast\to f^!
\]
is an isomorphism, and that $f^!\mathcal O_Y$ is naturally isomorphic to $\mathrm{det}(L_{X/Y})$. Here $L_{X/Y}$ is the cotangent complex, which is a perfect complex with amplitude $[0,1]$; and for a $2$-term complex of finite projective modules $[P_1\to P_0]$, one defines
\[
\mathrm{det}([P_1\to P_0]) = \bigwedge^{r_0}(P_0)[r_0]\otimes (\bigwedge^{r_1}P_1)^\vee[-r_1],
\]
with $r_i=\mathrm{rk}(P_i)$. (Hint: For the isomorphism $f^! \mathcal O_Y\cong \mathrm{det}(L_{X/Y})$, follow the argument in the smooth case, by base change to the case with a universal section, and a deformation to the normal cone. This time, this will introduce truly derived spaces even in case $A$ and $B$ are in degree $0$. Finally, analyze the universal case $[M_{r_0\times r_1}/\mathrm{GL}_{r_0}\times \mathrm{GL}_{r_1}]$ directly.)
\end{exercise}
\newpage

\section{Lecture XIV: Hirzebruch--Riemann--Roch, part I}

The set-up for the following discussion will be a pair $(X,V)$ where $X$ is a compact complex manifold, meaning $f:X\rightarrow\ast$ is a proper smooth map of complex analytic spaces without boundary in the terminology of the previous lectures, and $V$ is a vector bundle on $X$ (a locally free $\mathcal{O}_X$-module sheaf of finite rank).  By Theorem \ref{thm:grauert}, the total cohomology of $V$ is finite-dimensional in each degree:
$$f_\ast V = R\Gamma(X;V)\in\mathcal{D}_{pc}(\mathbb{C}).$$
But actually we can also see that $f_\ast V$ lives in only finitely many degrees, for example because the space $X$ has finite cohomological dimension.

In any case, a crucial question for applications is what the dimensions
$$\operatorname{dim}_\mathbb{C} H^i(X;V)$$
of the individual cohomology groups of $(X,V)$ are.  However, this is difficult to answer in generality, because these dimensions are highly sensitive to the geometry of $(X,V)$.  In particular, they can jump in families.  To obtain a more robust invariant, we instead look at the Euler characteristic
$$\chi(X,V) := \sum_{i=0}^\infty (-1)^i\operatorname{dim}_\mathbb{C} H^i(X;V).$$
As a consequence of the proper base change theorem, this integer is (locally) constant in families.  The \emph{Hirzebruch-Riemann-Roch (HRR) Theorem} provides a formula for $\chi(X,V)$ where all the terms in the formula are in some sense ``topological'' invariants of $(X,V)$.  In particular they are quite computable, and they also give another explanation for this robustness of $\chi(X,V)$.  The formula looks as follows:
$$\chi(X,V) = \int_X \operatorname{ch}(V)\cdot \operatorname{Td}(T_X).$$
To make sense of the right-hand side, we need to choose a cohomology theory on compact complex manifolds.  We will take what is in some sense the simplest (or at least the easiest to relate to coherent cohomology), namely \emph{Hodge cohomology}.

\begin{definition}
Let $X$ be a complex manifold.  Define:
$$\mathcal{H}dg(X) = \bigoplus_{i\geq 0} \Omega^i[i]\in \operatorname{Perf}(X)\subset C_X,$$
and
$$\operatorname{Hdg}(X) = R\Gamma(X;\mathcal{H}dg(X))\in \mathcal{D}(\operatorname{Liq}_p).$$
\end{definition}

Thus, in degree $0$ we get
$$H_0\operatorname{Hdg}(X) =  \bigoplus_{i\geq 0} H^i(X;\Omega^i).$$

\begin{remark}
As usual, for purposes of reducing questions to affinoids, it's useful to extend this definition to the case where $X$ is a complex manifold ``with boundary'', meaning $X$ is a complex analytic space which is locally isomorphic to an open subspace of a closed polydisk.  The bundle of one-forms $\Omega^1$, hence of $i$-forms $\Omega^i$, can be just as well defined for such $X$, namely $\Omega^1=I/I^2$ where $I$ is the ideal in $\mathcal{O}_{X\times X}$ giving the Zariski closed immersion $\Delta_X:X\rightarrow X\times X$.
\end{remark}

Note that $\mathcal{H}dg(X)$ and $\operatorname{Hdg}(X)$ are contravariantly functorial, via pullback of differential forms.  They also have commutative ring structures, induced by wedge product of differential forms.  The classes $ \operatorname{ch}(V), \operatorname{Td}(T_X)$ will lie in $H^0 \operatorname{Hdg}(X)$, and $ \int_X$ will be a map
$$ \int_X: \operatorname{Hdg}(X)\rightarrow\mathbb{C},$$
allowing to make sense of the right-hand side of the HRR formula.  To explain this in more detail, let's start with the following basic observation.

\begin{lemma}
Let $X$ be a compact complex manifold.  Then $\operatorname{Hdg}(X) \in \operatorname{Perf}(\mathbb{C})$, and for every other complex manifold with boundary $Y$ the K\"{u}nneth map
$$\operatorname{Hdg}(X)\otimes \operatorname{Hdg}(Y)\overset{\sim}{\rightarrow} \operatorname{Hdg}(X\times Y)$$
is an isomorphism.
\end{lemma}
\begin{proof}
Note that $\Omega^i_X=0$ for $i>\operatorname{dim}(X)$, so the direct sum defining $\mathcal{H}dg(X)$ is actually finite, so $\mathcal{H}dg(X)\in \operatorname{Perf}(X)$.  As remarked above, this implies that $\operatorname{Hdg}(X)\in\operatorname{Perf}(\mathbb{C})$ as well (see also Lemma \ref{preservesperf}).  The second claim follows from applying the proper base-change theorem to the pullback of $X\rightarrow\ast\leftarrow Y$.
\end{proof}

Next, we introduce the first Chern class for line bundles, with values in Hodge cohomology.  Line bundles are locally free sheaves of rank one, so they are classified by elements of $H^1(X;\mathcal{O}_X^\times)$, and we simply take the map
$$c_1: H^1(X;\mathcal{O}_X^\times)\rightarrow H^1(X;\Omega^1)\subset H_0 \operatorname{Hdg}(X)$$
induced by the homomorphism $d\log: \mathcal{O}_X^\times\rightarrow \Omega^1$ given by
$$d\log(f) = \frac{df}{f}.$$

With the first Chern class in hand, we can state the projective bundle formula.

\begin{proposition}
Let $X$ be a complex manifold with boundary and $V$ a vector bundle over $X$ of rank $d$.  For the associated projective bundle $\pi:\mathbb{P}(V)\rightarrow X$, we have that
$$\bigoplus_{i=0}^{d-1} \operatorname{Hdg}(X)\overset{\sim}{\rightarrow} \operatorname{Hdg}(\mathbb{P}(V)),$$
where the map on the $i^{th}$ summand is given by pulling back via $\pi$ and then multiplying with $c_1(\mathcal{O}(1))^i$.
\end{proposition}
\begin{proof}
Note that both sides satisfy descent in $X$, and the map is natural.  So it suffices to show this when $V$ is trivial.  But then $\pi:\mathbb{P}(V)\rightarrow X$ is pulled back from $\mathbb{P}^{d-1}\rightarrow \ast$, so by the K\"{u}nneth formula we reduce to calculating the Hodge cohomology of $\mathbb{P}^{d-1}$.  By GAGA we can perform this calculation in the algebraic category instead, and then it is a standard exercise.
\end{proof}

Following Grothendieck (\cite{grothendieck1958theorie}), this lets us define Chern classes for vector bundles.  First, there is the all-important \emph{splitting principle}:

\begin{lemma}
For any vector bundle $V$ over a complex manifold with boundary $X$, there exists a proper smooth map of complex manifolds with boundary $f:Y\rightarrow X$ such that:
\begin{enumerate}
\item $f^\ast V$ admits a full flag, i.e.~a filtration by sub-bundles where each successive quotient is a line bundle.
\item $f^\ast: H_0\operatorname{Hdg}(X)\hookrightarrow H_0\operatorname{Hdg}(Y)$ is injective.
\end{enumerate}
\end{lemma}
\begin{proof}
Take $Y$ to be the relative Grassmannian classifying full flags in $X$.  Then (1) holds by construction, and (2) follows by the projective bundle formula, since $Y\rightarrow X$ is a composition of projective bundles.
\end{proof}

\begin{theorem}
The exists a unique assignment $V\mapsto c(V)\in H_0\operatorname{Hdg}(X)[[t]]$ for an arbitrary vector bundle $V$ over a complex manifold with boundary $X$, such that:
\begin{enumerate}
\item $V\mapsto c(V)$ commutes with pullbacks;
\item If $0\rightarrow V'\rightarrow V\rightarrow V''\rightarrow 0$ is a short exact sequence of vector bundles on $X$, then $c(V) = c(V')\cdot c(V'')$;
\item If $V=\mathcal{L}$ is a line bundle, then $c(\mathcal{L}) = 1 + c_1(\mathcal{L})\cdot t.$
\end{enumerate}
\end{theorem}
\begin{proof}
Uniqueness is clear by the splitting principle.  For existence, define $c(V)$ to be given by the coefficients (with some signs, and in reverse order) of $c_1(\mathcal{O}(1))^d$ as a polynomial in the previous powers of $c_1(\mathcal{O}(1))$, using the projective bundle formula.  We refer to \cite{grothendieck1958theorie} for the details.
\end{proof}

Using these Chern classes, we can define the Chern character and Todd class, which appear in the HHR formula.

\begin{theorem}
There exists a unique assignment $V\mapsto \operatorname{ch}(V)\in H_0\operatorname{Hdg}(X)$ for an arbitrary vector bundle $V$ over a complex manifold with boundary $X$, such that:
\begin{enumerate}
\item $V\mapsto  \operatorname{ch}(V)$ commutes with pullbacks;
\item If $0\rightarrow V'\rightarrow V\rightarrow V''\rightarrow 0$ is a short exact sequence of vector bundles on $X$, then $ \operatorname{ch}(V) =  \operatorname{ch}(V')+ \operatorname{ch}(V'')$.
\item If $V=\mathcal{L}$ is a line bundle, then $ \operatorname{ch}(\mathcal{L}) = e^{c_1(\mathcal{L})}.$
\end{enumerate}
Moreover, we have $\operatorname{ch}(V\otimes W)= \operatorname{ch}(V)\cdot  \operatorname{ch}(W)$.  Here $e^{c_1(\mathcal{L})}$ is to be understood as a formal power series in $c_1(\mathcal{L})$, which is actually a polynomial as $c_1(\mathcal{L})^i\in H^i(X;\Omega^i)=0$ for $i>\operatorname{dim}(X)$.

Similarly, there exists a unique assignment $V\mapsto \operatorname{Td}(V)\in H_0\operatorname{Hdg}(X)$ for an arbitrary vector bundle $V$ over a complex manifold with boudnary $X$, such that:
\begin{enumerate}
\item $V\mapsto  \operatorname{Td}(V)$ commutes with pullbacks;
\item If $0\rightarrow V'\rightarrow V\rightarrow V''\rightarrow 0$ is a short exact sequence of vector bundles on $X$, then $ \operatorname{Td}(V) =  \operatorname{Td}(V')\cdot \operatorname{Td}(V'')$.
\item If $V=\mathcal{L}$ is a line bundle, then $ \operatorname{Td}(\mathcal{L}) = \frac{c_1(\mathcal{L})}{1-e^{-c_1(\mathcal{L})}}.$
\end{enumerate}

\end{theorem}
\begin{proof}
Uniqueness is clear by the splitting principle.  For existence, one writes down the appropriate universal power series in the Chern classes.  See \cite{hirzebruch1966topological} for the details.
\end{proof}

To finish explaining the meaning of the HRR formula, which we recall reads 

$$\chi(X,V) = \int_X \operatorname{ch}(V)\cdot \operatorname{Td}(T_X),$$

\noindent we need to define the ``trace'' map
 
$$ \int_X: \operatorname{Hdg}(X)\rightarrow\mathbb{C}$$

\noindent for a compact complex manifold $X$.  This is done using Serre duality.  Namely, by projecting on to the $d$-summand where $d$ is the dimension of $X$ (working separately on each connected component if $X$ is not connected), we get a map
$$\mathcal{H}dg(X)=\bigoplus_{i\geq 0} \Omega^i[i] \rightarrow \Omega^d[d]\simeq f^!\mathbb{C},$$
whence
$$\operatorname{Hdg}(X)\rightarrow f_\ast f^!\mathbb{C}\rightarrow \mathbb{C},$$
as desired.  Here the isomorphism $\Omega^d[d]\simeq f^!\mathbb{C}$ is Serre duality, and the map $f_\ast f^!\mathbb{C}\rightarrow \mathbb{C}$ is the counit for the adjunction.

This finishes the statement of the HRR formula.  But before moving on to start the proof, let us make a further remark about this trace map we just introduced.  Namely, the object $\mathcal{H}dg(X)\in \operatorname{Perf}(X)$ is canonically self-dual with respect to the invertible object $ \Omega^d[d]\simeq f^!\mathbb{C}$.  More precisely, the map
$$\mathcal{H}dg(X)\otimes \mathcal{H}dg(X)\rightarrow \mathcal{H}dg(X)\rightarrow \Omega^d[d]$$
induced by multiplication followed by projection to the top-dimensional component is a perfect pairing in $\operatorname{Perf}(X)$.  This is a general fact about exterior algebras.  Combining with Serre duality, we deduce that $\operatorname{Hdg}(X)\in\operatorname{Perf}(\mathbb{C})$ is also canonically self-dual via the perfect pairing
$$\operatorname{Hdg}(X)\otimes \operatorname{Hdg}(X)\rightarrow \operatorname{Hdg}(X)\overset{\int_X}{\rightarrow}\mathbb{C}$$
where the first map is multiplication.   In terms of this self-duality, we can say that $\int_X$ is dual to the unit map $\mathbb{C}\rightarrow \operatorname{Hdg}(X)$.

Now let's move towards the proof of HRR.  For this, we will actually need to state and prove a generalization, called the \emph{Grothendieck-Riemann-Roch (GRR) theorem}.  The idea is to at first forget about the Todd class, and focus only on the Chern character.  This is a map
$$\operatorname{ch}: \operatorname{Vect}(X)/\sim \longrightarrow H_0\operatorname{Hdg}(X)$$
from isomorphism classes of vector bundles to Hodge cohomology, but Grothendieck explains how to promote it to a natural transformation of cohomology theories, by introducing \emph{$K$-theory}, a cohomology theory which is formally built out of vector bundles.  Namely, $K_0(\operatorname{Vect}(X))$ is the free abelian group on $\operatorname{Vect}(X)/\sim$ modulo the relations
$$[V]=[V']+[V'']$$
running over all short exact sequences $0\rightarrow V'\rightarrow V\rightarrow V''\rightarrow 0$.  There is a ring structure on $K_0(\operatorname{Vect}(X))$  induced by tensor product of vector bundles, and then we can view the Chern character as a natural transformation
$$\operatorname{ch}: K_0(\operatorname{Vect}(X)) \longrightarrow H_0\operatorname{Hdg}(X)$$
of cohomology theories, i.e.\@ contravariant functors from complex manifolds to commutative rings.

Then Grothendieck rephrases the HRR question as that of how to address the compatibility of $\operatorname{ch}$ with the natural \emph{pushforward} structures which exist on both cohomology theories.  But actually, to do this we have to work with a modification of the $K$-theory of vector bundles, namely the $K$-theory of perfect complexes, since it is only on the level of perfect complexes that we actually have a pushforward functoriality.  The $K$-theory of perfect complexes, $K_0(\operatorname{Perf}(X))$, is defined analogously to that of vector bundles, but we replace short exact sequences with distinguished triangles (or cofiber sequences in the stable $\infty$-category language).

Then the pushforward on $K$-theory of perfect complexes comes from the following category-level statement.

\begin{lemma}\label{preservesperf}
Let $f:X\rightarrow Y$ be a map of compact complex manfiolds.  Then $f_\ast:C_X\rightarrow C_Y$ sends $\operatorname{Perf}(X)$ to $\operatorname{Perf}(Y)$.
\end{lemma}
\begin{proof}
Let us more generally prove the claim when $f$ is a boundaryless proper map between complex manifolds with boundary.  By descent and proper base-change, we can reduce to where $Y$ is a closed polydisk, in particular proper.  From Theorem \ref{thm:grauert} we know that $f_\ast \mathcal{F}\in \mathcal{D}_{pc}(Y)$ for $\mathcal{F}\in \mathcal{D}_{pc}(X)$, so in particular for $\mathcal{F}\in\operatorname{Perf}(X)$.  Thus it suffices to show that $f_\ast \mathcal{F}$ is a compact object of $C_Y$ when $\mathcal{F}\in\operatorname{Perf}(X)$.  But a perfect complex on a compact analytic space is a compact object (by descent from the affinoid case), so for this it suffices to show that $f_\ast$ preserves compact objects.  And for that it suffices to see that the right adjoint $f^!$ commutes with direct sums.  But we can factor $f$ as the composition of an lci Zariski closed immersion followed by a smooth map using the graph trick, so this follows claim follows from the formulas given for $f^!$ in those two cases in the previous lecture.
\end{proof}

Thus, for a map $f:X\rightarrow Y$ of compact complex manifolds, we have not just the contravariant
$$f^*:K_0(\operatorname{Perf}(Y))\rightarrow K_0(\operatorname{Perf}(X))$$
induced by pullback of perfect complexes, but also the covariant
$$f_\ast: K_0(\operatorname{Perf}(X))\rightarrow K_0(\operatorname{Perf}(Y))$$
induced by pushforward of perfect complexes.

We also have a pushforward structure on Hodge cohomology, defined in a completely different way: we can take
$$f_\ast: \operatorname{Hdg}(X)\rightarrow \operatorname{Hdg}(Y)$$
to be the dual to $f^\ast:\operatorname{Hdg}(Y)\rightarrow \operatorname{Hdg}(X)$ via the self-duality of $\operatorname{Hdg}(-)$ coming from Serre duality.

Then the fact of life is that our natural transformation
$$\operatorname{ch}: K_0(\operatorname{Vect}(X)) \longrightarrow H_0\operatorname{Hdg}(X)$$
does not intertwine these pushforward structures (why should it?).  However, this can be corrected by modifying the pushforward structure on Hodge cohomology, namely, by conjugating with the Todd classes of the tangent bundles.  Formally:

\begin{theorem}\label{GRR}
We work in the category of compact complex manifolds.
\begin{enumerate}
\item There is a natural transformation $\widetilde{\operatorname{ch}}:K_0(\operatorname{Perf}(X))\rightarrow H_0\operatorname{Hdg}(X)$ of multiplicative cohomology theories on compact complex manifolds extending $\operatorname{ch}: K_0(\operatorname{Vect}(X)) \longrightarrow H_0\operatorname{Hdg}(X)$.
\item If $f:X\rightarrow Y$ is a map of compact complex manifolds, then for $c\in K_0(\operatorname{Perf}(X))$, we have
$$\widetilde{\operatorname{ch}}(f_\ast c)\cdot\operatorname{Td}(T_Y)  = f_\ast(\widetilde{\operatorname{ch}}(c)\cdot\operatorname{Td}(T_X)) \in H_0\operatorname{Hdg}(Y).$$
\end{enumerate}
\end{theorem}

When $Y=\ast$, this GRR theorem exactly recovers the HRR theorem.  But to prove even that special case, we need to consider the general case.

\begin{remark}
In principle, one should state and prove this result more generally, for an arbitrary proper map of not-necessarily-compact complex manifolds.  The proof we present here does not directly give that generalization.
\end{remark}

Now we can explain the strategy of proof.  We will introduce an intermediary cohomology theory $\operatorname{HH}(X)$, which in some sense lies halfway between $K(\operatorname{Perf}(X))$ and $\operatorname{Hdg}(X)$.  This $\operatorname{HH}(X)$ also has a natural pushforward structure.  We will define $\widetilde{\operatorname{ch}}$ as a composition of natural transformations of multiplicative cohomology theoreis
$$K_0(\operatorname{Perf}(X))\overset{\alpha}{\longrightarrow} H_0\operatorname{HH}(X)\overset{\beta}{\longrightarrow} H_0\operatorname{Hdg}(X),$$
where:

\begin{enumerate}
\item $\alpha$ also commutes with pushforwards, on the nose, without any Todd classes or correction factors;
\item $\beta$ is an isomorphism.
\end{enumerate}

This reduces the GRR problem to showing that two different pushfowards structures on the same cohomology theory agree: namely, the cohomology theory is Hodge cohomology, and the two different pushforwards are the one from Serre duality twisted by the Todd class, and the one coming from $\operatorname{HH}(X)$ via the isomorphism $\beta$.

In the rest of this lecture, we will define $\operatorname{HH}(X)$ and produce the isomorphism $\beta$, and in the next lecture we will do the remainder of the work.

\begin{definition}
Let $X$ be a complex manifold with boundary, and let $\Delta:X\rightarrow X\times X$ denote the diagonal.  Define the \emph{Hochschild homology}
$$\mathcal{HH}(X) := \Delta^\ast \Delta_\ast \mathcal{O}_X \in \operatorname{Perf}(X),$$
and
$$\operatorname{HH}(X) = R\Gamma(X;\mathcal{HH}(X))\in \mathcal{D}(\operatorname{Liq}_p).$$
\end{definition}
Note that these are commutative algebra objects, contravariantly functorial in $X$ via pullback.  By descent, they are determined by the affinoid case, and if $X=\mathcal{M}(A)$ is affinoid, then
$$\operatorname{HH}(X) = A\otimes_{A\otimes A} A,$$
where all tensor products are derived liquid tensor products.  Here let's say that the $A$-algebra structure on $\operatorname{HH}(X)$, promoting it to $\mathcal{HH}(X)$, comes from the right hand factor.

The first step to constructing the isomorphism $\beta$ relating Hochschild homology to Hodge cohomology is the \emph{Hochschild-Kostant-Rosenberg (HKR) theorem}.

\begin{theorem}
Let $X$ be a complex manifold with boundary.  Then there is a natural isomorphism
$$H_\ast \mathcal{HH}(X) \simeq \bigoplus_i \Omega^i[i]$$
of graded commutative algebra objects in the category of coherent sheaves on $X$.
\end{theorem}
\begin{proof}
By descent, it suffices to construct such a natural isomorphism when $X=\mathcal{M}(A)$ is an affinoid isomorphic to a closed polydisk.  Take the boundary map associated to the fiber sequence
$$I\rightarrow A\otimes A\rightarrow A$$
and apply $-\otimes_{A\otimes A} A$; this gives a natural map of $A$-modules
$$\mathcal{HH}(A)\rightarrow \Sigma I\otimes_{A\otimes A} A,$$
whence on $H_1$ a natural map of $A$-modules
$$H_1 \mathcal{HH}(A)\rightarrow \Omega^1_A.$$
It suffices to show that this map is an isomorphism, and that
$$\Lambda^\ast_A H_1 \mathcal{HH}(A)\overset{\sim}{\rightarrow} H_\ast  \mathcal{HH}(A).$$
Choosing an isomorphism with the closed unit disk, the resulting map
$$\mathbb{C}[x_1,\ldots,x_n]\rightarrow A$$
is idempotent in liquid algebras and flat as a map of discrete rings, so we reduce to the analogous claims for $A=\mathbb{C}[x_1,\ldots,x_n]$, where it is a simple Koszul complex calculation.  (One can even use a K\"{u}nneth argument to reduce to the case $n=1$.)
\end{proof}

Now we would like to promote this to a natural commutative algebra isomorphism $\mathcal{HH}(X) \simeq \bigoplus_i \Omega^i[i]$ in $C_X$.  For that we will use a trick: the action of Adams operations on $\mathcal{HH}(X)$.  Again, by descent it suffices to consider the affinoid case $X=\mathcal{M}(A)$.  Then
$$\operatorname{HH}(A) = A\otimes_{A\otimes A}A = A^{\otimes S^1},$$
the coproduct of $A$ indexed by the anima $S^1$ in the $\infty$-category of commutative algebra objects of $\mathcal{D}(\operatorname{Liq}_p)$.  Indeed, this holds because the anima $S^1$ is the pushout of $\ast \leftarrow \ast\sqcup \ast\rightarrow\ast$.  Moreover, the $A$-algebra structure on $\operatorname{HH}(A)$, promoting it to $\mathcal{HH}(A)$, comes from the inclusion of a base-point in $S^1$.

We consider the degree two map $S^1\rightarrow S^1$.  It sends the base-point to the base-point, so it induces a natural commutative $A$-algebra endomorphism
$$\psi^2: \operatorname{HH}(A)\rightarrow \operatorname{HH}(A).$$
The key is the following.

\begin{lemma}
On $H_i  \operatorname{HH}(A)$, the endomorphism $\psi^2$ induces the map of multiplication by $2^i$.
\end{lemma}
\begin{proof}
By multiplicativity and HKR, it suffices to check this on $H_1$, and as in the previous proof we can reduce to the case of a polynomial algebra in one generator, where it is a simple calculation after choosing a simplicial model for the degree two map on $S^1$.
\end{proof}

Now we have the more refined HKR theorem.

\begin{theorem}\label{thm:refinedHKR}
Let $X$ be a complex manifold with boundary.  Then there is a natural isomorphism
$$\mathcal{HH}(X) \simeq \bigoplus_i \Omega^i[i] = \mathcal{H}dg(X)$$
of commutative algebra objects in $C_X$, and hence on global sections a natural isomorphism
$$\beta:\operatorname{HH}(X)\simeq\operatorname{Hdg}(X)$$
of commutative algebra objects in $\mathcal{D}(\operatorname{Liq}_p)$.
\end{theorem}
\begin{proof}
Note that $\mathcal{H}dg(X)$ is the free commutative algebra object on $\Omega^1[1]$.  Indeed, we are in characteristic $0$, so there are no higher derived functors of $\Sigma_n$-invariants, and as $\Omega^1[1]$ lives in odd degree the symmetric power becomes an alternating power due to the Koszul sign rule.  Thus, by the above HKR theorem on homology, it suffices to show that there's a natural map 
$$(H_1 \mathcal{HH}(X))[1]\rightarrow \mathcal{HH}(X)$$
in $C_X$ inducing an isomorphism on $H_1$.  For this, work on affinoids, and consider $\psi^2$ just as an $A$-module endomorphism of $\mathcal{HH}(A)$, giving an $A[T]$-module structure where $T$ acts by $\psi^2$.  Then as $T-2$ and $T-2^i$ generate the unit ideal for $i\neq 1$, the previous lemma implies that the $A$-module map
$$\mathcal{HH}(A)\rightarrow \mathcal{HH}(A)[\frac{1}{T-2}]$$
is an isomorphism on homology in all degrees $\neq 1$.  But in degree $1$ the target has vanishing homology because $T-2$ acts both as $0$ and as an isomorphism.  It follows that the fiber of this map identifies with $(H_1 \mathcal{HH}(A))[1]$, giving the required construction.
\end{proof}

Finally, in order to segue into the next lecture, we explain a categorical reinterpretation of this Hochschild homology $\operatorname{HH}(X)$ in the case where $X$ is compact (proper over $\ast$).  In fact, we will see in this case that $\operatorname{HH}(X)$ only depends on the $\infty$-category $C_X$, tensored over our base symmetric monoidal $\infty$-category $C_\ast = \mathcal{D}(\operatorname{Liq}_p)$.

More precisely, we will implicitly fix a cut-off cardinal in our condensed world, so that all our $\infty$-categories are presentable.  Then our ambient setting will be
$$\operatorname{Mod}_{C_\ast}(\operatorname{Pr}^L).$$
Here $\operatorname{Pr}^L$ is the $\infty$-category of presentable $\infty$-categories, with morphisms given by the colimit-preserving functors, or equivalently those which admit a right adjoint.  This $\operatorname{Pr}^L$ is a symmetric monoidal $\infty$-category with respect to Lurie's tensor product, which is characterized by the fact that maps
$$\mathcal{C}\otimes\mathcal{D}\rightarrow\mathcal{E}$$
in $\operatorname{Pr}^L$ are the same as functors
$$\mathcal{C}\times\mathcal{D}\rightarrow\mathcal{E}$$
which preserve colimits in each variable separately.  Moreover $\operatorname{Pr}^L$ admits all colimits, and the tensor product on $\operatorname{Pr}^L$ preserves colimits in each variable.  Note finally that $C_\ast$ is a commutative algebra object in this symmetric monoidal $\infty$-category $\operatorname{Pr}^L$, because of the liquid tensor product on $C_\ast$ which commutes with colimits in each variable.  Thus $\operatorname{Mod}_{C_\ast}(\operatorname{Pr}^L)$ is defined, and is itself symmetric monoidal, via relative tensor product over $C_\ast$.

As an example of a tensor product calculation in $\operatorname{Mod}_{C_\ast}(\operatorname{Pr}^L)$ (valid more generally with $C_\ast$ replaced by any commutative algebra object in $\operatorname{Pr}^L$), suppose given two algebras $A$ and $B$ in $C_\ast$.  Then $\operatorname{Mod}_A(C_\ast),\operatorname{Mod}_B(C_\ast) \in \operatorname{Mod}_{C_\ast}(\operatorname{Pr}^L)$, and we have
$$\operatorname{Mod}_A(C_\ast)\otimes_{C_\ast}\operatorname{Mod}_B(C_\ast)\simeq \operatorname{Mod}_{A\otimes B}(C_\ast),$$
see \cite[Remark 4.8.5.17]{lurie2017higher}.

We can use this basic fact to prove the following.

\begin{lemma}\label{lem:symmetricmonoidalCX}
Let $X$ and $Y$ be complex analytic spaces.  Then
$$C_X\otimes_{C_\ast} C_Y \overset{\sim}{\rightarrow} C_{X\times Y}.$$
\end{lemma}
\begin{proof}
If $X=\mathcal{M}(A)$ is affinoid, then $C_{X\times Y}\simeq \operatorname{Mod}_A(C_Y)$ and the claim follows as in the previous remarks.  Now suppose $X$ is quasiaffinoid, inside an affinoid $\mathcal{M}(A)$ with closed complement given by the idempotent algebra $B\in \operatorname{Mod}_A$.  The localization sequence
$$\operatorname{Mod}_B\overset{i_\ast}{\rightarrow} \operatorname{Mod}_A\overset{j^\ast}{\rightarrow} C_X$$
is $C_\ast$-linear, and the left adjoints are also $C_\ast$-linear.  It then follows from the symmetric monoidal $(\infty,2)$-category structure of $\operatorname{Mod}_{C_\ast}(\operatorname{Pr}^L)$ that $-\otimes_{C_\ast} C_Y$ sends this localization sequence to another such localization sequence, see \cite{hoyois2017higher}, and combining again with the remarks about tensoring with module categories this implies the claim.  For general $X$, by descent we have that in $\operatorname{Cat}_\infty$ the pullback functors give
$$C_X = \varprojlim_{U\subset X} C_U,$$
where $U\subset X$ runs over all quasi-affinoid open subsets of $X$.  As the pullback functors along open inclusions have left adjoints which are $C_\ast$-linear by the projection formula, this gives
$$C_X = \varinjlim_{U\subset X} C_U$$
in $\operatorname{Mod}_{C_\ast}(\operatorname{Pr}^L)$, via left adjoint to pullback, see \cite[5.5.3.18]{LurieHTT}.  In this manner we reduce the general case to the quasi-affinoid case, giving the result.
\end{proof}

Now we can give the promised categorified interpretation of $\operatorname{HH}(X)$ for $X$ proper.  This is based on the concept of ``dimension'' of a dualizable object in a symmetric monoidal $\infty$-category.  If $c\in\mathcal{C}^\otimes$ is dualizable, then by definition we have
$$\operatorname{dim}(c) = \left[1\rightarrow c\otimes c^\vee\rightarrow 1\right] \in \operatorname{End}_\mathcal{C}(1).$$
Here $1$ is the unit object of $\mathcal{C}$, the first map is the coevaluation for the duality, and the second map is the evaluation for the duality.

In our case, where $\mathcal{C} = \operatorname{Mod}_{C_\ast}(\operatorname{Pr}^L)$, we have
$$\operatorname{End}_\mathcal{C}(1) = C_\ast,$$
or more precisely the anima of objects of $C_\ast$, so given a dualizable object $M$ of $\operatorname{Mod}_{C_\ast}(\operatorname{Pr}^L)$ we get an object
$$\operatorname{dim}(M)\in C_\ast.$$
Then the result is the following.

\begin{theorem}
Let $X$ be a proper complex manifold with boundary.  Then there is a natural identification
$$\operatorname{dim}(C_X) = \operatorname{HH}(X).$$
\end{theorem}
\begin{proof}
We claim that $C_X$ is canonically self-dual as an object of $\operatorname{Mod}_{C_\ast}(\operatorname{Pr}^L)$, with coevaluation map given by pull-push along
$$\ast\leftarrow X\overset{\Delta}{\rightarrow} X\times X,$$
and evaluation map given by pull-push along
$$X\times X\overset{\Delta}{\leftarrow} X\rightarrow \ast.$$
Here we use that $C_{X\times X} = C_X\otimes C_X$ by the previous lemma to correctly interpret these maps as candidate evaluation and coevaluation maps.  To prove these give a duality datum, we just need to check the adjunction identities, but they follow immediately from the proper base-change theorem.  Composing coevaluation and evaluation we find exactly the definition of $\operatorname{HH}(X)$, proving the identification with the dimension of $C_X$.
\end{proof}

In the next lecture we will be more specific about the functoriality of this identification, and carry out the rest of the argument for GRR.

\begin{remark}
Even if $X$ is not proper, we still have that $C_X$ is canonically self-dual: one just uses proper pushforward instead of usual pushforward in the push-pull formalism.  The conclusion in that case is that $\operatorname{dim}(C_X)$ identifies with the \emph{compactly supported} cohomology of $\mathcal{HH}(X)$.
\end{remark}\newpage

\section{Lecture XV: Hirzebruch--Riemann--Roch, part II}

This is the final lecture of the course. The goal is to finish the proof of the (Grothendieck--)Hirzebruch--Riemann--Roch theorem for compact complex manifolds. Given the work done in the previous lecture, this naturally splits into two pieces:

\begin{enumerate}
\item For any compact complex manifold $X$, construct the (extended) Chern character
\[
\widetilde{\mathrm{ch}}: K_0(\mathrm{Perf}(X))\to H_0\mathrm{HH}(X)\cong H_0\mathrm{Hdg}(X)
\]
as a map of contravariant functors from compact complex manifolds to rings, and show that after restriction along the natural map $K_0(\mathrm{Vect}(X))\to K_0(\mathrm{Perf}(X))$ it agrees with the Chern character $\mathrm{ch}$ defined in the last lecture.
\item For any map $f: X\to Y$ of compact complex manifolds, show that the diagram
\[\xymatrix{
K_0(\mathrm{Perf}(X))\ar[rr]^-{\mathrm{td}(T_X)\cdot \widetilde{\mathrm{ch}}}\ar[d]^{f_\ast} && H_0\mathrm{Hdg}(X)\ar[d]^{f_\ast}\\
K_0(\mathrm{Perf}(Y))\ar[rr]^-{\mathrm{td}(T_Y)\cdot \widetilde{\mathrm{ch}}} && H_0\mathrm{Hdg}(Y)
}\]
commutes.
\end{enumerate}

For the first part, we will need to study the functoriality of the association taking any dualizable $C\in \operatorname{Mod}_{C_\ast}(\operatorname{Pr}^L)$ to the (relative) Hochschild homology $\mathrm{HH}(C) := \mathrm{HH}(C/C_\ast)$. Here $C_\ast=\mathcal D(\mathrm{Liq}_p)$ (or rather, a truncated version for some fixed $\kappa$). This discussion actually works more generally in symmetric monoidal $(\infty,2)$-categories $M$; for us $M$ will be $\operatorname{Mod}_{C_\ast}(\operatorname{Pr}^L)$. A detailed discussion of this functoriality is in the work of Hoyois--Scherotzke--Sibilla \cite{hoyois2017higher}. We note that we can actually get by with the same constructions for symmetric monoidal $2$-categories (ignoring all higher morphisms).

Recall that $X\in M$ is dualizable if there is some object $X^\vee\in M$, a coevaluation map $c: 1\to X\otimes X^\vee$, and an evaluation map $d: X^\vee\otimes X\to 1$, such that the composites
\[
X\xrightarrow{c\otimes 1} X\otimes X^\vee\otimes X\xrightarrow{1\otimes d} X
\]
and
\[
X^\vee\xrightarrow{1\otimes c} X^\vee\otimes X\otimes X^\vee\xrightarrow{d\otimes 1} X^\vee
\]
are equal (via implicit invertible $2$-morphisms) to the identity. In this case $X^\vee$ is necessarily the internal Hom from $X$ to $1$; in fact, for all $Y$, the internal Hom from $X$ to $Y$ is given by $X^\vee\otimes Y$.

Now given any endomorphism $f: X\to X$, one can define an endomorphism of $1$ in $M$ as the composite
\[
\mathrm{tr}_M(f|X):=(1\xrightarrow{c} X\otimes X^\vee\xrightarrow{f\otimes 1} X\otimes X^\vee\xrightarrow{d} 1)\in \mathrm{End}_M(1).
\]
This gives a natural functor of $(\infty,1)$-categories
\[
\mathrm{tr}_M(-|X): \mathrm{End}_M(X)\to \mathrm{End}_M(1).
\]
In particular, this can be applied to the identity $f=\mathrm{id}_X$, giving
\[
\mathrm{dim}_M(X)\in \mathrm{End}_M(1).
\]

An important property of the trace construction is its cyclic invariance. This mirrors the property $\mathrm{tr}(AB)=\mathrm{tr}(BA)$ for matrices. In this abstract situation, assume given dualizable $X,Y\in M$ and maps $a: X\to Y$ and $b: Y\to X$. Let $f=ba: X\to X$ and $g=ab: Y\to Y$ be the induced endomorphisms. Then there is an isomorphism
\[
\mathrm{tr}_M(f|X)\cong \mathrm{tr}_M(g|Y)
\]
defined by the following diagram
\[\xymatrix{
1\ar@{=}[d]\ar[r]^-{c_X} & X\otimes X^\vee\ar[dr]^{a\otimes b^\vee}\ar[r]^{f\otimes 1} & X\otimes X^\vee\ar[r]^-{d_X} & 1\ar@{=}[d]\\
1\ar[r]^-{c_Y} & Y\otimes Y^\vee\ar[r]^{g\otimes 1} & Y\otimes Y^\vee\ar[r]^-{d_Y} & 1.
}\]

Now we want to analyze the functoriality of $\mathrm{dim}_M(X)$ in $X$. It will be useful to study more generally the functoriality of $\mathrm{tr}_M(f|X)$. Assume given dualizable $X,Y\in M$ and a map $F: X\to Y$ that admits a right adjoint $G: Y\to X$ in $M$. (Recall that the theory of adjoints works in the generality of $2$-categories, or more generally $(\infty,2)$-categories -- it is a morphism in the opposite direction together with unit and counit transformations making some obvious diagrams commute. Adjoints are unique (up to unique isomorphism) when they exist.) Moreover, consider a diagram
\[\xymatrix{
X\ar[r]^F\ar[d]^f & Y\ar[d]^g\\
X\ar[r]^F & Y
}\]
(with an implicit $2$-morphism, which actually does not need to be invertible -- it can be any map in the direction $Ff\to gF$). Then one gets an induced map
\[
\mathrm{tr}_M(f|X)\to \mathrm{tr}_M(fGF|X)\cong \mathrm{tr}_M(FfG|Y)\to \mathrm{tr}_M(gFG|Y)\to \mathrm{tr}_M(g|Y).
\]

In \cite{hoyois2017higher}, it is shown that this gives in particular a symmetric monoidal functor of symmetric monoidal $(\infty,1)$-categories
\[
\mathrm{dim}_M: M^{\mathrm{dual}}\to \mathrm{End}_M(1)
\]
where $M^{\mathrm{dual}}\subset M$ is the sub-$(\infty,1)$-category whose objects are the dualizable objects of $M$, and whose morphisms are those morphisms that admit right adjoints (and whose $2$-morphisms are equivalences). As stated before, we will actually only require the version of this result that ignores $n$-morphisms for $n>2$.

Now for any proper map $\pi: X\to Y$ of complex-analytic spaces, the pullback functor
\[
\pi^\ast: C_Y\to C_X
\]
admits a right adjoint $\pi_\ast$ which satisfies the projection formula, and thus indeed defines a morphism in $M=\operatorname{Mod}_{C_\ast}(\operatorname{Pr}^L)$. In particular, using also Lemma~\ref{lem:symmetricmonoidalCX} and the dualizability (in fact selfduality) of $C_X$, the functor $X\mapsto C_X$ gives a symmetric monoidal functor from the category of compact complex manifolds to $M^{\mathrm{dual}}$. Composing with Hochschild homology, we get a contravariant symmetric monoidal functor
\[
X\mapsto \mathrm{HH}(X)\in \mathcal D(\mathrm{Liq}_p).
\]
(Let us stress again that since many lectures, all liquid vector spaces are over $\mathbb C$, not $\mathbb R$.) Moreover, Theorem~\ref{thm:refinedHKR} shows that, as a symmetric monoidal functor, this is isomorphic to
\[
X\mapsto \mathrm{Hdg}(X),
\]
giving the symmetric monoidal comparison isomorphism $\beta: \mathrm{HH}(X)\cong \mathrm{Hdg}(X)$.

Finally, we can discuss the construction of the refined Chern character
\[
\widetilde{\mathrm{ch}}: K_0(\mathrm{Perf}(X))\to H_0\mathrm{HH}(X)\cong H_0\mathrm{Hdg}(X).
\]
Indeed, given $E\in \mathrm{Perf}(X)$, we get the functor $-\otimes_{\mathbb C} E: C_\ast\to C_X$, whose right adjoint $R\Gamma(X,-\otimes_{\mathcal O_X} E^\vee): C_X\to C_\ast$ is also $C_\ast$-linear. In particular, it induces a map $\mathbb C=\mathrm{HH}(C_\ast)\to \mathrm{HH}(X)$, and we let
\[
\widetilde{\mathrm{ch}}(E)\in H_0\mathrm{HH}(X)
\]
be the image of $1\in \mathbb C$ under this map.

\begin{proposition}\label{prop:cherntriangle} For a cofiber sequence $E'\to E\to E''$ in $\mathrm{Perf}(X)$, one has
\[
\widetilde{\mathrm{ch}}(E) = \widetilde{\mathrm{ch}}(E') + \widetilde{\mathrm{ch}}(E'')\in H_0\mathrm{HH}(X).
\]
\end{proposition}

\begin{proof} Let $S_2(C_X)$ be the $C_\ast$-linear presentable stable $\infty$-category of cofiber sequence $A'\to A\to A''$ in $C_X$. It comes with three projections to $C_X$, given respectively by $A'$, $A$, and $A''$, and the given cofiber sequence induces a functor $C_\ast\to S_2(C_X)$ in $M^{\mathrm{dual}}$. It suffices to see that the projections $S_2(C_X)\to C_X$ given by $A'$ and $A''$ induce an isomorphism
\[
\mathrm{HH}(S_2(C_X))\cong \mathrm{HH}(C_X)\oplus \mathrm{HH}(C_X)
\]
under which the map $\mathrm{HH}(S_2(C_X))\to \mathrm{HH}(C_X)$ induced by the projection to $A$ is given by the sum map. The second claim actually follows from the first, using the natural splitting $C_X\times C_X\to S_2(C_X)$ given by split cofiber sequences.

For the first claim, it suffices to see that the inclusion $C_X\to S_2(C_X)$ given by $A'\mapsto (A'\to A'\to 0)$ and the projection $(A'\to A\to A'')\mapsto A''$ induce a cofiber sequence
\[
\mathrm{HH}(C_X)\to \mathrm{HH}(S_2(C_X))\to \mathrm{HH}(C_X).
\]
But the identity endofunctor $\mathrm{id}_{S_2(C_X)}: S_2(C_X)\to S_2(C_X)$ naturally sits in a cofiber sequence
\[
((A'\to A\to A'')\mapsto (A'\to A'\to 0))\to \mathrm{id}_{S_2(C_X)}\to ((A'\to A\to A'')\mapsto (0\to A''\to A''))
\]
which induces a cofiber sequence after applying $\mathrm{tr}(-|S_2(C_X))$. Using cyclic symmetry of the trace, one can identify the first term with the trace of the identity on $C_X$, and similarly for the last term, and this gives the desired result.
\end{proof}

Thus, we get the desired map
\[
\widetilde{\mathrm{ch}}: K_0(\mathrm{Perf}(X))\to H_0\mathrm{HH}(X)\cong H_0\mathrm{Hdg}(X),
\]
and it follows from the construction that it is a contravariant map of rings.

\begin{proposition}\label{prop:refinedchernextendschern} The restriction of $\widetilde{\mathrm{ch}}$ along $K_0(\mathrm{Vect}(X))\to K_0(\mathrm{Perf}(X))$ agrees with the Chern character $\mathrm{ch}$.
\end{proposition}

\begin{proof} By the splitting principle and Proposition~\ref{prop:cherntriangle}, it suffices to check this on the class of line bundles. One can now make an argument using the classifying stack $[\ast/\mathbb G_m]$, and computing its Hodge cohomology to be $\mathbb C[[c_1]]$; but this really requires the $\infty$-categorical functoriality of $\widetilde{\mathrm{ch}}$. We give a more ad hoc argument below.

We first observe that if $Z\subset X$ is a closed submanifold with coherent ideal sheaf $I\subset \mathcal O_X$, then one can look at the full subcategory $C_{X-Z}$ of $C_X$ of all objects whose reduction modulo $I$ vanishes; if $Z$ would be a Cartier divisor, this would amount to algebraically inverting $I$. Then $C_{X-Z}$ is still dualizable in $\operatorname{Mod}_{C_\ast}(\operatorname{Pr}^L)$, and still satisfies a Hochschild--Kostant--Rosenberg theorem. For any vector bundle $F$ on $X$, the map $R\Gamma(X,F)\to R\Gamma(X-Z,F)$ is an isomorphism in degrees less than $\mathrm{codim}(Z\subset X)-1$.

Now given a line bundle $L$ on $X$, consider the projective space $\tilde{X}=\mathbb P_X(L^N\oplus \mathcal O_X^N)$ for some large $N$. This contains $\mathbb P_X(L^N)\cong X\times \mathbb P^{N-1}$ and $\mathbb P_X(\mathcal O_X^N)\cong X\times \mathbb P^{N-1}$; let $Z\subset \tilde{X}$ be their disjoint union. Away from $Z$, there are two projection maps to $\mathbb P^{N-1}$, and the pullback of $L$ to $\tilde{X}-Z$ (i.e.~to $C_{\tilde{X}-Z}$) is isomorphic to the ratio of the pullbacks of the corresponding line bundles $\mathcal O_{\mathbb P^{N-1}}(1)$. Also, the pullback map $H_0\mathrm{Hdg}(X)\to H_0\mathrm{Hdg}(\tilde{X}-Z)$ is injective if $N$ is chosen large enough. All of this reduces us to the case of the line bundle $\mathcal O_{\mathbb P^{N-1}}(1)$.

Now $H_0\mathrm{Hdg}(\mathbb P^{N-1})\cong \mathbb C[c_1]/c_1^N$ where $c_1=c_1(\mathcal O(1))$. It follows that $\widetilde{\mathrm{ch}}(\mathcal O(1))$ is given by some truncated power series in $c_1(\mathcal O(1))$. In fact, these truncated power series are compatible under pullback, giving a formula $\widetilde{\mathrm{ch}}(\mathcal O(1))=f(c_1(\mathcal O(1)))$ for some $f\in \mathbb C[[t]]$. Computing for $\mathbb P^1$, one checks that $f(t)\equiv 1+t$ modulo $t^2$. On the other hand, $\widetilde{\mathrm{ch}}$ is multiplicative, and using this on $(\mathbb P^{N-1})^2$ along with $c_1(\mathcal L\otimes \mathcal L')=c_1(\mathcal L)+c_1(\mathcal L')$, we get
\[
f(t+t')=f(t)f(t)\in \mathbb C[[t,t']].
\]
This means that $f(t)=\mathrm{exp}(\lambda t)$ for some $\lambda\in \mathbb C$, which (by $f(t)\equiv 1+t$ modulo $t^2$) must be given by $\lambda=1$.
\end{proof}

This finishes the first part. It remains to prove the following theorem.

\begin{theorem}\label{thm:GHRRfinal} For any map $f: X\to Y$ of compact complex manifolds, the diagram
\[\xymatrix{
K_0(\mathrm{Perf}(X))\ar[rr]^{\mathrm{td}(T_X)\cdot \widetilde{\mathrm{ch}}}\ar[d]^{f_\ast} && H_0\mathrm{Hdg}(X)\ar[d]^{f_\ast}\\
K_0(\mathrm{Perf}(Y))\ar[rr]^{\mathrm{td}(T_Y)\cdot \widetilde{\mathrm{ch}}} && H_0\mathrm{Hdg}(Y)
}\]
commutes.
\end{theorem}

In fact, given the definition of $\widetilde{\mathrm{ch}}$, this naturally splits into two diagrams. Note that any map $f: X\to Y$ of compact complex manifolds is a local complete intersection (by factoring it via the graph) and hence the right adjoint $f^!: C_Y\to C_X$ to $f_\ast: C_X\to C_Y$ is given by $f^! \mathcal O_Y\otimes_{\mathcal O_X} f^\ast$, and in particular is $C_\ast$-linear. Thus, $f_\ast$ induces a covariant functoriality on Hochschild homology.

\begin{proposition}\label{prop:GHRRformal} For any map $f: X\to Y$ of compact complex manifolds, the diagram
\[\xymatrix{
K_0(\mathrm{Perf}(X))\ar[r]^{\widetilde{\mathrm{ch}}}\ar[d]^{f_\ast} & H_0\mathrm{HH}(X)\ar[d]^{f_\ast}\\
K_0(\mathrm{Perf}(Y))\ar[r]^{\widetilde{\mathrm{ch}}} & H_0\mathrm{HH}(Y)
}\]
commutes.
\end{proposition}

\begin{proof} This is a formal consequence of the functoriality of Hochschild homology. Indeed, take any $E\in \mathrm{Perf}(X)$. Going via the upper right corner, one gets the element of $H_0\mathrm{HH}(Y)$ corresponding to the effect on Hochschild homology of the functor
\[
C_\ast\xrightarrow{-\otimes_{\mathbb C} E} C_X\xrightarrow{f_\ast} C_Y.
\]
But by the projection formula, this functor agrees with
\[
C_\ast\xrightarrow{-\otimes_{\mathbb C} f_\ast E} C_Y,
\]
whose effect on Hochschild homology gives the element of $H_0\mathrm{HH}(Y)$ one obtains when going via the left lower corner.
\end{proof}

Thus, it remains to prove the following result.

\begin{proposition}\label{prop:GHRRtodd} For any map $f: X\to Y$ of compact complex manifolds, the diagram
\[\xymatrix{
\mathrm{HH}(X)\ar[r]^{\mathrm{td}(T_X)\cdot} \ar[d]^{f_\ast} & \mathrm{Hdg}(X)\ar[d]^{f_\ast}\\
\mathrm{HH}(Y)\ar[r]^{\mathrm{td}(T_Y)\cdot} & \mathrm{Hdg}(Y)
}\]
commutes (up to equivalence), where the upper maps are the HKR isomorphisms multiplied with the Todd class.
\end{proposition}

To prove this, we will actually pass to an abstraction, as follows.

\begin{definition}\label{def:abstractcohomology} Consider the category $\mathrm{Man}$ of compact complex manifolds, and let $(S,\otimes)$ be a symmetric monoidal $1$-category.
\begin{enumerate}
\item An abstract $S$-valued cohomology theory on $\mathrm{Man}$ is a contravariant symmetric monoidal functor
\[
H^\ast: \mathrm{Man}^{\mathrm{op}}\to S
\]
that takes disjoint unions to products, and about which we ask the following weak form of the projective bundle formula: For any $X\in \mathrm{Man}$ and any vector bundle $E$ of positive rank on $X$, the pullback map
\[
H^\ast(X)\to H^\ast(\mathbb P_X(E))
\]
is a monomorphism.
\item Let $H^\ast$ be an abstract $S$-valued cohomology theory on $\mathrm{Man}$. A pushforward structure on $H^\ast$ is a covariant functor
\[
H_\ast: \mathrm{Man}\to S
\]
with an identification
\[
H_\ast|_{\mathrm{Man}^{\simeq}}\cong H^\ast|_{\mathrm{Man}^{\simeq}}
\]
to the subcategory $\mathrm{Man}^{\simeq}\subset \mathrm{Man}$ of all objects, but only isomorphisms as morphisms, subject to the following conditions:
\begin{enumerate}
\item[{\rm (a)}] For any transverse pullback square
\[\xymatrix{
X'\ar[r]^{g'}\ar[d]^{f'} & X\ar[d]^f\\
Y'\ar[r]^g & Y,
}\]
one has
\[
g^\ast f_\ast = f'_\ast g^{\prime\ast}: H^\ast(X)\to H^\ast(Y'),
\]
where we denote $g^\ast = H^\ast(g)$ and $f_\ast = H_\ast(f)$, etc.
\item[{\rm (b)}] For any map $f: X\to Y$ and any $Z$, the map
\[
(f\times \mathrm{id}_Z)_\ast: H^\ast(X\times Z)=H^\ast(X)\otimes H^\ast(Z)\to H^\ast(Y\times Z)=H^\ast(Y)\otimes H^\ast(Z)
\]
agrees with $f_\ast\otimes H^\ast(Z)$.
\item[{\rm (c)}] Moreover, we ask a weak form of the projective bundle formula: for any $X$ and line bundle $L$ on $X$, the map
\[
H^\ast(\mathbb P_X(L\oplus 1))\xrightarrow{(\pi_\ast,\infty^\ast)} H^\ast(X)\times H^\ast(X)
\]
is a monomorphism.
\end{enumerate}
\end{enumerate}
\end{definition}

\begin{example} We will take $S$ to be $D(\mathrm{Liq}_p)$, the derived $1$-category of $p$-liquid $\mathbb C$-vector spaces. (In fact, everything will take place in $D(\mathbb C)\subset D(\mathrm{Liq}_p)$.) We have two pushforward structures on $H^\ast = \mathrm{HH}\cong \mathrm{Hdg}$: The first by the covariant functoriality of Hochschild homology, and the second as constructed in the last lecture, using Serre duality.
\end{example}

\begin{remark} If $(H^\ast,H_\ast)$ is any abstract cohomology theory with pushforwards, one gets a version of the projection formula. Namely, for any map $f: X\to Y$ and classes $x\in H^\ast(X)$ and $y\in H^\ast(Y)$ (where $x\in H^\ast(X)$ is meant to be read as $x: 1_S\to H^\ast(X)$), one has
\[
f_\ast x\cdot y = f_\ast(x\cdot f^\ast y).
\]
The product used here is defined as the composite of pullback to the diagonal and the exterior product (coming from the symmetric monoidal structure of $H^\ast$). Thus
\[
f_\ast x\cdot y = \Delta_Y^\ast (f_\ast x\otimes y) = \Delta_Y^\ast (f\times \mathrm{id}_Y)_\ast (x\otimes y)
\]
using condition (b). On the other hand
\[
f_\ast(x\cdot f^\ast y) = f_\ast \Delta_X^\ast(x\otimes f^\ast y)=f_\ast \Delta_X^\ast (\mathrm{id}_X\times f)^\ast(x\otimes y)
\]
by the symmetric monoidal nature of $H^\ast$ in the second equation. Finally, the result follows from the transverse pullback square
\[\xymatrix{
X\ar[r]^{(\mathrm{id}_X,f)} \ar[d]^f & X\times Y\ar[d]^{f\times\mathrm{id}_Y}\\
Y\ar[r]^{\Delta_Y} & Y\times Y.
}\]
\end{remark}

\begin{remark} Given any abstract cohomology theory with pushforward $(H^\ast,H_\ast)$, one can define some rudimentary theory of characteristic classes. We will only require Euler classes of line bundles: If $L$ is a line bundle on $X$, let $\mathbb P_X(L^\vee\oplus 1)$ be the projective space over $X$ that compactifies the $\mathbb A^1$-bundle corresponding to $L$ into a $\mathbb P^1$-bundle. It comes with a zero section $0: X\to \mathbb P_X(L^\vee\oplus 1)$, and we set
\[
e(L) = 0^\ast 0_\ast (1)\in H^\ast(X).
\]
\end{remark}

\begin{example} In Hodge cohomology, the Euler class of a line bundle is given by its first Chern class $c_1$. In Hochschild homology, the Euler class is given by $1-e^{-c_1}$. Indeed, using the pushforward compatibility of $0_\ast$ with the Chern character, $0_\ast(1)$ comes from the perfect complex that is given by the zero section, which is resolved by the structure sheaf and the ideal sheaf of the zero section. The first corresponds to $1$, the second corresponds to $\mathcal O(-1)$, whose Chern character is $e^{-c_1}$.
\end{example}

Note that in the previous example, one precisely sees a discrepancy of $c_1$ versus $1-e^{-c_1}$, getting us close to the definition of the Todd class.

\begin{remark} Let $(H^\ast,H_\ast)$ be an abstract cohomology theory with a pushforward structure. Assume that $t: K_0(\mathrm{Vect}(X))\to H^\ast(X)^\times$ is a contravariant natural transformation from the additive group $K_0(\mathrm{Vect}(X))$ to the multiplicative group of units of the ring $H^\ast(X)$ (i.e.~maps $1_S\to H^\ast(X)$ that admit a map $1_S\to H^\ast(X)$ that is a multiplicative inverse under $H^\ast(X)\otimes H^\ast(X)\cong H^\ast(X\times X)\xrightarrow{\Delta_X^\ast} H^\ast(X)$). Then one can define $t$-twisted pushforward maps $f_\ast^t = t(T_Y)^{-1} f_\ast t(T_X)$ for $f: X\to Y$, and it is easy to see that this is still a pushforward structure.
\end{remark}

The previous discussion then shows that it is enough to prove the following abstract result.

\begin{theorem}\label{thm:abstractGHRR} Let $(H^\ast,H_\ast)$ and $(H^{\prime\ast},H^\prime_\ast)$ be two abstract $S$-valued cohomology theories with pushforward structures, and let $\alpha: H^\ast\to H^{\prime\ast}$ be a map of abstract $S$-valued cohomology theories. Assume that $\alpha$ is compatible with Euler classes of line bundles, i.e.~for any $X$ and any line bundle $L$ on $X$, one has $\alpha(e(L))=e'(L)$.

Then $\alpha$ commutes with all pushforward maps.
\end{theorem}

\begin{proof} First, we note that given any such $(H^\ast,H_\ast)$, there is a natural self-duality of $H^\ast(X)\in S$, given by the coevaluation
\[
1_S=H^\ast(\ast)\xrightarrow{\pi_X^\ast} H^\ast(X)\xrightarrow{\Delta_{X\ast}} H^\ast(X\times X)\cong H^\ast(X)\otimes H^\ast(X)
\]
and evaluation
\[
H^\ast(X)\otimes H^\ast(X)=H^\ast(X\times X)\xrightarrow{\Delta_X^\ast} H^\ast(X)\xrightarrow{\pi_{X\ast}} H^\ast(\ast)=1_S.
\]
Moreover, with respect to this self-duality (which depends on the pushforward structure), the maps $f_\ast: H^\ast(X)\to H^\ast(Y)$ are the duals of $f^\ast: H^\ast(Y)\to H^\ast(X)$. In particular, to show that $\alpha$ commutes with the pushforward maps, it is enough to show that it commutes with the self-dualities, for which in turn it is enough to show that it is compatible with the coevaluation maps (as these determine the self-duality). Thus, it suffices to show that
\[
\alpha(\Delta_{X\ast}(1))=\Delta_{X\ast}'(1).
\]
As $H^\ast$ takes disjoint unions to products, we can assume that $X$ is connected.

We will now show more generally that for any closed immersion $i: Z\hookrightarrow X$ of compact complex manifolds (of codimension $c$), one has $\alpha(i_\ast(1))=i'_\ast(1)$. Consider first the case where $Z$ is of codimension $c=1$. In that case, we can define the line bundle $L=\mathcal O_X(Z)$, which comes with a natural section $s: X\to L$ vanishing exactly along $Z$, and $0: X\to L$. In particular, we get a transverse pullback square
\[\xymatrix{
Z\ar[r]^i\ar[d]^i & X\ar[d]^s\\
X\ar[r]^-0 & \mathbb P_X(L^\vee\oplus \mathcal O_X).
}\]
Now note that $s_\ast=0_\ast: H^\ast(X)\to H^\ast(\mathbb P_X(L^\vee\oplus \mathcal O_X))$: By the weak form of the projective bundle formula we asked for, it suffices to check this after pushing forward to $X$ (where it follows as $s$ and $0$ are both sections) and pulling back to $\infty$ (where it follows from the transverse pullback formula, and the emptiness of the intersection with $\infty$). Thus,
\[
i_\ast(1) = 0^\ast s_\ast(1) = 0^\ast 0_\ast(1)=e(L),
\]
and the claim follows from compatibility with Euler classes.

Note that if we know $\alpha(i_\ast(1))=i'_\ast(1)$ for some $i: Z\subset X$, then in fact for all $z\in H^\ast(Z)$ in the image of $i^\ast: H^\ast(X)\to H^\ast(Z)$, one has $\alpha(i_\ast(z))=i'_\ast(\alpha(z))$. Indeed, if $z=i^\ast(x)$, then
\[
\alpha(i_\ast(z))=\alpha(i_\ast(i^\ast(x))) = \alpha(i_\ast(1)\cdot x) = \alpha(i_\ast(1))\cdot \alpha(x) = i'_\ast(1)\cdot \alpha(x)=i'_\ast i'^\ast(\alpha(x)) = i'_\ast \alpha(z)
\]
using the compatibility of $\alpha$ with pullback and multiplication.

Next, consider the situation of a vector bundle $E$ of rank $r$ on a compact complex manifold $X$, and the embedding of the zero section $0: X\hookrightarrow \mathbb P_X(E^\vee\oplus \mathcal O_X)$. We want to show that $\alpha(0_\ast(1))=0'_\ast(1)$. By the splitting principle, we can assume that $E$ admits a complete flag
\[
0=E_0\subset E_1\subset \ldots\subset E_r=E.
\]
In that case, one can factor
\[
X=\mathbb P_X(\mathcal O_X)\subset \mathbb P_X(E_1^\vee\oplus \mathcal O_X)\subset \ldots\subset \mathbb P_X(E_r^\vee\oplus \mathcal O_X)=\mathbb P_X(E^\vee\oplus\mathcal O_X)
\]
into a series of codimension $1$ embeddings. Moreover, the corresponding line bundles
\[
\mathcal O_{\mathbb P_X(E_i^\vee\oplus \mathcal O_X)}(\mathbb P_X(E_{i-1}^\vee\oplus \mathcal O_X))
\]
on $\mathbb P_X(E_i^\vee\oplus \mathcal O_X)$ lift to $\mathbb P_X(E^\vee\oplus \mathcal O_X)$ (they are twists of $\mathcal O(1)$ by the pullback of the line bundle $E_i/E_{i-1}$ on $X$). Thus, the result can be inductively reduced to the case of codimension $1$.

Finally, we consider the general case $i: Z\hookrightarrow X$. The following argument is inspired by an argument in \cite[Chapter 15]{Fulton}. Let $D$ be the blow-up of $X\times \mathbb P^1$ along $Z\times \{0\}$. Then there is a projection $D\to \mathbb P^1$ with a map $\tilde{i}: Z\times \mathbb P^1\hookrightarrow D$ over $\mathbb P^1$, which after pullback to $\mathbb P^1\setminus \{0\}$ becomes isomorphic to $(Z\hookrightarrow X)\times (\mathbb P^1\setminus \{0\})$, while after pullback to $0\in \mathbb P^1$, we get
\[
Z\hookrightarrow \mathbb P_Z(N^\vee\oplus 1)\hookrightarrow D\times_{\mathbb P^1} \{0\}
\]
where $N$ is the normal bundle of $Z$ in $X$, and the second map is an inclusion of an irreducible component, where the other irreducible component does not meet $Z$.

Let $i_{Z,1}: Z=Z\times \{1\}\hookrightarrow Z\times \mathbb P^1$ and $i_{Z,0}: Z=Z\times \{0\}\hookrightarrow Z\times \mathbb P^1$ be the inclusions. Then $i_{Z,0\ast}(1) = i_{Z,1\ast}(1) = e(\mathcal O(1))$. Also consider $i_{X,1}: X=X\times\{1\}\hookrightarrow D$ and $i_0: \mathbb P_Z(N^\vee\oplus 1)\hookrightarrow D$ and $0: Z\hookrightarrow \mathbb P_Z(N^\vee\oplus 1)$. Then
\[
i_{X,1\ast} i_\ast(1) = \tilde{i}_\ast i_{Z,1\ast}(1) = \tilde{i}_\ast i_{Z,0\ast}(1)=i_{0\ast} 0_\ast (1).
\]
In particular,
\[
\alpha(i_{X,1\ast} i_\ast(1)) = \alpha(i_{0\ast} 0_\ast(1)) = i'_{0\ast} \alpha(0_\ast(1)) = i'_{0\ast} 0'_\ast(1)
\]
by the codimension $1$ case and embedding of the zero section case handled before (and noting that $0_\ast(1)$ is in the image of $i_0^\ast$, as it is in fact $i_0^\ast$ applied to $\tilde{i}_\ast(1)$).

Now we observe that also
\[
i_{X,1\ast}' \alpha(i_\ast(1)) = \alpha(i_{X,1\ast} i_\ast(1))
\]
using the codimension $1$ case for $i_{X,1}$ (and again noting that $i_\ast(1)$ is in the image of $i_{X,1}^\ast$, again when applied to $\tilde{i}_\ast(1)$), and
\[
i_{X,1\ast}'i'_\ast(1) = i'_{0\ast} 0'_\ast(1)
\]
as we proved the similar assertion for the unprimed version before. In total, we get
\[
i_{X,1\ast}' \alpha(i_\ast(1)) = \alpha(i_{X,1\ast} i_\ast(1)) = i'_{0\ast} 0'_\ast(1) = i_{X,1\ast}'i'_\ast(1).
\]
But $i_{X,1\ast}'$ is a (split) monomorphism (as $i_{X,1}: X\hookrightarrow D$ admits a splitting), so we get $\alpha(i_\ast(1))=i'_\ast(1)$, as desired.
\end{proof}

\bibliographystyle{amsalpha}

\bibliography{Complex}

\newcommand{\etalchar}[1]{$^{#1}$}
\providecommand{\bysame}{\leavevmode\hbox to3em{\hrulefill}\thinspace}
\providecommand{\MR}{\relax\ifhmode\unskip\space\fi MR }
\providecommand{\MRhref}[2]{%
  \href{http://www.ams.org/mathscinet-getitem?mr=#1}{#2}
}
\providecommand{\href}[2]{#2}
\begin{thebibliography}{BBBK18}

\bibitem[Ami64]{AmirSepInj}
D.~Amir, \emph{Projections onto continuous functions spaces}, Proc. Am. Math.
  Soc. \textbf{15} (1964), 396--402.

\bibitem[And21]{Andreychev}
G.~Andreychev, \emph{Pseudocoherent and {P}erfect {C}omplexes and {V}ector
  {B}undles on {A}nalytic {A}dic {S}paces}, arXiv:2105.12591, 2021.

\bibitem[ASC{\etalchar{+}}16]{SepInjBanach}
A.~Avil\'{e}s, F.~C. S\'{a}nchez, J.~M.~F. Castillo, M.~Gonz\'{a}lez, and
  Y.~Moreno, \emph{Separably injective {B}anach spaces}, Lecture Notes in
  Mathematics, vol. 2132, Springer, [Cham], 2016.

\bibitem[BBBK18]{BambozziBenBassatKremnizer}
F.~Bambozzi, O.~Ben-Bassat, and K.~Kremnizer, \emph{Stein domains in {Banach}
  algebraic geometry}, J. Funct. Anal. \textbf{274} (2018), no.~7, 1865--1927.

\bibitem[BBK17]{BenBassatKremnizer}
O.~Ben-Bassat and K.~Kremnizer, \emph{Non-{A}rchimedean analytic geometry as
  relative algebraic geometry}, Ann. Fac. Sci. Toulouse Math. (6) \textbf{26}
  (2017), no.~1, 49--126.

\bibitem[BBT18]{BakkerBrunebarbeTsimerman}
B.~Bakker, Y.~Brunebarbe, and J.~Tsimerman, \emph{o-minimal {GAGA} and a
  conjecture of {G}riffiths}, arXiv:1811.12230, 2018.

\bibitem[Ber90]{BerkovichSpectral}
V.~G. Berkovich, \emph{Spectral theory and analytic geometry over
  non-{Archimedean} fields}, Math. Surv. Monogr., vol.~33, Providence, RI:
  American Mathematical Society, 1990.

\bibitem[BF11]{BalmerFavi}
P.~Balmer and G.~Favi, \emph{Generalized tensor idempotents and the telescope
  conjecture}, Proc. Lond. Math. Soc. (3) \textbf{102} (2011), no.~6,
  1161--1185.

\bibitem[Bin78]{Bingener}
J.~Bingener, \emph{\"{U}ber formale komplexe {R}\"{a}ume}, Manuscripta Math.
  \textbf{24} (1978), no.~3, 253--293.

\bibitem[BKS20]{BalmerKrauseStevenson}
P.~Balmer, H.~Krause, and G.~Stevenson, \emph{The frame of smashing
  tensor-ideals}, Math. Proc. Cambridge Philos. Soc. \textbf{168} (2020),
  no.~2, 323--343.

\bibitem[Cem84]{Cembranos}
P.~Cembranos, \emph{{$C(K,\,E)$} contains a complemented copy of {$c_{0}$}},
  Proc. Amer. Math. Soc. \textbf{91} (1984), no.~4, 556--558.

\bibitem[Fri67]{Frisch}
J.~Frisch, \emph{Points de platitude d'un morphisme d'espaces analytiques
  complexes}, Invent. Math. \textbf{4} (1967), 118--138.

\bibitem[Ful98]{Fulton}
W.~Fulton, \emph{Intersection theory}, second ed., Ergebnisse der Mathematik
  und ihrer Grenzgebiete. 3. Folge. A Series of Modern Surveys in Mathematics
  [Results in Mathematics and Related Areas. 3rd Series. A Series of Modern
  Surveys in Mathematics], vol.~2, Springer-Verlag, Berlin, 1998.

\bibitem[GK00]{grosse2000rigid}
E.~Grosse-Kl{\"o}nne, \emph{Rigid analytic spaces with overconvergent structure
  sheaf}, Journal f\"ur die reine und angewandte Mathematik \textbf{519}
  (2000), 73--95.

\bibitem[Gle58]{Gleason}
A.~M. Gleason, \emph{Projective topological spaces}, Illinois J. Math.
  \textbf{2} (1958), 482--489.

\bibitem[Gro55]{GrothendieckTensor}
A.~Grothendieck, \emph{Produits tensoriels topologiques et espaces
  nucl\'{e}aires}, Mem. Amer. Math. Soc. \textbf{No. 16} (1955), Chapter 1: 196
  pp.; Chapter 2: 140.

\bibitem[Gro58]{grothendieck1958theorie}
Alexander Grothendieck, \emph{La th{\'e}orie des classes de chern}, Bulletin de
  la soci{\'e}t{\'e} math{\'e}matique de France \textbf{86} (1958), 137--154.

\bibitem[Gro61]{grothendieck1961elements}
\bysame, \emph{{\'E}l{\'e}ments de g{\'e}om{\'e}trie alg{\'e}brique: {III}.
  {\'e}tude cohomologique des faisceaux coh{\'e}rents, premi{\`e}re partie},
  Publications Math{\'e}matiques de l'IH{\'E}S \textbf{11} (1961), 5--167.

\bibitem[HBS66]{hirzebruch1966topological}
Friedrich Hirzebruch, Armand Borel, and RLE Schwarzenberger, \emph{Topological
  methods in algebraic geometry}, vol. 175, Springer Berlin-Heidelberg-New
  York, 1966.

\bibitem[Hou73]{Houzel}
C.~Houzel, \emph{Espaces analytiques relatifs et th\'eor\`eme de finitude},
  Math. Ann. \textbf{205} (1973), 13--54.

\bibitem[HSS17]{hoyois2017higher}
Marc Hoyois, Sarah Scherotzke, and Nicolo Sibilla, \emph{Higher traces,
  noncommutative motives, and the categorified chern character}, Advances in
  Mathematics \textbf{309} (2017), 97--154.

\bibitem[Hub93]{HuberContVal}
R.~Huber, \emph{Continuous valuations}, Math. Z. \textbf{212} (1993), no.~3,
  455--477.

\bibitem[Kal81]{KaltonConvexityType}
N.~J. Kalton, \emph{Convexity, type and the three space problem}, Studia Math.
  \textbf{69} (1980/81), no.~3, 247--287.

\bibitem[Lan77]{LangmannFrisch}
K.~Langmann, \emph{Zum {S}atz von {J}. {F}risch: ``{P}oints de platitude d'un
  morphisme d'espaces analytiques complexes'' ({I}nvent. {M}ath. {\bf 4}
  (1967), 118--138)}, Math. Ann. \textbf{229} (1977), no.~2, 141--142.

\bibitem[Loj64]{Lojasiewicz}
S.~Lojasiewicz, \emph{Triangulation of semi-analytic sets}, Ann. Scuola Norm.
  Sup. Pisa Cl. Sci. (3) \textbf{18} (1964), 449--474. \MR{173265}

\bibitem[Lur09]{LurieHTT}
J.~Lurie, \emph{Higher topos theory}, Annals of Mathematics Studies, vol. 170,
  Princeton University Press, Princeton, NJ, 2009.

\bibitem[Lur17]{lurie2017higher}
\bysame, \emph{Higher {A}lgebra},
  \url{https://www.math.ias.edu/~lurie/papers/HA.pdf}, 2017.

\bibitem[Mao21]{MaoPD}
Z.~Mao, \emph{Revisiting derived crystalline cohomology}, arXiv:2107.02921,
  2021.

\bibitem[Mat80]{Matsumura}
H.~Matsumura, \emph{Commutative algebra}, second ed., Mathematics Lecture Note
  Series, vol.~56, Benjamin/Cummings Publishing Co., Inc., Reading, Mass.,
  1980.

\bibitem[Oka50]{OkaVII}
K.~Oka, \emph{Sur les fonctions analytiques de plusieurs variables. {VII}.
  {S}ur quelques notions arithm\'{e}tiques}, Bull. Soc. Math. France
  \textbf{78} (1950), 1--27.

\bibitem[Ost17]{Ostrowski}
A.~Ostrowski, \emph{{\"U}ber einige {L{\"o}sungen} der {Funktionalgleichung}
  ${{\varphi}}(x){{\cdot}} {{\varphi}}(y)= {{\varphi}}(xy)$.}, Acta Math.
  \textbf{41} (1917), 271--284.

\bibitem[Rem84]{Remmert}
R.~Remmert, \emph{Funktionentheorie. {I}}, Grundwissen Mathematik [Basic
  Knowledge in Mathematics], vol.~5, Springer-Verlag, Berlin, 1984.

\bibitem[Sch13]{schatten2013norm}
R.~Schatten, \emph{Norm ideals of completely continuous operators}, vol.~27,
  Springer-Verlag, 2013.

\bibitem[Sch19]{Condensed}
P.~Scholze, \emph{Lectures on {C}ondensed {M}athematics},
  \url{people.mpim-bonn.mpg.de/scholze/Condensed.pdf}, 2019.

\bibitem[Sch20]{Analytic}
\bysame, \emph{Lectures on {A}nalytic {G}eometry},
  \url{people.mpim-bonn.mpg.de/scholze/Analytic.pdf}, 2020.

\bibitem[Sch26]{Gestalten}
\bysame, \emph{Geometry and {H}igher {C}ategory {T}heory},
  \url{https://people.mpim-bonn.mpg.de/scholze/Gestalten.pdf}, 2026.

\bibitem[Siu69]{SiuNoetherian}
Y.-T. Siu, \emph{Noetherianness of rings of holomorphic functions on {S}tein
  compact subsets}, Proc. Amer. Math. Soc. \textbf{21} (1969), 483--489.

\bibitem[SS72]{SchejaStorch}
G.~Scheja and U.~Storch, \emph{Differentielle {E}igenschaften der
  {L}okalisierungen analytischer {A}lgebren}, Math. Ann. \textbf{197} (1972),
  137--170.

\bibitem[Tay72]{Taylor}
J.~L. Taylor, \emph{A general framework for a multi-operator functional
  calculus}, Advances in Math. \textbf{9} (1972), 183--252.

\bibitem[Ull99]{UllrichWeierstrassPreparation}
P.~Ullrich, \emph{Division mit {R}est in {B}anach-{A}lgebren konvergenter
  {P}otenzreihen}, Arch. Math. (Basel) \textbf{72} (1999), no.~4, 289--292.

\bibitem[Wae06]{waelbroeck2006topological}
L.~Waelbroeck, \emph{Topological vector spaces and algebras}, vol. 230,
  Springer, 2006.

\end{thebibliography}

\end{document}